\documentclass[10pt]{amsart}
\usepackage{latexsym, amssymb, amsfonts, amscd, multicol}

\theoremstyle{plain}
\newtheorem{Thm}{Theorem}[subsection]
\newtheorem{Cor}[Thm]{Corollary}

\newtheorem{Prop}[Thm]{Proposition}
\newtheorem{Lem}[Thm]{Lemma}
\newtheorem{Cl}[Thm]{Claim}

\theoremstyle{definition}
\newtheorem{Emp'}{}[section]
\newtheorem{Def}[Thm]{Definition}

\newtheorem{Emp}[Thm]{}

\newtheorem{Not}[Thm]{Notation}
\newtheorem{Q}[Thm]{Question}
\newcommand{\fa}{\mathfrak a}
\numberwithin{equation}{section}

\newcommand{\qlbar}{\overline{\B{Q}}_{\ell}}
\newcommand{\fqbar}{\overline{\B{F}}_q}
\newcommand{\om}{\omega}
\newcommand{\cycl}{\langle u\rangle}
\newcommand{\La}{\Lambda}
\newcommand{\fg}{\frak{g}}

\newcommand{\ov}{\overline}

\newcommand{\fq}{\B{F}_q}

\newcommand{\B}[1]{\mathbb#1}

\newcommand{\C}[1]{\mathcal#1}

\newcommand{\gr}{\operatorname{gr}}

\newcommand{\isom}{\overset {\thicksim}{\to}}

\newcommand{\Om}{\Omega}
\newcommand{\Mod}{\on{Mod}}
\newcommand{\form}[1]{(\ref{Eq:#1})}
\newcommand{\ka}{\kappa}
\newcommand{\nr}{\operatorname{nr}}
\newcommand{\Ind}{\operatorname{Ind}}

\newcommand{\si}{\sigma}

\newcommand{\lra}{\longrightarrow}
\newcommand{\lla}{\longleftarrow}
\newcommand{\hra}{\hookrightarrow}
\newcommand{\wt}{\widetilde}
\newcommand{\wh}{\widehat}
\newcommand{\Gm}{\Gamma}
\newcommand{\G}{\bo{G}}
\newcommand{\bo}{\mathbf}

\newcommand{\gm}{\gamma}

\newcommand{\Dt}{\Delta}
\newcommand{\bs}{\backslash}

\newcommand{\lan}{\langle}
\newcommand{\ran}{\rangle}

\newcommand{\m}{^{\times}}

\newcommand{\al}{\alpha}

\newcommand{\la}{\lambda}

\newcommand{\rs}[1]{Section~\ref{S:#1}}
\newcommand{\rl}[1]{Lemma~\ref{L:#1}}
\newcommand{\rn}[1]{Notation~\ref{N:#1}}
\newcommand{\rcl}[1]{Claim~\ref{C:#1}}
\newcommand{\rp}[1]{Proposition~\ref{P:#1}}

\newcommand{\re}[1]{\ref{E:#1}}
\newcommand{\rco}[1]{Corollary~\ref{C:#1}}

\newcommand{\rt}[1] {Theorem~\ref{T:#1}}

\newcommand{\sm}{\smallsetminus}

\newcommand{\on}{\operatorname}

\newcommand{\inv}{\operatorname{inv}}
\newcommand{\I}{\mathbf I}
\renewcommand{\P}{\mathbf P}
\renewcommand{\Q}{\mathbf Q}
\newcommand{\T}{\mathbf T}
\newcommand{\ind}{\operatorname{ind}}
\newcommand{\sep}{\operatorname{sep}}
\newcommand{\can}{\operatorname{can}}
\newcommand{\pr}{\operatorname{pr}}
\newcommand{\ev}{\operatorname{ev}}
\newcommand{\Out}{\operatorname{Out}}
\newcommand{\Ker}{\operatorname{Ker}}
\newcommand{\Coker}{\operatorname{Coker}}

\newcommand{\val}{\operatorname{val}}

\newcommand{\Int}{\operatorname{Int}}
\newcommand{\im}{\operatorname{Im}}

\newcommand{\Spec}{\operatorname{Spec}}
\newcommand{\Aut}{\operatorname{Aut}}

\newcommand{\Ad}{\operatorname{Ad}}

\newcommand{\Gal}{\operatorname{Gal}}
\newcommand{\Tr}{\operatorname{Tr}}
\newcommand{\rss}{\operatorname{rss}}
\newcommand{\Fr}{\operatorname{Fr}}

\newcommand{\st}{\operatorname{st}}
\newcommand{\red}{\operatorname{red}}
\newcommand{\der}{\operatorname{der}}
\newcommand{\tor}{\operatorname{tor}}
\renewcommand{\sp}{\operatorname{sp}}

\newcommand{\Lie}{\operatorname{Lie}}
\newcommand{\reg}{\operatorname{reg}}
\newcommand{\ad}{\operatorname{ad}}
\renewcommand{\sc}{\operatorname{sc}}
\newcommand{\Id}{\operatorname{Id}}
\newcommand{\Fl}{\operatorname{Fl}}
\newcommand{\Rep}{\operatorname{Rep}}
\newcommand{\Proj}{\operatorname{Proj}}

\newcommand{\End}{\operatorname{End}}

\newcommand{\rk}{\operatorname{rk}}
\newcommand{\Hom}{\operatorname{Hom}}

\newcommand{\gen}{\operatorname{gen}}
\newcommand{\adj}{\operatorname{adj}}
\newcommand{\colim}{\operatorname{colim}}

\newcommand{\e}{\par \noindent}
\begin{document}
%Topmatter

\title[Affine Springer fibers and depth zero $L$-packets]%
{Affine Springer fibers and depth zero $L$-packets}

\author{Roman Bezrukavnikov}
\address{Department of Mathematics\\
Massachusetts Institute of Technology\\
77 Massachusetts Avenue\\
Cambridge, MA 02139, USA} \email{bezrukav@math.mit.edu}
\date{\today}
\author{Yakov Varshavsky}
\address{Einstein Institute of Mathematics\\
Edmond J. Safra Campus\\
The Hebrew University of Jerusalem\\
Givat Ram, Jerusalem, 9190401, Israel}
\email{yakov.varshavsky@mail.huji.ac.il}

\thanks{This research was partially supported by
the BSF grant 2020189. The research of R.B. was also partially supported by NSF grant
DMS-2101507. The research of Y.V. was partially supported by the
ISF grant 2091/21.}
%\author[]{}
%\address{} \email{}
\date{\today}

%\abstract
\begin{abstract}
Let $G$ be a connected reductive group over a field $F=\fq((t))$ splitting over $\fqbar((t))$. Following \cite{KV,DR}, a
tamely unramified Langlands parameter $\la:W_F\to{}^L G(\qlbar)$ in general position gives rise to a finite set $\Pi_{\la}$ of irreducible admissible representations of $G(F)$, called the $L$-packet.

The main goal of this work is to provide a geometric description of characters $\chi_{\pi}$ of $\pi\in\Pi_{\la}$ and of their endoscopic linear combinations $\chi_{\la}^{\ka}$ in terms of homology of affine Springer fibers, thus establishing an analog of Lusztig conjectures in this case.
Furthermore, each $\chi_{\la}^{\ka}$ can be described as the trace of Frobenius function of a conjugation equivariant perverse sheaf on the
loop group by the sheaf-function correspondence.

As another application, we prove that the sum $\chi_{\la}^{\st}:=\sum_{\pi\in\Pi_{\la}}\chi_{\pi}$ is stable and show that the $\chi_{\la}^{\st}$'s are compatible with inner twistings. More generally, we prove that each $\chi_{\la}^{\ka}$ is $\C{E}_{\la,\ka}$-stable.
\end{abstract}
\maketitle

%\centerline{Preliminary version}

\tableofcontents

\section*{Introduction}

\begin{Emp'}
{\bf The goal of this work.}

\smallskip

Let $F$ be a local non-archimedean field, $\Gm_F$ the absolute Galois group of $F$,
$I^w_F\subseteq I_F\subseteq W_F\subseteq\Gm_F$ the wild inertia, the inertia and the Weil group of $F$,
$\ell$ a prime number different
from the characteristic of $F$, $G$ a connected reductive group over
$F$ split over a maximal unramified extension $F^{\nr}$ of $F$, and
${}^LG=\wh{G}\rtimes\Gm_F$ the Langlands dual group of $G$ over $\qlbar$.

\smallskip

The local Langlands conjecture predicts that every Langlands parameter $\la:W_F\to{}^L G(\qlbar)$ gives rise
to a certain finite set $\Pi_{\la}$ (called the {\em $L$-packet}) of irreducible admissible representations of
$G(F)$. It is also believed that at least in the tempered case a certain explicit linear combination
$\chi^{\st}_{\la}=\sum_{\pi\in\Pi_{\la}}d_{\pi}\chi_{\pi}$ of characters of representations from $\Pi_{\la}$
is stable.

\smallskip

More generally (see \cite{Lan}), every semisimple element $\ka\in Z_{\wh{G}}(\la)\subseteq \wh{G}$ in the centralizer of $\la$ gives rise to  an {\em endoscopic datum} $\C{E}_{\la,\ka}$ for $G$ (see Section~\re{exenddatum}(b)), and a certain explicit linear combination $\chi^{\ka}_{\la}=\sum_{\pi\in\Pi_{\la}}d_{\pi}(\ka)\chi_{\pi}$ is expected to be {\em $\C{E}_{\la,\ka}$-stable}
(see Section~\re{estable}(d)). Furthermore, we have $\chi_{\la}^{\ka}=\chi_{\la}^{\st}$ when $\ka=1$, and
the $\chi^{\ka}_{\la}$'s spans the same vector space as  $\{\chi_{\pi}\}_{\pi\in\Pi_{\la}}$.

%, the linear span $\on{Span}_{\qlbar}\{\chi_{\pi}\,|\,\pi\in\Pi_{\la}\}$ is expected to have an explicit {\em endoscopic basis} $\{\chi^{\ka}_{\la}\}_{\ka}$, parameterized by the set of conjugacy classes in a finite group $S_{\la}:=\pi_0(Z_{\wh{G}}(\im\la))/Z(\wh{G})^{\Gm_F})$, where $Z_{\wh{G}}(\im\la)\subseteq \wh{G}$ denotes the centralizer and $Z(\wh{G})^{\Gm_F}$ denotes the group $\Gm_F$-invariants in the center of $\wh{G}$. Namely, each conjugacy class $\ka$ in $S_{\la}$ gives rise to an {\em endoscopic datum}  $\C{E}_{\la,\ka}$ for $G$, and each invariant generalized $\chi^{\ka}_{\la}$ on $G(F)$ is expected to be {\em $\C{E}_{\la,\ka}$-stable}.
%In particular, one  expects to have a map
%$\la\mapsto\chi^{\st}_{\la}$ from the set of Langlands parameters
%to the space $D_{\st}(G(F),\qlbar)$ of stable generalized functions on $G(F)$.

\smallskip

Though the conjectural generalized functions $\chi^{\ka}_{\la}$ are
predicted in some cases, there is no general direct procedure  producing the generalized
functions $\chi^{\ka}_{\la}$ for a given $\la$. It is expected, however, that
it should be possible to carry out such a construction using geometry of $\ell$-adic sheaves.

\smallskip

In the present article we realize this approach in a special case. In fact, due to the recent work \cite{BouKaV},
our result implies that the endoscopic linear combinations of characters in these $L$-packets arise as
the trace of Frobenius function of a conjugation equivariant perverse sheaves on the loop group $LG$.\footnote{As of now, this is only done 
under the assumption that the derived group $G^{\der}$ of $G$ is simply connected. Moreover, the perverse $t$-structure is only constructed on the locus of bounded regular semi-simple elements.}

\smallskip

These perverse sheaves are natural candidates for character sheaves on $LG$, even though this notion has not yet been defined in the literature in full generality; see, however, \cite{NY}, \cite{Lu3} for related notions. Thus, we establish the first nontrivial connection of this emerging theory to irreducible characters and endoscopy. Note that our earlier work  \cite{BKV1} connected similar $\ell$-adic sheaves to elements in the Bernstein center.

%\smallskip

%Conjugation equivariant sheaves on the loop group do not explicitly appear in this work. However, our result provides a new piece of evidence for existence of a loop group analogue of the theory of character sheaves: from this perspective our present results describe stalks of a hypothetical generic depth zero character sheaf (see \cite{Lu3}, where hypothetical unipotent character sheaves on the loop group have been introduced, and \cite{BKV1}, where related sheaves were used to categorify elements in Bernstein center).
%We plan to return to this topic in a future work.
\end{Emp'}

\begin{Emp'} \label{E:0}
{\bf Main results.}

\smallskip

Assume that $G$ splits over a maximal unramified extension $F^{\nr}$ of $F$, that Langlands parameter $\la$ is {\em elliptic}, that is, does not factor through ${}^LM\subseteq {}^LG$ for a proper Levi subgroup $M\subsetneq G$, that $\la$ is {\em tamely ramified}, that is, $I_F^w\subseteq\Ker \la$, and that $\la$ is in a {\em general position}, that is, $Z_{\wh{G}}(\la(I))\subseteq\wh{G}$ is a maximal torus $\wh{T}$.
In this case, $\la$ factors through ${}^LT\subseteq{}^LG$ for a certain elliptic maximal unramified torus $T\subseteq G$ over $F$, we have $Z_{\wh{G}}(\la)=\wh{T}^{\Gm_F}$.

\smallskip

This setting was considered first in the work \cite{KV}, where the expected $L$-packet $\Pi_{\la}$ and generalized functions $\chi_{\la}^{\ka}$ were introduced (for all fields). Moreover, the $\C{E}_{\la,\ka}$-stability of $\chi_{\la}^{\ka}$ was proven there in the mixed characteristic case by analytic methods (see \cite[Theorem~2.1.6]{KV} and Section~\re{KV} below). Slightly later, the stability for $\ka=1$ was also shown in \cite{DR} by a similar method.

%\smallskip
%Namely, the case of mixed characteristic was first studied by \cite{KV},
%non geometric (in the sense of geometry of $\ell$-adic sheaves)
%and the $\C{E}_{\la,\ka}$-stability of $\chi_{\la}^{\ka}$ was shown using a generalization of a theorem of Waldspurger asserting that
%Fourier transform preserves $\C{E}$-stability.   \footnote{For $\kappa=1$ the stability was also shown by a similar method in a slightly later work \cite{DR}.}.

%see \cite{KV} and (a slightly later work with a weaker result) \cite{DR}.

\smallskip
Assume now that $F=\fq((t))$. The main goal of this paper is to give a geometric description of the restriction of
%characters $\chi_{\pi}|_{G^{\rss}(F)},\pi\in\Pi$ and of their linear combinations
$\chi^{\ka}_{\la}|_{G^{\rss}(F)}$ to the set of  regular semisimple elements of $G(F)$. Namely, we express it in terms of a trace of Frobenius on the derived coinvariants of the homology of affine Springer fibers.\footnote{We do it under the assumption that the characteristic $p$ of $\fq$ satisfies $p>2h$, where $h$ is the Coxeter number of $G$.} In particular, our results can be viewed as analogs of Lusztig's conjectures \cite{Lu4} for non-unipotent depth zero representations.
%and to use this description in order to show the $\C{E}_{\la,\ka}$-stability
%relating the endoscopy phenomenon to a natural geometric statement. Namely, we do it under a mild restriction on the characteristic $p$ of $\fq$.
%\footnote{Namely, we do it under an assumption that the characteristic $p$ of $\fq$ satisfies $p>2h$, where $h$ is the Coxeter number of $G$.}

\smallskip

More precisely, to every compact element $\gm\in G^{\rss}(F)$ one can associate the affine Springer fiber $\Fl_{G,\gm}$ at $\gm$. By the local Langlands correspondence for tori, the Langlands parameter $\la$ gives rise to an $\ell$-adic local system $\C{F}^{\st}_{\theta}$ on $\Fl_{G,\gm}$ of rank $|W|$, the cardinality of the Weyl group $W$ of $G$ (the local system is trivial geometrically but it carries a nontrivial action of Frobenius, see Sections~\re{1}, \re{2}).

\smallskip

Generalizing \cite{Lu2}, \cite{Yun}, we equip each homology group $H_i(\Fl_{G,\gm},\C{F}^{\st}_{\theta})$ with an action of the semidirect product $\wt{W}_G\rtimes\Gal(\fqbar/\fq)$ of the extended affine Weyl group $\wt{W}_G$ of $G$ and the absolute Galois group of $\fq$.
Moreover, we show that  $H_i(\Fl_{G,\gm},\C{F}^{\st}_{\theta})$ is a finitely generated $\qlbar[\wt{W}_G]$-module, hence the group cohomology
$H_j(\wt{W}_G,H_i(\Fl_{G,\gm},\C{F}^{\st}_{\theta}))$ is a finite-dimensional $\qlbar[\Gal(\fqbar/\fq)]$-module, which is non-zero only for finitely many pairs $i,j$.

\smallskip

One of the main results of this work asserts that we have an equality
\begin{equation} \label{Eq:Main}
\chi^{\st}_{\la}(\gm)=\sum_{i,j}(-1)^{i+j}\Tr(\si,H_j(\wt{W}_G,H_i(\Fl_{G,\gm},\C{F}^{\st}_{\theta}))),
\end{equation}
where $\si\in\Gal(\fqbar/\fq)$ denotes the arithmetic Frobenius element. Furthermore, we have a similar geometric description for each $\chi^{\ka}_{\la}(\gm)$ (see equality~\form{Mainka}). In order to prove the result, we establish a geometric description of $\chi_{\pi}(\gm)$ for every $\pi\in\Pi_{\la}$ (see equality~\form{Main4}).

\smallskip

In fact, due to the recent work \cite{BouKaV},
our result implies that the endoscopic linear combinations of characters in these $L$-packets arise as
the trace of Frobenius function of a conjugation equivariant perverse sheaves on the loop group $LG$ (see Section~\re{char} below).

\smallskip

Using identity \form{Main} and (a group analog of) the result of \cite{Yun} on compatibility between the Springer--Lusztig action
and the centralizer action (see Section~\re{4}), we show that the restriction $\chi^{\st}_{\la}|_{G^{\rss}(F)}$ is stable, that is,
for every pair of stably conjugate elements $\gm,\gm'\in G^{\rss}(F)$ we have an equality
$\chi^{\st}_{\la}(\gm)=\chi^{\st}_{\la}(\gm')$. Furthermore, we show that the stable characters $\chi^{\st}_{\la}|_{G^{\rss}(F)}$ are compatible with inner twistings.\footnote{By a theorem of Harish--Chandra, in the mixed characteristic case the restriction of $\chi_{\la}^{\st}$ to the  regular semisimple locus is a locally constant function, which is locally $L^1$.
Therefore, to show stability of $\chi_{\la}^{\st}$ (and compatibility with inner twistings) it suffices to show the corresponding assertions for restrictions to the  regular semisimple locus.
On the other hand, the analogue of the theorem of Harish-Chandra does not seem to be known in equal characteristic.}

\smallskip

We also obtain a similar geometric description of $\chi^{\ka}_{\la}|_{G^{\rss}(F)}$
for each $\ka\in\wh{T}^{\Gm_F}$, and deduce (using theorem of Yun) that $\chi^{\ka}_{\la}|_{G^{\rss}(F)}$ is $\C{E}_{\la,\ka}$-stable.

\smallskip

Thus we provide the first proof of stability and endoscopy property (on the regular semisimple locus) for the above $L$-packets in the equal  characteristic case. In fact, it may also be possible to deduce these results from the mixed characteristic case
established in \cite{KV}
%\footnote{but not from the result of \cite{DR} since the bound on $p$ used in {\em loc. cit.} depends on the ramification over $\B{Q}_p$}
by the  techniques of \cite{Ka}.
Thus we view the geometric interpretation of characters and of their linear combinations
 rather than the proof of stability per se as the main novelty of this work.

\end{Emp'}

We now describe our results in more detail and outline the strategy of proofs.

\begin{Emp'} \label{E:1}
{\bf The construction of \cite{KV}.}

\smallskip

(a) By the local Langlands correspondence for tori, the Langlands parameter $\la:W_F\to{}^LT\subseteq {}^LG$
corresponds to a tamely ramified homomorphism
\[
\theta=\theta_{\la}:T(F)\to\qlbar\m.
\]

Let $\C{O}$ and $\fq$ be the ring of integers and the residue field of $F$, respectively. Since $T$ is unramified, it has a natural structure $T_{\C{O}}$ over $\C{O}$, hence gives rise to a torus $\ov{T}$ over $\fq$. Since $\theta$ is tamely ramified, the restriction $\theta|_{T(\C{O})}:T(\C{O})\to\qlbar\m$ factors through a character
$\ov{\theta}:\ov{T}(\fq)\to\qlbar\m$.

\smallskip

(b) Let $\fa:T\hra G$ be an embedding, which is stably
conjugate to the inclusion $\fa_0:T\hra G$. Since $\fa(T)\subseteq G$ is a maximal elliptic torus, there exists
a unique parahoric subgroup $G_{\fa}\subseteq G(F)$ such that
$\fa(T(\C{O}))\subseteq G_{\fa}$.

We denote by $LG$ the loop group ind-scheme of $G$, by $L^+(G_{\fa})\subseteq LG$ the arc-group scheme of the parahoric subgroup $G_{\fa}$,\footnote{that is, the arc-group scheme of the corresponding Bruhat--Tits group scheme over $\C{O}$, whose generic fiber is $G$.}
by $M_{\fa}$ the quotient of $L^+(G_{\fa})$ by its pro-unipotent radical, by $\ov{\fa}:\ov{T}\hra M_{\fa}$ the map induced by $\fa$.

Let  $R_{\ov{\fa}}^{\ov{\theta}}$ be the virtual Deligne--Lusztig representation of $M_{\fa}(\fq)$
corresponding to the maximal torus $\ov{\fa}(\ov{T})\subseteq M_{\fa}$ and
character $\ov{\theta}$. Since character $\theta$ is in general position,
$R_{\ov{\fa}}^{\ov{\theta}}$ is an irreducible cuspidal representation up to a sign.

\smallskip

(c) Let $Z(G)$ be the center of $G$, and denote by $R_{\fa}^{\theta}$ the representation of group $\wt{G}_{\fa}:=G_{\fa}\cdot Z(G)(F)$ such that the restriction of $R_{\fa}^{\theta}$ to $G_{\fa}$ is the inflation of $R_{\ov{\fa}}^{\ov{\theta}}$
and restriction of $R_{\fa}^{\theta}$ to $Z(G)(F)$ is the restriction of $\theta$.

Next, let $\pi_{\fa,\theta}$ be the induced representation $\ind_{\wt{G}_{\fa}}^{G(F)}(R_{\fa}^{\theta})$. Then each $\pi_{\fa,\theta}$ is an irreducible cuspidal representation up to a sign.

\smallskip

(d) Let $\fa_1,\fa_2,\ldots,\fa_n$ be a set of representatives of
conjugacy classes of embeddings $T\hra G$, which are stably
conjugate to the inclusion $\fa_0$. For every $i=1,\ldots,n$, we can consider the {\em general position} $\inv(\fa_i,\fa_0)\in H^1(F,T)$. Then, by the Tate-Nakayama duality, for every $\ka\in\wh{T}^{\Gm_F}$ we can form
a pairing $\lan\ka,\inv(\fa_i,\fa)\ran\in\qlbar\m$.

\smallskip

We denote by $\chi_{\la}^{\st}:=\sum_i\chi_{\pi_{\fa_i,\theta}}$
the sum of characters of the $\pi_{\fa_i,\theta}$'s. More generally, for every $\ka\in\wh{T}^{\Gm_F}$,
we set $\chi_{\la}^{\ka}:=\sum_{i=1}^n\lan\ka,\inv(\fa_i,\fa_0)\ran\chi_{\pi_{\fa_i,\theta}}$. By definition, we have
$\chi_{\la}^{\ka}=\chi_{\la}^{\st}$ when $\ka=1$.

%The result of \cite{KV}, \cite{DR} asserts that under some mild characteristic of $\fq$, the generalized function $\chi_{\la}^{\st}$ is stable.
\end{Emp'}

\begin{Emp'} \label{E:2}
{\bf Geometric interpretation of $\chi^{\st}_{\la}|_{G^{\rss}(F)_{c}}$.}

\smallskip

We now describe our geometric expression \form{Main} for the restriction of
$\chi^{\st}_{\la}$ to the locus $G^{\rss}(F)_c$ of compact
elements of $G^{\sc}(F)$ in more detail.\footnote{This implies the description of the restriction to the locus of all  regular semisimple elements, because characters
$\chi_{\pi_{\fa_i,\theta}}$ are supported on elements which are products of a compact element and a central element.}
%vanish on noncompact elements so we focus on compact regular semisimple elements.

\smallskip

(a) By Lang's isogeny, a character $\ov{\theta}:\ov{T}(\fq)\to\qlbar\m$ gives rise to a rank one local system $\C{L}_{\theta}$ on $\ov{T}$, equipped with a Weil structure.

\smallskip

(b) Let $\ov{T}_G$ be the ``abstract Cartan'' of the loop group $LG$ of $G$, that is, $\ov{T}_G$ is a torus over $\fq$,
which is canonically isomorphic to the quotient $\I/\I^+$ for every Iwahori subgroup scheme $\I$ of $LG$ (over $\fqbar$) with pro-unipotent radical $\I^+$.\footnote{Notice that though $\I$ is only defined over $\fqbar$, its Frobenius Galois conjugate ${}^{\si}\I$ is another  Iwahori subgroup scheme of $LG$, hence we have a natural isomorphism between $\ov{T}_G=\I/\I^+$ and ${}^{\si}\I/{}^{\si}\I^+\simeq {}^{\si}(\I/\I^+)={}^{\si}\ov{T}_G$. Therefore the torus $\ov{T}_G$, which is a priori is only defined over $\fqbar$, has a natural structure over $\fq$.}

\smallskip

(c) Let $\I$ be an Iwahori subgroup scheme of $LG$, defined over $\fqbar$ and containing the arc-group scheme
$L^+(T)=L^+(T_{\C{O}})$ of $T$. Then $\I$ defines an isomorphism $\varphi=\varphi_{T,\I}:\ov{T}\isom\I/\I^+\isom\ov{T}_G$, where the isomorphism $\ov{T}\isom\I/\I^+$ is induced by the composition $L^+(T)\hra\I\to\I/\I^+$. Such an isomorphism $\varphi$ we call {\em admissible}.

\smallskip

(d) Consider the quotient $\ov{W}=\ov{W}_G:=\wt{W}_G/\La_G$, where $\wt{W}_G$ is the extended Weyl group of $G$ and  $\La_G:=X_*(\ov{T}_G)$ is a lattice of cocharacters of $\ov{T}_G$. Then $\ov{W}$ is non-canonically isomorphic to the Weil group $W=W_G$ of $G$, and the collection of admissible isomorphisms $\varphi:\ov{T}\isom\ov{T}_G$ form a $\ov{W}$-torsor,
 and is stable under the $\Gal(\fqbar/\fq)$-action.

\smallskip

(e) Every admissible $\varphi$ gives rise to the rank one local system $\C{L}_{\theta,\varphi}:=\varphi_*(\C{L}_{\theta})$ on $\ov{T}_G$. Then it follows from part~(d) that $\C{L}^{\st}_{\theta}:=\bigoplus_{\varphi}\C{L}_{\theta,\varphi}$ is naturally a $\ov{W}$-equivariant local system on $\ov{T}_G$ of rank $|\ov{W}|$, equipped with a Weil structure. Also we have a canonical isomorphism $\C{L}^{\st}_{\theta}\simeq\bigoplus_{\ov{w}\in \ov{W}}\ov{w}_*(\C{L}_{\theta,\varphi})$
for every $\varphi$.
%}

\smallskip

(f) For every $\gm\in G^{\rss}(F)_c$, let $\Fl_{G,\gm}$ be the affine Springer fiber of $G$ at $\gm$, and
let $\red_{\gm}:\Fl_{G,\gm}\to\ov{T}_G$ be the natural projection. Namely, $\Fl_{G,\gm}$ classifies the locus of all $[g]\in LG/\I$ such that $g^{-1}\gm g\in\I$ and  $\red_{\gm}([g])$ is
the projection of $g^{-1}\gm g\in\I$ under the natural projection $\I\to\I/\I^+\simeq\ov{T}_G$, where  $\I$ is a fixed Iwahori subgroup.

\smallskip

(g) Since $\C{L}^{\st}_{\theta}$ is a local system on $\ov{T}_G$ equipped with a Weil structure,  its pull back $\C{F}^{\st}_{\theta}:=\red_{\gm}^*(\C{L}^{\st}_{\theta})$ is a (geometrically trivial) Weil local system on $\Fl_{G,\gm}$.
%since $\C{L}^{\st}_{\theta}$ is $W$-equivariant,

We show (see \rp{whaction}) that each homology group $H_i(\Fl_{G,\gm},\C{F}^{\st}_{\theta})$ is equipped with a $\wt{W}_G$-action. Moreover, we have a decomposition
\[
H_i(\Fl_{G,\gm},\C{F}^{\st}_{\theta})\simeq \bigoplus_{\varphi} H_i(\Fl_{G,\gm},\C{F}_{\theta,\varphi}),
\]
where $\C{F}_{\theta,\varphi}:=\red_{\gm}^*(\C{L}_{\theta,\varphi})$ a (geometrically trivial) rank one local system  on $\Fl_{G,\gm}$, and an element $w\in \wt{W}_G$ induces an isomorphism
\[
H_i(\Fl_{G,\gm},\C{F}_{\theta,\varphi})\isom H_i(\Fl_{G,\gm},\C{F}_{\theta,\ov{w}\circ\varphi}),
\]
where $\ov{w}\in\ov{W}$ is the class of $w$, for every $\varphi$.

\smallskip

(h) It turns out (see \rp{fingen}) that each $H_i(\Fl_{G,\gm},\C{F}^{\st}_{\theta})$ is a
finitely generated $\wt{W}_G$-module, therefore the group homology
$H_j(\wt{W}_G, H_i(\Fl_{G,\gm},\C{F}^{\st}_{\theta}))$
is a finite dimensional vector space equipped with an
action of $\Gal(\fqbar/\fq)$. Moreover, $H_j(\wt{W}_G, H_i(\Fl_{G,\gm},\C{F}^{\st}_{\theta}))$
is non-zero only for finitely many pairs $(i,j)$. Thus
\[
H_*(\wt{W}_G, H_*(\Fl_{G,\gm},\C{F}^{\st}_{\theta})):=\sum_i\sum_{j}(-1)^{i+j}H_j(\wt{W}_G, H_i(\Fl_{G,\gm},\C{F}^{\st}_{\theta}))
\]
is a virtual finite dimensional $\qlbar[\Gal(\fqbar/\fq)]$-module, and our identity \form{Main} can be rewritten as
\begin{equation} \label{Eq:Main2}
\chi_{\la}^{\st}(\gm)=\Tr(\si, H_*(\wt{W}_G, H_*(\Fl_{G,\gm},\C{F}^{\st}_{\theta}))).\footnote{see \form{Mainka} with $\ka=1$ for a slightly different expression of $\chi_{\la}^{\st}(\gm)$
which is less canonical but more economical, closer resembling the description of the stalk of a character sheaf on a finite dimensional group, and has a generalization to $\chi_{\la}^{\ka}(\gm)$ for all $\ka$.}
\end{equation}
%As a consequence, we give a geometric proof of the fact that the character
%$\chi_{\la}^{\st}$ is stable. Furthermore, we show that the $\chi_{\la}^{\st}$'s are {\em compatible with inner twistings}.
\end{Emp'}

\smallskip

To prove \form{Main2}, we provide a geometric interpretation of $\chi_{\pi_{\fa,\theta}}|_{G^{\rss}(F)_{c}}$ for each $\fa$ as in Section~\ref{E:1}(b).

\begin{Emp'} \label{E:3}
{\bf Geometric interpretation of $\chi_{\pi_{\fa,\theta}}|_{G^{\rss}(F)_{c}}$: geometrically elliptic case.}

\smallskip

(a) Fix an admissible isomorphism $\varphi:\ov{T}\isom\ov{T}_G$ (see Section~\re{1}). %It gives rise to the (geometrically trivial) rank one local system   $\C{F}_{\theta,\varphi}:=\red_{\gm}^*(\C{L}_{\theta,\varphi})$ on $\Fl_{G,\gm}$.
Let $G^{\sc}$ be the simply connected covering of the derived group of $G$, $\wt{W}:=\wt{W}_{G^{\sc}}$ the affine Weyl group of $G$, $\La:=X_*(\ov{T}_{G^{sc}})$ the group of cocharacters,  $T_{\fa}\subseteq G$ the image of $\fa$, and $\varphi_{\al}:\ov{T}_{\fa}\isom\ov{T}_G$ the composition $\varphi\circ\fa^{-1}$.

\smallskip

(b) We choose an Iwahori subgroup scheme $\I_{\fa,\varphi}\supseteq L^+(T_{\fa})$ of $LG$ defined over $\fqbar$ such $\varphi_{T_{\fa},\I_{\fa,\varphi}}=\varphi_{\fa}$. Then the  Galois conjugate ${}^{\si}\I_{\fa,\varphi}$ is another Iwahori subgroup of $LG$, so we can consider their relative position $\wt{w}_{\fa,\varphi}:=\wt{w}_{\I_{\fa,\varphi},{}^{\si}\I_{\fa,\varphi}}\in\wt{W}$, getting thus an element $\wt{u}_{\fa,\varphi}:=\wt{w}_{\fa,\varphi}\si\in\wt{W}\si\subseteq\wt{W}_G\rtimes\lan\si\ran$.

By construction, element $\wt{u}_{\fa,\varphi}$ is unique up to a $\La$-conjugacy, and each homology group $H_i(\Fl_{G,\gm},\C{F}_{\theta,\varphi})$ is naturally a $\qlbar[\La_G\rtimes\lan\wt{u}_{\fa,\varphi}\ran]$-module.

\smallskip

(c) Assume that $G$ is semisimple, and that element $\gm$ is geometrically elliptic. In this case, the affine Springer fiber $\Fl_{G,\gm}$ is a scheme  of finite type over $\fq$, hence each homology group $H_i(\Fl_{G,\gm},\C{F}_{\theta,\varphi})$ is a finite-dimensional $\qlbar$-vector space.

Therefore $H_*(\Fl_{G,\gm},\C{F}_{\theta,\varphi}):=\sum_i(-1)^i H_i(\Fl_{G,\gm},\C{F}_{\theta,\varphi})$ is a virtual finite-dimensional
$\qlbar[\lan\wt{u}_{\fa,\varphi}\ran]$-module, and our result asserts that in this case we have an equality
\begin{equation} \label{Eq:Main3}
\chi_{\pi_{\fa,\theta}}(\gm)=\Tr(\wt{u}_{\fa,\varphi},H_*(\Fl_{G,\gm},\C{F}_{\theta,\varphi})).
\end{equation}
In particular, the right hand side of \form{Main3} is independent of the choice of $\varphi$ and $\I_{\fa,\varphi}$.

\smallskip

(d) The identity \form{Main3} is an affine analog of a theorem of Lusztig (see \cite{Lu} or \cite{La}), asserting that characters of  Deligne--Lusztig representations of reductive groups over finite fields can be described as the traces of Frobenius on the cohomology of Springer fibers.
\end{Emp'}

The interpretation of $\chi_{\pi_{\fa,\theta}}$ in general is based on the following purely algebraic construction.

\begin{Emp'} \label{E:gentr}
{\bf Generalized trace} (see Appendix A){\bf.}

\smallskip

(a) Let  $\La$ be a finitely generated free abelian group, and let $\psi\in \Aut(\La)$ be an {\em elliptic} automorphism of finite order, where {\em elliptic} means that $\La^{\psi}=\{0\}$. Consider the semi-direct product $\Dt:=\La\rtimes\lan u\ran$, where $\lan u\ran$ is the infinite cyclic group generated by $u$, and $u$ acts on $\La$ as $\psi$, that is, $u\la u^{-1}=\psi(\la)$ for every $\la\in\La$.

\smallskip

(b) Recall that a $\qlbar[\Dt]$-module $M$ is called {\em $u$-locally finite}, if it is a union of $u$-invariant finite dimensional vector $\qlbar$-subspaces.
\smallskip

(c) To a $u$-locally finite $\qlbar[\Dt]$-module $M$, which is finitely generated as a $\qlbar[\La]$-module,
we associate a {\em generalized  trace} $\Tr_{\gen}(u,M)\in\qlbar$ characterized by the condition that the assignment
$M\mapsto \Tr_{\gen}(u,M)$ has the following three properties:

\smallskip

\quad\quad(i) ({\em additivity}) for every short exact sequence $0\to M'\to M\to M''\to 0$, we have an equality
 \[
 \Tr_{\gen}(u,M)=\Tr_{\gen}(u,M')+\Tr_{\gen}(u,M'');
 \]

\smallskip

\quad\quad(ii) we have  $\Tr_{\gen}(u,M)=\Tr(u,M)$, if $M$ is finite-dimensional over $\qlbar$;

\smallskip

\quad\quad(iii) for a one-dimensional representation $V$ of $\lan u\ran$ over $\qlbar$, viewed as a representation of $\Dt$, trivial on $\La$,
we have an equality
\[
\Tr_{\gen}(u,M\otimes_{\qlbar} V)= \Tr_{\gen}(u,M)\cdot\Tr(u,V).
\]

\smallskip

Explicitly, we choose a $\psi$-stable finite subset $\La_1\subseteq\La$, generating $\La$  as a monoid and a $u$-invariant  finite-dimensional subspace $M_0\subseteq M$, which generates $M$ as a $\qlbar[\La]$-module. This data gives a $u$-stable increasing filtration
$\{M_n\}_n$ of $M$, where $M_{n+1}=\sum_{\mu\in\La_1}\mu(M_n)\subseteq M$  for all $n\in\B{Z}_{\geq 0}$. Then every $M_n$ is finite-dimensional over $\qlbar$, we have an equality
\[
\Tr_{\gen}(u,M)=\Tr(u,M_n)\text{ for every sufficiently large }n.
\]
\end{Emp'}

\begin{Emp'} \label{E:3'}
{\bf Geometric interpretation of $\chi_{\pi_{\fa,\theta}}$: the general case.}

\smallskip

(a) Notice that in the notation of Section~\re{3} each $\qlbar[\La_{G}\rtimes\lan\wt{u}_{\fa,\varphi}\ran]$-module $H_i(\Fl_{G,\gm},\C{F}_{\theta,\varphi})$ is $\wt{u}_{\fa,\varphi}$-locally finite and finitely generated as a $\qlbar[\La_G]$-module.

\smallskip

(b) Let $Z_0\subseteq Z(G)$ be the maximal split torus, and let $\La_0:=X_*(Z_0)$ be the group of cocharacters. Then we have a natural isomorphism $\La_G/\La_0\simeq \La_{G/Z_0}$ and the group of coinvariants $H_i(\Fl_{G,\gm},\C{F}_{\theta,\varphi})_{\La_0}$ is naturally a $\qlbar[\La_{G/Z_0}\rtimes\lan\wt{u}_{\fa,\varphi}\ran]$-module, which is $\wt{u}_{\fa,\varphi}$-locally finite and finitely generated as a $\qlbar[\La_{G/Z_0}]$-module.

\smallskip

(c) Moreover, since the torus $T_{\fa}\subseteq G$ is elliptic, we see that the automorphism $\wt{u}_{\fa,\varphi}|_{\La_{G/Z_0}}\in\Aut( \La_{G/Z_0})$ is elliptic. Therefore all assumptions of Section~\re{gentr} are satisfied, thus we can form a generalized trace $\Tr_{\gen}(\wt{u}_{\fa,\varphi},H_i(\Fl_{G,\gm},\C{F}_{\theta,\varphi})_{\La_0})$ for each $i$, hence a generalized trace
\[
\Tr_{\gen}(\wt{u}_{\fa,\varphi},H_*(\Fl_{G,\gm},\C{F}_{\theta,\varphi})_{\La_0}):=
\sum_{i}(-1)^i\Tr_{\gen}(\wt{u}_{\fa,\varphi},H_i(\Fl_{G,\gm},\C{F}_{\theta,\varphi})_{\La_0}).
\]

\smallskip

(d) The main technical result of this work asserts that for every $\gm\in G^{\rss}(F)_c$, we have an equality
\begin{equation} \label{Eq:Main4}
\chi_{\pi_{\fa,\theta}}(\gm)=\Tr_{\gen}(\wt{u}_{\fa,\varphi},H_*(\Fl_{G,\gm},\C{F}_{\theta,\varphi})_{\La_0}).
\end{equation}
\end{Emp'}

\begin{Emp'} \label{E:ka}
{\bf Geometric interpretation of $\chi_{\la}^{\ka}|_{G^{\rss}(F)_{c}}$.} Fix $\ka\in\wh{T}^{\Gm_F}$.

\smallskip
(a) Fix an admissible isomorphism $\varphi:\ov{T}\isom \ov{T}_G$ (see Section~\re{3}), let $\fa_0:T\hra G$ be an inclusion, and let $\wt{u}_{\fa_0,\varphi}$ be as in Section~\re{3}(b). Then $\varphi$ induces an isomorphism $X_*(T)\simeq X_*(\ov{T})\simeq X_*(\ov{T}_G)=\La_G$. By construction, it interchanges the action of $\si$ on $X_*(T)$ with the action of $\wt{u}_{\fa_0,\varphi}$ on $\La_G$. Therefore
element $\ka\in \wh{T}^{\Gm_F}=\Hom(X_*(T)_{\Gm_F},\qlbar\m)$
can be viewed as a character $\ka:\La_G\ltimes\lan\wt{u}_{\fa_0,\varphi}\ran\to\qlbar\m$, trivial on $\wt{u}_{\fa_0,\varphi}$.

\smallskip

(b) As in Section~\ref{E:1}, $H_i(\Fl_{G,\gm},\C{F}_{\theta,\varphi})$ is a representation of $\La_G\ltimes\lan\wt{u}_{\fa_0,\varphi}\ran$. Let
$(\qlbar)_{\ka}$ be a one-dimensional representation of $\La_G\ltimes\lan\wt{u}_{\fa_0,\varphi}\ran$, corresponding to the character $\ka$ from part~(a). Taking tensor product, we get a representation
\[
H_i(\Fl_{G,\gm},\C{F}_{\theta,\varphi})_{\ka}:=H_i(\Fl_{G,\gm},\C{F}_{\theta,\varphi})\otimes_{\qlbar}(\qlbar)_{\ka}
\]
of $\La_G\ltimes\lan\wt{u}_{\fa_0,\varphi}\ran$. As in Section~\ref{E:2}(h), it is a finitely generated $\qlbar[\La_G]$-module.

\smallskip

(c) Consider an element $\ov{w}_{T,\varphi}\in \ov{W}$ such that ${}^{\si}\varphi=\ov{w}_{T,\varphi}^{-1}\circ\varphi$ and an element
$\ov{u}_{T,\varphi}:=\ov{w}_{T,\varphi}\si\in \ov{W}\si\subseteq \ov{W}\rtimes\lan\si\ran$. Then $\ov{u}_{T,\varphi}$ is the image of
$\wt{u}_{\fa_0,\varphi}$,  so the space
$H_j(\La_G,H_i(\Fl_{G,\gm},\C{F}_{\theta,\varphi})_{\ka})$ is a finite-dimensional $\qlbar[\lan\ov{u}_{T,\varphi}\ran]$-module. As in Section~\ref{E:2}(h), it is non-zero only for finitely many pairs $(i, j)$, so we can form a virtual finite-dimensional $\qlbar[\lan\ov{u}_{T,\varphi}\ran]$-module
\[
H_*(\La_G,H_*(\Fl_{G,\gm},\C{F}_{\theta,\varphi})_{\ka}):=\sum_{i,j}(-1)^{i+j}H_j(\La_G,H_i(\Fl_{G,\gm},\C{F}_{\theta,\varphi})_{\ka}).
\]
One of our main results asserts that we have an equality
\begin{equation} \label{Eq:Mainka}
\chi_{\la}^{\ka}(\gm)=\Tr(\ov{u}_{T,\varphi}, H_*(\La_G, H_*(\Fl_{G,\gm},\C{F}_{\theta,\varphi})_{\ka})).
\end{equation}
\end{Emp'}

\begin{Emp'}
{\bf Outline of proofs.}

\smallskip

(a) First we show  identity \form{Main4}. In the case when $G$ is semisimple and $\gm$ is geometrically elliptic
(hence identity \form{Main4} specializes to identity \form{Main3}), the assertion follows from an analogous assertion for finite groups (Lusztig's theorem \cite[Propositions~8.15 and 9.2]{Lu}) using Lefschetz trace formula and the observation that $\chi_{\pi_{\fa,\theta}}(\gm)$ can be written as an orbital integral of the character of the Deligne--Lusztig representation (compare \cite{Ar}).

The proof in the general case is much more technically involved and uses in addition finiteness properties of the affine Springer fibers, a group version of a theorem of Yun \cite{Yun} (see Section~\ref{E:4} below) and the injectivity property of the homology of affine Springer fibers established in \cite{BV}.

\smallskip

(b) Next we claim that identity \form{Main2} is equivalent to the identity \form{Mainka} for $\ka=1$.
Indeed, for every $\qlbar[\wt{W}_G\rtimes\lan\si\ran]$-module $V$, which is a finitely generated $\qlbar[\wt{W}_G]$-module, and every $j\in\B{N}$ we have an equality (see Section~\re{remtr})
\[
\Tr(\si,H_j(\wt{W}_G,V))=\Tr(\si,H_j(\La_G,V)^{\ov{W}})=\frac{1}{|\ov{W}|} \sum_{\ov{w}\in \ov{W}} \Tr(\ov{w}\si,H_j(\La_G,V)).
\]
Using this identity, the assertion follows from the identity $\C{L}^{\st}_{\theta}=\bigoplus_{\ov{w}\in \ov{W}}\ov{w}_*(\C{L}_{\theta,\varphi})$.

\smallskip
%\vskip 4 truept
(c) It remains to deduce identity \form{Mainka} from identity \form{Main4}. Assume that $G$ is semi-simple and simply connected. In this case, we have a formula
\begin{equation} \label{Eq:M}
\Tr(\ov{u}_{T,\varphi},H_*(\La_G,H_*(\Fl_{G,\gm},\C{F}_{\theta,\varphi})_{\ka}))=\sum_{\wt{u}}
\Tr_{\gen}(\wt{u},H_*(\Fl_{G,\gm},\C{F}_{\theta,\varphi})_{\ka}),
\end{equation}
where $\wt{u}$ runs over a set of representatives of the set of $\La_G$-conjugacy classes in
$\wt{W}\si\subseteq\wt{W}\rtimes\lan\si\ran$, whose image in $\ov{W}\rtimes\lan\si\ran$ equals $\ov{u}_{T,\varphi}$.

Indeed, this follows from a purely algebraic assertion (see \rp{trform}(b)) asserting that for every $u,\La$ and $M$ as in Section~\re{gentr}
we have an equality
\begin{equation} \label{Eq:M1}
\Tr(u,H_*(\La,M))=\sum_{\wt{u}}\Tr_{\gen}(\mu u,M),
\end{equation}
where $\mu u$ runs over the set of $\La$-conjugacy classes in $\La u$.

\smallskip

(d) To show formula \form{M1}, we note that both of its sides are additive in $M$. Set $\La_{\psi}:=\La/(\psi-1)\La$.
Using localization theorem in equivariant $K$-theory we reduce to the case, when $M=M_{\xi}$ corresponds to a character
\[
\La\rtimes\lan u\ran\to\La_{\psi}\times\lan u\ran\to \La_{\psi}\overset{\xi}{\to}\qlbar\m.
\]

In this case, we have to show that the equality
\[
\Tr(\psi,H_*(\La,M_{\xi}))=\sum_{\mu\in\La_{\psi}}\la(\xi).
\]

If $\xi$ is non-trivial, then both sides of the last equality
vanish, while if $\xi$ is trivial, both sides are equal to $\det(1-{\psi},\La)=|\La_\psi|$.

\smallskip

(e) By Galois cohomology arguments, the correspondence $\fa_i\mapsto \wt{u}_{\fa_i,\varphi}$ defines a bijection between the set representatives of conjugacy classes of embeddings $\fa:T\hra G$, stably conjugate to the inclusion $\fa_0:T\hra G$, and the set of representatives of conjugacy classes in the summation of the right hand side of formula \form{M}. Moreover, unwinding definitions we have an equality
\[
\Tr_{\gen}(\wt{u}_{\fa_i,\varphi},H_*(\Fl_{G,\gm},\C{F}_{\theta,\varphi})_{\ka})=\lan\ka,\inv(\fa_i,\fa_0)\ran
\Tr_{\gen}(\wt{u}_{\fa_i,\varphi},H_*(\Fl_{G,\gm},\C{F}_{\theta,\varphi})).
\]
Therefore identity \form{Mainka} is an immediate consequence of the combination of formulas \form{Main4} and \form{M}. For a general reductive group $G$, the strategy is similar.

\end{Emp'}

\begin{Emp'} \label{E:char} 
{\bf Relation to equivariant perverse sheaves.} Let $\frak{C}_{\bullet}\subseteq LG$ be the locally closed sub-indscheme such that $\frak{C}_{\bullet}(\fq)\subseteq LG(\fq)=G(F)$ is the locus of compact regular semisimple elements.

\smallskip 

(a) The local system $\C{L}_{\theta,\varphi}$ on $\ov{T}_G$ (see Section~\re{2}(e)) gives rise to the affine Grothendieck--Springer sheaf $\mathcal{S}_{\theta,\varphi,\bullet}:=\mathcal{S}_{\C{L}_{\theta,\varphi},\bullet}$ on the quotient stack $[\frak{C}_{\bullet}/LG]$, equipped with an action of the group $\La_G$ (see \cite[Corollary~3.1.4(b)]{BouKaV}). Moreover, an element $\ka\in\wh{T}^{\Gm_F}$ gives rise to a character $\ka$ of $\La_G$ (see Section~\re{ka}(a)), and
\cite[Theorem~3.2.1]{BouKaV} asserts that the sheaf of $\ka$-coinvariants $\mathcal{S}_{\theta,\varphi,\bullet,\ka}$ of $\mathcal{S}_{\theta,\varphi,\bullet}$ is perverse.

\smallskip 

(b) Furthermore, the local system $\C{L}_{\theta,\varphi}$ on $\ov{T}_G$ is equipped with a Weil structure, and this structure induces Weil  structures on $\mathcal{S}_{\theta,\varphi,\bullet}$ and $\mathcal{S}_{\theta,\varphi,\bullet,\ka}$. Then, combining
formula \form{Mainka} with the ind-fp-proper base change one can show that the function $\chi_{\la}^{\ka}$ on $G^{\rss}(F)$ is obtained from
perverse sheaf $\mathcal{S}_{\theta,\varphi,\bullet,\ka}$ by the Grothendieck sheaf-to-function correspondence (see \cite{BKV2}). 
\end{Emp'}

\begin{Emp'} \label{E:4}
{\bf A group analog of a theorem of Yun and applications.} Our arguments rely heavily on the following result,
which is essentially due to Yun \cite{Yun}.\footnote{More precisely,
Yun proved the corresponding result for Lie algebras using global
methods, while we deduce its group version (see \rt{action}) using the
topological Jordan decomposition and quasi-logarithm maps (see Appendix B).}

\smallskip

(a) Each the homology group
$H_i(\Fl_{G,\gm},\C{F}_{\theta,\varphi})$ is equipped with an action of the loop group $LG^0_{\gm}:=L(G_{\gm}^0)$ of the connected centralizer
$G^0_{\gm}=(G_{\gm})^0\subseteq G$, commuting with the action of $\La_G$. Moreover,
this action factors through the group of connected components
$\pi_0(LG^0_{\gm})$.

\smallskip

(b) Note that there is a canonical homomorphism
$\pr'_{\gm}:\qlbar[\La_G]^{W_G}\to\qlbar[\pi_0(LG^0_{\gm})]$ of group algebras (see Section~\re{setup comp}(c)).
Namely, we have natural isomorphisms $\qlbar[\pi_0(LG^0_{\gm})]\simeq\qlbar[\wh{G^0_{\gm}}^{\Gm_F}]$
and  $\qlbar[\La_G]^{W_G}\simeq\qlbar[c_{\wh{G}}]$, where algebras on the right are the algebras of regular functions and
$c_{\wh{G}}$ denotes the Chevalley space of $\wh{G}$.

Under these identifications, homomorphism $\pr'_{\gm}$ is defined to be the pullback map
$(\can'_{\gm})^*: \qlbar[c_{\wh{G}}]\to \qlbar[\wh{G^0_{\gm}}^{\Gm_F}]$, where  $\can'_{\gm}:\wh{G^0_{\gm}}^{\Gm_F}\to c_{\wh{G}}$ denotes the composition
\[
\wh{G^0_{\gm}}^{\Gm_F}\overset{\iota}{\lra}\wh{G^0_{\gm}}^{\Gm_F}\hra\wh{G}\overset{\nu_{\wh{G}}}{\lra}c_{\wh{G}},
\]
where $\iota$ is the map $g\mapsto g^{-1}$, and $\nu_{\wh{G}}$ is the Chevalley map.

\smallskip

(c) Theorem of Yun asserts that there is a finite filtration of $H_i(\Fl_{G,\gm},\C{F}_{\theta,\varphi})$, stable under the action of  $\La_G\times \pi_0(LG^0_{\gm})$, such that the induced action of $\qlbar[\La_G]^{W_G}$ on each graded
piece is induced from the action of $\pi_0(LG^0_{\gm})$ via
homomorphism $\pr'_{\gm}$.

\smallskip

%Recall that
%$\Fl_{G,\gm}$ is an ind-scheme and there exists a lattice $\La_{G,\gm}\subseteq LG^0_{\gm}$  such that the quotient
%$\La_{G,\gm}\bs \Fl_{G,\gm}$ is of finite type (compare Section~\re{cent}). Therefore each
%It is well-known that each $H_i(\Fl_{G,\gm},\C{F}_{\theta,\varphi})$ is a finitely generated
%$\qlbar[\La_{G,\gm}]$-module, hence a finitely generated
%$\qlbar[\pi_0(LG^0_{\gm})]$-module.

(d) It is well-known that each $H_i(\Fl_{G,\gm},\C{F}_{\theta,\varphi})$ is a finitely generated $\qlbar[\pi_0(LG^0_{\gm})]$-module.
Using the theorem of Yun from part~(c) we deduce that each $H_i(\Fl_{G,\gm},\C{F}_{\theta,\varphi})$ is a finitely
generated $\qlbar[\La_G]^{W_G}$-module, hence a finitely generated $\qlbar[\La_G]$-module.\footnote{The statement that homologies of affine Springer fibers are finitely generated over $\wt{W}_G$ appears also as Conjecture 3.6 in \cite{Lu3}. It is also
mentioned in {\em loc. cit.} that the statement should follow from the result of \cite{Yun}.}
%
%\footnote{The statement that cohomology of an affine Springer fiber is finitely generated over $\wt{W}$ and the idea
%to prove it using the result of \cite{Yun} is due to Lusztig \cite{Lu3}.}
This finiteness property is needed in order to make sure that the right hand side of formulas \form{Main2}, \form{Main4} and \form{Mainka} above
are defined.
%
%the trace $\Tr(\si,H_*(\wt{W}_G, H_*(\Fl_{G,\gm},\C{F}^{\st}_{\theta})))$
%in Section~\re{1} and generalized trace $\Tr_{\gen}(\wt{u}_{\fa,\varphi},H_*(\Fl_{G,\gm},\C{F}_{\theta,\varphi})_{\La_0})$ in Section~\re{3'}.

\smallskip

(e) Theorem of Yun also implies that $\pi_0(LG^0_{\gm})$ acts unipotently on each homology group
$H_j(\wt{W}_G, H_i(\Fl_{G,\gm},\C{F}^{\st}_{\theta}))$. Combining this fact with formula \form{Main2},  stability of $\chi^{\st}_{\la}|_{G^{\rss}(F)}$ follows.
%This implies the stability of the function
%$\gm\mapsto\Tr(\si,H_*(\wt{W}_G, H_*(\Fl_{G,\gm},\C{F}^{\st}_{\theta})))$.
\footnote{In a work in progress \cite{BKV2} we found an alternative (purely local) proof of the fact that each $H_i(\Fl_{G,\gm},\C{F}^{\st}_{\theta})$ is a finitely generated $\qlbar[\wt{W}_G]$-module. However,
we do not know an alternative proof of stability.}

\smallskip

(f) More generally, theorem of Yun plays a central role in showing that the restriction $\chi^{\ka}_{\la}|_{G^{\rss}(F)}$ is $\C{E}_{\la,\ka}$-stable for every $\ka\in\wh{T}^{\Gm_F}$. Namely, using the algebra isomorphism $\qlbar[\pi_0(LG^0_{\gm})]\simeq \qlbar[\wh{G^0_{\gm}}^{\Gm_F}]$,
each representation $H_j(\La_G, H_i(\Fl_{G,\gm},\C{F}_{\theta,\varphi})_{\ka})$ of $\pi_0(LG^0_{\gm})$ gives rise to a coherent sheaf $\C{A}$ on $\wh{G^0_{\gm}}^{\Gm_F}$ and theorem of Yun implies that $\C{A}$ is supported on the finite subset
$(\can'_{\gm})^{-1}(\nu_{\wh{G}}(\ka^{-1}))$.

Combining this fact with formula \form{Mainka} we conclude that $\chi^{\ka}_{\la}|_{G^{\rss}(F)}$ is $\C{E}_{\la,\ka}$-stable, if the center $Z(G)$ of $G$ is assumed to be connected. To show the assertion in general, we observe that the assertion for
$G$ follows from that for $G_1:=G\times^{Z(G)}T$ and that the center of $G_1$ is connected.
\end{Emp'}

\begin{Emp'} \label{E:KV}
{\bf Comparison with \cite{KV}.}

\smallskip

(a) While the method of the current paper is very different from that of \cite{KV},
the proofs of the $\C{E}_{\la,\ka}$-stability of $\chi_{\la}^{\ka}$ have one common feature: though the arguments in both papers are mostly local, each of them uses one crucial ingredient, whose proof is global. Namely, the argument of \cite{KV} is based on a generalization of a   theorem of Waldspurger \cite{Wa} asserting that Fourier transform preserves $\C{E}$-stability, whose proof uses Arthur--Selberg trace formula, while the argument of the current paper is based on a theorem of Yun, described above, whose proof uses geometry of the Hitchin fibration.

\smallskip

(b) As our proof of $\C{E}_{\la,\ka}$-stability of $\chi_{\la}^{\ka}$ is mostly geometric, the argument for a general $\ka$ is not significantly more difficult than the case of $\ka=1$. On the other hand, the situation is very different in the setting of \cite{KV}.
Namely, there the $\C{E}_{\la,\ka}$-stability of $\chi_{\la}^{\ka}$ is deduced from the fact that the restrictions of $\chi_{\la}^{\ka}$ to topologically unipotent elements are compatible with inner twistings, and  even the formulation of the latter assertion is significantly
more difficult for $\ka\neq 1$.
\end{Emp'}

\begin{Emp'}
{\bf Remark.}
It is tempting to speculate that the parallel roles of the  Arthur--Selberg trace formula and of the geometry of the Hitchin fibration in the two approaches to the present story
(see Section~\re{KV}(a))
 can be explained by a direct connection between the two theories. In particular, one can hope for a categorification of the trace formula
involving sheaves on the Hitchin space.
\end{Emp'}

\begin{Emp'} \label{E:convention}
{\bf Notation and conventions.} Let $k$ \label{a:k} be an algebraically closed field, and  $\ell$ \label{a:ell} be a prime different from the characteristic of $k$.

\smallskip

(a) For a field $L$, we denote by $\ov{L}$ \label{a:lbar} and $L^{\sep}$ \label{a:lsep}the algebraic and the separable closure of $L$, respectively, and by $\Gm_L$ \label{a:gml} the absolute Galois group $\Gal(L^{\sep}/L)$ of $L$. We denote by $\fq$ the finite field with $q$ elements and by $\si\in\Gm_{\fq}$ \label{a:si} the arithmetic Frobenius.

\smallskip

(b) Let $Y$ be a scheme of finite type over $k$.
For every $\C{F}\in D_c^b(Y,\qlbar)$ and $i\in\B{Z}$ we set $H_i(Y,\C{F}):=H^i(Y,\C{F})^{*}$\label{a:hi}. Every morphism $f:Y'\to Y$  induces a morphism on homology $f_*:H_i(Y',f^*\C{F})\to H_i(Y,\C{F})$. For an ind-scheme
$Y=\colim_i Y_i$ we denote by $H_i(Y,\C{F})$ the inductive limit $\on{colim}_i H_i(Y,\C{F}|_{Y_i})$, where $\C{F}|_{Y_i}$ denoted the $*$-pullback of $\C{F}$ to $Y_i$.

\smallskip

(c) For every scheme $Y$ of finite type over a finite field $\fq$, we will identify every $\C{F}\in D_c^b(Y,\qlbar)$ with the corresponding
Weil sheaf, that is, a pair $(\C{F},\varphi)$, where $\C{F}\in D_c^b(Y\otimes_{\fq}\fqbar,\qlbar)$ and $\varphi$ is an isomorphism
$\si^*\C{F}\isom\C{F}$. Also to simplify the notation, we will always write $H^i(Y,\C{F})$ instead of
 $H^i(Y\otimes_{\fq}\fqbar,\C{F})$ and similarly for $H_i(Y,\C{F})$.

\smallskip

(d) Let $G$ be a connected reductive group over $K=k((t))$, and let $G^{\rss}\subseteq G$ be the locus of regular semisimple elements. Following Yun \cite[Section~1.4]{Yun} we assume that either $\on{char} k=0$ or $\on{char} k > 2h$, where $h$ is
the Coxeter number of $G$.

\smallskip

(e) In Sections~4 and 5 we assume that $k=\fqbar$ and that
 $G$ is a connected reductive group over $F=\fq((t))$, usually split over $F^{\nr}:=F\otimes_{\fq}\fqbar$.

%one plus the sum of coefficients of the highest root of $G$ written in terms of simple roots.
\end{Emp'}

\begin{Emp'}
{\bf Plan of the paper.} The paper is organized as follows:

\smallskip

In Section 1 we introduce basic notation and formulate a theorem of Lusztig \cite{Lu} describing characters of
Deligne--Lusztig representations in terms of the cohomology of Springer fibers.

\smallskip

In Section 2 we recall the construction of the Lusztig action of the extended affine group on the homology of the affine Springer fibers, slightly generalizing \cite{Lu2}. We also formulate a group version of a theorem of Yun mentioned above.

\smallskip

In Section 3, we will establish finiteness properties of the homology affine Springer fibers, which we need later.

\smallskip

In Section 4 we review basic notions related to endoscopy, classify maximal unramified tori in unramified reductive groups over $p$-adic fields,
and formulate our main theorems. The proofs of these results are carried out in Section 5.

\smallskip

We finish this paper by two appendixes used in Section 5 and having independent interest:

\smallskip

In Appendix A we introduce the generalized trace, described above, and show its basic properties.

\smallskip

In Appendix B we study properties of the affine Springer fibers and deduce the group version of a theorem of Yun
from its original (Lie algebra) version.
\end{Emp'}

\begin{Emp'}
{\bf Acknowledgements.} The idea of this work dates back to 2009,
and its present form has been influenced by the many discussions with David Kazhdan and Zhiwei Yun
held since then.\footnote{The first draft of this work under the title ``Homology of affine Springer fibers, characters of Deligne--Lusztig representations, and stability", appeared around 2011 and was cited in \cite{Yun} and \cite{BKV1}.}  It is a pleasure to express our gratitude to both of them.
Our key point here is reduction of stability to a geometric question resolved in \cite{Yun}, so we are also grateful to Yun for this contribution that made our work possible.
The first author wants to thank Bob Kottwitz for discussions that influenced his understanding of the subject reflected, in particular, in this work.

%The authors were supported by the ISF, BSF and NSF.

\end{Emp'}

\section{Springer fibers and Deligne--Lusztig representations}
In this section we will recall mostly standard facts concerning action of the Weyl group on the cohomology of Springer fibers
and its relation with Deligne--Lusztig representations.

\subsection{The abstract Cartan and the Weyl group} \label{S:prel}
Let $H$ be a connected algebraic group over an algebraically closed field $k$.

\begin{Emp} \label{E:abscartan}
{\bf The construction}. Following \cite[Section~1.1]{DL}, we define the abstract Cartan $T_H$ \label{a:th} of
$H$, the Weyl group  $W_H$  \label{a:wh}  of $H$, and the set of simple reflections $S_H\subseteq W_H$. \label{a:sh}

\smallskip

(a) For every pair $(T,B)$, where $B\subseteq H$ is a Borel subgroup and $T\subseteq B$ is a maximal torus, we set $W_{H,T}:=N_H(T)/T$,  \label{a:wht} where $N_H(T)\subseteq H$  \label{a:nht} is the normalizer of $T$, and let $S_{H,B,T}\subseteq
W_{H,T}$  \label{a:shbt}  be the corresponding set of simple reflections. Then group $W_{H,T}$ acts on $T$ by conjugation, and the inclusion of groups $N_T(H)\hra H$ induces a bijection of sets $W_{H,T}\isom B\bs H/B$.

\smallskip

(b) For every two pairs $(B,T)$ and $(B',T')$ as in part~(a),
there exists a unique isomorphism $\phi:T\isom T'$, which is of the form
$t\mapsto hth^{-1}$ for some $h\in H$ such that $h(B,T)h^{-1}=(B',T')$.
Then isomorphism $\phi$ induces a group isomorphism $W_{H,T}\isom W_{H,T'}$ and a bijections $S_{H,T,B}\isom
S_{H,T',B'}$.

\smallskip

(c) We denote by  $T_H$ the projective limit $\lim_{(B,T)}T$, taken over the set of pairs
$(B,T)$ as in part~(a) where the transition maps are as in part~(b), and call $T_H$ the {\em abstract Cartan} of $H$.
Similarly, we set $W_H:=\lim_{(B,T)}W_{H,T}$ and call it the {\em abstract Weyl group} of $H$, and
denote by $S_H:=\lim_{(B,T)}S_{H,B,T}\subseteq W_H$ the {\em set of simple reflections} in $W_H$.
Note that $T_H$ is a torus, and the Weyl group $W_H$ naturally acts on $T_H$.
\end{Emp}

\begin{Emp} \label{E:weylprop}
{\bf Properties.}

\smallskip

(a) By construction, for each pair $(B,T)$ as in Section~\re{abscartan}(a), we have a canonical isomorphism $\varphi_{T,B}:T\isom T_H$
\label{a:varphitb} of tori,
a canonical isomorphism $W_{H,T}\isom W_{H}$ of groups, and a canonical bijection  $S_{H,T,B}\isom S_{H}$ of sets.

\smallskip

(b) For every pair $(T,B)$ as in Section~\re{abscartan}(a), we have a natural isomorphism of tori
$T\isom B/R_u(B)$, defined to be the composition $T\hra B\to B/R_u(B)$.  Moreover, the composition $B/R_u(B)\isom T\overset{\varphi_{T,B}}{\lra}
T_H$ is independent of $T$. We denote by $\pr_B$  \label{a:prb} the composition
$\pr _B:B\to B/R_u(B)\isom T_H$.

\smallskip

(c) For every two Borel subgroups $B,B'\subseteq H$, we can define
their relative position $w_{B,B'}\in W_H$  \label{a:wbb'} to be the image of
$h\in H$ such that $hBh^{-1}=B'$ under the  composition $H\to B\bs
H/B\isom W_{H,T}\isom W_H$ (for some maximal torus $T\subseteq B$).

By direct calculation, one sees that $w_{B,B'}$ is independent
of the choice of $T\subseteq B$, and that for every three Borel
subgroups $B,B',B''$ containing a maximal torus $T$, we have an equality
$w_{B,B''}=w_{B,B'}\cdot w_{B',B''}$.

\smallskip

(d) Consider the {\em Chevalley space} $c_H:=\Spec k[H]^H$  \label{a:ch} of $H$, where $H$ acts on $k[H]$ via the adjoint action,
and let $\nu_H$  \label{a:nuh} be the natural projection $H\to c_H$.

Then $\nu_H$
induces a morphisms $\nu_{T_H}:T_H\to c_H$  \label{a:nuth}  defined as a limit of morphisms
$\nu_H|_T:T\to c_H$, taken over all $(T,B)$ as in Section~\re{abscartan}(a).

Moreover, morphism $\nu_{T_H}$ induces an isomorphism $W_H\bs T_H\isom c_H$, and the restriction
$\nu^{\rss}_{T_H}:T^{\rss}_H\to c^{\rss}_H$ to the
 regular semisimple locus $c_H^{\rss}\subseteq c_H$ is a $W_H$-torsor.
\end{Emp}

\begin{Emp}  \label{E:admis}
{\bf Admissible isomorphisms.} Let $T\subseteq H$ be a maximal torus.

\smallskip

(a) We say that an isomorphism $\varphi:T\isom T_H$ of tori is {\em
admissible},  \label{a:adisom}
if there exists a Borel subgroup $B\supseteq T$ of $H$
such that $\varphi=\varphi_{T,B}$ (see Section~\re{weylprop}(a)).

\smallskip

(b) Note that for every two Borel subgroups $B,B'\supseteq T$, we have a
formula $\varphi_{T,B'}=w_{B,B'}^{-1}\circ \varphi_{T,B}$.
Therefore the set of admissible isomorphisms is a $W_H$-torsor,
and the map $B\mapsto\varphi_{T,B}$ is a bijection between the set of
Borel subgroups of $H$ containing $T$ and the set of admissible isomorphisms
$T\isom T_H$.

\smallskip

(c) Note that every admissible isomorphism $\varphi:T\isom T_H$ induces an isomorphism
$\varphi^{-1}_*:X_*(T_H)\isom X_*(T)$ between groups of cocharacters, hence an isomorphism
$\varphi^{-1}_*:L[X_*(T_H)]\isom L[X_*(T)]$ between the corresponding group algebras
for every field $L$. Moreover, since the set of admissible isomorphisms is a $W_H$-torsor (see part~(b)),
its restriction
\[
\varphi_*^{-1}|_{L[X_*(T_H)]^{W_H}}: L[X_*(T_H)]^{W_H}\hra L[X_*(T_H)]\isom L[X_*(T)]
\]
to the algebra of $W_H$-invariants is independent of $\varphi$, and makes $L[X_*(T)]$ a finite $L[X_*(T_H)]^{W_H}$-algebra.
\end{Emp}

\begin{Emp} \label{remadm}
{\bf Remark.} Notice that an isomorphism $\varphi:T\isom T_H$ of tori is admissible if and only if
$\nu_{T_H}\circ\varphi=\nu_H|_{T}$. Indeed, for every admissible isomorphism $\varphi$, the equality $\nu_{T_H}\circ\varphi_{T,B}=\varphi|_T$
follows by the definition of $\nu_{T_H}$ (see Section~\re{weylprop}(d)). Thus the
assertion follows from the fact that $\Aut_{c_H}(T_H)=W_H$, and that the set of
admissible isomorphisms is a $W_H$-torsor.
\end{Emp}

\begin{Emp} \label{E:rationality}
{\bf Rationality.} Let $H$ be a connected reductive group over $k$, defined over its subfield $f\subseteq k$.
Then we claim that the abstract Cartan $T_H$ has a natural structure of a torus over $f$, the Weyl group $W_H$ is equipped with a $\Gm_{f}$-action, the action of $W_H$ on $T_H$ is $\Gm_f$-equivariant, and the subset $S_H\subseteq W_H$ is $\Gm_f$-invariant.
Furthermore, $T_H$ is split over $f$, if $H$ is split over $f$.

\smallskip

\begin{proof}

First we claim that if $H$ splits over $f$, then $T_H$ has a natural structure of a split torus over $f$. Indeed, choose a maximal torus $T$ of $H$ split over $f$ and a Borel subgroup $B\supseteq T$ of $H$. Then the canonical isomorphism $\varphi_{T,B}:T\isom T_H$ from Section~\re{weylprop}(a)) gives to $T_H$ a structure of a split torus over $f$, and it is easy to see that this structure is independent of the choice of $(T,B)$.

\smallskip

Since $H$ splits over $f^{\sep}$, it follows from the proven above that $T_H$ has a natural structure of a torus over $f^{\sep}$. Next we claim that for every element $\tau\in\Gm_f$ of the absolute Galois group of $f$ we have natural isomorphisms ${}^{\tau}T_H\isom T_H$ and $\tau:W_H\isom W_H$.

\smallskip

Choose a maximal torus $T$ of $H$ defined (and hence split) over $f^{\sep}$ and a Borel subgroup $B\supseteq T$ of $H$ (automatically defined over $f^{\sep}$). Then the Galois conjugates ${}^{\tau}B\subseteq H$ and   ${}^{\tau}T\subseteq {}^{\tau}B$ are a Borel subgroup and a maximal torus of $H$ as well, and it is easy to see that the composition
\[
\varphi_{{}^{\tau}T,{}^{\tau}B}\circ {}^{\tau}\varphi_{T,B}^{-1}:{}^{\tau}T_H\isom {}^{\tau}T \isom T_H
\]
is independent of the choice of $(T,B)$.

\smallskip

Similarly, the map $\tau:T\to {}^{\tau}T:t\mapsto {}^{\tau}t$ induces a group isomorphism
\[
\tau:W_{H,T}=N_H(T)/T\isom N_H({}^{\tau}T)/{}^{\tau}T=W_{H,{}^{\tau}T},
\]
and the composition
\[
W_H\isom W_{H,T}\isom W_{H,{}^{\tau}T}\isom W_H
\]
is independent of the choice of $(T,B)$. Finally, the assertion about $S_H$ follows from the fact that isomorphism
$\tau:W_{H,T}\isom W_{H,{}^{\tau}T}$ maps $S_{H,B,T}$ to $S_{H,{}^{\tau}B,{}^{\tau}T}$.
\end{proof}
\end{Emp}

%\begin{Emp} \label{E:levi}
%{\bf Parabolic subgroups.} (a) For each parabolic subgroup
%$P\subseteq H$, we have a canonical isomorphism $T_P\isom T_H$, and
%canonical inclusions $W_P\hra W_H$ and $S_P\hra S_H$. Indeed, this follows from
%the fact that every Borel subgroup $B$ of $P$ is a Borel subgroup of $H$.

%(b) We will call the subset $S_P\subseteq S_H$ the {\em type} of $P$. Note that two
%parabolic subgroups $P$ and $P'$ of $H$ are conjugate is and only
%if $S_P=S_{P'}$.

%(c) Let $M=M_P:=P/R_u(P)$ be the reductive quotient of $P$.
%Then the projection $P\to M$ induces isomorphisms $T_P\isom T_M$,
%$W_P\isom W_M$ and $S_P\isom S_M$.
%\end{Emp}

\subsection{The Grothendieck--Springer resolution and the $W_H$-action} \label{S:spr}

Suppose that we are in the situation of Section \ref{S:prel}.

\begin{Emp} \label{E:whaction}
{\bf Construction.}

\smallskip

(a) Note that we have a commutative diagram
\begin{equation} \label{Eq:spr}
\CD
[\frac{B}{B}]   @>\pr_B>>  T_H \\
@V[p_H]VV       @VV\nu_{T_H}V \\
[\frac{H}{H}]                     @>\nu_H>>   c_H
\endCD
\end{equation}
of Artin stacks, where the top horizontal morphism is a composition of the map $[\frac{B}{B}]\to [\frac{T_H}{T_H}]$, induced by
$\pr_B$ (see Section~\re{abscartan}(c)) with the canonical projection $[\frac{T_H}{T_H}]\to T_H$, and morphism $[p_H]$  \label{a:[ph]} is induced by the inclusion $B\hra H$.

\smallskip

(b) Note that Artin stack $[\frac{B}{B}]$ is smooth of dimension zero, while morphism $[p_H]$ is proper and small.
Therefore for every pair of local systems $\C{F}_1,\C{F}_2$ on $[\frac{B}{B}]$, the
push-forwards $\ov{\C{F}}_1:=[p_H]_*\C{F}_1$  and $\ov{\C{F}}_2:=[p_H]_*\C{F}_2$ are perverse
sheaves on $[\frac{H}{H}]$, and the restriction map
$\Hom(\ov{\C{F}},\ov{\C{F}}')\to\Hom(\ov{\C{F}}_1|_U,\ov{\C{F}}_2|_U)$
is an isomorphism for every open dense substack $U\subseteq[\frac{H}{H}]$.

\smallskip

(c) Let $\C{L}\in D_c^b(T_H,\qlbar)$ be a local system. Since $\nu_{T_H}\circ w=\nu_{T_H}$ for each $w\in W_H$, we
have a canonical isomorphism
\[
a_{w,\C{L}}:(\nu_{T_H})_*\C{L}\isom(\nu_{T_H})_*(w_*\C{L}).
\]
Moreover, for $w_1,w_2\in W_H$, the isomorphism $(w_1w_2)_*\C{L}\simeq (w_1)_*(w_2)_*\C{L}$ identifies
$a_{w_1w_2,\C{L}}$ with $a_{w_1,(w_2)_*\C{L}}\circ a_{w_2,\C{L}}$.

\smallskip

(d) Consider local system $\C{F}_{\C{L}}:=\pr_B^*(\C{L})$  \label{a:fl}  on $[\frac{B}{B}]$, and set $\ov{\C{F}}_{\C{L}}:=[p_H]_*\C{F}_{\C{L}}$.  \label{a:flbar} Notice that over the  regular semisimple locus $c^{\rss}_H$ the
diagram \form{spr} is Cartesian. Therefore it follows from
proper base change that the isomorphism $a_{w,\C{L}}$
from part~(c) induces an isomorphism
\[
a^{\rss}_{w,\C{L}}:\ov{\C{F}}_{\C{L}}|_{[\frac{H^{\rss}}{H}]}\isom\ov{\C{F}}_{w_*\C{L}}|_{[\frac{H^{\rss}}{H}]}.
\]
%Since the diagram Section~\re{spr}  is $H^{ad}$ equivariant, this isomorphism commute with the action of $H^{ad}$.

\smallskip

(e) It now follows from part~(b) that the isomorphism $a^{\rss}_{w,\C{L}}$ from part~(d) uniquely extends
to an isomorphism $\ov{\C{F}}_{\C{L}}\isom\ov{\C{F}}_{w_*\C{L}}$, which we will denote again by
$a_{w,\C{L}}$. \label{a:awl}
\end{Emp}

\begin{Emp} \label{E:spr}
{\bf Springer fibers.}

\smallskip

(a) For a Borel subgroup $B\subseteq H$, we set $\ov{\Fl}_{H,B}:=H/B$. \label{a:ovflhb}  Note that if $B'\subseteq H$ is another Borel subgroup, then there exists
$g\in H$ such that $B'=gBg^{-1}$, and the map $hB\mapsto hBg^{-1}=hg^{-1}B'$ defines an isomorphism $\ov{\Fl}_{H,B}\isom \ov{\Fl}_{H,B'}$,
independent of the choice $g$.

We set $\ov{\Fl}_H:=\lim_B \ov{\Fl}_{H,B}$  \label{a:ovflh} and call it the {\em flag variety}  \label{a:flvar} of $H$. By definition, we have a
canonical isomorphism $\ov{\Fl}_{H}\isom \ov{\Fl}_{H,B}$ for every Borel subgroup $B\subseteq H$.

\smallskip

(b) The group $H$ acts naturally on each $\ov{\Fl}_{H,B}$, and that each isomorphism $\ov{\Fl}_{H,B}\isom \ov{\Fl}_{H,B'}$ is $H$-equivariant.
Therefore the group $H$ naturally acts on $\ov{\Fl}_H$ such that the identification $\ov{\Fl}_{H}\isom \ov{\Fl}_{H,B}$ from part~(a) is $H$-equivariant.

\smallskip

(c) For every element $\gm\in H$, we denote by $\ov{\Fl}_{\gm}\subseteq\ov{\Fl}_H$  \label{a:ovflgm} the scheme of fixed points of
$\gm$, called the {\em Springer fiber}.  \label{a:sprfib} Then for every Borel subgroup $B\subseteq H$, we have a Cartesian diagram
\[
\begin{CD}
\ov{\Fl}_{\gm}  @>>> \ov{\Fl}_H @>>> \on{pt}\\
@V\red_{\gm}VV  @V\red_{\gm}VV @VV\eta_{\gm}V\\
[\frac{B}{B}] @>>>[\frac{H}{B}] @>>>[\frac{H}{H}],
\end{CD}
\]
where $\eta_{\gm}$ is the map corresponding to $\gm$, and $\red_{\gm}$  \label{a:redgm} is the map $hB\mapsto [h^{-1}\gm h]$.

\smallskip

(d) For a local system $\C{L}$ as in Section~\re{whaction}(c), let $\C{F}_{\C{L}}$ be as in Section~\re{whaction}(d), and denote the pullback $\red_{\gm}^*(\C{F}_{\C{L}})\in D_c^b(\ov{\Fl}_{\gm},\qlbar)$ again by $\C{F}_{\C{L}}$.  \label{a:fl2}

Notice that it follows from proper base change applying to the Cartesian diagram of part~(c) that the fiber of $\ov{\C{F}}_{\C{L}}$ at $[\gm]\in[\frac{H}{H}]$ is naturally isomorphic to $H^*(\ov{\Fl}_{\gm}, \C{F}_{\C{L}})$. Therefore the isomorphism $a_{w,\C{L}}$ of Section~\re{whaction}(e) induces an isomorphism
\[
a_{w,\C{L}}: H^*(\ov{\Fl}_{\gm}, \C{F}_{\C{L}})\isom H^*(\ov{\Fl}_{\gm}, \C{F}_{w_*\C{L}}). \label{a:awl2}
\]

%(f) As in (a), all the isomorphisms satisfy cocycle condition
%$a_{\C{L},w_1w_2}=a_{(w_2)^{-1})^*\C{L},w_1}\circ a_{\C{L},w_2}$
%for all $w_1,w_2\in W_H$.
\end{Emp}

\subsection{Characters of Deligne-Lusztig representations}
Let $H$ be a connected reductive group over $\fqbar$, defined over a finite field $\fq$.

\begin{Emp} \label{E:fincor}
{\bf Maximal subtori over finite fields.}

\smallskip

(a) Following
Deligne-Lusztig (\cite{DL}) we consider pairs $(T,B)$,
where $T\subseteq H$ is a maximal torus defined over $\fq$, and
$B\supseteq T$ is a Borel subgroup of $H$ defined over
$\fqbar$. Then ${}^{\si}B\supseteq T$ is another Borel
subgroup of $H$, and we let $w_{B}$ \label{a:wb} be the relative
position $w_{B,{}^{\si}B}\in W_H$ (see Section~\re{abscartan}(d)).

Deligne--Lusztig showed that the correspondence
$(T,B)\mapsto w_{B}$ is a bijection between
conjugacy classes of pairs $(B,T)$ as above and elements of $W_H$.

\smallskip

(b) Recall that the correspondence
$B\mapsto\varphi_{T,B}$ gives a bijection between
Borel subgroups $B\subseteq H$ containing $T$ and
admissible  isomorphisms $T\isom T_H$ (see Section~\re{admis}(b)).
For every  admissible  isomorphisms $\varphi:T\isom T_H$, we
define $B_{\varphi}\supseteq T$  \label{a:bvarphi}  to be the unique Borel
subgroup such that $\varphi_{T,B_{\varphi}}=\varphi$, and
set $w_{\varphi}:=w_{B_{\varphi}}$.  \label{a:wvarphi}

\smallskip

(c) Using the identity
\[
{}^{\si}\varphi_{T,B}=\varphi_{T,{}^{\si}B}=
w_{B,{}^{\si}B}^{-1}\circ \varphi_{T,B}=
w_{B}^{-1}\circ\varphi_{T,B}
\]
(see
Section~\re{admis}(b)) for each pair $(T,B)$ as in part~(a), we get
that ${}^{\si}\varphi=w_{\varphi}^{-1}\circ \varphi$ for
each pair $(T,\varphi)$ as in part~(b).

\smallskip

(d) Let $X_*(T):=\Hom_{\fqbar}(\B{G}_m,T)$  \label{a:xt} and
$X_*(T_H):=\Hom_{\fqbar}(\B{G}_m,T_H)$ be the groups of
cocharacters defined over $\fqbar$. By part~(c), we have an equality
\[
\si\circ\varphi\circ\si^{-1}=w_{\varphi}^{-1}\circ
\varphi:X_*(T)\isom X_*(T_H),
\]
hence an equality
\[
w_{\varphi}\circ\si=\varphi\circ\si\circ\varphi^{-1}:X_*(T_H)\isom
X_*(T_H).
\]
In particular, a torus $T$ is unisotropic if and only if
$X_*(T)^{\si}=\{0\}$ if and only if $
X_*(T_H)^{w_{\varphi}\circ\si}=\{0\}$.
\end{Emp}

\begin{Emp} \label{E:DL}
{\bf Notation.} (a) Let $T\subseteq H$ be a maximal torus defined over $\fq$, let
${\theta}:T(\fq)\to\qlbar\m$   be a character, and let
$\varphi:T\isom T_H$  \label{a:varphi} be an admissible isomorphism (see
Section~\re{admis}(a)).

\smallskip

(b) Recall that character $\theta$ gives rise to the one-dimensional local
system $\C{L}_{\theta}$  \label{a:ltheta} on $T$. Then isomorphism $\varphi$ gives
rise to  local systems
$\C{L}_{\theta,\varphi}:=\varphi_*(\C{L}_{\theta})$  \label{a:lthetavarphi} on $T_H$ and $\C{F}_{\theta,\varphi}:=\C{F}_{\C{L}_{\theta,\varphi}}$ on $[\frac{B}{B}]$,  \label{a:fthetavarphi} hence to the perverse sheaf $\ov{\C{F}}_{\theta,\varphi}:=\ov{\C{F}}_{\C{L}_{\theta,\varphi}}$  \label{a:ovfthetavarphi} on $[\frac{H}{H}]$ (see Section~\re{whaction}(d)).

\smallskip

(c) Note that local system $\C{L}_{\theta}$ is equipped with an isomorphism $\si_*\C{L}_{\theta}\isom \C{L}_{\theta}$. Therefore using Section~\re{fincor}(c) we have
a natural isomorphism
\[
\si_*\C{L}_{\theta,\varphi}=\si_*\varphi_*\C{L}_{\theta}=({}^{\si}\varphi)_*\si_*\C{L}_{\theta}\isom ({}^{\si}\varphi)_*\C{L}_{\theta}=\C{L}_{\theta,{}^{\si}\varphi}=\C{L}_{\theta,w_{\varphi}^{-1}\circ\varphi},
\]
which  induces isomorphisms
$\si_*\C{F}_{\theta,\varphi}\isom
\C{F}_{\theta,w_{\varphi}^{-1}\circ\varphi}$ and $\si_*\ov{\C{F}}_{\theta,\varphi}\isom
\ov{\C{F}}_{\theta,w_{\varphi}^{-1}\circ\varphi}$.

\smallskip

(d) Similarly, for every $w\in W_H$ we have an  isomorphism
$w_*\C{L}_{\theta,\varphi}\isom
\C{L}_{\theta,w\circ\varphi}$, hence by Section~\re{whaction}(e) it
induces an isomorphism
\[
a_{w,\theta,\varphi}=a_{w,\C{L}_{\theta,\varphi}}:\ov{\C{F}}_{\theta,\varphi}\isom \ov{\C{F}}_{\theta,w\circ\varphi}.  \label{a:awthetavarphi}
\]

\smallskip

(e) Composing isomorphisms of parts (c) and (d),  one get an isomorphism
 \[
a_{\theta,\varphi}:\si_*\ov{\C{F}}_{\theta,\varphi}\isom
\ov{\C{F}}_{\theta,w_{\varphi}^{-1}\circ\varphi}\isom \ov{\C{F}}_{\theta,\varphi}.  \label{a:athetavarphi}
\]

\smallskip

(f) As in Section~\re{spr}(d), for every $\gm\in H(\fq)$, we denote the pullback $\red_{\gm}^*\C{F}_{\theta,\varphi}$ on $\ov{\Fl}_{\gm}$ by
$\C{F}_{\theta,\varphi}$. Recall that the fiber of $\ov{\C{F}}_{\theta,\varphi}$ at point $[\gm]\in[\frac{H}{H}]$ is isomorphic to
$H^*(\ov{\Fl}_{\gm},\C{F}_{\theta,\varphi})$. Therefore, taking the fiber at $[\gm]$ of the isomorphism  $a_{\theta,\varphi}$ of part~(e), we get an  isomorphism
\[
a_{\theta,\varphi,\gm}:H^*(\ov{\Fl}_{\gm},\C{F}_{\theta,\varphi})\isom
H^*(\ov{\Fl}_{\gm},\C{F}_{\theta,\varphi}).  \label{a:athetavarphigm}
\]
\end{Emp}

\begin{Emp} \label{E:geom}
{\bf Remark.} Let $\Fr_q:[\frac{H}{H}]\to[\frac{H}{H}]$  \label{a:Frq}  be the
geometric Frobenius morphism over $\fq$. Then we have a
canonical isomorphism $\si_*\ov{\C{F}}_{\theta,\varphi}\isom
\Fr_q^*\ov{\C{F}}_{\theta,\varphi}$. Therefore the isomorphism $a_{\theta,\varphi}$ of Section~\re{DL}(e) can be viewed as
$\Fr^*_q\ov{\C{F}}_{\theta,\varphi}\isom \ov{\C{F}}_{\theta,\varphi}$.
\end{Emp}

The following important result of Lusztig (see \cite[Propositions~8.15 and 9.2]{Lu} or \cite[Corollary~2.3.2]{La}) is crucial for our
work.

\begin{Thm}  \label{T:lus}
Assume that there exists a character $\theta_0:T(\fq)\to\qlbar\m$ in general position, and
let $R_T^{\theta}$ be the virtual Deligne--Lusztig representation of $H(\fq)$,
corresponding to a pair $(T,\theta)$. Then for every $\gm\in H(\fq)$ we have an equality
\begin{equation} \label{Eq;lus}
\Tr(R_T^{\theta}(\gm))=\Tr(a_{\theta,\varphi,\gm}, H^*(\ov{\Fl}_{\gm},\C{F}_{\theta,\varphi})):=\sum_i (-1)^i \Tr(a_{\theta,\varphi,\gm}, H^i(\ov{\Fl}_{\gm},\C{F}_{\theta,\varphi})).
\end{equation}
\end{Thm}

\section{Affine Weyl groups, and homology of the affine Springer fibers}

\subsection{The extended affine Weyl group} \label{S:aff}

\begin{Emp} \label{E:saffsetup}
{\bf Set up.}

\smallskip

(a) Let $K:=k((t))$ be the field of Laurent power series over $k$, and denote by
$\C{O}=\C{O}_K=k[[t]]$ the ring of integers.

\smallskip

(b) Let $G$ be a split reductive connected group over $K$, and let
$G^{\sc}$  \label{a:gsc} be the simply connected covering of the derived group $G^{\der}$  \label{a:gder} of $G$.

\smallskip

(c) %By Iwahori (resp. parahoric) subgroup of $G$ we will mean the corresponding group scheme over $\C{O}$.
For a split torus $T$ over $K$, we denote by $T_{\C{O}}$  \label{a:to}  the corresponding torus
over $\C{O}$ and by $\ov{T}$  \label{a:ovt} the special fiber of $T_{\C{O}}$, also referred as the {\em reduction} of $T$.

\smallskip

(d) Let $LG$ be the loop group of $G$, and let $\T:=L^+(T_{\C{O}})$  \label{a:boldt} be the arc-group.
We denote by $\I$ and $\P$ the Iwahori and the parahoric subgroup schemes of $LG$, respectively.
% and $\P:=L^+(P)$ be the arc-groups
%of $T_{\C{O}}$ and the Iwahori subgroup $I$, and the parahoric subgroup $P$, respectively. We call $\I$ (res. $\P$) the Iwahori (resp. parahoric) subgroup of $LG$, respectively.
\end{Emp}

\begin{Emp} \label{E:affweyl}
{\bf Construction.} Generalizing the construction of Section~\re{abscartan}(a), we are going to define
the extended affine Weyl group  $\wt{W}_G$  of $G$, its subgroup $\La_G\subseteq\wt{W}_G$, the set of simple affine refections $\wt{S}\subseteq \wt{W}_G$,  \label{a:wtwg} the affine Weyl group $\wt{W}\subseteq \wt{W}_G$,  the decomposition $\wt{W}_G\simeq\wt{W}\rtimes\Om_G$,  \label{a:omg} and the torus $\ov{T}_G$  \label{a:ovtg} over $k$.

\smallskip

(a) Consider the set of pairs $(T,\I)$, there $T\subseteq G$ is a
maximal split torus over $K$ and $\I\supseteq \T$ an
Iwahori subgroup of $LG$. We set $\wt{W}_{G,T}:=N_{LG}(LT)/\T=N_{LG}(\T)/\T$,  \label{a:wtwgt} $\La_{G,T}:=LT/\T=N_{LG}(\T)/\T\subseteq\wt{W}_{G,T}$,  \label{a:lagt}
let $\ov{T}_{G,T}$  \label{a:ovtgt} be a torus over $k$ defined to be the reduction of $T$, set $\Om_{G,\I}:=N_{LG}(\I)/\I$,  \label{a:omgi}
and let $\wt{S}_{T,\I}\subseteq \wt{W}_{G,T}$ \label{a:wtsgti}  be the set of simple affine reflections, corresponding to $\I$.

\smallskip

(b) For each two such pairs $(T,\I)$ and $(T',\I')$ as in part~(a),
there exists a unique isomorphism of tori $\phi:T\isom T' $ such that
$\phi$ is the map $t\mapsto gtg^{-1}$ for some $g\in L(G^{\sc})$
such that $g\I g^{-1}=\I'$. Again, isomorphism $\phi$ induces isomorphisms
$\wt{W}_{G,T}\isom \wt{W}_{G,T'}$, $\La_{G,T}\isom \La_{G,T'}$, $\ov{T}_{G,T}\isom \ov{T}_{G,T'}$, $\Om_{G,\I}\isom \Om_{G,\I'}$ and a bijection
$\wt{S}_{T,\I}\isom\wt{S}_{T',\I'}$.

\smallskip

(c) By analogy to the finite-dimensional case, we set $\wt{W}_G:=\lim_{T,\I}\wt{W}_{G,T}$, $\La_G:=\lim_{T,\I}\La_{G,T}$,
$\Om_G:=\lim_{T,\I}\Om_{G,\I}$, $\wt{S}:=\lim_{T,I}\wt{S}_{T,\I}$, and $\ov{T}_G:=\lim_{T,\I}\ov{T}_{G,T}$.

\smallskip

(d) We denote by $\wt{W}\subseteq\wt{W}_G$  \label{a:wtw} the subgroup, generated by $\wt{S}$. It is easy to see that
the subgroup $\wt{W}\subseteq\wt{W}_G$ is normal, that the canonical morphism $\wt{W}\rtimes\Om_G\to \wt{W}_G$
is an isomorphism, and that we have a canonical isomorphism $\wt{W}\simeq \wt{W}_{G^{\sc}}$.

\smallskip

(e) For every proper subset $J\subsetneq\wt{S}$, we denote by $W_J\subseteq\wt{W}$  \label{a:wj} the subgroup, generated by $J$.
\end{Emp}

\begin{Emp} \label{E:prop}
{\bf Properties.}

\smallskip

(a) For each pair $(T,\I)$ as in Section~\re{affweyl}(a), we have a canonical
isomorphisms $\varphi_{T,\I}:\ov{T}\isom\ov{T}_G$,  \label{a:varphiti}  $\wt{W}_{G,T}\isom \wt{W}_{G}$, $\La_{G,T}\isom \La_{G}$ and a bijection $\wt{S}_{T,\I}\isom \wt{S}$.

\smallskip

(b) We denote by $\ov{W}=\ov{W}_G:=\wt{W}_G/\La_G$  \label{a:w} the quotient group and
by $\pi_G:\wt{W}_G\to \ov{W}_G$  \label{a:pig} the natural projection. The group
$\ov{W}$ naturally and faithfully acts on $\ov{T}_G$. For each pair $(T,\I)$ as in Section~\re{affweyl}(a),
the isomorphisms of part~(a) induce an isomorphism $\ov{W}\isom N_{LG}(LT)/LT\simeq N_G(T)/T$.

\smallskip

(c) We claim that there is a natural isomorphism of groups $\La_G\simeq X_*(\ov{T}_G)$, where $X_*(-)$  \label{a:x*}
denotes the group of cocharacters (over an algebraic closure). Namely, for every pair $(T,\I)$ as in Section~\re{affweyl}(a), we have
an isomorphism $X_*(T)\isom LT/\T$ of groups, which maps $\la\in X_*(T)$ to the class of $\la(t)\in T(k((t)))=LT(k)$.
Moreover, this isomorphism together with those of part~(a) gives rise to the composition
\[
\La_G\isom \La_{G,T}=LT/\T\isom X_*(T)\isom X_*(\ov{T})\isom X_*(\ov{T}_G),
\]
which is independent of the choice of $(T,\I)$.

\smallskip

(d) For every pair $(T,\I)$ as in Section~\re{affweyl}(a), the composition $\T\hra \I\to \I/\I^{+}$, where $\I^+\subseteq \I$ is the pro-unipotent radical,  \label{a:i+} induces an isomorphism $\ov{T}\isom \I/\I^{+}$ of tori over $k$. Moreover, the composite isomorphism  $\ov{T}_G\overset{\varphi_{T,\I}^{-1}}{\lra}\ov{T}\isom \I/\I^{+}$ is independent of $T$.
We denote by $\pr_{\I}$ the composition $\I\to\I/\I^+\isom \ov{T}_G$.  \label{a:pri}

\smallskip

(e) For every two Iwahori subgroups $\I,\I'$ of $G$, we can
define their relative position $\wt{w}_{\I,\I'}\in \wt{W}$  \label{a:wtwii'}  to be
the image of an element $g\in LG^{\sc}$ such that $g\I g^{-1}=\I'$ under the
composition $LG^{\sc}\to \I\bs LG^{\sc}/\I\isom\wt{W}$. Also we denote by $\ov{w}_{\I,\I'}\in \ov{W}$ the
projection of $\wt{w}_{\I,\I'}\in\wt{W}$ to $\ov{W}$. As in the finite-dimensional case, for
every three Iwahori subgroups $\I,\I',\I''$ containing $\T$,  we have an equality
$\wt{w}_{\I,\I''}=\wt{w}_{\I,\I'}\cdot\wt{w}_{\I',\I''}$.
\end{Emp}

\begin{Emp} \label{E:canonical}
{\bf A canonical isomorphism.}

\smallskip

Consider triple $(T,\I,B)$, where $(T,\I)$ is an Section~\re{affweyl}(a), and $B\supseteq T$ is a Borel subgroup of $G$. This triple gives rise to an isomorphism of groups
\[
\phi: W_G\overset{\re{weylprop}(a)}{\lra} W_{G,T}=N_G(T)/T\overset{\re{prop}(b)}{\lra}\ov{W}_G
\]
and to a $\phi$-equivariant isomorphism of groups
\[
\psi: X_*(T_G) \overset{\varphi^{-1}_{T,B}}{\lra} X_*(T)\isom X_*(\ov{T})\overset{\varphi_{T,\I}}{\lra}X_*(\ov{T}_G) \overset{\re{prop}(c)}{\lra}\La_G,
\]
where  $X_*(T)\isom X_*(\ov{T})$ is a canonical isomorphism. Moreover, $\psi$ gives rise to a $\phi$-equivariant isomorphism of groups algebras
$L[X_*(T_G)]\isom L[\La_G]$ for every field $L$, hence to an isomorphism of algebras
\begin{equation} \label{Eq:can}
L[X_*(T_G)]^{W_G}\isom L[\La_G]^{\ov{W}_G}.
\end{equation}

Furthermore, isomorphism \form{can} is independent of the choice of $(T,\I,B)$.
\end{Emp}

\begin{Emp}  \label{E:admisaff}
{\bf Admissible isomorphisms}.

\smallskip

(a) Let $T\subseteq G$ be a split
maximal torus over $K$. We say that an isomorphism
$\varphi:\ov{T}\isom \ov{T}_G$ is {\em admissible},  \label{a:admisom2} if there
exists an Iwahori subgroup $\I\supseteq\T$ of $G$ such that
$\varphi=\varphi_{T,\I}$ (see Section~\re{prop}(a)).

\smallskip

(b) For every two Iwahori subgroups $\I,\I'\supseteq\T$, we
have a formula $\varphi_{T,\I'}=\ov{w}_{\I,\I'}^{-1}\circ
\varphi_{T,\I}$. In particular, every two admissible isomorphisms
$\varphi:\ov{T}\isom \ov{T}_G$ are $\ov{W}$-conjugate.
\end{Emp}

\begin{Emp} \label{E:parahoric}
{\bf Parahoric subgroups.}

\smallskip

(a) For each parahoric subgroup $\P$
of $LG$, let $\P^+\subseteq \P$  \label{a:p+} be the pro-unipotent radical.
Then $M_{\P}:=\P/\P^+$  \label{a:mp} is a connected reductive group over $k$.

\smallskip

(b) For every pair $(T,\I)$ as in Section~\re{affweyl}(a) satisfying $\I\subseteq\P$, the quotient $\ov{B}_{\I,\P}:=\I/\P^+$  \label{a:ovbip} is a Borel subgroup of $M_{\P}$, and the composition
\[
\T\hra\I\to\I/\P^{+}=\ov{B}_{\I,\P}
\]
induces an embedding $\ov{T}\hra \ov{B}_{\I,\P}$, identifying $\ov{T}$ with a maximal torus of  $\ov{B}_{\I,\P}$.

\smallskip

(c) We claim that there is a natural isomorphism $\ov{T}_G\isom T_{M_{\P}}$ and a natural embedding $W_{M_\P}\hra \wt{W}_G$.
Namely, for every pair $(T,\I)$ as in part~(b), the projection $\P\to M_{\P}$ induces isomorphisms
\begin{equation} \label{Eq:abscartan}
\ov{T}_G\overset{\re{prop}(d)}{\lra}\I/\I^{+}\isom \ov{B}_{\I,\P}/R_u(\ov{B}_{\I,\P})\overset{\re{weylprop}(b)}{\lra}T_{M_{\P}}
\end{equation}
and
\[
N_{\P}(\T)/\T\isom N_{M_{\P}}(\ov{T})/\ov{T}\overset{\re{weylprop}(a)}{\lra}W_{M_{\P}}.
\]
Moreover, both isomorphism \form{abscartan} and the composition
\[
W_{M_{\P}}\isom N_{\P}(\T)/\T\hra N_{LG}(\T)/\T\overset{\re{prop}(a)}{\lra}\wt{W}_G
\]
are independent of the choice of $(T,\I)$.

By construction, the embedding $W_{M_\P}\hra \wt{W}_G$ maps $S_{M_{\P}}$ into $\wt{S}$.

(d) We denote by $J_{\P}\subseteq\wt{S}$  \label{a:jp} the image of the embedding $S_{M_{\P}}\hra\wt{S}$ from part~(c) and call $J_{\P}$ the {\em type} of $\P$. Then the embedding $W_{M_{\P}}\hra\wt{W}_G$ induces an isomorphism $W_{M_{\P}}\isom W_{J_{\P}}$.
%Note that for two parabolic subgroups $\P_1$ and
%$\P_2$ are conjugate is and only if $S_{\P_1}=S_{\P_2}$.

\smallskip

(e) For every Iwahori subgroup $\I\subseteq LG$, the correspondence $\P\mapsto J_{\P}$ induces
a bijection between parahoric subgroups $\P\supseteq\I$ and proper subsets $J\subsetneq\wt{S}$. The inverse map we denote
by $J\mapsto\P_{J,\I}$.

\smallskip

%(f) For every subset $J\subsetneq\wt{S}$ and two pairs $(T,\I)$ and $(T',\I')$ as in Section~\re{affweyl}, we have an equality
%$W_{\P_{J,\I}}=W_{\P_{J,\I'}}\subseteq\wt{W}_G$, and we denote this subgroup simply by $W_J$.
%\end{Emp}

%\begin{Emp} \label{E:compat}
%{\bf Compatibility.}

\smallskip

(f) Let $\I,\I'$ be two Iwahori subgroups of $LG$  contained in
$\P$. Then $\ov{B}_{\I,\P}$ and $\ov{B}_{\I',\P}$ are
two Borel subgroups of $M_{\P}$ (see part~(b)), and its relative position
$\wt{w}_{\I,\I'}\in\wt{W}_{G}$ equals the image of $w_{\ov{B}_{\I,\P},\ov{B}_{\I',\P}}\in
W_{M_{\P}}$ under the embedding of part~(c).

Indeed, choose $g\in \P$ such that $g\I g^{-1}=\I'$.
Then the image $\ov{g}\in M_{\P}$ of $g$ satisfies
$\ov{g}\ov{B}_{\I,\P}\ov{g}^{-1}=\ov{B}_{\I',\P}$, so the assertion follows from
definitions of relative positions (see Sections~\re{abscartan}(d) and \re{affweyl}(e)).
\end{Emp}

\begin{Emp} \label{E:affine roots}
{\bf Affine roots.} To every $G$ as above, one can associate the set of roots $\ov{\Phi}=\ov{\Phi}_G$ of $G$, the set of affine roots $\wt{\Phi}=\wt{\Phi}_G$ and a subset of positive affine roots $\wt{\Phi}^+\subseteq\wt{\Phi}$:

\smallskip

(a) To every pair $(T,\I)$ as in Section~\re{affweyl}(a) we denote by $\ov{\Phi}_T\subseteq X^*(\ov{T}_G)$ the image of the set of roots
$\Phi(G,T)\subseteq X^*(T)$ under the canonical isomorphism $X^*(T)\simeq X^*(\ov{T})\simeq X^*(\ov{T}_G)$.

The torus $T$ acts on the Lie algebra $\frak{g}=\Lie G$, which induces a decomposition $\frak{g}=\bigoplus_{\ov{\al}\in \{0\}\cup\Phi_T}\frak{g}_{\ov{\al}}$.
%$\frak{I}:=\Lie\I$ decomposes as $\frak{I}=\bigoplus_{\al}\frak{I}_{\al}$, where each $\frak{I}_{\al}\subseteq \frak{G}_{\al}$ is a $\C{O}$-module.
Next, we denote by $\wt{\Phi}_{T}$ be the set of pairs $\al=(\ov{\al},\C{L})$, where $\ov{\al}\in\ov{\Phi}_T$ and $\C{L}\subseteq \frak{g}_{\ov{\al}}$ is an $\C{O}$-lattice. Finally, we denote by $\wt{\Phi}^+_{T,\I}\subseteq \wt{\Phi}_{T}$ be the set of affine roots $\al=(\ov{\al},\C{L})$ such that $\C{L}$ is contained in $\frak{I}:=\Lie\I$.

\smallskip

(b) Note that  if $(T',\I')$ is another pair as in Section~\re{affweyl}(a), then the canonical isomorphism $T\isom T'$ from Section~\re{affweyl}(b) induces bijections $\ov{\Phi}_T\isom \ov{\Phi}_{T'}$, $\wt{\Phi}_T\isom \wt{\Phi}_{T'}$ and $\wt{\Phi}^+_{T,\I}\isom \wt{\Phi}^+_{T',\I'}$.

Thus, as in Section~\re{affweyl}(b), we set  $\ov{\Phi}:=\lim_{(T,\I)}\ov{\Phi}_{T}$, $\wt{\Phi}:=\lim_{(T,\I)}\wt{\Phi}_{T}$ and
$\wt{\Phi}^+:=\lim_{(T,\I)}\wt{\Phi}^+_{T,\I}$. For $\al\in\wt{\Phi}$ we write $\al>0$ (resp. $\al<0$) if
$\al\in\wt{\Phi}^+$ (resp. $\al\notin\wt{\Phi}^+$).
\end{Emp}

\begin{Emp} \label{E:splitting}
{\bf Splitting.}
Let $\P\subseteq LG$ be a hyperspecial parahoric subgroup.
\smallskip

(a) We claim that a choice of $\P$ gives rise to a subset $\ov{\Phi}^+\subseteq\ov{\Phi}$ of {\em positive} roots and an identification $\wt{\Phi}\simeq \ov{\Phi}\times\B{Z}$ such that $\wt{\Phi}^+$ corresponds to pairs $(\al,r)\in \ov{\Phi}\times\B{Z}$ such that either $r>0$ or ($r=0$ and $\al\in \ov{\Phi}^+$).

\smallskip

Namely, let $(T,\I)$ be as in Section~\re{affine roots}(a) such that $\I\subseteq\P$. Since $\P$ is hyperspecial, the canonical isomorphism
$X^*(\ov{T}_G)\simeq X^*(\ov{T})$ induces a bijection between $\ov{\Phi}_T$ and the set of roots $\Phi(M_{\P},\ov{T})$ of $M_{\P}$ with respect to $\ov{T}$. Then a subset $\ov{\Phi}^+\subseteq\ov{\Phi}$ corresponds under this bijection to a subset $\Phi^+(M_{\P},\ov{T},\ov{B}_{\I,\P})$ of positive roots $\Phi(M_{\P},\ov{T})$ relative to the Borel subgroup $\ov{B}_{\I,\P}=\I/\P^+\subseteq M_{\P}$.

Moreover, the bijection   $\wt{\Phi}_G\simeq \ov{\Phi}\times\B{Z}$ is characterised by the condition that $(\ov{\al},\C{L})\in\wt{\Phi}_T$
corresponds to a pair $(\ov{\al},r)\in \ov{\Phi}_T\times\B{Z}$, where $r$ is the smallest integer such that $t^r\C{L}\subseteq\Lie\P$.
\smallskip

(b) A choice of $\P$ defines a section of the projection $\pi_G:\wt{W}_G\to \ov{W}$, thus gives rise to decompositions $\wt{W}_G\simeq \La_G\rtimes\ov{W}$ and $\wt{W}\simeq \La\rtimes\ov{W}$, where we set $\La:=\La_{G^{\sc}}$. Namely, for each pair $(T,\I)$ as in part~(a), notice that the composition
\begin{equation*}
\zeta: N_{\P}(\T)/\T\hra N_{LG}(\T)/\T= N_{LG}(LT)/\T\to N_{LG}(LT)/LT\overset{\re{prop}(b)}{\underset{\sim}{\lra}}\ov{W}
\end{equation*}
is an isomorphism, and the composition
\[
s:\ov{W}\overset{\zeta^{-1}}{\underset{\sim}{\lra}}  N_{\P}(\T)/\T\hra N_{LG}(\T)/\T \overset{\re{prop}(a)}{\underset{\sim}{\lra}}\wt{W}_G
\]
is a section of $\pi_G$. Moreover, section $s$ is independent of the choice of $(T,\I)$.
\end{Emp}

The following variant of a notion of \cite{BV} plays a central role in \rt{inj} below.

\begin{Emp} \label{E:mreg}
{\bf $m$-regularity.}

\smallskip

(a) For every elements $w\in\wt{W}$ and $\ov{\al}\in\ov{\Phi}$, we denote by $|\lan\ov{\al},w\ran|$ the number of all affine roots
$\al\in \wt{\Phi}$ such that $\al=(\ov{\al},\C{L})$ and we have
\[
\text{ either }(\al>0\text{ and }w^{-1}(\al)<0)\text{ or }(\al<0\text{ and }w^{-1}(\al)>0).
\]

\smallskip

(b) Notice that the expression $|\lan\ov{\al},w\ran|$ coincides with the absolute value of the usual pairing $\lan\ov{\al},w\ran$
when $w=\la\in \La\subseteq\wt{W}$.

\smallskip

(c) Let $m\in\B{Z}_{\geq 0}$. We say that an element $w\in\wt{W}$ is {\em $m$-regular}  \label{a:mregular} if for every $\ov{\al}\in\ov{\Phi}$ we have $|\lan\ov{\al},w\ran|\geq m$.

\smallskip

(d) By definition, for every elements $w\in\wt{W}$ and $\ov{\al}\in\ov{\Phi}$, we have an equality
\[
|\lan\ov{\al},w\ran|=|\lan w^{-1}(\ov{\al}),w^{-1}\ran|.
\]
Therefore an element $w$ is $m$-regular if and only if element $w^{-1}$ is $m$-regular.
\end{Emp}

\begin{Emp} \label{E:mreggeom}
{\bf Geometric description.} The expression $|\lan\ov{\al},w\ran|$ from Section~\re{mreg}(a) has the following geometric description:

\smallskip

Choose a pair $(T,\I)$ as in Section~\re{affine roots}(a), then we have an equality $\Ad g(\frak{I})=\frak{I}$ for every $g\in\T$,
 thus we can form an $\C{O}$-lattice $\Ad w(\frak{I})\subseteq\fg$, and we set $\frak{I}_{\ov{\al}}:= \frak{I}\cap\frak{g}_{\ov{\al}}$ and $\frak{I}_{w,\ov{\al}}:= \Ad w(\frak{I})\cap\frak{g}_{\ov{\al}}$.

Then $\frak{I}_{\ov{\al}}$ and $\frak{I}_{w,\ov{\al}}$ be two $\C{O}$-lattices in a one-dimensional vector $K$-space $\frak{g}_{\ov{\al}}$,
and the expression $|\lan\ov{\al},w\ran|$ is equal to  $\dim_k(\frak{I}_{\ov{\al}}/\frak{I}_{w,\ov{\al}})$, if  $\frak{I}_{w,\ov{\al}}\subseteq \frak{I}_{\ov{\al}}$, and is equal to $\dim_k(\frak{I}_{w,\ov{\al}}/\frak{I}_{\ov{\al}})$, if  $\frak{I}_{w,\ov{\al}}\supseteq \frak{I}_{\ov{\al}}$.
\end{Emp}

\begin{Emp} \label{E:remmreg}
{\bf $m$-regularity in the sense of \cite{BV}.}

\smallskip

(a) Assume that we are given a parahoric subgroup $\P\subseteq LG$ thus a splitting $\wt{W}\simeq\La\rtimes \ov{W}$ (see Section~\re{splitting}(b)). In this case, we have another (more elementary) notion of $m$-regularity, used in \cite{BV}.
Namely, an element $\la u\in\wt{W}$ with $\la\in\La$ and $u\in \ov{W}$ is called $m$-regular in the sense of \cite[Introduction]{BV} if for every
$\ov{\al}\in \ov{\Phi}$, we have an inequality $|\lan\ov{\al},\la\ran|\geq m$.

\smallskip

(b) Clearly, the notion of $m$-regularity introduced in Section~\re{mreg} is not equivalent to that of \cite[Introduction]{BV}.
Moreover, the notion of \cite{BV} depends on a choice of parahoric subgroup $\P\subseteq LG$, and thus less canonical and does not satisfy a crucial property of Section~\re{affweyl2}(d) we need.

\smallskip

(c) In the situation of part~(a), for every $\la\in\La, u\in \ov{W}$ and
$\ov{\al}\in \ov{\Phi}$, we have an inequality
\[
|\lan\ov{\al},\la\ran|-1\leq |\lan\ov{\al},\la u\ran|\leq |\lan\ov{\al},\la\ran|+1.
\]
Therefore in an element $w\in\wt{W}$ is $m$-regular in the sense of Section~\re{mreg}, then it is $(m-1)$-regular in the sense of \cite[Introduction]{BV}, and if it is $m$-regular in the sense of \cite[Introduction]{BV}, then it is $(m-1)$-regular in the sense of
Section~\re{mreg}. Moreover, by Section~\re{mreg}(b), the two notions coincide when $w=\la\in\La$.

%(d) Using the description of (a) one get that if $w$ is $m$-regular, then for every $s\in\wt{S}$ the element $sw$ is $(m-1)$-regular.
\end{Emp}

%defines decompositions $\wt{W}_G\isom \La_G\rtimes W$,
%$\wt{W}_G\isom \La\rtimes W$ and

%\begin{Emp} \label{E:splitting}
%{\bf Splitting.} (a) Notice that if $\P$ is a hyperspecial parahoric
%subgroup of $LG$, then the natural composition
%$W_{\C{P}}\hra\wh{G}\to W$ is an isomorphism. Thus we have
%(non-canonical) $\Gm_k$-equivariant  isomorphism
%$W\rtimes\La\isom W_{P}\rtimes\La\isom\wt{W}_G$.

%(b) The choice of hyperspecial $\P$ also give us
%(non-caninical) isomorphisms $\wt{T}_{G}\isom T_G$ and $W\isom
%W$.
%\end{Emp}

\subsection{Affine Springer fibers and $\wt{W}_G$-action.}

Suppose that we are in the situation of Section \ref{S:aff}.

\begin{Emp} \label{E:affflvar}
{\bf Affine flag variety.}

\smallskip

(a) For every Iwahori subgroup schemes $\I$ of $LG$, we set
$\Fl_{G,\I}:=LG/\I$. For every pair of Iwahori subgroups $\I,\I'$ of $LG$ there exists $h\in L(G^{\sc})$ such that
$h\I'h^{-1}=\I$. Then the map $g\I\mapsto g\I h=gh\I'$ defines an $LG$-equivariant isomorphism $\Fl_{G,\I}\isom \Fl_{G,\I'}$, independent of
the choice of $h$.

Thus we can form a projective limit $\Fl_G:=\lim_\I\Fl_{G,\I}$ and call it the {\em affine flag variety} of $G$.
In particular, for every Iwahori subgroup $\I$ of $LG$, we have a canonical $LG$-equivariant isomorphism $\Fl_G\isom LG/\I$.

\smallskip

(b) Notice that for every Iwahori subgroup $\I$ of $LG$, the maps $h\mapsto (g\I\mapsto hg\I)$ and $u\mapsto (g\I\mapsto g\I u^{-1}=gu^{-1}\I)$
define commuting actions of $LG$ and $\Om_{G,\I}$ in $\Fl_{G,\I}$. Furthermore,  these actions are compatible with isomorphisms
$\Fl_{G,\I}\isom \Fl_{G,\I'}$ and $\Om_{G,\I}\isom\Om_{G,\I'}$, thus induce an action of $LG\times\Om_{G}$ on $\Fl_{G}$.

\smallskip

(c) More generally, to every proper subset $J\subsetneq \wt{S}$, and every Iwahori subgroup $\I$ of $LG$,  we associate an ind-scheme $\Fl_{G,J,\I}:=LG/\P_{J,\I}$, where parahoric subgroup $\P_{J,\I}$ was defined in Section~\re{parahoric}(e).
As in part~(a), for every two Iwahori subgroups $\I,\I'$ of $LG$ we have a canonical isomorphism $\Fl_{G,J,\I}\isom \Fl_{G,J,\I'}$, and
we denote by $\Fl_{G,J}$ the limit $\Fl_{G,J}:=\lim_\I\Fl_{G,J,\I}$. We have a natural $LG$-equivariant
projection $\pi_J:\Fl_G\to\Fl_{G,J}$.
\end{Emp}

\begin{Emp} \label{E:affsprfib}
{\bf Affine Springer fibers.}

\smallskip

(a) For every element $\gm\in LG(k)=G(K)$, we denote by  $\Fl_{G,\gm}$ (resp. $\Fl_{G,J,\gm}$) the closed  sub-ind-scheme of $\Fl_G$ (resp. $\Fl_{G,J}$) of $\gm$-fixed points. Explicitly, for every Iwahori subgroup $\I$, the affine Springer fiber  $\Fl_{G,J,\gm}$ consists of all $g\P_{J,\I}\in LG/\P_{J,\I}$ such that $g^{-1}\gm g\in\P_{J,\I}$. We denote the
natural projection $\Fl_{G,\gm}\to\Fl_{G,J,\gm}$ by $\pi_J$.

\smallskip

(b) We have a natural projection map $\red_{\gm}:\Fl_{G,\gm}\to
\ov{T}_G$, which sends $g\I$ to $\pr_\I(g^{-1}\gm g)$.  For each local system
$\C{L}$ on $\ov{T}_G$, we set $\C{F}_{\C{L}}:=\red_{\gm}^*(\C{L})$.

\smallskip

(c) For every $g\in LG$, the (left) action of $g$ on $\Fl_{G,J}$ induces an isomorphism $l_g:\Fl_{G,J,\gm}\isom\Fl_{G,J,g\gm g^{-1}}$.
Moreover, the projection  $\red_{\gm}:\Fl_{G,\gm}\to\ov{T}_G$ of part~(b) decomposes as a composition
$\red_{g\gm g^{-1}}\circ l_g:\Fl_{G,\gm}\to\Fl_{G,g\gm g^{-1}}\to\ov{T}_G$.

\smallskip

(d) Notice that since the actions of $LG$ and $\Om_G$ on $\Fl_G$ commute,  for every $u\in\Om_G$ the automorphism $u:\Fl_G\to\Fl_G$ preserves
$\Fl_{G,\gm}$. Moreover, the following diagram is commutative
\begin{equation} \label{Eq:CDredgm}
\begin{CD}
\Fl_{G,\gm} @>\red_{\gm}>> \ov{T}_G\\
@VuVV @VV\ov{u}V\\
\Fl_{G,\gm} @>\red_{\gm}>> \ov{T}_G.
\end{CD}
\end{equation}
Indeed, this  follows from the identity $\pr_\I(ug^{-1}\gm gu^{-1})=\ov{u}(\pr_{\I}(g^{-1}\gm g))$.

\smallskip

(e) For each $i\in\B{N}$ and  every local system $\C{L}$ on $\ov{T}_G$ we can consider the $i$-th homology group
$H_i(\Fl_{G,\gm},\C{F}_{\C{L}})$. Then for every element $g\in LG$ the isomorphisms $l_g$ of part~(c) induce isomorphisms
\[
l_g:H_i(\Fl_{G,\gm},\C{F}_{\C{L}})\isom H_i(\Fl_{G,g\gm g^{-1}},\C{F}_{\C{L}}).
\]

\smallskip

(f) Let $G_{\gm}\subseteq G$ be the centralizer of $\gm$. Then each subscheme $\Fl_{G,\gm}\subseteq \Fl_G$ is $LG_{\gm}$-invariant, and the projection $\red_{\gm}$ is $LG_{\gm}$-equivariant with respect to the trivial action of
$LG_{\gm}$ on $\ov{T}_G$ (see part~(c)). Therefore the local system $\C{F}_{\C{L}}$ from part~(b) is $LG_{\gm}$-equivariant, and the induced action of $LG_{\gm}$ on $H_i(\Fl_{G,\gm},\C{F}_{\C{L}})$ is given by isomorphisms $l_g$ from part~(e).
\end{Emp}

\begin{Emp} \label{E:triv}
{\bf Remark.} Note that the image of the projection $\red_{\gm}:\Fl_{G,\gm}\to\ov{T}_G$ is a single $\ov{W}$-orbit. In particular, $\im\red_{\gm}\subseteq \ov{T}_G$ is finite, thus there exists a
(non-canonical) isomorphism $\C{L}|_{\im\red_{\gm}}\simeq \qlbar^{\rk\C{L}}$, hence $\C{F}_{\C{L}}\simeq \qlbar^{\rk\C{L}}$.
\end{Emp}

\begin{Emp} \label{E:regss}
{\bf Main particular case.}
We will be mostly interested in the case, when $\gm\in LG(k)\simeq G(K)$ is a  regular semisimple  element of $G$.
Then, arguing as in \cite{KL}, each reduced ind-scheme $(\Fl_{G, J,\gm})_{\red}$ is actually a scheme, locally of finite type over $k$.
\end{Emp}

 \begin{Emp} \label{E:wsaction}
{\bf The $W_J$-action and $\Om_G$-action on $H_i(\Fl_{G,\gm},\C{F}_{\C{L}})$}.

\smallskip

For every proper subset $J\subsetneq\wt{S}$ and an Iwahori subgroup $\I\subseteq LG$, we set $\P_J:=\P_{J,\I}$, $M_J:=M_{\P_J}$, and
$\ov{B}_J:=\I/\P_{J}^+\subseteq M_J$.

\smallskip

(a) Note that we have the following natural Cartesian diagram
\begin{equation} \label{Eq:SGaff}
\CD
\Fl_{G,\gm}   @>\wt{\red}_{\gm}>>  [\frac{\I}{\I}] @>>>  [\frac{\ov{B}_J}{\ov{B}_J}] @>\pr_{\ov{B}_J}>> \ov{T}_G\\
@V\pi_J VV      @VV[\pi_{\P_J}]V       @VV[p_{M_J}]V \\
\Fl_{G,J,\gm}  @>\wt{\red}_{J,\gm}>> [\frac{\P_{J}}{\P_{J}}] @>>>  [\frac{M_J}{M_J}].
\endCD
\end{equation}

\smallskip

(b) Recall that we have a natural isomorphism $W_J\simeq W_{M_J}$ (see Section~\re{parahoric}(d)). Using the Cartesian diagram of part~(a) and proper base change one gets that for every $w\in W_J\simeq W_{M_J}$ the isomorphism $a_{w,\C{L}}$ from Section~\re{whaction}(e) induces an isomorphism
\[
a_{J,w,\C{L}}:(\pi_J)_*\C{F}_{\C{L}}\isom (\pi_J)_*\C{F}_{\ov{w}_*\C{L}}
\]
on $\Fl_{G,J,\gm}$.

\smallskip

(c) Using identifications $H_i(\Fl_{G,\gm},\C{F}_{\C{L}})=H_i(\Fl_{G,J,\gm},(\pi_J)_*\C{F}_{\C{L}})$, the isomorphisms of part~(b) induce isomorphisms
\[
a_{J,w,\C{L}}:H_i(\Fl_{G,\gm},\C{F}_{\C{L}})\isom H_i(\Fl_{G,\gm},\C{F}_{\ov{w}_*\C{L}}) \text{ for each }w\in W_J.
\]
Moreover, these isomorphisms are independent of the choice of $\I$.

\smallskip

(d) Using Section~\re{whaction}(f), the isomorphisms of part~(c) satisfy
\[
a_{J,w_1w_2,\C{L}}=a_{J,w_1,\ov{w_2}_*\C{L}}\circ a_{J,w_2,\C{L}}\text{  for all }w_1,w_2\in W_J.
\]

\smallskip

(e) Using commutative diagram \form{CDredgm}, for every $u\in\Om_G$, one get an isomorphism $u_*(\C{F}_{\C{L}})\simeq \C{F}_{\ov{u}_*(\C{L})}$ of local systems on $\Fl_{G,\gm}$, hence an isomorphism
\[
a_{\Om,u,\C{L}}:H_i(\Fl_{G,\gm},\C{F}_{\C{L}})\isom H_i(\Fl_{G,\gm},\C{F}_{\ov{u}_*\C{L}}).
\]
\smallskip

(f) By construction, the isomorphisms of part~(e) satisfy
\[
a_{\Om,u_1u_2,\C{L}}=a_{\Om,u_1,\ov{u_2}_*\C{L}}\circ a_{\Om,u_2,\C{L}}\text{  for all }u_1,u_2\in\Om_G.
\]

\end{Emp}

\begin{Emp} \label{E:truncation}
{\bf Truncated version.}

\smallskip

(a) Since projection $\pi_J$ is proper, for every closed subscheme $Y\subseteq\Fl_G$ the image $Y_J:=\pi_J(Y)\subseteq\Fl_{G,J}$ is closed as well. Hence we have a commutative diagram
\begin{equation} \label{Eq:Y}
\CD
Y_{\gm}   @>>>   \Fl_{G,\gm}\\
@V\pi^Y_{J,\gm} VV       @VV\pi_{J,\gm}V \\
Y_{J,\gm}   @>>>    \Fl_{G,J,\gm},
\endCD
\end{equation}
where we put $Y_{\gm}:=Y\cap \Fl_{G,\gm}$ and $Y_{J,\gm}:=Y_J\cap \Fl_{G,J,\gm}$.

\smallskip

(b) Assume furthermore that $Y=\pi_J^{-1}(Y_J)$. Then the diagram \form{Y} is Cartesian, so repeating the argument of Section~\re{wsaction}, for every  $w\in W_J$ we get a natural isomorphism
\[
a^Y_{J,w,\C{L}}:H_i(Y_{\gm},\C{F}_{\C{L}})\isom H_i(Y_{\gm},\C{F}_{\ov{w}_*\C{L}}),
\]
lifting the isomorphism $a_{J,w,\C{L}}$ from Section~\re{wsaction}(c).

\smallskip

(c) In particular, in the situation of part~(b), the isomorphism $a_{J,w,\C{L}}$ from Section~\re{wsaction}(c) induces an isomorphism
\[
H'_i(Y_{\gm},\C{F}_{\C{L}})\isom H'_i(Y_{\gm},\C{F}_{\ov{w}_*\C{L}}),
\]
where  $H'_i(Y_{\gm},\C{F}_{\C{L}})$ denotes the image of the canonical map $H_i(Y_{\gm},\C{F}_{\C{L}})\to H_i(\Fl_{G,\gm},\C{F}_{\C{L}})$.
\end{Emp}

\begin{Prop} \label{P:whaction} The isomorphisms of Sections~\re{wsaction}(c),(e) and \re{affsprfib}(e) have the following properties:

\smallskip

\noindent (a) For proper subsets $J_1\subseteq J_2\subsetneq \wt{S}$ and an element $w\in
W_{J_1}\subseteq W_{J_2}$, the isomorphisms  $a_{J_1,w,\C{L}}$  and
$a_{J_2,w,\C{L}}$ coincide.

\smallskip

\noindent (b) For a proper subset $J\subsetneq \wt{S}$ and elements $w\in W_J$ and $u\in\Om_G$, the element $uwu^{-1}$ belongs to $W_{u(J)}$, and we have an equality of isomorphisms
\[
a_{u(J),uwu^{-1},\C{L}}=a_{\Om,u,\ov{w}_*(\ov{u}^{-1})_*\C{L}}\circ a_{J,w,(\ov{u}^{-1})_*\C{L}}\circ a_{\Om,u^{-1},\C{L}}.
\]

\smallskip

\noindent(c) There exists a unique collection of isomorphisms
\[
a_{w,\C{L}}:H_i(\Fl_{G,\gm},\C{F}_{\C{L}})\isom H_i(\Fl_{G,\gm},\C{F}_{\ov{w}_*\C{L}})
\]
parameterized by $w\in\wt{W}_G$ such that

\smallskip

(i) $a_{w,\C{L}}=a_{J,w,\C{L}}$ for each $w\in W_J$;

\smallskip

 (ii) $a_{w,\C{L}}=a_{\Om,w,\C{L}}$  for each $w\in\Om_G$;

\smallskip

(iii) $a_{w_1w_2,\C{L}}=a_{J,w_1,\ov{w}_{2*}\C{L}}\circ a_{w_2,\C{L}}$ for all ${w}_1,{w}_2\in \wt{W}_G$.

\smallskip

\noindent(d) For every $g\in LG$, set $\gm':=g\gm g^{-1}$. Then for every $w\in\wt{W}_G$ the following diagram is commutative
\[
\begin{CD}
H_i(\Fl_{G,\gm},\C{F}_{\C{L}}) @>a_{w,\C{L}}>> H_i(\Fl_{G,\gm},\C{F}_{\ov{w}_*\C{L}})\\
@Vl_g VV @VVl_gV\\
H_i(\Fl_{G,\gm'},\C{F}_{\C{L}}) @>a_{w,\C{L}}>> H_i(\Fl_{G,\gm'},\C{F}_{\ov{w}_*\C{L}}).
\end{CD}
\]
\end{Prop}
\begin{proof}
(a) follows from definitions using Cartesian diagram
\[
\begin{CD}
\Fl_{G,\gm} @>>> \Fl_{G,J_1,\gm} @>>> \Fl_{G,J_2,\gm}\\
@VVV @VVV @VVV\\
[\frac{\I}{\I}] @>>> [\frac{\P_{J_1}}{\P_{J_1}}] @>>> [\frac{\P_{J_2}}{\P_{J_2}}].
\end{CD}
\]

\smallskip

(b) is straightforward.

\smallskip

(c) Formally follows from a combination of parts (a),(b) and the assertions of  Section~\re{wsaction}(d),(f):

Recall that the affine Weyl group $\wt{W}$ is generated by $\wt{S}$ subject to relations $(s_1s_2)^{m_{s_1,s_2}}=1$ for each $s_1,s_2\in\wt{S}$. Therefore  there exists at most one collection of isomorphisms $\{a_{w,\C{L}}\}_{w\in\wt{W}}$ satisfying condition (iii) and such that $a_{s,\C{L}}=a_{\{s\},s,\C{L}}$ for each $s\in\wt{S}$. Next, combining part~(a) with
the assertion of Section~\re{wsaction}(d) we see that such a collection
indeed exists and satisfies property (i). Finally, the existence of collection of isomorphisms $\{a_{w,\C{L}}\}_{w\in\wt{W}}$ satisfying (i)-(iii) follows from decomposition $\wt{W}_G=\wt{W}\ltimes \Om_G$, part~(b) and property of  Section~\re{wsaction}(f).
\smallskip

(d) is straightforward.
\end{proof}

\rp{whaction} immediately implies the following corollary.

\begin{Cor} \label{C:whaction}
(a) Each $H_i(\Fl_{G,\gm},\C{F}_{\C{L}})$ is a equipped with an action of the group $\La_G$,
commuting with the $LG_{\gm}$-action of Section~\re{affsprfib}(f).

\smallskip

(b) Moreover, if $\C{L}$ is an $\ov{W}$-equivariant local system, then each $H_i(\Fl_{G,\gm},\C{F}_{\C{L}})$
is equipped with an action of $\wt{W}_G$, commuting with the $LG_{\gm}$-action of Section~\re{affsprfib}(f).
\end{Cor}

\begin{Emp} \label{E:variant}
{\bf Notation.} Let $i:G^{\sc}\to G$ be the natural map.

\smallskip

(a) We set $\Fl:=\Fl_{G^{\sc}}$. Notice that we have a natural $LG^{\sc}$-equivariant embedding
$\eta:\Fl\hra\Fl_G$. Indeed, every Iwahori subgroup $\I\subseteq LG$ gives
rise to the Iwahori subgroup $\I^{\sc}:=i^{-1}(\I)\subseteq LG^{\sc}$ and we define $\eta$ to be the map
$g\I^{\sc}\mapsto i(g)\I$.

\smallskip

(b) For every element $\gm\in LG$, we set $\Fl_{\gm}:=\Fl\cap\Fl_{G,\gm}\subseteq\Fl_G$. Explicitly, for every Iwahori subgroup $\I$ of $LG$, the  affine Springer fiber $\Fl_{\gm}$ consists of all $g\I^{\sc}\in LG^{\sc}/\I^{\sc}$ such that $g^{-1}\gm g\in\I$.

\smallskip

(c) For every $u\in\Om_G$, consider the composition $\eta_u:\Fl\overset{\eta}{\lra}\Fl_G\overset{u}{\lra}\Fl_G$.
Then the induced maps
\[
\bigsqcup_{u\in\Om_G}\eta_u:\bigsqcup_{u\in\Om_G}\Fl\to \Fl_G\text{ and  }\bigsqcup_{u\in\Om_G}\eta_u:\bigsqcup_{u\in\Om_G}\Fl_{\gm}\to \Fl_{G,\gm}
\]
are universal homeomorphisms (see \rl{sprtjd}(b) below).
\end{Emp}

The proof of the following assertion is given in Section~\re{pflind} below.

\begin{Lem} \label{L:ind}
(a) For every $i\in\B{Z}$ and $w\in \wt{W}$ the isomorphism $a_{w,\C{L}}$ from \rp{whaction}(c) restricts to an isomorphism
$a_{w,\C{L}}:H_i(\Fl_{\gm},\C{F}_{\C{L}})\isom H_i(\Fl_{\gm},\C{F}_{\ov{w}_*(\C{L})})$.

\smallskip

(b) The subspace $H_i(\Fl_{\gm},\C{F}_{\C{L}})\subseteq H_i(\Fl_{G,\gm},\C{F}_{\C{L}})$ is
$\La$-invariant, and have a natural isomorphism of $\La_G$-representations
\[
H_i(\Fl_{G,\gm},\C{F}_{\C{L}})\simeq\ind_{\La}^{\La_G} H_i(\Fl_{\gm},\C{F}_{\C{L}}).
\]

\smallskip

(c) Moreover, if $\C{L}$  is $\ov{W}$-equivariant, then the subspace $H_i(\Fl_{\gm},\C{F}_{\C{L}})\subseteq H_i(\Fl_{G,\gm},\C{F}_{\C{L}})$ is
$\wt{W}$-invariant, and we have a natural isomorphism of $\wt{W}_G$-representations
\[
H_i(\Fl_{G,\gm},\C{F}_{\C{L}})\simeq\ind_{\wt{W}}^{\wt{W}_G} H_i(\Fl_{\gm},\C{F}_{\C{L}}).
\]
\end{Lem}

\begin{Emp} \label{E:functoriality}
{\bf Functoriality.}

\smallskip

(a) Every isomorphism $\phi:G\isom G'$ of split reductive groups over $K$ induces an isomorphism $\phi_*:\Fl_G\isom \Fl_{G'}$ of affine flag varieties, an isomorphism $\phi_*:\ov{T}_G\isom \ov{T}_{G'}$ of tori, an isomorphism $\wt{W}_G\isom\wt{W}_{G'}$ of extended affine Weyl groups,  etc.

For example,
if $T\subseteq G$ is a maximal split tori and $\I$ is an Iwahori subgroup scheme of $LG$, containing $\T$, then
$T':=\phi(T)\subseteq G$ is a maximal split tori and $\I':=\phi(\I)$ is an Iwahori subgroup scheme of $LG'$ containing $\T':=L^+(T')$, hence
$\phi$ induces an isomorphism $\phi_*:\Fl_G\simeq LG/\I\isom LG'/\I'\simeq \Fl_{G'}$. Moreover, the latter isomorphism
is independent of the choice of $(T,\I)$.

\smallskip

(b) The isomorphism $\phi_*:\Fl_G\isom \Fl_{G'}$ is $\phi$-equivariant, therefore it induces an isomorphism $\phi_*:\Fl_{G,\gm}\isom \Fl_{G',\phi(\gm)}$ of affine Springer fibers for every $\gm\in LG$.

\smallskip

(c) A local system $\C{L}$ on $\ov{T}_G$ gives rise to a local system $\C{L}':=\phi_*(\C{L})$ on $\ov{T}_{G'}$.
Since the isomorphism $\phi_*:\ov{T}_G\isom \ov{T}_{G'}$ is $\ov{W}_G\simeq \ov{W}_{G'}$-equivariant, we conclude that $\C{L}'$ is
$\ov{W}_{G'}$-equivariant, if $\C{L}$ is $\ov{W}_{G}$-equivariant.

\smallskip

(d) An isomorphism $\phi_*:\Fl_{G,\gm}\isom \Fl_{G',\phi(\gm)}$ from part~(b) induces an  isomorphism
\[
\phi_*:H_i(\Fl_{G,\gm},\C{F}_{\C{L}}) \isom H_i(\Fl_{G',\phi(\gm)},\C{F}_{\C{L}'})
\]
for every $\gm\in LG$ and $i\in\B{N}$.

\smallskip

(e) For every $g\in LG$ and $w\in\wt{W}_G$, the following diagram is commutative
\[
\begin{CD}
H_i(\Fl_{G,\gm},\C{F}_{\C{L}}) @>a_{w,\C{L}}>> H_i(\Fl_{G,\gm},\C{F}_{\ov{w}_*\C{L}})\\
@V\phi_* VV @VV\phi_*V\\
H_i(\Fl_{G',\phi(\gm)},\C{F}_{\C{L}'}) @>a_{w,\C{L}'}>> H_i(\Fl_{G',\phi(\gm)},\C{F}_{\ov{w}_*\C{L}'}).
\end{CD}
\]

\smallskip

(f) Furthermore, if the local system $\C{L}$ is $\ov{W}_G$-equivariant, then it follows from part~(e) that the isomorphism of part~(d) is $\phi_*:\wt{W}_G\isom\wt{W}_{G'}$-equivariant, and thus induces an isomorphism
\[
\phi_*:H_j(\wt{W}_G, H_i(\Fl_{G,\gm},\C{F}_{\C{L}}))\isom H_j(\wt{W}_{G'}, H_i(\Fl_{G',\phi(\gm)},\C{F}_{\C{L}'}))
\]
for all $j\in\B{N}$.
\end{Emp}

\begin{Emp} \label{E:inaut}
{\bf Inner automorphisms.}
Let $\phi:G\isom G$ be an inner isomorphism $\Ad_g$ for some $g\in G^{\ad}(K)=L(G^{\ad})(k)$.

\smallskip

(a) Consider projection $LG^{\ad}\to LG^{\ad}/LG^{\sc}\simeq\Om_G$, let $\om_g\in \Om_{G^{\ad}}$ be the image of $g$, and
let $\ov{w}_g\in \ov{W}$ be the image of $\om_g$ under the composition $\Om_G\hra\wt{W}_G\to \ov{W}$. Unwinding the definitions,
the induced morphism $\phi_*:\ov{T}_G\isom \ov{T}_G$ equals $\ov{w}_g$.

\smallskip

(b) It follows from part~(a) for every $\ov{W}$-equivariant local system $\C{L}$ on $\ov{T}_G$, we have a natural isomorphism
$\phi_*(\C{L})\simeq \C{L}$.

\smallskip

(c)  Assume that $g\in LG^{\ad}_{\gm}$, that is, $\Ad_g(\gm)=\gm$. Then, combining part~(b) and Section~\re{functoriality}(f), we see that
for every $\ov{W}$-equivariant local system $\C{L}$ on $\ov{T}_G$, the automorphism $\phi=\Ad_g$ of $G$ induces an automorphism $\phi_*=(\Ad_g)_*$
of $H_j(\wt{W}_G, H_i(\Fl_{G,\gm},\C{F}_\C{L}))$. Moreover, the correspondence $g\mapsto (\Ad_g)_*$ induces the action of $LG^{\ad}_{\gm}$ on $H_j(\wt{W}_G, H_i(\Fl_{G,\gm},\C{F}_\C{L}))$.
%
%(e) Each $\eta_{\gm}$ it induces an isomorphism
%
%\[
%H_j(\wt{W},H_i(\Fl_{\gm'},\C{L}^{\st}_{\theta}))\isom
%H_j(\wt{W},H_i(\Fl_{\gm},\C{L}^{\st}_{\theta})).
%\]
%
%In particular, the centralizer $G_{\gm}(F^{\nr})$ acts on each
%$H_j(\wt{W},H_i(\Fl_{\gm},\C{L}^{\st}_{\theta}))$.
\end{Emp}

%\begin{Lem} \label{L:commute}
%Each $\eta_g$ commutes with the action of $\wt{W}$.
%\end{Lem}

%\begin{proof}
% Indeed, it suffice to show that  commutes with each $W_S$. But this follows
%from the diagram
%\begin{equation} \label{Eq:1}
%\CD
%\Fl_{\gm'}   @>\Int g^{-1}>>    \Fl_{\gm} @>red'_{\gm,S}>>[(L_S)^{ad}\bs Spr(L_S)]\\
%@VV\pi_SV       @VV\pi_SV       @V[p_{L_S}]VV\\
%\Fl_{S,\gm'}   @>\Int g^{-1}>>    \Fl_{S,\gm} @>red_{\gm,S}>>[(L_S)^{ad}\bs L_S],\\
%\endCD
%\end{equation}
%identities $red'_{\gm,S}\circ Int
%g^{-1}=red'_{\gm',S}:\Fl_{\gm'}  \to [(L_S)^{ad}\bs Spr(L_S)]$
%and $red_{\gm,S}\circ Int g^{-1}=red_{\gm',S}:\Fl_{S,\gm'}
%\to [(L_S)^{ad}\bs L_S]$ and definition of the $W_S$-action (see
%Section~\re{wsaction}).
%\end{proof}

%\begin{Cor}
%Each $\eta_{\gm}$ it induces an isomorphism
%
%\[
%H_j(\wt{W},H_i(\Fl_{\gm'},\C{L}^{\st}_{\theta}))\isom
%H_j(\wt{W},H_i(\Fl_{\gm},\C{L}^{\st}_{\theta})).
%\]
%
%In particular, the centralizer $G_{\gm}(F^{\nr})$ acts on each
%$H_j(\wt{W},H_i(\Fl_{\gm},\C{L}^{\st}_{\theta}))$.
%\end{Cor}

\subsection{Compatibility of actions}

\begin{Emp} \label{E:pi0}
{\bf Kottwitz isomorphism.} For every torus $S$ over $K$, we have a natural isomorphism
$\pi_0(LS)\simeq X_*(S)_{\Gm_K}$, where $(-)_{\Gm_K}$ denotes coinvariants, functorial in $S$. Moreover, it is characterized
by the condition that a cocharacter $\la\in X_*(S)$ corresponds to the connected component of $\la(t)\in S(K)=LS(k)$ when $S$ is split. Indeed, the functor $S\mapsto \pi_0(LS)$ satisfies the conditions of \cite[Section~2.1]{Ko}. Therefore it is canonically isomorphic to
$S\mapsto X_*(S)_{\Gm_K}$ by \cite[Lemma~2.2]{Ko}.
\end{Emp}

%(b) Explicitly the isomorphism $\phi:
%X_*(S)_{\Gm}\isom\pi_0(\un{S})$ is as follows (DO WE NEED IT?)  Choose a finite
%separable extension $F'/F$ splitting $S$, and let $\pi'$ be an
%uniformizer of $F'$. Then the composition
%\[
%X_*(S)\overset{\ev_{\pi'}}{\lra} S(F')\overset{N_{F'/F}}{\lra}
%S(F)=\un{S}(k)\overset{\pr}{\lra}\pi_0(\un{S})
%\]
%is independent of $F'$, factors through $X_*(S)_{\Gm}$ and
%$\varphi:X_*(S)_{\Gm}\to \pi_0(\un{S})$ is the induced map.
%\end{Emp}

\begin{Emp} \label{E:can}
{\bf The canonical map.} Let $S\subseteq G$ be a maximal torus over $K$.

\smallskip

(a) We denote by $\pr_{S}$ the composition
\[
\qlbar[\La_G]^{\ov{W}_G}\isom\qlbar[X_*(T_G)]^{W_G}\hra \qlbar[X_*(S)]\overset{\pr}{\lra}
\qlbar[X_*(S)_{\Gm_K}]\isom \qlbar[\pi_0(LS)],
\]
where

\smallskip

\quad\quad$\bullet$ the first isomorphism is the inverse of isomorphism \form{can} from Section~\re{canonical};

\smallskip

\quad\quad$\bullet$ the second map was defined in Section~\re{admis}(c);

\smallskip

\quad\quad$\bullet$ and the last map is the Kottwitz isomorphism from Section~\re{pi0}.

\smallskip

(b) We denote by $\pr'_{S}$ the composition $\qlbar[\La_G]^{\ov{W}_G}\overset{\pr_{S}}{\lra}\qlbar[\pi_0(LS)]\overset{\iota}{\to}\qlbar[\pi_0(LS)]$, where $\iota$ induced by the inverse map $\iota:LS\isom LS:g\mapsto g^{-1}$.

\smallskip
(c) Notice that both $\pr_{S}$ and $\pr'_{S}$ make $\qlbar[\pi_0(LS)]$ a finite $\qlbar[\La_G]^{\ov{W}_G}$-algebra.
\end{Emp}

\begin{Emp} \label{E:setup comp}
{\bf Set up.} Fix an element $\gm\in G(K)=LG(k)$, and let $\C{L}$ be a local system on $\ov{T}_G$.

\smallskip

(a) By Section~\re{affsprfib}(f), each homology group $H_i(\Fl_{G,\gm},\C{F}_{\C{L}})$ is equipped with an action of the group
$LG_{\gm}$. Moreover, since the action any connected algebraic group induces a trivial action on the homology, the action of $LG_{\gm}$ on $H_i(\Fl_{G,\gm},\C{F}_{\C{L}})$ factors through the action of the group of connected components $\pi_0(LG_{\gm})$.

\smallskip

(b) Using  \rco{whaction}(a) each $H_i(\Fl_{G,\gm},\C{F}_{\C{L}})$ is therefore
equipped with an action of the group $\La_G\times\pi_0(LG_{\gm})$. Furthermore, it is even
equipped with an action of the group $\wt{W}_G\times\pi_0(LG_{\gm})$, if the local system $\C{L}$ is $\ov{W}_G$-equivariant (see \rco{whaction}(b)).

\smallskip

(c) Assume now that $\gm\in G^{\rss}(K)$. Then $G^0_{\gm}:=(G_{\gm})^0\subseteq G$ is a maximal torus over $K$, so the construction of Section~\re{can} applies, and we set $LG^0_{\gm}:= L(G^0_{\gm})$, $\pr_{\gm}=\pr_{G,\gm}:=\pr_{G^0_{\gm}}$ and $\pr'_{\gm}=\pr'_{G,\gm}:=\pr'_{G^0_{\gm}}$. Note that $\pi_0(LG^0_{\gm})$ is a subgroup of $\pi_0(LG_{\gm})$.
\end{Emp}

The following assertion is a crucial technical result of the paper
and will be proved in Appendix B.

\begin{Thm} \label{T:action}
For every $i\in \B{Z}$, there exists a finite
filtration $\{F^j H_i(\Fl_{G,\gm},\C{F}_{\C{L}})\}_j$ of $H_i(\Fl_{G,\gm},\C{F}_{\C{L}})$, stable under
the action of $\La_G\times\pi_0(LG^0_{\gm})$,
such that the action of a subalgebra $\qlbar[\La_G]^{\ov{W}_G}\subseteq \qlbar[\La_G]$ on each graded piece
$\gr^j H_i(\Fl_{G,\gm},\C{F}_{\C{L}})$ is induced from the action of $\qlbar[\pi_0(L{G}_{\gm})]$ via homomorphism $\pr'_{\gm}$. Furthermore, this filtration can be chosen to be stable under the action of $\wt{W}_G\times\pi_0(LG^0_{\gm})$, if $\C{L}$ is $\ov{W}_G$-equivariant.
\end{Thm}

Our strategy will be to deduce the result from an analogous result for Lie algebras shown by Yun \cite{Yun} (see \rt{yun} and Remark~\re{sign}).

\section{Finiteness of homology of affine Springer fibers}

\subsection{Finiteness of homology}

Let $L$ be an algebraically closed field.

\begin{Lem} \label{L:fingen}
Let $X$ be a scheme locally of finite type over $L$ equipped with an action of a group $\La$
such that there exists a closed subscheme $Y\subseteq X$ of finite type
over $L$ such that $\La(Y)=X$.

If $\qlbar[\La]$ is Noetherian, then  for every $\La$-equivariant
object $\C{F}\in D^b_c(X,\qlbar)$ and every $i\in\B{Z}$ the
$\qlbar[\La]$-module $H_i(X,\C{F})$ is finitely generated.
\end{Lem}

\begin{proof}
Notice that for every distinguished triangle
\[
\C{F}'\to\C{F}\to\C{F}''\to
\]
in $D^b_c(X,\qlbar)$ we have an exact sequence of
$\qlbar[\La]$-modules
\[
H_i(X,\C{F}'')\to H_i(X,\C{F})\to H_i(X,\C{F}').
\]

Since $\qlbar[\La]$ is Noetherian, the assertions
for $\C{F}'$ and $\C{F}''$ implies that for  $\C{F}$. Applying
this to the distinquished triangle
\[
\iota_*\iota^!\C{F}\to\C{F}\to j_*j^*\C{F}\to,
\]
where $j:U\hra X$ is a  $\La$-invariant open
embedding, and $\iota:X\sm U\hra X$ is the corresponding closed
embedding, we get that the assertion for $(X,\C{F})$ follows, if we
show assertions for $(U,j^*\C{F})$ and $(X\sm U,\iota^!\C{F})$.

\smallskip

Since we can always assume that the assertion for $(X\sm
U,\iota^!\C{F})$ holds by the Noetherian induction on $Y$, in the
course of the proof we can replace $X$ by its open non-empty
$\La$-equivariant subset $U$ and $Y$ by $Y\cap U$.

Let $\{X_{\al}\}_{\al}$ be the set of all irreducible components
of $X$. Then the union $\bigcup_{\al\neq\beta}(X_{\al}\cap X_{\beta})\subseteq X$ is closed.
Indeed, it suffices to show that for every open subscheme $V\subseteq X$ of finite type over $L$, the intersection
$V\cap(\bigcup_{\al\neq\beta}(X_{\al}\cap X_{\beta}))\subseteq V$ is closed.
It suffices to show that the set of the $X_{\al}$'s such that $V\cap X_{\al}\neq\emptyset$ is finite.
But this follows from the fact that every $V\cap X_{\al}$ is either empty or an irreducible component
of $V$ and the fact that $V$ has finitely many irreducible components.

Replacing $X$ by its open subscheme $X\sm\bigcup_{\al\neq\beta}(X_{\al}\cap X_{\beta})$,
we can assume that each connected component $X_{\al}$ of $X$ irreducible. Moreover,
enlarging $Y$ one can assume that $Y$ is a union of connected
components of $X$. But then $H_i(X,\C{F})$ is generated as a
$\La$-module by its finite dimensional subspace $H_i(Y,\C{F})$.
Therefore it is finitely generated.
\end{proof}

\begin{Emp} \label{E:filtr}
{\bf Notation.} Let $X$ be either a scheme locally of finite type
over $L$ or an {\em ind-scheme} over $L$, that is, an inductive limit $\on{colim}_{\al} X_{\al}$ of schemes of finite type over
$L$ such that all transition maps are closed embeddings.

\smallskip

(a) By a {\em filtration} on $X$, we mean an increasing sequence of closed subschemes
$X_1\subseteq X_2\subseteq \ldots\subseteq X$ of $X$ such that $X=\bigcup_i X_i$.

\smallskip

(b) Assume that $X$ is equipped with an action of a group $\La$,
equipped with a filtration $\{\La_i\}_i$ (see Section~\re{fil}(a)).

We say that a filtration
$\{X_i\}_i$ on $X$ is {\em compatible} with $\{\La_i\}_i$, if
we have an inclusion
$\La_i(X_j)\subseteq X_{i+j}$ for each $i,j\in\B{N}$.

Moreover, we say that a filtration $\{X_i\}_i$ is {\em finitely generated}
over $\{\La_i\}_i$, if each $X_i$ is of finite type over $L$, and there
exists $N\in\B{N}$ such that for every $i>N$, we have an equality
$X_i=\bigcup_{j=1}^N \La_{i-j}(X_j)$.

\smallskip

(c) As in section Section~\re{fil}(d), filtration $\{\La_i\}_i$ gives rize to a filtration on
the group algebra $A:=\qlbar[\La]$, and we denote by
$R(A)=\bigoplus_i A_i$ the corresponding Rees algebra.

For every
$\La$-equivariant object $\C{F}\in D^b_c(X,\qlbar)$ and every
$j\in\B{Z}$, the graded vector space
\[
R(H_j(X,\C{F})):=\bigoplus_i H_j(X_i,\C{F})
\]
is a graded $ R(A)$-module: an element $\la\in\La_{i}$ induced a closed embedding
$\la:X_{i'}\hra X_{i+i'}:x\mapsto \la(x)$, hence  induces a morphism
$\la_*:H_j(X_i,\C{F})\to H_j(X_{i+i'},\C{F})$, dual to the
pull-back $\la^*:H^j(X_{i+i'},\C{F})\to H_j(X_i,\C{F})$.

\smallskip

(d) For every $i$, we denote by $H'_j(X_i,\C{F})\subseteq H_j(X,\C{F})$
the image of the natural map $H_j(X_i,\C{F})\to H_j(X,\C{F})$, and
set
\[
R(H'_j(X,\C{F})):=\bigoplus_i H'_j(X_i,\C{F}).
\]
Then the natural surjective maps $H_j(X_i,\C{F})\to H'_j(X_i,\C{F})$ induce a surjective homomorphism
$R(H_j(X,\C{F}))\to R(H'_j(X,\C{F}))$ of graded $R(A)$-modules.
\end{Emp}

\begin{Lem} \label{L:rees}
In the situation of Section~\re{filtr}, assume that $\La$ acts
on the set of irreducible components of $X$ with finite stabilizers, that
$ R(A)$ is Noetherian and that filtration $\{X_i\}_i$ is finitely generated
over $\{\La_i\}_i$.

Then for every $\La$-equivariant object $\C{F}\in
D^b_c(X,\qlbar)$ and every $j\in\B{Z}$ the $ R(A)$-modules
$R(H_j(X,\C{F}))$ and $R(H'_j(X,\C{F}))$ are finitely generated.
\end{Lem}

\begin{proof}
Since $R(H'_j(X,\C{F}))$ is a factor module of $R(H_j(X,\C{F}))$, it suffices to show the assertion for
$R(H_j(X,\C{F}))$.

\smallskip

Fix $N\in\B{N}$ such that for every $i>N$ we have
$X_i=\bigcup_{j=1}^N \La_{i-j}(X_j)$. In particular, we have $X=\La(X_N)$, thus we are in
the situation of \rl{fingen} with $Y=X_N$.

\smallskip

Now our argument is very similar to that of \rl{fingen}: Notice
that for every distinguished triangle
\[
\C{F}'\to\C{F}\to\C{F}''\to
\]
in $D^b_c(X,\qlbar)$ we have an exact sequence of
$ R(A)$-modules
\[
R(H_j(X,\C{F}''))\to R(H_j(X,\C{F}))\to R(H_j(X,\C{F}')).
\]

Since $ R(A)$ is a Noetherian ring, the assertion
for $\C{F}'$ and $\C{F}''$ implies that for  $\C{F}$. Thus, arguing
as in \rl{fingen}, we can replace $X$ by its open
$\La$-equivariant subscheme, thus assuming that each connected
component $X_{\al}$ of $X$ irreducible. In this case, $\La$ acts
on the set of connected components of $X$ with finite stabilizers.

\smallskip

Next we reduce the assertion to the case when each $X_i$ is a
union of (possible empty set) of connected components of $X$:
Since $X_i=\bigcup_{j=1}^N \La_{i-j}(X_j)$ for each $i>N$, it is
enough to show the assertion only for $i=1,\ldots,N$. In other
words, we want to get to the situation that for every $X_{\al}$
and every $i=1,\ldots, N$ the intersection $X_{\al}\cap X_i$ is
either $X_{\al}$ or empty.

Notice that for every $i=1,\ldots,N$, the intersection $X_{\al}\cap X_i$ is non-empty only for
finitely many $X_{\al}$'s. Therefore there exist only finitely many pairs $(X_{\al},X_i)$
such that $\emptyset\neq X_{\al}\cap X_i\subsetneq X_{\al}$, and we will call such a pair
{\em bad}.

Assume that a bad pair $(X_{\al},X_i)$ exists. Since the stabilizer of
$X_{\al}\subseteq X$ in $\La$ is finite, the union $\La(X_{\al}\cap X_i)\subseteq X$ is closed.
Thus, replacing $X$ by an open $\La$-equivariant subset $X\sm
\La(X_{\al}\cap X_i)$, we decrease the number of bad pairs by at least one. This process will stop after finitely
many steps.

\smallskip

By the shown above, we can assume that each $X_i$ is a union of
(possible empty set) of connected components of $X$. Since
$X_i=\bigcup_{j=1}^N \La_{i-j}(X_j)$ for each $i>N$, we deduce that
the $ R(A)$-module $ R(H_j(X,\C{F}))$ is generated by its finite
dimensional subspace $\bigoplus_{i=1}^N H_j(X_j,\C{F})$. Therefore it
is finitely generated.
\end{proof}

\subsection{Combinatorics of the affine Weyl groups}

In this subsection we assume that $G$ is a split, semisimple and simply connected group over $K$, fix a pair $(T,\I)$ as in Section~\re{affweyl}(a), and set $\wt{W}:=\wt{W}_G$ and $\ov{W}:=\ov{W}_G$.

\begin{Emp} \label{E:bruhatlength}
{\bf Bruhat decomposition and length function.}

\smallskip

(a) Recall that the inclusion $N_{ LG}(LT)\hra LG$ induces a bijection
\[
\wt{W}\simeq N_{ LG}(LT)/ LT\isom \I\bs LG/\I,
\]
which depends on $\I$, but is independent of $T$.

\smallskip

(b) For every element $g\in LG$, we denote by $w(g)\in\wt{W}$ the unique
element such that $g\in \I w(g)\I$, and let $l(g)=l_\I(g)$ be the length of $w(g)$.

\smallskip

(c) Consider the embedding $\eta:\La=X_*(T_G)\simeq X_*(T)\hra  LT\subseteq LG$, induced by the embedding $\la\mapsto \la(t):X_*(T)\to LT$.
Note that for every $\la\in\La\subseteq \wt{W}$ the length $l(\la)$ coincides with $l_\I(\eta(\la))$ (see part~(b)).

\smallskip

(d) For shortness, we will usually omit $\eta$, that is, write $\la\in LG$ instead of $\eta(\la)$.
\end{Emp}

\begin{Emp} \label{E:bruhat}
{\bf Properties}.

\smallskip

(a) Recall that for every $w\in\wt{W}$ and every simple affine reflection $s=s_{\al}$,
we have $\I w\I\cdot \I s\I=\I ws\I$, if $ws>w$, and $\I w\I\cdot\I s\I=\I ws\I\cup \I w\I$, is $ws<w$.

\smallskip

%(b) By part~(a), for every $g\in\I w\I$ and $h\in \I s\I$, we have
%$l(gh)=l(g)+1$, if $w(\al)>0$, and $l(gh)\in\{l(g),l(g)-1\}$, otherwise.
%Moreover, equality $l(gh)=l(g)$ is equivalent to the equality $w(gh)=w(g)$.
%
%\smallskip

(b) Using part~(a) and induction, we see that for every $w,w'\in\wt{W}$,
we have an inclusion
\[
\I w\I\cdot \I w'\I\subseteq\bigcup_{w''\leq w'}\I ww''\I.
\]
Moreover, we have
$\I w\I\cdot \I w'\I=\I ww'\I$ if and only if $l(ww')=l(w)+l(w')$.
Similarly, we have an inclusion
\[
\I w\I\cdot \I w'\I\subseteq\bigcup_{w''\leq w}\I w''w'\I.
\]

\smallskip

(c) By part~(b), for every $g,h\in LG$ we have $l(gh)\leq l(g)+l(h)$. Moreover, we have
$l(gh)=l(g)+l(h)$ if and only if $\I g\I\cdot \I h\I=\I gh\I$, in which case we have $w(gh)= w(g)w(h)$.

\smallskip

(d) It is a standard fact that for every $\la\in\La$, we have an equality
\begin{equation} \label{Eq:lla}
l(\la)=\sum_{\ov{\al}\in\ov{\Phi}}\max\{\langle\ov{\al},\la\rangle,0\}.
\end{equation}
In particular, we have $l(w(\la))=l(\la)$ for every $w\in \ov{W}$.
\end{Emp}

\begin{Emp} \label{E:lengthfun}
{\bf Length functions.} Let $T'\subseteq G$ be a (not necessarily split) torus.

\smallskip

(a) Assume that $T'$ is maximal, and let $\varphi:T'\isom T_G$ be an admissible isomorphism (over $\ov{K}$).
Then $\varphi$ induces an isomorphism $X_*(T')\isom X_*(T_{G})\simeq\La$. Hence the length function on $\La\subseteq\wt{W}$ induces a length
function on $X_*(T')$. Moreover, since $\varphi$ is unique up to a $\ov{W}$-action, while the length function on $\La$ is $\ov{W}$-invariant
(see Section~\re{bruhat}(d)), the length function $l_{T'}(\cdot)$ on  $X_*(T')$ is independence of $\varphi$.

\smallskip

(b) For an arbitrary torus $T'$, choose a maximal torus $T'_1\supseteq T'$ of $G$. Then we have $X_*(T'_1)\supseteq X_*(T')$, so
the construction of part~(a) defines a length function $l_{T'}:=l_{T'_1}|_{X_*(T')}$ on $X_*(T')$. We claim that
the length function $l_{T'}$ is independent of the maximal torus $T'_1$. Indeed, if $T'_2\supseteq T'$ is
another maximal torus, then there exists  $g\in Z_G(T')$ such that $gT'_1g^{-1}=T'_2$, so the assertion follows from part~(a).

\smallskip

(c) It follows from part~(b) that for every torus $T''\supseteq T'$ of $G$, we have an equality
\[
l_{T'}=l_{T''}|_{X_*(T')}.
\]

\smallskip

(d) Assume now that $T'$ is split over $K$. Then we have a natural embedding $X_*(T')\hra LT'\subseteq LG:\la\mapsto \la(t)$.
Hence for every Iwahori subgroup $\I\subseteq LG$ can also consider the length function $l_\I$ on $X_*(T')$.
\end{Emp}

\begin{Lem} \label{L:lengthfun}
In the situation of Section~\re{lengthfun}(d), assume that $L^+(T')\subseteq\I$. Then the length function $l_{T'}$ of
Section~\re{lengthfun}(b) equals $l_\I$ of Section~\re{lengthfun}(d).
\end{Lem}

\begin{proof}
First, if $T'\subseteq G$ is a maximal split torus, is clear (compare Section~\re{bruhatlength}(b)). Next, using observation
Section~\re{lengthfun}(c) it suffices to show that there exists a maximal split torus $T\supseteq T'$ of $G$ such that $L^+(T)\subseteq\I$.
But this is clear: since $T'$ is split, the centralizer $H:=Z_{G}(T')$ a split connected group,
and $\I_H:=\I\cap LH$ is an Iwahori subgroup of $LH$. Then any maximal split torus $T\subseteq H$ such that $L^+(T)\subseteq\I_H$ satisfies the required property.
\end{proof}

The following assertion plays a central role in the proof of Theorem \ref{T:fingen} below.

\begin{Lem} \label{L:bound}
Let $\La''\subseteq \La$ be a subgroup. For every finite subset
$A\subseteq\wt{W}$ there exists $r\in\B{N}$ such that such that for
every $g\in \La''\cdot \I A\I$ we have either $l(g)\leq r$ or
there exists $\la\in\La''$ such that $\la\neq 1$ and
$l(g)=l(\la)+l(\la^{-1}g)$.
\end{Lem}

The proof of \rl{bound} is based on the following claim.

\begin{Cl} \label{C:fin}
Let $\La'\subseteq \La$ be a sub-semigroup such that for every root $\ov{\al}\in\ov{\Phi}$
we have either $\langle\ov{\al},\La'\rangle\geq 0$ or
$\langle\ov{\al},\La'\rangle\leq 0$.

Then for every $w\in\wt{W}$ there
exists $\la_0\in\La'$ such that for every $g\in \I w\I$ and
$\la\in\La'$ we have $l(\la\la_0g)=l(\la)+l(\la_0g)$.
\end{Cl}

\begin{proof}
To expose the structure of the argument we  divide it into steps.

\smallskip

\noindent{\bf Step 1.} For every $\la,\la'\in\La'$, we have
$l(\la\la')=l(\la)+l(\la')$. Indeed, the assertion follows from equality \form{lla}
and our assumption on  $\La'$.

\smallskip

\noindent{\bf Step 2.} For every affine root
$\wt{\al}\in\wt{\Phi}$ there exists
$\la_0\in\La'$ such that for every $\la\in\La'$ we have
$\la\la_0(\wt{\al})>0$ if and only if $\la_0(\wt{\al})>0$.

Indeed, for every affine root $\wt{\al}=(\al,k)$ and $\la\in\La$ we have
$\la(\wt{\al})=(\al,k-\langle\al,\la\rangle)$. So the assertion follows from our assumption
on $\La'$.

\smallskip

\noindent{\bf Step 3.} Assume that elements $w\in\wt{W}$ and $\la_0\in\La'$ satisfy the property that for every $\la\in\La'$ we have  $l(\la\la_0w)=l(\la)+l(\la_0w)$. Then for every $\la,\la_0'\in\La'$, we have
\[
l(\la\la'_0\la_0w)=l(\la\la'_0)+l(\la_0w)=l(\la)+l(\la'_0)+l(\la_0w)=l(\la)+l(\la'_0\la_0w),
\]
where the middle equality follows from Step 1, while the first and the last equalities follow by assumption.

\smallskip

\noindent{\bf Step 4.} For every finite subset $A\subseteq\wt{W}$, there
exists $\la_0\in\La'$ such that for every $\la\in\La'$ and $w\in A$, we have an equality
$l(\la\la_0w)=l(\la)+l(\la_0w)$.

\begin{proof}
Using Step 3, it is enough to show the assertion in the case when $A=\{w\}$ consists of one
element. We are going to prove the assertion by induction
on $l(w)$. If $l(w)=0$, then, by Step 1, every $\la_0\in\La'$ satisfies
the property.

\smallskip

Assume now that $l(w)>0$ and choose a simple affine reflection
$s=s_{\al}$ such that $w':=ws<w$. Hence, by the induction
hypothesis, there exists $\la'_0\in\La'$ such that
\begin{equation} \label{Eq:la'}
\text{for every }\la\in\La',\text{ we have }l(\la\la'_0 w')=l(\la)+l(\la'_0 w').
\end{equation}

Next, by Step 2, there exists $\la''_0\in\La'$ such that
\begin{equation} \label{Eq:la''}
\text{for every }\la\in\La'\text{ we have }
\la\la''_0(\la'_0w'(\al))>0\text{ if and only if }
\la''_0(\la'_0w'(\al))>0.
\end{equation}

We claim that $\la_0:=\la''_0\la'_0$ satisfies the required property. Using formula \form{la'} and Step 3, it suffices to show that for every $\la\in\La'$ we have an equality
\[
l(\la\la_0w)-l(\la\la_0w')=l(\la_0w)-l(\la_0w'),
\]
or, equivalently that $(\la\la_0w')(\al)>0$ if and only if  $(\la_0w')(\al)>0$. But this follows from formula \form{la''}.
\end{proof}

\noindent{\bf Step 5.} We claim that element $\la_0\in\La'$, satisfying property of Step 4 for a finite subset $A:=\wt{W}^{\leq w}\subseteq\wt{W}$, satisfies property of the claim.

\smallskip

Indeed, by   Section~\re{bruhat}(b), for every $g\in\I w\I$ there exists $w'\in \wt{W}^{\leq w}$ such that
$\la_0g\in \I\la w'\I$. Then, by our assumption, for every element $\la\in \La'$, we have $l(\la\la_0w')=l(\la)+l(\la_0w')$.
Thus, we have $\la\la_0g\in\I\la\I\cdot\I\la w'\I=\I\la\la_0w'\I$ (by Section~\re{bruhat}(b)), therefore
\[
l(\la\la_0g)=l(\la\la_0w')=l(\la)+l(\la_0w')=l(\la)+l(\la_0g),
\]
as claimed.
\end{proof}

\begin{Cor} \label{C:bound}
For every finitely generated semigroup $\La'\subseteq\La$ satisfying
the assumption of \rcl{fin} and a finite set $A\subseteq\wt{W}$, there exists $r\in\B{N}$
 such that for every $g\in \La'\cdot \I A\I$ we have
either $l(g)\leq r$ or there exists $\la\in\La'$ such that
$\la\neq 1$, $\la^{-1}g\in\La'\cdot \I A\I$ and
$l(g)=l(\la)+l(\la^{-1}g)$.
\end{Cor}
\begin{proof}
We are going to prove the assertion by induction on the minimal number of generators
$\la_1,\ldots,\la_n$ of the semigroup $\La'$. If $n=0$, the assertion is clear.

Note that it is enough to show the assertion in the case when $A=\{w\}$ consists of one
element. In this case, by \rcl{fin}, there exists $\la_0\in\La'$ such that for every $g'\in \I w\I$ and
$\la\in\La'$, we have $l(\la\la_0g')=l(\la)+l(\la_0g')$.

Write $\la_0$ as $\prod_{i=1}^n \la_i^{m_i}$.  For every $i=1,\ldots,n$ we denote by $\La'_i\subseteq\La'$
the sub-semigroup generated by $\{\la_j\}_{j\neq i}$.
Then  the assertion follows for each
$\La'_i$ by induction hypothesis, and we claim that the assertion for $\La'$ follows from that for the $\La'_i$'s.

\smallskip

More precisely, let $r_i$ be the constant associated to the sub-semigroup $\La'_i$ and set
$A_i:=\{w'w\,|\, w'\leq\la_i^{m_i}\}$, and we claim that constant $r:=\max_{i=1}^n r_i$ satisfies the property
for $\La'$ and $\{w\}$.

Indeed, this follows from the decomposition
\begin{equation*} \label{Eq:dec}
\La'=(\La'-\{1\})\la_0\cup(\bigcup_{i=1}^n\bigcup_{j=1}^{m_i}\La'_i\la_i^j),
\end{equation*}
our choice of $\la_0$ and inclusion
\[
\la_i^j\cdot \I w\I\subseteq\bigcup_{w'\leq\la_i^j}\I w'w\I\subseteq \bigcup_{w'\leq\la_i^{m_i}}\I w'w\I=\I A_i \I
\] (which uses Section~\re{bruhat}(b) and observation that $\la_i^j\leq\la_i^{m_i}$ for all $j\leq m_i$).
\end{proof}

Now we are ready to prove \rl{bound}.

\begin{proof}[Proof of \rl{bound}]
Consider the vector space $V:=\La\otimes_{\B{Z}}\B{R}$ over $\B{R}$. Then the hyperplanes
$\{H_{\ov{\al}}:=\ker\ov{\al}\}_{\ov{\al}\in\ov{\Phi}}$ decompose $V$ as a union of faces. For each
closure of a face $F$, we set $\La'_F:=\La''\cap F$. Then by
Gordan's lemma (see, for example, \cite[Lemma~3.4, page~154]{Ew}), each $\La'_F$ is a finitely generated semigroup satisfying
the assumption of \rcl{fin}.
Since $\La''=\bigcup_F \La'_F$, the assertion follows from \rco{bound}.
\end{proof}

\begin{Emp} \label{E:filtr0}
{\bf Length filtrations.}

\smallskip

(a) For $n\in\B{N}$, we set $\wt{W}^{\leq
n}=\{w\in\wt{W}\colon l(w)\leq n\}$. Since $l(ww')\leq l(w)+l(w')$, we conclude that
$\{\wt{W}^{\leq n}\}_n$ is a filtration of a monoid $\wt{W}$ (see Section~\re{fil}(a)).

\smallskip

(b) For every subgroup $\Dt\subseteq
\wt{W}$, we consider the induced filtration $\{\Dt^{\leq n}\}_n$ of $\Dt$, where
$\Dt^{\leq n}:=\Dt\cap \wt{W}^{\leq n}$.
\end{Emp}

\begin{Cor} \label{C:fg}
For every subgroup $\Dt\subseteq \La$, the length filtration $\{\Dt^{\leq n}\}_n$
is finitely generated. Therefore the Rees algebra $R(\qlbar[\Dt])$ is a finitely generated $\qlbar$-algebra, hence a Noetherian ring.
\end{Cor}

\begin{proof}
The first assertion follows from \rl{bound} for $A=\{1\}$. Since $\Dt$ is commutative, the second assertion follows from the first one
and Section~\re{fil}(e).
\end{proof}

\begin{Emp} \label{E:winv}
{\bf Example.}
Since filtration $\{\La^{\leq n}\}_n$ is $\ov{W}$-invariant (see Section~\re{bruhat}(d)), it induces
a filtration on $\qlbar[\La]^{\ov{W}}$, and the corresponding Rees algebra satisfies $R(\qlbar[\La]^{\ov{W}})\simeq R(\qlbar[\La])^{\ov{W}}$.
Since $R(\qlbar[\La])$ is a finitely generated $\qlbar$-algebra (by \rco{fg}), we therefore conclude that $R(\qlbar[\La]^{\ov{W}})$ is a finitely generated
$\qlbar$-algebra as well, hence a Noetherian ring.
\end{Emp}

\begin{Emp} \label{E:notdem}
{\bf Notation.}

\smallskip

(a) For $u\in\wt{W}$, we set $\wt{W}^{\leq u}:=\{u'\in\wt{W}\,|\,u'\leq u\}$. For $u,v\in\wt{W}$, we denote by
$\wt{W}^{\leq u}\cdot \wt{W}^{\leq v}\subseteq\wt{W}$ the set of products $\{u'v'\,|\,u'\in \wt{W}^{\leq u}, v'\in \wt{W}^{\leq v}\}$.

\smallskip

(b) Let $\wt{S}$ be the set of simple affine reflections of $G$ (see Section~\re{affweyl}). For every proper subset $J\subsetneq\wt{S}$, let
$W_J\subseteq\wt{W}$ be the subgroup generated by $J$.
\end{Emp}

The following lemma is well-known to specialists.

\begin{Lem} \label{L:dem}
(a) For every $u,v\in\wt{W}$ there exists a unique element $w:=u*v$ of $\wt{W}$ (called the {\em Demazure product}) such
that the product $\wt{W}^{\leq u}\cdot \wt{W}^{\leq v}\subseteq\wt{W}$ equals $\wt{W}^{\leq w}$.
Moreover, $u*v=uv'=u*v'$ for some $v'\leq v$ with $l(uv')=l(u)+l(v')$.

In particular, we have $l(u*v)\leq l(u)+l(v)$, and equality $u*v=uv$ holds if and only if $l(uv)=l(u)+l(v)$. Similarly,
$u*v=u'v=u'*v$ for some $u'\leq u$ with $l(u'v)=l(u')+l(v)$.

\smallskip

(b) For every non-empty subset $J\subsetneq\wt{S}$, let $w_J\in
W_J\subseteq\wt{W}$ be the longest element. Then for every
$w\in\wt{W}$ the Demazure product $w*w_J$ (resp. $w_J*w$) is the longest element of the
coset $wW_J$ (resp. $W_Jw$).

\smallskip

(c) For every $w\in \wt{W}$ and $\la\in\La$, we have inequality
\[
l(w*w_J*\la*w_J)\leq l(w*w_J)+l(\la).
\]
\end{Lem}

\begin{proof}
(a) Argue as in Section~\re{bruhat}(a),(b) or see \cite[Lemma~1]{He}.

%We prove both assertions by induction on $l(v)$. If $l(v)=0$, then $v=1$, so there is nothing to prove. If $l(v)=1$, then $v=s$ is
%a simple affine reflection, then the product $\wt{W}^{\leq u}\cdot\wt{W}^{\leq s}$ equals
%$\wt{W}^{\leq us}$, if $us>u$, and $\wt{W}^{\leq u}$, if $us<u$.
%\smallskip
%Let $l(v)>1$ and choose a simple reflection $s$ such that
%$vs<v$, then the product $\wt{W}^{\leq u}\cdot\wt{W}^{\leq v}$ equals
%$\wt{W}^{\leq u}\cdot\wt{W}^{\leq vs}\cdot\wt{W}^{\leq s}$, by the
%case of a simple affine reflection.
%ince $l(vs)<l(v)$, we conclude by the induction hypothesis
%that the product $\wt{W}^{\leq u}\cdot\wt{W}^{\leq vs}\cdot\wt{W}^{\leq s}$ has the form
%$\wt{W}^{\leq w'}\cdot\wt{W}^{\leq s}$ (by the induction hypothesis), hence $\wt{W}^{\leq w}$
%(by the case $v=s$).
%The associativity of $*$ follows from the associativity of the usual product in $\wt{W}$.
%Moreover, by the induction hypothesis, $w'=uv'=u*v'$ for some $v'\leq vs$ such
%that $l(uv')=l(u)+l(v')$, hence $w$ is either $w'=uv'$ (if $w's<w$) or $uv's$ (if $uv's>uv'$).
%In the second case, we have $v's>v'$, thus $l(v's)=l(v')+1$, hence
%\[
%l(uv's)=l(uv')+1=l(u)+l(v')+1=l(u)+l(v's).
%\]

\smallskip

(b) By part~(a), we have $w*w_J=wu$ for some $u\leq w_J$. Then $u\in
W_J$, thus $w*w_J\in wW_J$. Next, for every $u\in W_J$ we
have $u\leq w_J$, hence $wu\leq w*w_J$ by the definition of $*$.
Thus $w*w_J$ is the longest element of $wW_J$, as claimed.

\smallskip

(c) By part~(a), there exist elements $u,u'\leq w_J$ such that
$w_J*\la=u'\la$ and $(w_J*\la)*w_J=u'\la u$.
Thus
\[
w_J*\la*w_J=u'\la u=u'u(u^{-1}\la u).
\]
Note that $u,u'\in W_J$, thus $u'u\in W_J$, hence $l(u'u)\leq l(w_J)$.
Since  $l(u^{-1}\la u)=l(\la)$
(by formula \form{lla}), we thus conclude that
\begin{equation} \label{Eq:ineq1}
l(w_J*\la*w_J)\leq l(uu')+l(u^{-1}\la u)\leq l(w_J)+l(\la).
\end{equation}
Using part~(a) again, there exists $w'\in\wt{W}$ such that we have $w*w_J=w'*w_J$ and
$l(w*w_J)=l(w')+l(w_J)$. Then we have an inequality
\[
l(w*w_J*\la*w_J)=l(w'*w_J*\la*w_J)\leq l(w')+l(w_J*\la*w_J)
\]
and equality
\[
l(w*w_J)+l(\la)=l(w')+l(w_J)+l(\la).
\]
Now the assertion follows from inequality \form{ineq1}.
\end{proof}

\begin{Not} \label{N:k-reg}
%(a) Let $A_0\subseteq \La\otimes_{\B{Z}}\B{R}$ be the fundamental alcove, that is, $A_0$ is the set of $x\in\La\otimes_{\B{Z}}\B{R}$ such that $\wt{\al}(x)>0$ for all positive affine roots $\wt{\al}$ of $(G,T)$. Notice that this notion depends on our choice of $\I$.

%For $m\in\B{N}$, we say that $w\in\wt{W}$ is {\em $m$-regular},
%if for every root $\al\in\Phi(G,T)$ we have $|\langle \al,w(A_0)\rangle|> m$, that is,
%$|\langle \al,w(x)\rangle|> m$ for all $x\in A_0$.

%(c) We say that an closed $I$-invariant subscheme $\C{P}\subseteq
%\Fl_G$ is {\em $m$-regular}, if $\C{P}=\cup_{i=1}^l\Fl^{\leq
%w_i}$ and each $w_i$ is $m$-regular.
For every proper subset $J\subsetneq\wt{S}$, we say that an element
$w\in\wt{W}$ is {\em right} (resp. {\em left}) $J$-longest, if $w$ is the longest element of the
coset $wW_J\subseteq\wt{W}$ (resp. $W_Jw\subseteq\wt{W}$). Notice that by \rl{dem}(b) this happens if and only if
$w*w_J=w$ (resp. $w_J*w=w$).
%(b) Using (a) and (d) for every $m\in\B{N}$ and $J\subsetneq\wt{W}$ there exists an element $\la\in\La\subseteq\wt{W}$, which is
%left $J$-longest and $m$-regular.
\end{Not}

%\begin{Emp} \label{E:k-reg}
%{\bf Remarks.}
%(a) Note that if $w\in\wt{W}$ is $m$-regular, then coweight
%$w(0)\in w(\ov{A_0})\in\La$ satisfies $|\langle \al,w(0)\rangle|\geq m$ for all $\al\in\Phi(G,T)$, that is,
%$w$ is $m$-regular in the sense of \cite[Introduction]{BV}.

%(b) Conversely, if $w$ is $m$-regular in the sense of \cite[Introduction]{BV}, then $w(0)\in w(\ov{A_0})\in\La$ satisfies $|\langle \al,w(0)\rangle|\geq m$ for all $\al\in\Phi(G,T)$, thus $w\in\wt{W}$ is $(m-1)$-regular in the sense of \rn{k-reg}.

%(c) By definition, $w$ is $m$-regular if and only for every $\al\in\Phi(G,T)$ we have either
%$w^{-1}(\al)-m>0$ or $w^{-1}(\al)+m<0$. Thus this happens if and only if we have either
%$w^{-1}(-\al,m)<0$ or $w^{-1}(\al,m)<0$, that is, either $s_{\al,m}w<w$ or $s_{-\al,m}w<w$.

%(d) An element $w\in \wt{W}$ is left $J$-longest if and only if $s_{\al}w<w$ for every $s_{\al}\in J$, that is,
%$w^{-1}(\al)<0$ for every simple affine root $\al$ such that $s_{\al}\in J$.

%(e) Using (a) and (d) for every $m\in\B{N}$ and $J\subsetneq\wt{W}$ there exists an element $\la\in\La\subseteq\wt{W}$, which is
%left $J$-longest and $m$-regular.

%and $w\leq_R w'$, that is $w$ is smaller than $w'$ with respect to
%right weak order (see \cite[Def 3.1.1]{BB}) than  $w\in\wt{W}$ is
%right $m$-regular (see \cite[Prop 3.1.3]{BB}). On the other hand
%this is false for the left weak order. This explains our
%terminology.
%\end{Emp}

\begin{Lem} \label{L:k-reg}

%(a) If $w\in\wt{W}$ is right $m$-regular, then it is $m$-regular
For each elements $w,w'\in\wt{W}$  we have $w*w'\geq w$ and $w'*w\geq w$. Moreover, if an element
$w\in\wt{W}$ is  $m$-regular (see Section~\re{mreg}), then elements $w*w'$ and $w'*w$ are
$m$-regular as well.
\end{Lem}

\begin{proof}
%(a) If $w=w_1\mu w_2$ is  $m$-regular, where $w_1,w_2\in W$
%and $\mu$ is dominant, then $w(0)=w_1(\mu)\in w(\ov{A_0})$.
%Therefore for every root $\al$, we have
%$|\langle\mu,w_1^{-1}(\al)\rangle|=|\langle w_1(\mu),\alpha|\geq
%k$.
Since the map $w\mapsto w^{-1}$ preserves $m$-regularity (see Section~\re{mreg}(d)) and Bruhat ordering, and therefore we have
$(w*w')^{-1}=(w')^{-1}*w^{-1}$, it suffices to show the assertion for $w*w'$.

By induction on $l(w')$, it is enough to show the assertion in the
case when $l(w')=1$, thus $w'=s$ is a simple affine reflection. In this
case, we see that $w*s=\max \{w,ws\}\geq w$. If $w*s=w$, there is nothing to prove, while if $ws>w$, then
for every $\al\in\wt{\Phi}$ such that $\al>0$ and $w^{-1}(\al)<0$, we have  $(ws)^{-1}(\al)=sw^{-1}(\al)<0$,
from which the $m$-regularity assertion follows.
%Next, $w\in\wt{W}$ is  $m$-regular means that for every $\al\in\Phi(G,T)$ we have either
%$s_{\al,m}w<w$ or $s_{-\al,m}w<w$. Therefore the assertion follows from the observation that if $ws>w$ and $tw<w$
%for some reflection $t$, then $tws<ws$ (use, for example, \cite[Prop 5.9]{Hu}).
%From this the assertion follows (see \cite[Prop 3.1.3]{BB})
\end{proof}

\subsection{Application to  homology of affine Springer fibers} \hfill

\smallskip

Let $G$ be a split connected reductive group over $K$, and let $\gm\in G(K)$ be a  regular semisimple element.

\begin{Emp} \label{E:cent}
{\bf Set up.}

\smallskip

(a) Denote by $S_{G,\gm}\subseteq G_{\gm}$ the maximal split torus, and set $\La_{G,\gm}:=X_*(S_{G,\gm})$. Consider an embedding
\[
\La_{G,\gm}\hra L(S_{G,\gm})\subseteq LG_{\gm}:\mu\mapsto\mu(t).
\]
Then the group $\La_{G,\gm}$ acts on $\Fl_{G,\gm}$. Moreover, arguing as in \cite{KL}, there exists a closed subscheme of finite type $Y\subseteq \Fl_{G,\gm}$
such that $\Fl_{G,\gm}=\La_{G,\gm}(Y)$.

\smallskip

(b) Let $\C{L}$ be a local system on $\ov{T}_G$, and let
$\C{F}_{\C{L}}=\red_{\gm}^*(\C{L})$ be the corresponding local system on $\Fl_{G,\gm}$ (see Section~\re{affsprfib}(b)).
Recall that each homology group $H_i(\Fl_{G,\gm},\C{F}_{\C{L}})$  has a
structure of a $\qlbar[\La_G]$-module, while each $H_i(\Fl_{\gm},\C{F}_{\C{L}})$ has a structure of a $\qlbar[\La]$-module (see \rco{whaction}(a) and \rl{ind}(b)).

\smallskip

(c) Recall that the group $LG_{\gm}$ acts on $\Fl_{G,\gm}$, and the induced action on each $H_i(\Fl_{G,\gm},\C{F}_{\C{L}})$
factors through $\pi_0(LG_{\gm})$. Furthermore, the subgroup  $LG^{\sc}_{\gm}:=L(G^{\sc}_{\gm})\subseteq LG_{\gm}$ stabilizes
$\Fl_{\gm}=\Fl_{G^{\sc},\gm}\subseteq\Fl_{G,\gm}$, and hence $\pi_0(LG^{\sc}_{\gm})$ acts on
$H_i(\Fl_{\gm},\C{F}_{\C{L}})$.

%(d) Furthermore,  $\pi_0(LG^{\sc}_{\gm})\simeq X_*(G^{\sc}_{\gm})_{\Gm}$ is finitely
%generated abelian group, so its group algebra $\qlbar[\pi_0(LG^{\sc}_{\gm})]$
%is finitely generated over $\qlbar$, hence Noetherian.
\end{Emp}

\begin{Prop} \label{P:fingen}
(a) Each $H_i(\Fl_{G,\gm},\C{F}_{\C{L}})$ is a finitely generated $\qlbar[\La_G]$-module.

\smallskip

(b) Each $H_i(\Fl_{\gm},\C{F}_{\C{L}})$ is a finitely generated $\qlbar[\La]$-module.
\end{Prop}

\begin{proof}
(a) Let $\{F^j H_i(\Fl_{G,\gm},\C{F}_{\C{L}})\}_j$ be the
$\La_G\times\pi_0(LG^0_{\gm})$-invariant filtration from \rt{action}. It is enough to show that each graded piece $\gr^j
H_i(\Fl_{G,\gm},\C{F}_{\C{L}})$ is a finitely generated
$\qlbar[\La_G]$-module. We will show that each
$\gr^j{H}_i(\Fl_{G,\gm},\C{F}_{\C{L}})$ is a finitely generated
$\qlbar[\La_G]^{\ov{W}}$-module.

\smallskip

By \rt{action}, the action of $\qlbar[\La_G]^{\ov{W}}$ on
$\gr^j{H}_i(\Fl_{G,\gm},\C{F}_{\C{L}})$ factors through homomorphism
$\pr'_{\gm}:\qlbar[\La_G]^{\ov{W}}\to \qlbar[\pi_0(LG^0_{\gm})]$, making
$\qlbar[\pi_0(LG^0_{\gm})]$ a finite
$\qlbar[\La_G]^{\ov{W}}$-algebra. Therefore it is enough to show that each
$\gr^j{H}_i(\Fl_{G,\gm},\C{F}_{\C{L}})$ is a finitely generated
$\qlbar[\pi_0(LG^0_{\gm})]$-module.

\smallskip

Since the action of $\La_{G,\gm}$ on $\gr^j{H}_i(\Fl_{G,\gm},\C{F}_{\C{L}})$
factors through the homomorphism $\La_{G,\gm}\to \pi_0(LG^0_{\gm})$, it suffices to show that
each $\gr^j{H}_i(\Fl_{G,\gm},\C{F}_{\C{L}})$ is a
finitely generated $\qlbar[\La_{G,\gm}]$-module. Using \rl{fingen} and the observation of Section~\re{cent}(a),
we conclude that $H_i(\Fl_{G,\gm},\C{F}_{\C{L}})$ is a finitely generated
$\qlbar[\La_{G,\gm}]$-module. Thus the assertion follows from the
fact that $\qlbar[\La_{G,\gm}]$ is a Noetherian ring.

\smallskip

(b) Recall that the $\qlbar[\La_G]$-module $H_i(\Fl_{G,\gm},\C{F}_{\C{L}})$ is isomorphic
to the tensor product $\qlbar[\La_G]\otimes_{\qlbar[\La]}H_i(\Fl_{\gm},\C{F}_{\C{L}})$ (see \rl{ind}(b)).
Since $H_i(\Fl_{G,\gm},\C{F}_{\C{L}})$ is a Noetherian $\qlbar[\La_G]$-module, by part~(a), while $\qlbar[\La_G]$ is a free (hence flat) $\qlbar[\La]$-module, the $\qlbar[\La]$-module $H_i(\Fl_{\gm},\C{F}_{\C{L}})$ is Noetherian, thus finitely generated.
\end{proof}

\begin{Emp}
{\bf Remark.} The statement that homologies of affine Springer fibers are finitely generated over $\wt{W}_G$ appears also as Conjecture 3.6 in \cite{Lu3}. It is also mentioned in {\em loc. cit.} that the statement should follow from the result of \cite{Yun}.
\end{Emp}

\begin{Emp} \label{E:lengthfilt}
{\bf Length filtrations.}

\smallskip

(a) Let $S_{\gm}\subseteq G^{\sc}_{\gm}$ be the maximal split torus.
We equip $\La_{\gm}:=X_*(S_{\gm})$ and $X_*(G^{\sc}_{\gm})$ with the canonical  length functions (see Section~\re{lengthfun}(b)), and let $\{\La_{\gm}^{\leq n}\}_n$ and  $\{X_*(G^{\sc}_{\gm})^{\leq n}\}_n$ be the corresponding filtrations.

\smallskip

(b) Define $\{\pi_0(LG^{\sc}_{\gm})^{\leq n}\}_n$ to be the filtration on $\pi_0(LG^{\sc}_{\gm})$ defined as
the image of $\{X_*(G^{\sc}_{\gm})^{\leq n}\}_n$ under the surjection $X_*(G^{\sc}_{\gm})\to X_*(G^{\sc}_{\gm})_{\Gm_K}\simeq\pi_0(LG^{\sc}_{\gm})$ (see Section~\re{pi0}).
\end{Emp}

\begin{Emp} \label{E:filhom}
{\bf Filtrations on affine flag varieties.} Fix an Iwahori subgroup $\I\subseteq LG$.

\smallskip

(a) As in Section~\re{variant}, we denote by $\I^{\sc}$ the corresponding Iwahori subgroup of $LG^{\sc}$,
and set $\Fl=\Fl_{G^{\sc}}$ and $\wt{W}=\wt{W}_{G^{\sc}}$.

\smallskip

(b) Recall (see Section~\re{bruhatlength}(a)) that the correspondence  $w\mapsto \Fl_{\I}^w:=\I^{\sc} w\I^{\sc}$ gives
a canonical bijection between elements of $\wt{W}$ and $\I^{\sc}$-orbits in $\Fl=\Fl_{\I}$.

\smallskip

(c) For every element $w\in\wt{W}$, we denote by $\Fl^{\leq w}=\Fl^{\leq w}_{\I}\subseteq \Fl$ be the closure of $\Fl_{\I}^w\subseteq \Fl$.
%Then $\Fl^{\leq w}$ has a stratification $\{\Fl^u_{\I}\}_{u\leq w}$.

\smallskip

(d) Using Section~\re{bruhatlength}, for every proper subset $J\subsetneq\wt{S}$, an element $w$ is right $J$-longest if and only if $\pi_J^{-1}(\pi_J(\Fl^{\leq w}))=\Fl^{\leq w}$, where $\pi_J$ is the projection $\Fl\to\Fl_J$, and element $w$ is left $J$-longest if and only if $\Fl^{\leq w}\subseteq\Fl$ is
$\P_{J,\I}$-invariant.

\smallskip

(e) For each $n\in\B{N}$, we denote by  $\Fl^{\leq n}=\Fl_{\I}^{\leq n}$ the union
$\bigcup_{w\in\wt{W}^{\leq n}} \Fl_{\I}^{\leq w}$. Then $\{\Fl^{\leq n}\}_n$ is a filtration of $\Fl$ (compare Section~\re{filtr}(a)).

\smallskip

(f) For all $n,m\in\B{N}$ and $J,J'\subsetneq \wt{S}$, we denote by
\[
\Fl^{\leq n;J_r;J'_l;m_{\reg}}=\Fl_{\I}^{\leq n;J_r;J'_l;m_{\reg}}\subseteq\Fl^{\leq n}
\]
the union $\bigcup_{w}\Fl_{\I}^{\leq w}$, taken over $w\in\wt{W}^{\leq n}$ such that $w$ is right $J$-longest, left $J'$-longest and $m$-regular (see Section~\re{mreg} and \rn{k-reg}).

\smallskip

(g) We claim that $\{\Fl^{\leq n;J_r;J'_l;m_{\reg}}\}_n$ is also a filtration on $\Fl$. Equivalently, we claim that for every $w\in\wt{W}$ there exists $w'\geq w$ such that $w'$ is right $J$-longest, left $J'$-longest and $m$-regular. But it follows from Lemmas~\ref{L:k-reg} and \ref{L:dem}(b) that for every $m$-regular element $\la\in\La$, the element $w':=w_{J'}*\la*w*w_J$   satisfies all these properties.
\end{Emp}

\begin{Emp} \label{E:dep}
{\bf Dependence on $\I$}. Assume that we are in the situation of Section~\re{filhom}.

\smallskip

(a) Let $\I'\subseteq LG$ be another Iwahori subgroup, and let $h\in LG^{\sc}$ be such that $\I'=h\I h^{-1}$.
Then for every $w\in\wt{W}$ we have an equality  $\Fl^w_{\I'}=h\Fl_{\I}^w\subseteq \Fl$, hence $\Fl^{\leq w}_{\I'}=h\Fl_{\I}^{\leq w}\subseteq \Fl$.

\smallskip

(b) In the situation of part~(a) assume further that $\I'\subseteq \P_{J';\I}$ for some subset $J'\subsetneq\wt{S}$ and that element $w\in\wt{W}$ is  left $J'$-longest. Then $h\in\P^{\sc}_{J';\I}$, hence by a combination of part~(a) and Section~\re{filhom}(d), we have $\Fl^{\leq w}_{\I'}=h\Fl_{\I}^{\leq w}=\Fl_{\I}^{\leq w}\subseteq \Fl$.

\smallskip

(c) By part~(b), the filtration $\{\Fl^{\leq n;J_r;J'_l;m_{\reg}}\}_n$ does not change, if the Iwahori subgroup $\I$ is replaced by  $\I'\subseteq \P_{J';\I}$.
\end{Emp}

\begin{Emp} \label{E:filhom1}
{\bf Filtrations on affine Springer fibers and their homology.}

\smallskip

(a) For a closed subscheme $Y\subseteq \Fl$, we denote by
$Y_{\gm}:=Y\cap \Fl_{\gm}$ the corresponding closed subscheme of
$\Fl$, and denote by $H'_i(Y_{\gm},\C{F}_{\C{L}})$ the image of the
natural map
\[
H_i(Y_{\gm},\C{F}_{\C{L}})\to H_i(\Fl_{\gm},\C{F}_{\C{L}}).
\]

\smallskip

(b) By Section~\re{filhom}, both  $\{\Fl^{\leq
n}_{\gm}\}_n$ and $\{\Fl^{\leq n;J_r;J'_l;m_{\reg}}_{\gm}\}_n$ are filtrations of
$\Fl_{\gm}$, thus both $\{H'_i(\Fl^{\leq
n}_{\gm},\C{F}_{\C{L}})\}_n$ and $\{H'_i(\Fl^{\leq
n;J_r;J'_l;m_{\reg}}_{\gm},\C{F}_{\C{L}})\}_n$ are filtrations of
$H_i(\Fl_{\gm},\C{F}_{\C{L}})$.
\end{Emp}

\begin{Emp} \label{E:goodpos}
{\bf Good position.} We say that element $\gm\in LG$ is in {\em good position} with a parahoric subgroup $\P\subseteq LG$, if the maximal split subtorus $S_{G,\gm}\subseteq G_{\gm}$ satisfies $L^+(S_{G,\gm})\subseteq \P$.
\end{Emp}

%The goal of this subsection is to show that filtrations of Section~\re{filhom1}(b) are compatible with the length filtration
%$\La$. Moreover, these filtrations are finitely generated, if
%$\gm$ is in {\em good position with $\I$}.

\begin{Lem} \label{L:good pos}
Let $\gm$ be in good position with $\P$. Then

\smallskip

(a) There exists an Iwahori subgroup $\I\subseteq\P$ of $LG$ such that
$\gm$ is in good position with $\I$.

\smallskip

(b) There exists a maximal split torus $T\supseteq S_{G,\gm}$ such that  $L^+(T)\subseteq \P$.
\end{Lem}
\begin{proof}
(a) The composition $L^+(S_{G,\gm})\hra\P\to\P/\P^+=M_{\P}$ induces an embedding $\ov{S}_{G,\gm}\hra M_{\P}$. Choose a Borel subgroup $B_{\P}\subseteq M_{\P}$, containing $\ov{S}_{G,\gm}$. Then the preimage $\I\subseteq\P$ of  $B_{\P}\subseteq M_{\P}$ is an Iwahori subgroup such that $L^+(S_{G,\gm})\subseteq \P$.

\smallskip

(b) Was shown during the course of the proof of \rl{lengthfun}.
\end{proof}

\begin{Lem} \label{L:good}
Assume that $\gm$ is in good position with $\I$. Then the filtration $\{\Fl^{\leq n}_{\gm}\}_n$ of $\Fl_{\gm}$ is compatible with and finitely generated over the filtration $\{\La_{\gm}^{\leq n}\}_n$ of $\La_{\gm}$.
\end{Lem}

\begin{proof}
First we claim that filtration $\{\Fl^{\leq n}_{\gm}\}_n$ is
compatible with filtration $\{\La^{\leq n}_{\gm}\}_n$ on
$\La_{\gm}$. For this we have to show that for every two elements
$\la\in\La_{\gm}$ and $g\in LG^{\sc}$ we have $l_{\I}(\la g)\leq l(\la)+l_{\I}(g)$. But this immediately follows from a combination of Section~\re{bruhat}(c) and \rl{lengthfun}.

It now follows from Section~\re{cent}(a) that we have $\La_{\gm}(\Fl_{\gm}^{\leq n})=\Fl_{\gm}$ for some $n\in\B{N}$.
Now the fact that $\{\Fl^{\leq n}_{\gm}\}_n$ is finitely generated over $\{\La^{\leq n}_{\gm}\}_n$ follows immediately from
\rl{bound} applied to the subgroup $\La':=\La_{\gm}$ and  subset $A:=\wt{W}^{\leq n}$.
\end{proof}

\begin{Lem} \label{L:image}
Let $w_1,\ldots,w_n\in\wt{W}$. Then for every $w\in\wt{W}$ the isomorphism $a_{w,\C{L}}$ from \rl{ind}(a) satisfies
\begin{equation} \label{Eq:dem}
a_{w,\C{L}}(H'_i(\cup_j \Fl_{\gm}^{\leq w_j},\C{F}_{\C{L}}))\subseteq
H'_i(\cup_j \Fl_{\gm}^{\leq w_j*w^{-1}},\C{F}_{\ov{w}_*(\C{L})}).
\end{equation}
\end{Lem}

\begin{proof}
We are going to show the assertion by induction on $l(w)$.

\smallskip

Let $w=s$ be a simple affine reflection. Then the closed subscheme
$Y:=\bigcup_j \Fl^{\leq w_j*s}$ of $\Fl$ satisfies
$Y=\pi_s^{-1}(\pi_s(Y))$ (see Section~\re{filhom}(b) and \rl{dem}(b)), thus it follows from Section~\re{truncation}(c)
that we have an equality
\[
a_{s,\C{L}}(H'_i(Y_{\gm},\C{F}_{\C{L}}))=H'_i(Y_{\gm},\C{F}_{\ov{w}_*(\C{L})}).
\]
Since $w_j\leq w_j*s$ (see \rl{k-reg}), we have $H'_i(\bigcup_j \Fl_{\gm}^{\leq
w_j},\C{F}_{\C{L}})\subseteq H'_i(Y_{\gm},\C{F}_{\C{L}})$, and
inclusion \form{dem} follows in this case.

\smallskip

For an arbitrary $w$ choose $s\in\wt{S}$ such that $w':=sw<w$.
Then $w=sw'$, and $w^{-1}=w'^{-1}*s$, therefore inclusion
\form{dem} for $w$ follows from that for $w'$ and $s$.
\end{proof}

\begin{Cor} \label{C:filt}
Filtrations  $\{H'_i(\Fl^{\leq n}_{\gm},\C{F}_{\C{L}})\}_n$,  %$\{H'_i(\Fl^{\leq n;J'_l}_{\gm},\C{F}_{\C{L}})\}_n$
and $\{H'_i(\Fl^{\leq n;J_r;J'_l;m_{\reg}}_{\gm},\C{F}_{\C{L}})\}_n$ of $H_i(\Fl_{\gm},\C{F}_{\C{L}})$ are
compatible with the length filtration $\{\La^{\leq n}\}_n$ of $\La$.
\end{Cor}

\begin{proof}
By \rl{image}, to show the assertion for $\{H'_i(\Fl^{\leq
n}_{\gm},\C{F}_{\C{L}})\}_n$, it suffices to show an inclusion
\[
\bigcup_{w\in\wt{W}^{\leq n}}\Fl^{\leq w*\la^{-1}}\subseteq \Fl^{\leq
n+l(\la)}
\]
for every $n\in\B{N}$ and $\la\in\La$. Since $l(w^{-1})=l(w)$
for every $w\in\wt{W}$, we have to show that for every $w\in\wt{W}$ and $\la\in\La$ we have an inclusion
$\Fl^{\leq w*\la}\subseteq \Fl^{\leq l(w)+l(\la)}$. This follows from  the inequality
$l(w*\la)\leq l(w)+l(\la)$ (use \rl{dem}(a)).

%Next, the assertion for $\{H'_i(\Fl^{\leq n;J'_l}_{\gm},\C{F}_{\C{L}})\}_n$ follows from by the same argument as above together with observation that if $w$ if $J'$-longest, then $w*\la$ is $J'$-longest as well (use \rl{dem}(b)).

\smallskip

Similarly, to show the assertion for  $\{H'_i(\Fl^{\leq
n;J_r;J'_l;m_{\reg}}_{\gm},\C{F}_{\C{L}})\}_n$, it is enough to show that for every
$\la\in\La$ and every $m$-regular, right $J$-longest and left $J'$-longest element $w'\in\wt{W}$, we have an inclusion
\[
\Fl^{\leq w'*\la}\subseteq \Fl^{\leq l(w')+l(\la);J_r;J'_l;m_{\reg}}.
\]

Consider element $w:=w'*\la*w_J$. Then $w'*\la\leq w$ (see \rl{k-reg}), thus $\Fl^{\leq w'*\la}\subseteq \Fl^{\leq w}$, and
it remains to show that element $w$ is right $J$-longest, left $J'$-longest, $m$-regular, and satisfies $l(w)\leq   l(w')+l(\la)$.

\smallskip

Note that $w$ is right $J$-longest by \rl{dem}(b), and $w$ is  $m$-regular, since $w'$ is such (use \rl{k-reg}). Next, since $w'$ is left $J'$-longest, we have an equality $w'=w_{J'}*w'$ (by \rl{dem}(b)), hence $w_{J'}*w=w$, thus $w$ is left $J'$-longest (by \rl{dem}(b)). Finally, since $w'$ is right $J$-longest, we have $w'=w'*w_J$ (by \rl{dem}(b)), hence $l(w)\leq l(w')+l(\la)$ (by \rl{dem}(c)).
\end{proof}

%The goal of this section is to prove the following result.

\begin{Thm} \label{T:fingen}
Assume that $\gm$ is in good position with $\I$. Then for each
$i\in\B{Z}$, the filtration $\{H'_i(\Fl^{\leq n}_{\gm},\C{F}_{\C{L}})\}_n$
of $H_i(\Fl_{\gm},\C{F}_{\C{L}})$ is finitely
generated over $\{\La^{\leq n}\}_n$.
\end{Thm}

\begin{proof} We are going to deduce the result from \rl{good} and \rco{filt},
 mimicking the argument of \rp{fingen}. For simplicity, we divide our argument into six steps:

\smallskip

\noindent{\bf Step 1.}  We equip $\La, \La_{\gm}, \qlbar[\La]^{\ov{W}}$ and $\pi_0(LG^{\sc}_{\gm})$ with length filtrations
(see Sections~\re{filtr0} and \re{lengthfilt}), and let $R(\qlbar[\La]), R(\qlbar[\La_{\gm}], R(\qlbar[\La]^{\ov{W}})$ and $R(\qlbar[\pi_0(LG^{\sc}_{\gm})])$
be the corresponding Rees algebras (see Section~\re{filtr}(c)). By definitions, the homomorphisms
\[
\qlbar[\La]\supseteq \qlbar[\La]^{\ov{W}}\overset{\pr'_{\gm}}{\lra}\qlbar[\pi_0(LG^{\sc}_{\gm})]\overset{\eta}{\lla}  \qlbar[\La_{\gm}]
\]
preserve length filtrations, thus give rise to homomorphisms
\[
R(\qlbar[\La])\supseteq R(\qlbar[\La]^{\ov{W}})\overset{R(\pr'_{\gm})}{\lra}R(\qlbar[\pi_0(LG^{\sc}_{\gm})])\overset{R(\eta_{\gm})}{\lla} R(\qlbar[\La_{\gm}]).
\]

\smallskip

\noindent{\bf Step 2.} Let $R(H_i(\Fl_{\gm},\C{F}_{\C{L}}))$ be the Rees module, corresponding to the filtration
$\{H'_i(\Fl^{\leq n}_{\gm},\C{F}_{\C{L}})\}_n$ of $H_i(\Fl_{\gm},\C{F}_{\C{L}})$. We have to show that $R(H_i(\Fl_{\gm},\C{F}_{\C{L}}))$ is a finitely generated
$R(\qlbar[\La])$-module (use \rco{filt} and see Section~\re{Rees}(c)).

\smallskip

Let $\{F^j H_i(\Fl_{\gm},\C{F}_{\C{L}})\}_j$ be the finite
$\La\times \pi_0(LG^{\sc}_{\gm})$-invariant filtration from \rt{action}. It induces a finite filtration
$\{F^j R(H_i(\Fl_{\gm},\C{F}_{\C{L}}))\}_j$ of the Rees module $R(H_i(\Fl_{\gm},\C{F}_{\C{L}}))$, and the
graded piece $\gr^j R(H_i(\Fl_{\gm},\C{F}_{\C{L}}))$ is naturally identified with the Rees module
$R(\gr^j H_i(\Fl_{\gm},\C{F}_{\C{L}}))$ corresponding to the induced filtration
$\{\gr^j H'_i(\Fl^{\leq n}_{\gm},\C{F}_{\C{L}})\}_n$
of $\gr^jH_i(\Fl_{\gm},\C{F}_{\C{L}})$. It suffices to show that each  $R(\gr^j
H_i(\Fl_{\gm},\C{F}_{\C{L}}))$ is a finitely generated
$R(\qlbar[\La])$-module.
%We will show that $R(\gr^j H_i(\Fl_{\gm},\C{F}_{\C{L}}))$ is a
%finitely generated $R(\qlbar[\La]^W)$-module.

\smallskip

\noindent{\bf Step 3.} Unfortunately, we do not know whether the filtration $\{\gr^j
H'_i(\Fl^{\leq n}_{\gm},\C{F}_{\C{L}}))\}_n$ of $\gr^j
H_i(\Fl_{\gm},\C{F}_{\C{L}})$ is compatible with filtration
$\{\pi_0(LG^{\sc}_{\gm})^{\leq n}\}_n$.

\smallskip

To overcome this difficulty, we define a subspace
\[
\qlbar[\pi_0(LG^{\sc}_{\gm})]'_n\subseteq\qlbar[\pi_0(LG^{\sc}_{\gm})^{\leq n}]
\]
consisting of all $x\in\qlbar[\pi_0(LG^{\sc}_{\gm})^{\leq n}]$ such that for all $j$ and $m$ we have an inclusion
\[
x(\gr^j H'_i(\Fl^{\leq m}_{\gm},\C{F}_{\C{L}}))\subseteq\gr^j H'_i(\Fl^{\leq n+m}_{\gm},\C{F}_{\C{L}}),
\]
and set
\[
R'(\qlbar[\pi_0(LG^{\sc}_{\gm})]):=\bigoplus_n\qlbar[\pi_0(LG^{\sc}_{\gm})]'_n
\subseteq R(\qlbar[\pi_0(LG^{\sc}_{\gm})]).
\]
By construction, $R'(\qlbar[\pi_0(LG^{\sc}_{\gm})])\subseteq R(\qlbar[\pi_0(LG^{\sc}_{\gm})])$
is a graded subalgebra, and $R(\gr^j H_i(\Fl_{\gm},\C{F}_{\C{L}}))$ is a
graded $R'(\qlbar[\pi_0(LG^{\sc}_{\gm})])$-module.

\smallskip

\noindent{\bf Step 4.} By \rt{action}, the action of $\qlbar[\La]^{\ov{W}}$ on $\gr^j
H_i(\Fl_{\gm},\C{F}_{\C{L}})$ factors through homomorphism $\pr'_{\gm}$. Since
filtration  $\{\gr^j H'_i(\Fl^{\leq n}_{\gm},\C{F}_{\C{L}})\}_n$ of
$\gr^j H_i(\Fl_{\gm},\C{F}_{\C{L}})$  is compatible with $\{\qlbar[\La^{\leq n}]^{\ov{W}}\}_n$, the
image of the homomorphism $R(\pr'_{\gm})$ (see Step~1)
lies in $R'(\qlbar[\pi_0(LG^{\sc}_{\gm})])$, and the action of $R(\qlbar[\La]^{\ov{W}})$ on $R(\gr^j H'_i(\Fl_{\gm},\C{F}_{\C{L}}))$
is induced by the $R'(\qlbar[\pi_0(LG^{\sc}_{\gm})])$-action via
\[
R(\pr'_{\gm}):R(\qlbar[\La_{\Gm_K}]^{\ov{W}})\to R'(\qlbar[\pi_0(LG^{\sc}_{\gm})]).
\]

\smallskip

\noindent{\bf Step 5.} By the construction of the length filtration on $\pi_0(LG^{\sc}_{\gm})$ (see Section~\re{lengthfilt}), the map $R(\pr'_{\gm}):R(\qlbar[\La])\to R(\qlbar[\pi_0(LG^{\sc}_{\gm})])$ is surjective.  Since $R(\qlbar[\La])$ is a finitely generated $\qlbar$-algebra, it is a finite $R(\qlbar[\La_{\Gm_K}]^{\ov{W}})=R(\qlbar[\La_{\Gm_K}])^{\ov{W}}$-algebra.
Therefore $R(\qlbar[\pi_0(LG^{\sc}_{\gm})])$ is finite $R(\qlbar[\La]^{\ov{W}})$-algebra.
Since $R(\qlbar[\La_{\Gm_K}]^{\ov{W}})$ is a Noetherian ring (see Section~\re{winv}), the subalgebra
\[
R'(\qlbar[\pi_0(LG^{\sc}_{\gm})])\subseteq R(\qlbar[\pi_0(LG^{\sc}_{\gm})])
\]
is a finite $R(\qlbar[\La]^{\ov{W}})$-algebra as well. Thus it suffices to show that each Rees module $R(\gr^j H'_i(\Fl_{\gm},\C{F}_{\C{L}}))$ is finitely generated over $R'(\qlbar[\pi_0(LG^{\sc}_{\gm})])$.

\smallskip

\noindent{\bf Step 6.}
Since filtration  $\{\gr^j H'_i(\Fl^{\leq
n}_{\gm},\C{F}_{\C{L}})\}_n$ is compatible
with $\{\La_{\gm}^{\leq n}\}_n$ (by \rl{good}), the homomorphism
$R(\eta_{\gm})$ factors through $R'(\qlbar[\pi_0(LG^{\sc}_{\gm})])$, thus it suffices to show
that each
$R(\gr^j H'_i(\Fl_{\gm},\C{F}_{\C{L}}))$ is  a finitely generated
$R(\qlbar[\La_{\gm}])$-module.

On the other hand, it follows from the combination of Lemmas~
\ref{L:good} and \ref{L:rees}, that the Rees module
$R(H'_i(\Fl_{\gm},\C{F}_{\C{L}}))$ is a finitely generated
$R(\qlbar[\La_{\gm}])$-module. Since the Rees module $R(\gr^j H_i(\Fl_{\gm},\C{F}_{\C{L}}))\simeq\gr^jR(H_i(\Fl_{\gm},\C{F}_{\C{L}}))$ is a
subquotient of $R(H_i(\Fl_{\gm},\C{F}_{\C{L}}))$, the assertion follows from the fact that the Rees algebra
$R(\qlbar[\La_{\gm}])$ is Noetherian (see \rco{fg}).
\end{proof}

\begin{Cor} \label{C:fingen}
Assume that $\gm$ is in good position with $\P_{J',\I}$. Then the filtration $\{H'_i(\Fl^{\leq n;J_r;J'_l;m_{\reg}}_{\gm},\C{F}_{\C{L}})\}_n$
of $H_i(\Fl_{\gm},\C{F}_{\C{L}})$ is finitely generated over $\{\La^{\leq n}\}_n$.
\end{Cor}

\begin{proof}
By \rl{good pos}(a), there exists an Iwahori subgroup $\I'\subseteq \P_{J',\I}$ such that $\gm$ is in good position with $\I'$.
Then, by Section~\re{dep}(c), we can replace $\I$ by $\I'$, thus assuming that $\gm$ is in good position with $\I$.

Hence, by \rt{fingen}, the Rees module $\bigoplus_n H'_i(\Fl^{\leq n}_{\gm},\C{F}_{\C{L}})$  is a finitely generated
$R(\qlbar[\La])$-module. Since
$R(\qlbar[\La])$ is Noetherian (see \rco{fg}), the submodule $\bigoplus_n H'_i(\Fl^{\leq n;J_r;J'_l;m_{\reg}}_{\gm},\C{F}_{\C{L}})$ of
$\bigoplus_n H'_i(\Fl^{\leq n}_{\gm},\C{F}_{\C{L}})$ is
finitely generated over $R(\qlbar[\La])$ as well. Therefore the filtration
$\{H'_i(\Fl^{\leq n;J_r;J'_l;m_{\reg}}_{\gm},\C{F}_{\C{L}})\}_n$ is finitely
generated over $\{\La^{\leq n}\}_n$, as claimed.
\end{proof}

\begin{Thm} \label{T:inj}
Assume that $\gm$ is in good position with $\P_{J',\I}$. Then for every sufficiently large $m\in\B{N}$ and every $n\in\B{N}$
the canonical surjection
\[
H_i(\Fl^{\leq n;J_r;J'_l;m_{\reg}}_{\gm},\C{F}_{\C{L}})\to H'_i(\Fl^{\leq n;J_r;J'_l;m_{\reg}}_{\gm},\C{F}_{\C{L}})
\]
is an isomorphism.
\end{Thm}

\begin{proof}
We have to show that for every sufficiently large $m$, the morphism
\[
H_i(\Fl^{\leq n;J_r;J'_l;m_{\reg}}_{\gm},\C{F}_{\C{L}})\to H_i(\Fl_{\gm},\C{F}_{\C{L}})
\]
is injective.
By Section~\re{triv}, we can assume that $\C{F}_{\C{L}}=\qlbar$. Next, arguing as in the proof of \rco{fingen}, we can assume that $\gm$ is in good position with $\I$. Then by \rl{good pos}(b), there exists a maximal split torus $T\supseteq S_{G,\gm}$ such that $\T\subseteq \I$. Thus the assertion follows from a combination of Section~\re{remmreg}(c) and \cite[Theorem~0.1]{BV}.
\end{proof}

%The following result, proven in \cite{BV}, is crucial.

%\begin{Thm} \label{T:inj}
%For every $\gm\in G^{rss}(F)$ in good position there exists an
%integer $k$ such that $H_i(\C{P}_{\gm})\to H_i(\Fl_{\gm})$ is
%injective for every $i\in \B{Z}$ and every closed subscheme
%$\C{P}\subseteq \Fl$ of the form
%$\C{P}=\cup_{i=1}^l\Fl^{\leq w_i}$, where each $w_i\in \wt{W}$
%is $m$-regular.
%\end{Thm}

%\begin{Rem}
%We presume that  $\wh{H}_i^{J;\geq_R
%k}(\Fl_{\gm})$ is finitely generated for all $k$ and, in
%particular,  $\wh{H}_i^J(\Fl_{\gm})$ is
%finitely generated, but this does not follow from our method.
%\end{Rem}

%\begin{Thm} \label{L:good}
%(a) If $\C{L}$ is $W$-equivariant then $V^i_{\gm}(\C{L})$
%is a finitely generaled $\qlbar[\wt{W}]$-module.

%(b) Equip $V^i_{\gm}(\C{L})$ with a filtration induced from the
%length filtration on $\Fl$ and $\qlbar[\wt{W}]$ with a length
%filtration. Then $V^i_{\gm}(\C{L})$ is a good filtered
%$\qlbar[\wt{W}]$-module.
%\end{Thm}

%\begin{proof}
%It is well-known that (a) is a consequence of (b), so we restrict
%ourself with the proof of (b). The fact that the filtration on
%$V^i_{\gm}(\C{L})$ is compatible with the filtration on
%$\qlbar[\wt{W}]$ follows from ..... So it remains to show that the
%filtration is good.

%TO FINISH!!!
%\end{proof}

\section{Geometric interpretation of characters and stability}

Let $G$ be a connected reductive group over $F=\fq((t))$. \label{a:F} Starting from Section~\ref{S:rationality} we will assume that $G$ splits over $F^{\nr}:=F\otimes_{\fq}\fqbar$,  \label{a:fnr} hence split over $K:=\fqbar((t))$,  \label{a:K} thus the construction of \rs{aff} applies.

\subsection{Endoscopy}

\begin{Emp} \label{E:cohom}
{\bf Cohomology of tori.}

\smallskip

(a) Recall that for every torus $S$ over $F$ we have $H^1(F^{\nr},S)=1$, thus we have an isomorphism
$H^1(F,S)\simeq H^1(\Gm_{F^{\nr}/F}, S(F^{\nr}))$.

\smallskip

(b) Furthermore, it follows from the results of Kottwitz (\cite[Lemma~2.2
and Proposition~2.3]{Ko1}) or Borovoi (\cite{Bo}) that we have a natural isomorphism
\[
H^1(F,S)\simeq H^1(\Gm_{F^{\nr}/F}, S(F^{\nr}))\simeq X_*(S)_{\Gm_F,\tor},
\]
functorial in $S$, where $(-)_{\Gm_F}$ stands for $\Gm_F$-coinvariants,  \label{a:gmfcoinv}  and $(-)_{\tor}$ for torsion.

\smallskip

(c) More precisely, if $x\in H^1(\Gm_{F^{\nr}/F}, S(F^{\nr}))$ is the class of the cocycle
$\tau\mapsto x_{\tau}$, then its image in $X_*(S)_{\Gm_F}$ is the projection of the class
$[x_{\si}]\in \pi_0(LS)\simeq  X_*(S)_{\Gm_K}$ (see Section~\re{pi0}) of $x_{\si}\in S(K)=(LS)(\fqbar)$.
\end{Emp}

\begin{Emp} \label{E:stableconj}
{\bf Stable conjugacy.}  \label{a:stableconj}

\smallskip

(a) Recall that two elements $\gm,\gm'\in G^{\rss}(F)$ are called {\em $G(F^{\sep})$-conjugate}, if there exists an element $g\in G(F^{\sep})$ such that $g\gm g^{-1}=\gm'$. Note that in this case, for every $\tau\in\Gm_F$ the element $g^{-1}{}^{\tau}g\in G(F^{\sep})$ belongs to
$G_{\gm}(F^{\sep})$.

Furthermore, two elements $\gm,\gm'\in G^{\rss}(F)$ are called {\em stably conjugate}, if there exists $g$ as above such that
$g^{-1}{}^{\tau}g\in G^0_{\gm}(F^{\sep})$ for every $\tau\in\Gm_F$. Then we can form an invariant $\inv(\gm,\gm')\in H^1(F,G^0_{\gm})$,  \label{a:inv} depending only on the $G(F)$-conjugacy classes of $\gm$ and $\gm'$, defined to be the class of the cocycle $\tau\mapsto g^{-1}{}^{\tau}g$.

\smallskip

(b) Since $H^1(F^{\nr},G^0_{\gm})=1$ (see Section~\re{cohom}(a)), for every stable conjugate elements  $\gm,\gm'\in G^{\rss}(F)$ there exists $g\in G(F^{\nr})$ such that $g\gm g^{-1}=\gm'$ and
\[
g^{-1}{}^{\si}g\in G^0_{\gm}(F^{\nr})\subseteq G^0_{\gm}(K)=LG^0_{\gm}(\fqbar).
\]
Moreover, the image of $\inv(\gm,\gm')\in H^1(F,G^0_{\gm})$ in $X_*(G^0_{\gm})_{\Gm_F}$ (see Sections~\re{cohom}(b),(c)) equals the projection of
$[g^{-1}{}^{\si}g]\in \pi_0(LG^0_{\gm})\simeq  X_*(LG^0_{\gm})_{\Gm_K}$.

\smallskip

(c) Recall that two embeddings $\fa,\fa':S\hra G$ of a maximal torus defined over $F$ are called {\em stably conjugate} if there exists $g\in G(F^{\sep})$ such that $\fa'=\Ad_g\circ \fa$ and the induced morphism $\Ad_g:\fa(S)\isom\fa'(S)$ is defined over $F$. Moreover, since $H^1(F^{\nr},S)=1$, there exists $g\in G(F^{\nr})$ satisfying this property.

As in part~(a), we can form an invariant $\inv(\fa,\fa')\in H^1(F,S)$, depending only on the $G(F)$-conjugacy classes of $\gm$ and $\gm'$.  Explicitly, $\inv(\fa,\fa')$ is the class of the cocycle $\tau\mapsto \fa^{-1}(g^{-1}{}^{\tau}g)\in S(F^{\sep})$.
\smallskip

(d) Notice that elements $\gm,\gm'\in G^{\rss}(F)$ are stably conjugate if and only if there exists an embedding $\fa'_{\gm}:G^0_{\gm}\hra G$, which is stably conjugate to the inclusion $\fa_{\gm}:G^0_{\gm}\hra G$ and satisfies $\fa'_{\gm}(\gm)=\gm'$.

\smallskip

(e) In the situation of part~(c), assume that $G$ is semisimple and simply connected. Then $H^1(F,G)=1$, therefore the correspondence
$\fa'\mapsto \inv(\fa,\fa')$ defines a bijection between the set of $G(F)$-conjugacy classes of elmeddings $\fa':S\hra G$,
stably conjugate to $\fa$ and $H^1(F,S)$.
\end{Emp}

\begin{Emp} \label{E:dualgp}
{\bf Langlands dual groups.}

\smallskip

(a) Let $\wh{G}$  \label{a:ghat} be the (connected) Langlands dual group over $\qlbar$. Then $\wh{G}$ is equipped with a canonical continuous homomorphism $\rho_G:\Gm_F\to \on{Out}(\wh{G})$. In particular, we have a natural action of
$\Gm_F$ on $Z(\wh{G})$, so we can consider the group of fixed points $Z(\wh{G})^{\Gm_F}$. Moreover, composing $\rho_G$ with a section of the natural projection $\Aut(\wh{G})\to\on{Out}(\wh{G})$, we get an action of $\Gm_F$ on $\wh{G}$, defined uniquely up to a $\wh{G}$-conjugacy.

\smallskip

(b) Every embedding  $\fa:S\hra G$ of a maximal torus over $F$ gives rise to a $\Gm_F$-invariant $\wh{G}$-conjugacy class of embeddings
$[\wh{\fa}]:\wh{S}\hra \wh{G}$ of maximal tori, hence canonical $\Gm_F$-equivariant embedding $Z_{\fa}:Z(\wh{G})\hra\wh{S}$. Furthermore, two
embeddings $\fa,\fa':S\hra G$ are stably conjugate if and only if $[\wh{\fa'}]=[\wh{\fa}]$.

\smallskip

(c) Notice that for every torus $S$ over $F$, we have a natural isomorphism of groups $\Hom(X_*(S),\qlbar\m)\simeq \wh{S}$, hence
$\Hom(X_*(S)_{\Gm_F},\qlbar\m)\simeq \wh{S}^{\Gm_F}$. In particular, for every $\xi\in  \wh{S}^{\Gm_F}$ and $x\in H^1(F,S)\simeq X_*(S)_{\Gm_F}$, we can form a pairing $\lan\xi,x\ran\in\qlbar\m$. \label{a:pairing}
\end{Emp}

\begin{Emp} \label{E:enddatum}
{\bf Endoscopic data} (compare \cite[$\S$7]{Ko1} and \cite[Section~1.3]{KV}).
 Suppose we are in the situation of Section~\re{dualgp}.

\smallskip

(a) Slightly modifying the terminology of \cite[$\S$7]{Ko1}, by an {\em endoscopic datum}  \label{a:enddatum} for $G$, we mean a triple $\C{E}=(s,\wh{H},\rho)$, where

\smallskip

\quad\quad$\bullet$ $s\in\wh{G}$ is a semisimple element;

\smallskip

\quad\quad$\bullet$ $\wh{H}:=(\wh{G}_s)^0\subseteq \wh{G}$ is a connected centralizer;

\smallskip

\quad\quad$\bullet$ $\rho:\Gm_F\to\on{Out}(\wh{H})$ is a continuous homomorphism

\smallskip

%\noindent
such that

\smallskip

\quad\quad$\bullet$ the $\wh{G}$-conjugacy class of the inclusion $\wh{H}\hra\wh{G}$ is $\Gm_F$-invariant,
%so we have a $\Gm_F$-equivariant embedding $Z(\wh{G})\hra Z(\wh{H})$;

\smallskip

\quad\quad$\bullet$ we have $s\in Z(\wh{H})^{\Gm_F}\subseteq \wh{H}\subseteq\wh{G}$.

\smallskip

(b) We say that two endoscopic data $\C{E}=(s,\wh{H},\rho)$ and $\C{E}'=(s',\wh{H}',\rho')$ {\em conjugate} and write $\C{E}\simeq\C{E}'$,
if there exists $g\in\wh{G}$ such that $gsg^{-1}=s'$ and the induced isomorphism $(\Ad_g)_*:\Out(\wh{H})\isom \Out(\wh{H}')$ induced by isomorphism $\Ad_g:H\isom H'$ satisfies $\rho'=\rho\circ(\Ad_g)_*$.
\end{Emp}

\begin{Emp} \label{E:exenddatum}
{\bf Examples.}

\smallskip

(a) As in \cite[Section~1.3.8]{KV}, to each pair $(S, \kappa)$, where $S\subseteq G$ is a maximal torus over $F$ and $\kappa\in \wh{S}^{\Gm_F}$,  we associate an endoscopic datum $\C{E}_{S,\kappa}=(s,\wh{H},\rho)$ \label{a:eskappa} for $G$, defined up to a $\wh{G}$-conjugacy.

Namely, let $\fa:S\hra G$ be the inclusion and choose an embedding $\wh{\fa}:\wh{S}\hra\wh{G}$ belonging to a canonical conjugacy class of embeddings $[\wh{\fa}]$ from Section~\re{dualgp}(b), set $s:=\wh{\fa}(\ka)\in\wh{G}$, put $\wh{H}:=(\wh{G}_s)^0$, and let $\rho:\Gm_F\to \on{Out}(\wh{H})$ be the unique homomorphism such that $s\in Z(\wh{H})^{\Gm_F}$ and the $\wh{H}$-conjugacy class of the inclusion $\wh{\fa}:\wh{S}\hra\wh{H}$ is $\Gm_F$-invariant.

\smallskip

(b)\footnote{The example in part~(b) (and compatibility in part~(c) below) is not used in this work, it is included only in order to relate
results of this work to the  general picture outlined in the introduction.}
Let $\la:W_F\to {}^LG(\qlbar)$ be a continuous homomorphism, lifting the inclusion $W_F\hra\Gm_F$. Following \cite{Lan}, every semisimple element $s\in Z_{\wh{G}}(\la)$ gives rise to the endoscopic datum $\C{E}_{\la,\ka}=(s,\wh{H},\rho)$, where $\wh{H}=(\wh{G}_s)^0$, and $\rho:\Gm_F\to \Out(\wh{H})$ is characterised by the condition that the restriction $\rho|_{W_F}$ is the composition
\begin{equation} \label{Eq:endoscopy}
W_F\overset{\la}{\lra}Z_{{}^LG}(s)\overset{\Ad}{\lra}\Aut(\wh{H})\overset{\pr}{\lra}\Out(\wh{H}).
\end{equation}
Namely, it is easy to see that the composition \form{endoscopy} is continuous with finite image, so it uniquely extends to $\Gm_F$.

\smallskip

(c) The construction in parts~(a) and (b) are compatible. Namely, assume that $Z_{\wh{G}}(\la)=\wh{T}^{\Gm_F}$ for some maximal torus $T\subseteq G$. Then for every $\ka\in Z_{\wh{G}}(\la)=\wh{T}^{\Gm_F}$ the endoscopic datum $\C{E}_{T,\ka}$ from part~(a) is conjugate to the endoscopic datum $\C{E}_{\la,\ka}$ from part~(b).
\end{Emp}

\begin{Emp} \label{E:remendosc}
{\bf Remark.} One can show that if $G$ is quasi-split, then every endoscopic datum for $G$ is conjugate to an endoscopic datum $\C{E}_{S,\ka}$ obtained by the construction of Section~\re{exenddatum}(a). For an arbitrary $G$ the assertion holds if the endoscopic datum is {\em elliptic}.
We do not need these facts in this work.

%Indeed, note that there exists a quasi-split group $H$ over $F$ (unique up to an isomorphism) such that $(\wh{H},\rho_{H})\simeq (\wh{H},\rho)$.
%We choose an embedding of a maximal torus $\fa_H:T\hra H$ over $F$ and an embedding $\wh{\fa_H}:\wh{S}\hra \wh{H}$ in $[\wh{\fa_H}]$, let
%$Z_{\fa}:Z(\wh{H})\hra \wh{S}$ be the canonical $\Gm_F$-invariant embedding (see Section~\re{dualgp}(b)), set $\xi:= Z_{\fa}(s)\in \wh{S}^{\Gm_F}$to finish(!!!).

\end{Emp}

\begin{Emp} \label{E:estable}
{\bf $\C{E}$-stable functions.} Let $\C{E}=(s,\wh{H},\rho)$ be an endoscopic datum for $G$.

\smallskip

(a) Fix a pair $(\gm,\xi)$, where $\gm\in G^{\rss}(F)$ and $\xi\in \wh{G^0_{\gm}}^{\Gm_F}$, and
let $\C{E}_{\gm,\xi}:=\C{E}_{G^0_{\gm},\xi}$  \label{a:egmxi} be the endoscopic datum, defined in Section~\re{exenddatum}(a). We  write $(\gm,\xi)\in\C{E}$, if $\C{E}_{\gm,\xi}\simeq\C{E}$.

\smallskip

(b) For a pair  $(\gm,\xi)$ as in part~(a), let $\fa_{\gm}:G^0_{\gm}\hra G$ be the inclusion. Then we have $(\gm,\xi)\in\C{E}$ if and only if there exists an embedding
$\wh{\fa_{\gm}}:\wh{G^0_{\gm}}\hra \wh{G}$ such that

\smallskip

\quad\quad$\bullet$ we have $\wh{\fa_{\gm}}\in[\wh{\fa_{\gm}}]$ (see Section~\re{dualgp}(b));

\smallskip

\quad\quad$\bullet$  we have $\wh{\fa_{\gm}}(\xi)=s$, thus  $\wh{\fa_{\gm}}(\wh{G^0_{\gm}})\subseteq\wh{H}$;

\smallskip

\quad\quad$\bullet$ the $\wh{H}$-conjugacy class of $\wh{\fa_{\gm}}:\wh{G^0_{\gm}}\hra\wh{H}$ is $\Gm_F$-invariant.

\smallskip

(c) Let $\on{Orb}_{G(F)}^{\st}(\gm)\subseteq G^{\rss}(F)$  \label{a:orbstgm} be the locus of all stable conjugates of $\gm$ also referred as the {\em stable orbit}
of $\gm$. Then using Sections~\re{stableconj}(a) and \re{dualgp}(c), every $\chi\in \wh{G^0_{\gm}}^{\Gm_F}$ defines a function
$\ev^{\xi}_{\gm}: \on{Orb}_{G(F)}^{\st}(\gm)\to \qlbar\m$  \label{a:evxigm} given by the formula
\[
\ev^{\xi}_{\gm}(\gm')=\lan\xi,\inv(\gm,\gm')\ran.
\]

\smallskip

(d) We say that an $\Ad G(F)$-invariant function $f:G^{\rss}(F)\to\qlbar$ is {\em $\C{E}$-stable},  \label{a:estable} if for every element $\gm\in G^{\rss}(F)$ the restriction $f|_{\on{Orb}_{G(F)}^{\st}(\gm)}$ belongs to the linear span of $\ev^{\xi}_{\gm}$, where $\xi$ runs over all elements of $\wh{G^0_{\gm}}^{\Gm_F}$ such that $(\gm,\xi)\in\C{E}$.
\end{Emp}

\begin{Emp} \label{E:stable}
{\bf Stable case.}
For every connected reductive group $G$ we have a trivial endoscopic datum $\C{E}_{\on{triv}}:=(1,\wh{G},\rho_G)$. Then for every pair $(\gm,\xi)$ as in Section~\re{estable}(a) we have $(\gm,\xi)\in \C{E}_{\on{triv}}$ if and only if $\xi=1$. Therefore a function  $f:G^{\rss}(F)\to\qlbar$ is $\C{E}_{\on{triv}}$-stable if and only if
its restriction to each stable orbit $\on{Orb}_{G(F)}^{\st}(\gm)$ is constant, and such function is usually called {\em stable}.  \label{a:stable}
\end{Emp}

The following simple observation will play an important role later.

\begin{Emp} \label{E:connected}
{\bf Observations.} Let $(\gm,\xi)$ and $\fa_{\gm}$ be as in Section~\re{estable}(b).

\smallskip

(a) The composition $\on{can}_{\gm}:=\nu_{\wh{G}}\circ \wh{\fa_{\gm}}:\wh{G^0_{\gm}}\to \wh{G}\to c_{\wh{G}}$  \label{a:cangm} is independent of $\wh{\fa_{\gm}}\in [\wh{\fa_{\gm}}]$. It follows from the description of \re{estable}(c) that if $(\gm,\xi)\in \C{E}$, then $\on{can}_{\gm}(\xi)=\nu_{\wh{G}}(s)$.

\smallskip

(b) Conversely, assume now that $Z(G)$ is connected, then for every pair $(\gm,\xi)$ as in Section~\re{estable}(a) such that $\on{can}_{\gm}(\xi)=\nu_{\wh{G}}(s)$, we have $(\gm,\xi)\in \C{E}$.

\begin{proof}
By assumption, we have an equality $\nu_{\wh{G}}\circ\wh{\fa_{\gm}}(\xi)=\nu_{\wh{G}}(s)$ for every $\wh{\fa_{\gm}}\in [\wh{\fa_{\gm}}]$. Therefore there exits $\wh{\fa_{\gm}}\in [\wh{\fa_{\gm}}]$ such that  $\wh{\fa_{\gm}}(\xi)=s$. So it remains to show that the $\wh{H}$-conjugacy class of $\wh{\fa_{\gm}}$ is $\Gm_F$-invariant, that is, every $\Gm_F$-Galois conjugate ${}^{\tau}\wh{\fa_{\gm}}$ of $\wh{\fa_{\gm}}$ is $\wh{H}$-conjugate.

Since the $\wh{G}$-conjugacy classes of the inclusion $\eta:\wh{H}\hra\wh{G}$ and of the composition
$\eta\circ \wh{\fa_{\gm}}:\wh{G^0_{\gm}}\hra\wh{H}\hra\wh{G}$ are $\Gm_F$-invariant, we conclude that the composition
$\eta\circ{}^{\tau}\wh{\fa_{\gm}}$ is a $\wh{G}$-conjugate of $\eta\circ\wh{\fa_{\gm}}$. Since $\eta$ is an embedding, elements $\xi\in\wh{G^0_{\gm}}$ and $s\in\wh{H}$ are $\Gm_F$-invariant, while $\wh{\fa_{\gm}}(\xi)=s$, we thus conclude that ${}^{\tau}\wh{\fa_{\gm}}$ is a $\wh{G}_s$-conjugate of $\wh{\fa_{\gm}}$.

Since $\wh{H}=(\wh{G}_s)^0$, it therefore suffices to show that the centralizer $\wh{G}_s$ is connected. But this follows from the fact
that $Z(G)$ is connected, thus the derived group of $\wh{G}$ is simply connected.
\end{proof}
\end{Emp}

\begin{Emp}
{\bf Remark.} The assertion of Section~\re{connected}(b) is completely false, if $Z(G)$ is not connected.
For example, it is false for $G=\on{SL}_2$.
\end{Emp}

The following notion will be used in order to deduce the general case to the case of groups with connected center.

\begin{Def} \label{D:quasiisogeny}
Following \cite[Definition~1.1.12]{KV}, by  a {\em quasi-isogeny} we call a homomorphism $\pi:G\to G_1$ of connected reductive groups over $F$ such that $\pi(Z(G))\subseteq Z(G_1)$ and the induced homomorphism $\pi^{\ad} : G^{\ad}\to G_1^{\ad}$ of adjoint groups is an
isomorphism.
\end{Def}

\begin{Emp} \label{E:quasiisogeny}
{\bf Functorial properties.} Let $\pi:G\to G_1$ be a quasi-isogeny.  \label{a:quasiisogeny}

\smallskip

(a) Notice that for every embedding $\fa_1:S_1\hra G_1$ of a maximal torus over $F$ its pullback $\fa:S\to G$
is an embedding of a maximal torus and $\pi$ induces a morphism $S\to S_1$ hence a $\Gm$-equivariant morphism
$\wh{\pi}:\wh{S_1}\to \wh{S}$.

\smallskip

(b) Conversely, every embedding $\fa:S\hra G$ of a maximal torus over $F$ gives rise to an embedding of a maximal torus $\fa_1:S_1\hra G_1$, whose pullback is isomorphic to $\fa$, hence to a morphism $\wh{\pi}:\wh{S_1}^{\Gm_F}\to \wh{S}^{\Gm_F}$.

\smallskip

(c) Note that $\pi$ induces a quasi-isogeny $\wh{\pi}:\wh{G_1}\to\wh{G}$ of dual groups, canonical up to a $\wh{G}$-conjugacy.
In particular, the conjugacy class of $\wh{\pi}$ is $\Gm_F$-invariant.

\smallskip

(d) As in \cite[Lemma~1.3.10(a)]{KV}, an endoscopic datum $\C{E}=(s_1,\wh{H}_1,\rho_1)$ for $G_1$ gives rise to the endoscopic datum
$\wh{\pi}(\C{E})=(s,\wh{H},\rho)$ for $G$, unique up to conjugacy. Namely, it is characterised by the property that if $\wh{\pi}$ be a quasi-isogeny from part~(b), then $s=\wh{\pi}(s_1)$, $\wh{H}=((\wh{G_1})_{s_1})^0$ and $\rho:\Gm_F\to\on{Out}(\wh{H})$ is the unique homomorphism such that the conjugacy classes of $\wh{\pi}:\wh{H_1}\to \wh{H}$ and $\wh{H}\hra\wh{G}$ are $\Gm_F$-invariant.

\smallskip

(e) In the situation of part~(a), for every $\kappa\in \wh{S_1}^{\Gm_F}$ we have  $\wh{\pi}(\kappa)\in \wh{S}^{\Gm_F}$.
Moroever, the endoscopic datum $\wh{\pi}(\C{E}_{\fa_1,\ka})$ (see part~(d) and Section~\re{exenddatum}(a)) is conjugate to  $\C{E}_{\fa,\wh{\pi}(\ka)}$.
\end{Emp}

%\begin{Emp} \label{E:constr}
%{\bf Construction.} Let $G$ be a connected reductive group over $F$, and let $\fa:S\hra G$ be an embedding of a maximal torus over $F$.
%We claim that there exists an injective quasi-isogeny $\pi:G\hra G_1$ such that the induced morphism
%$\wh{\pi}:\wh{S_1}^{\Gm}\to \wh{S}^{\Gm}$ from Section~\re{quasiisogeny}(b) is surjective.

%Namely, namely $\fa$ induces an embedding $Z(G)\hra S:z\mapsto \fa^{-1}(z)$, and we set $G_1:=G\times^{Z(G)}S$.
%We claim that the quasi-isogeny $\pi:G\hra G_1:g\mapsto [g,1]$ satisfies the required property.

%Indeed, $\fa$ gives rise to an embedding $\fa_1:S_1:=S\times^{Z(G)}S\hra G\times^{Z(G)}S=G_1$ of a maximal torus and the
%embedding $S\hra S_1:a\mapsto [a,1]$ has a left inverse ($[a,b]\mapsto ab$), hence  $\wh{\pi}:\wh{S_1}\to \wh{S}$ has a $\Gm$-equivariant
%right inverse.
%\end{Emp}

\begin{Emp}  \label{E:qisog}
{\bf Observations.} Let $\pi:G\to G_1$ be a quasi-isogeny.

\smallskip

(a) Note that $\pi$ satisfies $\pi(G^{\rss}(F))\subseteq G_1^{\rss}(F)$ and  $\pi^{-1}(G_1^{\rss}(F))\subseteq G^{\rss}(F)$.
In particular, for every $\gm\in G^{\rss}(F)$ we have $\pi(\gm)\in G_1^{\rss}(F)$ and
$\pi$ induces a homomorphism $\pi:G^0_{\gm}\to (G_1)^0_{\pi(\gm)}$ hence a homomorphism
$\pi_*:H^1(F,G^0_{\gm})\to H^1(F,(G_1)^0_{\pi(\gm)})$ and a
$\Gm_F$-equivariant homomorphism $\wh{\pi}:\wh{(G_1)^0_{\pi(\gm)}}\to \wh{G^0_{\gm}}$.

\smallskip

(b) The homomorphisms of part~(a) are compatible with the pairing of Section~\re{dualgp}(c). In other words, for every
$x\in H^1(F,G^0_{\gm})$ and $\xi\in \wh{(G_1)^0_{\pi(\gm)}}^{\Gm_F}$, we have an equality
\[
\lan \wh{\pi}(\xi),x\ran= \lan \xi,\pi_*(x)\ran.
\]

\smallskip

(c) For every $\gm'\in G^{\rss}(F)$, stably conjugate to $\gm$, its image $\pi(\gm')\in  G^{\rss}_1(F)$ is stably conjugate to $\pi(\gm)$ and we have an equality $\inv(\pi(\gm),\pi(\gm'))=\pi_*(\inv(\gm,\gm'))$.

\smallskip

(d) Let $\C{E}$ be as in Section~\re{quasiisogeny}(b). Using Section~\re{estable}(b) one sees that for every $\gm\in G^{\rss}(F)$ and $\xi\in \wh{(G_1)_{\pi(\gm)}}$ such that $(\pi(\gm),\xi)\in\C{E}$, we have $(\gm,\wh{\pi}(\xi))\in\wh{\pi}(\C{E})$.
\end{Emp}

\begin{Lem} \label{L:quasi-isogeny}
Let $\pi:G\to G_1$ is a quasi-isogeny of connected reductive groups over $F$, and let $\C{E}$ be an endoscopic triple for $G_1$.
Then for every $\C{E}$-stable function $f$ on $G^{\rss}_1(F)$ its pullback $\pi^*(f)$ to $G^{\rss}(F)$ is $\wh{\pi}(\C{E})$-stable.
\end{Lem}

\begin{proof}
Though the result can be deduced from \cite[Corollary~1.6.13]{KV} where a more general result is proven,
we provide the argument for completeness.

\smallskip

We want to show that for any element $\gm\in G^{\rss}(F)$, the restriction $\pi^*(f)|_{\on{Orb}^{\st}_{G(F)}(\gm)}$ lies in the span of
$\ev^{\xi'}_{\gm}$, where $\xi'\in \wh{G^0_{\gm}}^{\Gm_F}$  such that $(\gm,\xi')\in\wh{\pi}(\C{E})$.

\smallskip

By assumption, function $f$ is $\C{E}$-stable. Since $\pi(\gm)\in  G^{\rss}_1(F)$ the restriction $f|_{\on{Orb}^{\st}_{G_1(F)}(\pi(\gm))}$ is a finite sum $\sum_{\xi} a_{\xi}\ev^{\xi}_{\pi(\gm)}$, where $\xi$ runs over elements of $\wh{(G_1)_{\pi(\gm)}}^{\Gm_F}$  such that $(\pi(\gm),\xi)\in\C{E}$.

\smallskip

Then by Sections~\re{qisog}(b),(c), we conclude that the restriction $\pi^*(f)|_{\on{Orb}^{\st}_{G(F)}(\gm)}$
equals $\sum_{\xi} a_{\xi}\ev^{\wh{\pi}(\xi)}_{\gm}$, so the assertion follows from Section~\re{qisog}(d).
\end{proof}

%\begin{Emp}
%{\bf Remark.} All notions of this section can be extended from strongly regular semisimple elements to regular semisimple elements.
%After this is done the results of Section~\re{qisog} can be extended to arbitrary (not necessarily injective) quasi-isogenies.
%We don't need this generality for our work.
%\end{Emp}

%For simplicity we will fix a splitting of $G$, thus fix a
% Borel subgroup $B_0\subseteq G$ a maximal split torus $T_0\subseteq B$,
%and Iwahori subgroup $I_0\subseteq G(F)$, and hence also an apartment
%$\C{A}(G)\subseteq \C{B}(G)$ and a fundamental alcove
%$C\subseteq \C{A}$.

%Also we denote by  $W$ the Weyl group of $G$,  by $\La$ the group of
%cocharacters $X_*(T_0)$ and by
%$\wt{W}=W\ltimes\La$ the affine Weyl group of $G$

\subsection{Rationality} \label{S:rationality}

\begin{Emp} \label{E:genrat}
{\bf General observations.}

\smallskip

(a) For every algebraic group $H$ over $F$, the loop group $LH$ is defined over $\fq$. For example, this applies
to $H=G$, to $H=G^0_{\gm}$ where $\gm\in G^{\rss}(F)$ and more generally for every torus $S$ over $F$. In particular,
the group of connected components $\pi_0(LS)$ is equipped with an action of $\Gm_{\fq}$.

\smallskip

(b) Notice that since $F^{\nr}\subseteq K$ is dense the restriction map $\Gm_{K}\to\Gm_{F^{\nr}}$ is an isomorphism.
In particular, for every torus $S$ over $F$, the group of coinvariants $X_*(S)_{\Gm_K}$ and invariants $\wh{S}^{\Gm_K}$ are equipped with an action of $\Gm_{\fq}$ and we have identifications  $(X_*(S)_{\Gm_K})_{\si}\simeq X_*(S)_{\Gm_F}$ and $(\wh{S}^{\Gm_K})^{\si}=\wh{S}^{\Gm_F}$.

\smallskip

(c) For a torus $S$ over $F$, the Kottwitz isomorphism $\pi_0(LS)\simeq X_*(S)_{\Gm_K}$ (see Section~\re{pi0}) is $\Gm_{\fq}$-equivariant.
\end{Emp}

\begin{Emp} \label{E:affweyl2}
{\bf The affine Weyl group.}

\smallskip

(a) We claim that torus $\ov{T}_G$ is defined over $\fq$, the affine Weyl group $\wt{W}_G$ is equipped with a $\Gm_{\fq}$-action, while the subgroups
$\wt{W},\La_G$ and $\Om_G$ and $\Gm_{\fq}$-invariant. Indeed, this follows repeating the argument of Section~\re{rationality} word-by-word. In particular, the factor group $\ov{W}=\wt{W}_G/\La_G$ is equipped with a $\Gm_{\fq}$-action, and the action of $\ov{W}$ on $\ov{T}_{G}$ is $\Gm_{\fq}$-equivariant.

\smallskip

(b) By the same argument as in part~(a) one sees that the $\Gm_{\fq}$-action preserves the Bruhat order on $\wt{W}$. In particular, the subset
$\wt{S}\subseteq\wt{W}$ is $\Gm_{\fq}$-invariant, and the $\Gm_{\fq}$-action preserves the Demazure product (see \rl{dem}).

\smallskip

(c) For every pair of Iwahori subgroups $\I_1,\I_2\subseteq LG$ over $\fqbar$, we have an equality $\wt{w}_{{}^{\si}\I_1, {}^{\si}\I_2}={}^{\si}( \wt{w}_{\I_1,\I_2})$.

\smallskip

(d) Note that the sets of roots $\ov{\Phi}$, of affine roots $\wt{\Phi}$, and positive affine roots $\wt{\Phi}^+$ are equipped with a
$\Gm_{\fq}$-action, and that the action of $\wt{W}$ on $\wt{\Phi}$ is $\Gm_{\fq}$-equivariant. Therefore for every $m$-regular element $w\in\wt{W}$ its Galois conjugate ${}^{\si}w$ is $m$-regular as well.

\smallskip

(e) Since $G$ splits over $F^{\nr}$, it splits over $F^{(n)}=\B{F}_{q^n}((t))$ for some $n\in\B{N}$. Then $\si^n$ acts trivially
on $\wt{W}_G$, and the torus $\ov{T}_G$ splits over $\B{F}_{q^n}$.
%(d) Note that if $\varphi:G\to G'$ is an inner twisting, then the induced isomorphism $\wt{W}_G\isom \wt{W}_{G'}$ is $\Gm_{\fq}$-equivariant. Indeed, this follows from the fact that the composition $\varphi^{-1}\circ{}^{\si}\varphi$ is an inner isomorphism of $G$, therefore
%$\varphi^{-1}\circ{}^{\si}\varphi$ induces an identity automorphism on $\wt{W}_G$ (see Section~\re{???}).
\end{Emp}

\begin{Emp} \label{E:parahoric2}
{\bf Parahoric subgroups.}

\smallskip

(a) Note that the correspondence $\P\mapsto \P(\fq)$ defines a bijection between
$\si$-invariant parahoric subgroups of $LG$ and parahoric subgroups of $LG(\fq)=G(F)$ in the classical sense.

\smallskip

(b) Notice that if $\P$ is a $\si$-invariant parahoric subgroup of $LG$, then its type $J_{\P}\subseteq\wt{S}$ is
$\si$-invariant.
%(b) Notice that for every two $\Gm_{\fq}$-invariant parahoric subgroups $\P_1$ and $\P_2$ of $LG$ there exists $g\in G^{\sc}(F)$ such that the intersection $\P_1\cap g\P_2g^{-1}$ contains a $\si$-invariant parahoric subgroup $LG$. Indeed, since every parahoric subgroup of $G(F)$ contains an Iwahori subgroup, and all Iwahori subgroups in $G(F)$ are $G^{\sc}(F)$-conjugate, the assertion follows from (a).

\smallskip

(c) For every torus $S\subseteq G$ over $F$, which  splits over $F^{\nr}$, there exists a $\si$-stable parahoric subgroup
$\P\subseteq LG$  containing $L^+(S)$.

\begin{proof}
First we claim that there exists a maximal torus $T'\subseteq G$ defined over $F$, which contains $S$  and splits over $F^{\nr}$. Indeed, since $G$ and  $S$ split over $F^{\nr}$, the centralizer $H:=G_{S}$ splits over $F^{\nr}$ as well. Then every maximal torus $T'\subseteq H$ defined over $F$ and splitting over $F^{\nr}$ satisfies the required  property.

Next, the embedding $T'\hra G$ defines a morphism of Bruhat-Tits buildings $i:\C{B}(T')\hra\C{B}(G)$, and every point in the image of $i$ defines a parahoric subgroup of $G(F)$, hence (by part~(a)) a $\si$-invariant parahoric subgroup $\P$ of $LG$, containing $L^+(T')\supseteq L^+(S)$.
\end{proof}
\end{Emp}

\begin{Emp} \label{E:affflvar2}
{\bf Affine flag varieties and affine Springer fibers.}

\smallskip

(a) Since $G$ is defined over $F$, the loop group $LG$ is defined over $\fq$. Moreover, the affine flag variety $\Fl_G$ is equipped with a structure over $\fq$. Namely, for every Iwahori subgroup $\I$ of $LG$ the canonical isomorphism ${}^{\si}\Fl_G\isom\Fl_G$ decomposes as
\[
{}^{\si}\Fl_G\isom {}^{\si}(LG/\I)\isom LG/{}^{\si}{\I}\isom \Fl_G,
\]
where the first and the last isomorphisms are from Section~\re{affflvar}(a). Moreover, the closed ind-subscheme $\Fl\subseteq\Fl_G$
and the action map $LG\times \Fl_G\to\Fl_G$ are $\fq$-rational.

\smallskip

(b) For every parahoric subgroup $\Q$ of $LG^{\sc}$ and $w\in \wt{W}$,
we set $\Fl_{\Q}^w:=\Q w\subseteq\Fl$. \label{a:flqw} Then the isomorphism ${}^{\si}\Fl\isom\Fl$ induces an identification
${}^{\si}\Fl_{\Q}^w\isom \Fl_{{}^{\si}\Q}^{{}^{\si}w}$.

\smallskip

(c) Let $\I\subseteq LG$ be an Iwahori subgroup, and let $J,J'\subseteq\wt{S}$
be $\si$-invariant subsets such that parahoric subgroup $\P_{J',\I}\subseteq LG$ is $\si$-invariant.
Then it follows from parts~(a),(b) and Section~\re{affweyl2} that the filtration
$\{\Fl_{\I}^{\leq n;J_r;J'_l;m_{reg}}\}_n$ of $\Fl$ is $\si$-invariant.

\smallskip

(d) For every $\gm\in G(K)=LG(\fqbar)$, the isomorphism ${}^{\si}\Fl_G\isom\Fl_G$ from part~(a) induces an isomorphism
${}^{\si}\Fl_{G,\gm}\simeq \Fl_{G,{}^{\si}\gm}$ and this isomorphism identifies  the reduction map ${}^{\si}\red_{\gm}:{}^{\si}\Fl_{G,\gm}\to {}^{\si}\ov{T}_G\simeq\ov{T}_G$ with  $\red_{{}^{\si}\gm}:\Fl_{G,{}^{\si}\gm}\to\ov{T}_G$.

In particular, for every $\gm\in G(F)=LG(\fq)$,  the affine Springer fiber $\Fl_{G,\gm}\subseteq\Fl_G$ and the reduction map $\red_{\gm}:\Fl_{G,\gm}\to \ov{T}_G$ are $\fq$-rational.
\end{Emp}

\begin{Emp} \label{E:homaffspr}
{\bf Application to homology.} Let $\gm\in LG(\fqbar)$ and let $\C{L}$ be a local system on $\ov{T}_G$.

\smallskip

(a) Using observations of Section~\re{affflvar2}(d), we have a natural isomorphism
\[
a_{\si,\C{L}}: H_i(\Fl_{G,\gm},\C{F}_{\C{L}})\isom H_i({}^{\si}\Fl_{G,\gm},\si_*\C{F}_{\C{L}})\isom H_i(\Fl_{G,{}^{\si}\gm},\C{F}_{\si_*\C{L}}).  \label{a:asiL}
\]
Furthermore, for every $g\in LG(\fqbar)$ the following diagram is commutative
\[
\begin{CD}
H_i(\Fl_{G,\gm},\C{F}_{\C{L}}) @>a_{\si,\C{L}}>> H_i(\Fl_{G,{}^{\si}\gm},\C{F}_{\si_*\C{L}})\\
@Vl_gVV @VVl_{{}^{\si}g}V \\
H_i(\Fl_{G,\gm'},\C{F}_{\C{L}}) @>a_{\si,\C{L}}>> H_i(\Fl_{G,{}^{\si}\gm'},\C{F}_{\si_*\C{L}}),
\end{CD}
\]
with $\gm'=g\gm g^{-1}$ (compare \rp{whaction}(d)).
\smallskip

(b) For every $w\in\wt{W}_G$ we have an natural isomorphism $\si_*\ov{w}_*\C{L}\simeq ({}^{\si}\ov{w})_*\si_*\C{L}$.
Furthermore, this isomorphism identifies an isomorphism
\[
a_{\si,w_*\C{L}}\circ a_{w,\C{L}}:H_i(\Fl_{G,\gm},\C{F}_{\C{L}})\isom H_i(\Fl_{G,{}^{\si}\gm},\C{F}_{\si_*\ov{w}_*\C{L}})
\]
with

\[
a_{{}^{\si}w,\si_*\C{L}}\circ a_{\si,\C{L}}: H_i(\Fl_{G,\gm},\C{F}_{\C{L}})\isom H_i(\Fl_{G,{}^{\si}\gm},\C{F}_{({}^{\si}\ov{w})_*\si_*\C{L}}).
\]
Indeed, it suffices to show this assertion for $w=s_{\al}$ with $\al\in\wt{S}$ and $w\in\Om_G$. In both cases, the assertion follows
by unwinding the definitions (see part~(a) and \rp{whaction}).

\smallskip

(c) Assume now that $\gm\in LG(\fq)$. Then, by part~(b), for every $u\in \wt{W}_G\rtimes\lan\si\ran$  we have a natural isomorphism
\[
a_{u,\C{L}}: H_i(\Fl_{G,\gm},\C{F}_{\C{L}})\isom  H_i(\Fl_{G,\gm},\C{F}_{\ov{u}_*\C{L}}),
\]
where $\ov{u}\in \ov{W}\rtimes\lan\si\ran$ is the image of $u$, which coincides with that of part~(a) when $u=\si$, coincides with that of
\rp{whaction} when $u\in\wt{W}_G$, and such that the assignment $u\mapsto a_{u,\C{L}}$ is compatible with products/compositions.

\smallskip

(d) Assume now that $\gm\in LG(\fq)$ and $\C{L}$ is a $\ov{W}$-equivariant Weil sheaf. Then for every $\ov{u}\in  \ov{W}\rtimes\lan\si\ran$ we have a canonical
isomorphism $\ov{u}_*\C{L}\isom\C{L}$, thus the construction of part~(c) provides each $H_i(\Fl_{G,\gm},\C{F}_{\C{L}})$ with a
$(\wt{W}_G\rtimes\lan\si\ran)$-action.

%\smallskip

%(e) Furthermore, using \rp{whaction}(d), commutative diagram of part~(a) and Section~\re{setup comp}(b), we see that in the situation of part~(d)
%each $H_i(\Fl_{G,\gm},\C{F}_{\C{L}})$ is equipped with an action of $(\wt{W}_G\times LG^0_{\gm})\rtimes\lan\si\ran$, hence
%of $(\wt{W}_G\times \pi_0(LG^0_{\gm}))\rtimes\lan\si\ran$.

%\smallskip
%(e) More generally, assume that $\C{L}$ is a $\lan \ov{u} \ran$-equivariant local system for some element $\ov{u}\in  W\rtimes\lan\si\ran$.
%Then the construction of part~(c) provides each $H_i(\Fl_{G,\gm},\C{F}_{\C{L}})$ with an action of a subgroup
%$\pi_G^{-1}(\lan\ov{u}\ran)\subseteq \wt{W}_G\rtimes\lan\si\ran$.
\end{Emp}

\subsection{Maximal unramified tori}
The goal of this section is to construct a generalization of the correspondence
Section~\re{fincor} to the affine setting.

\begin{Emp} \label{E:basic pair}
{\bf Relative position.} % Notice that $K:=\fqbar((t))$ is a completion of $F^{\nr}$, so $G$ splits over $K$,
%thus construction  of Section~\re{?????} applies.

\smallskip

(a) Let $T\subseteq G$ be a maximal torus over $F$, split over $F^{\nr}$, and let $\I$ be an Iwahori subgroup of $LG$, defined over
$\fqbar$ and containing $\T$. Then ${}^{\si}\I$ is another Iwahori subgroup of $LG$, so we can consider their relative position  $\wt{w}_{\I}:=\wt{w}_{\I,{}^{\si}\I}\in\wt{W}$,  \label{a:wtwI} and let $\ov{w}_{\I}\in \ov{W}$  \label{a:ovwI} be the image of $\wt{w}_{\I}$.

\smallskip

(b) Let $\varphi=\varphi_{T,\I}:\ov{T}\isom \ov{T}_G$ be an admissible isomorphism. Then ${}^{\si}\varphi$ is also an admissible isomorphism $\varphi_{T,{}^{\si}{\I}}$, and we have an equality ${}^{\si}\varphi=\ov{w}^{-1}_{\I}\circ\varphi$ (see Section~\re{admisaff}(b)).
In particular, element $\ov{w}_{\I}\in \ov{W}$ only depends on $T$ and $\varphi$, so we can set $\ov{w}_{T,\varphi}:=\ov{w}_{\I}$.  \label{a:ovwTvarphi}
\end{Emp}

\begin{Emp} \label{E:example conj}
{\bf Example.} Let $(T,\I)$ be as in Section~\re{basic pair}, and let $\pr_{\I}:N_{LG}(LT)/\T\isom\wt{W}_G$ be the canonical isomorphism. Then ${}^{\si}{\I}\supseteq \T$ is an Iwahori subgroup as well, and we have an equality
$\pr_{{}^{\si}{\I}}=\wt{w}_{\I}^{-1}\circ\pr_{\I}\circ\wt{w}_{\I}$. Therefore for every element $g\in  N_{LG}(LT)$ we have an equality
\[
{}^{\si}\pr_{\I}([g])=\pr_{{}^{\si}\I}([{}^{\si}g])=\wt{w}_{\I}^{-1}\circ\pr_{\I}([{}^{\si}g])\circ\wt{w}_{\I}.
\]
\end{Emp}

\begin{Not} \label{N:sitor}\hfill

\smallskip

(a) Denote by  $\wt{W}_{\si-\tor}$ the set of  \label{a:wtWsitor}  {\em $\si$-torsion}
elements of $\wt{W}$, that is, those elements
$\wt{w}\in\wt{W}_G$ such that
$N_n(\wt{w}):=\wt{w}\cdot{}^\si\wt{w}\cdot\ldots\cdot{}^{\si^{n-1}}\wt{w}$ is $1$
for some $n\in\B{N}$.

\smallskip

(b) For every pair $(T,\I)$ as in Section~\re{basic pair}, we denote by $T^{\sc}$ (resp $\I^{\sc})$) the corresponding maximal torus of $G^{\sc}$ (resp. Iwahori subgroup of $LG^{\sc}$), and we set $\T^{\sc}:=L^+(T^{\sc})$.  In particular, we have a canonical isomorphism
$\wt{W}\simeq N_{LG^{\sc}}(LT^{\sc})/\T^{\sc}$.

\smallskip

(c) For every two pairs  $(T,\I)$ and  $(T',\I')$ as in Section~\re{basic pair} there exists an element
$g\in G^{\sc}(F^{\nr})\subseteq G^{\sc}(K)=LG^{\sc}(\fqbar)$ such that $(T',\I')=g(T,\I)g^{-1}$.
Then element $h:=g^{-1}{}^{\si}g$ satisfies $hTh^{-1}=T$, thus $h\in N_{LG^{\sc}}(LT^{\sc})$.

\smallskip

(d) We say that $(T,\I)$ and  $(T',\I')$ from part~(c) are {\em stably conjugate},  \label{a:stablyconj2} if there exists $g\in G^{\sc}(F^{\nr})$ as in part~(c) such that the isomorphism $\Ad_g:T\isom T':t\mapsto gtg^{-1}$ is defined over $F$. Notice that this happens if and only if element $h:=g^{-1}{}^{\si}g$ belongs to $LT^{\sc}$.
\end{Not}

\begin{Lem} \label{L:cor}
(a) A correspondence $(T,\I)\mapsto \wt{w}_\I$ gives an injection
from the set of $G^{\sc}(F)$-conjugacy classes of pairs $(T,\I)$ as is Section~\re{basic pair} to the set $\wt{W}_{\si-\tor}$.

\smallskip

(b) A correspondence $(T,\I)\mapsto \ov{w}_\I$ gives an injection from the set of stably conjugacy classes of pairs $(T,\I)$ as above to
the set $\ov{W}$.

%(c) Two pairs $(T,\I)$ and $(T',\I')$ are $G(F)$-conjugate if and only if the invariants $\wt{w}_{\I},\wt{w}_{\I'}\in\wt{W}$
%are $\Om^{\si}$-conjugate.
\smallskip

(c) If $G$ is quasi-split, then the correspondences of parts~(a) and (b) are bijective.
\end{Lem}

\begin{proof}
For the convenience of the reader, we will divide the proof into five steps:

\smallskip \noindent{\bf Step 1.}  For every pair $(T,\I)$ as in Section~\re{basic pair}, the conjugate ${}^{\si}\I$ is another Iwahori subgroup
of $LG$ containing $\T$. Hence there exists $g\in  N_{LG^{\sc}}(LT^{\sc})$ such that ${}^{\si}\I=g\I g^{-1}$, and by definition
$\wt{w}_\I$ is the class $[g]\in N_{LG^{\sc}}(LT^{\sc})/\T^{\sc}\simeq\wt{W}$.

\smallskip \noindent{\bf Step 2.} By definition, $T$ splits over a finite unramified extension $F^{(n)}:=\B{F}_{q^n}((t))$ of $F$. Then ${}^{\si^n}\I=\I$. Since each
${}^{\si^i}\I\supseteq \T$ is an Iwahori subgroup of $LG$, we deduce from Section~\re{prop}(e) that
\[
N_n(\wt{w}_{\I})=\wt{w}_{\I}\cdot{}^{\si}\wt{w}_{\I}\cdot \ldots\cdot {}^{\si^{n-1}}\wt{w}_{\I}=\wt{w}_{\I,{}^{\si}\I}\cdot \wt{w}_{{}^{\si}\I,{}^{\si^2}\I}\cdot
\ldots\cdot\wt{w}_{{}^{\si^{n-1}}\I, {}^{\si^{n}}\I}=\wt{w}_{\I,{}^{\si^n}\I}=1,
\]
thus $\wt{w}_{\I}$ is $\si$-torsion.

\smallskip \noindent{\bf Step 3.} Let $(T,\I)$ and $(T',\I')$ be two pairs as in Section~\ref{N:sitor}(c), choose element $g\in G^{\sc}(F^{\nr})$ such that
$(T',\I')=g(T,\I)g^{-1}$, and set $h:=g^{-1}{}^{\si}g$. Then we have $h\in N_{LG^{\sc}}(LT^{\sc})$, and we have an equality
\begin{equation} \label{Eq:basic eq}
\wt{w}_{\I'}=[h]\cdot \wt{w}_{\I}.
\end{equation}
Indeed, choose element $g_0\in G^{\sc}(F^{\nr})$ such that
$(T,{}^{\si}\I)=g_0(T,\I)g_0^{-1}$. Then
\[
\wt{w}_{\I',{}^{\si}\I'}=\wt{w}_{g\I g^{-1},({}^{\si} g)({}^{\si}\I)({}^{\si} g)^{-1}}=\wt{w}_{\I,h({}^{\si}\I)h^{-1}}=
\wt{w}_{\I,(hg_0)\I(hg_0)^{-1}},
\]
and it follows from Step 1 that
\[
\wt{w}_{\I',{}^{\si}\I'}=[hg_0]=[h]\cdot[g_0]=[h]\cdot \wt{w}_{\I,{}^{\si}\I}.
\]

\smallskip \noindent{\bf Step 4.} Assertions (a) and (b) easily follows from equality~\form{basic eq}: Indeed, if pairs $(T,\I)$ and $(T',\I')$ are $G^{\sc}(F)$-conjugate,
then we can assume that $g\in G^{\sc}(F)$, thus $h=1$, hence $\wt{w}_{\I'}=\wt{w}_{\I}$. Conversely, if $\wt{w}_{\I'}=\wt{w}_{\I}$, then $[h]=1$, thus $h\in L^+(T^{\sc})$. Then, by Lang's theorem, there exists an element $t\in L^+(T^{\sc})$ such that $h=t^{-1}{}^{\si}t$. Then $g^{-1}{}^{\si}g=t^{-1}{}^{\si}t$, hence $gt^{-1}\in G^{\sc}(F)$, so pair
\[
(T',\I')=g(T,\I)g^{-1}=gt^{-1}(T,\I)(gt^{-1})^{-1}
\]
is $G^{\sc}(F)$-conjugate to $(\T,I)$.

Next, $\ov{w}_\I=\ov{w}_{\I'}$ if and only if
$h=g^{-1}g^{\si}\in LT^{\sc}$, or, what is the same, if and only if pairs $(T,\I)$ and $(T',\I')$ are  stably conjugate.

\smallskip \noindent{\bf Step 5.} Assume now that $G$ is quasi-split, thus there exists a pair
$(T_0,\I_0)$ as in Section~\re{basic pair}, where $\I_0$ is defined over $\fq$. We set $N_0:=N_G(T_0)$.

To show that the correspondence of part~(a) is surjective, we need to show that  every $\si$-torsion element
$\wt{w}\in\wt{W}$ comes from some pair $(T,\I)$. First we claim that element
$\wt{w}$ lifts to a $\si$-torsion element $h_{\wt{w}}\in
N_0(F^{\nr})$. Indeed, choose $m$ such that $N_m(\wt{w})=1$ and
 $\si^m$ acts trivially on $\wt{W}$. Then $\wt{w}$ gives
rise to a $1$-cocycle
\[
c_{\wt{w}}\in Z^1(F^{(m)}/F,\wt{W}_G)\subseteq Z^1(F^{\nr}/F,\wt{W}_G)
\]
such that
$c_{\wt{w}}(\si)=\wt{w}$. Since $\Gm_{F^{\nr}/F}\simeq \wh{\B{Z}}$,
we have $H^2(F^{\nr}/F,A)=0$ for every
$\Gm_{F^{\nr}/F}$-module $A$. Hence by \cite[Corollary, page~54]{Se}, $1$-cocycle $c_{\wt{w}}$ lifts to a  $1$-cocycle $\wt{c}_{\wt{w}}\in
Z^1(F^{\nr}/F,N_0(F^{\nr}))$.

\smallskip

Since $G^{\sc}$ is semisimple and simply connected, we get that
\[
H^1(F^{\nr}/F,G^{\sc}(F^{\nr}))=H^1(F,G^{\sc})=0.
\]
Thus there exists $g\in G^{\sc}(F^{\nr})$ such that $\wt{c}_{\wt{w}}(\si)=g^{-1}{}^{\si}g$. Then
$(T,\I):=g(T_0,\I_0)g^{-1}$ is the pair satisfying
$\wt{w}_I=\wt{w}$.

\smallskip

Finally, since $G$ is quasi-split and split over $F^{\nr}$, group $G(F)$ has a hyperspecial parahoric subgroup.
This subgroup gives rise to a $\si$-invariant parahoric subgroup $L^+G\subseteq LG$, hence (by Section~\re{splitting}(b)) to a $\si$-equivariant section
$s:\ov{W}\to \wt{W}$ of the projection $\pi_{G^{\sc}}:\wt{W}\to \ov{W}$. Therefore every $\ov{w}\in \ov{W}$ has a $\si$-torsion
lift $s(\ov{w})\in \wt{W}$. Thus the correspondence of part~(b) is surjective as well.
\end{proof}

\begin{Emp} \label{E:comp}
{\bf Compatibility.} The constructions of Sections~\re{fincor} and \re{basic pair} are
compatible: Let $T\subseteq G$ be a maximal
unramified torus over $F$ such that $\T=L^+(T)$ is contained in a parahoric
subgroup $\P$ of $LG$. Then $\ov{T}$ is a maximal torus of
$M_{\P}$ over $\fq$ (compare Section~\re{parahoric}(a)), while the finite Weyl group $W_{M_{\P}}$ is naturally a
subgroup of $\wt{W}_{G}$.

Moreover, any Borel subgroup
$\ov{B}\supseteq\ov{T}$ of $M_{\P}$ (defined over $\fqbar$) gives rise to
the unique Iwahori subgroup $\I$ of $LG$ over $\fqbar$  such that
$\T\subseteq \I\subseteq \P$ and $\I/\P^+=\ov{B}$.

We claim that element $w_{\ov{T},\ov{B}}\in W_{M_{\P}}\subseteq\wt{W}$
equals $\wt{w}_{\I}$. Indeed, since
$w_{\ov{B}}=w_{\ov{B},{}^{\si}\ov{B}}$, while
$\wt{w}_\I=\wt{w}_{\I,{}^{\si}\I}$, the assertion follows from
the equality $w_{\ov{B},{}^{\si}\ov{B}}=\wt{w}_{\I,{}^{\si}\I}$ from
Section~\re{parahoric}(f).
 \end{Emp}

\begin{Not} \label{N:conj}
Consider two pairs $(T,\varphi)$ and $(T',\varphi')$, where $T,T'\subseteq G$
are maximal unramified tori and $\varphi:\ov{T}\isom \ov{T}_G$ and $\varphi':\ov{T}'\isom\ov{T}_G$
are admissible isomorphisms (see Section~\re{admisaff}(a)).

\smallskip

(a) We say that $(T,\varphi)$ and $(T',\varphi')$  are {\em conjugate} (resp. {\em $G^{\sc}(F)$-conjugate}), if there exists $g\in G(F)$ (resp. $g\in G^{\sc}(F)$) such that we have $gTg^{-1}=T'$ and $\varphi=\varphi'\circ\Int g:\ov{T}\isom\ov{T}_G$.

\smallskip

(b) We say that $(T,\varphi)$ and $(T',\varphi')$  are {\em stably conjugate},  \label{a:stablyconj3} if there exists an element $g\in G^{\sc}(F^{\nr})$ such that $gTg^{-1}=T', \varphi=\varphi'\circ\Ad_g:\ov{T}\isom\ov{T}_G$, and the inner isomorphism $\Ad_g:T\isom T'$ is defined over $F$.\footnote{As in Section~\re{stableconj}(b), this happens if and only if  there exists $g\in G(F^{\on{sep}})$ satisfying this property.}
\end{Not}

%\begin{Emp}
%{\bf Remark.} Using the fact $H^1(F^{\nr},T)=1$ for every torus $T$, one can show that $(T,\varphi)$ and $(T',\varphi')$ are stably conjugate if and only if there exists $g\in G(F^{\on{sep}})$ such that $gTg^{-1}=T', \varphi=\varphi'\circ\Int g:\ov{T}\isom\ov{T}_G$, and the inner isomorphism $\Ad_g:T\isom T'$ is defined over $F$.
%\end{Emp}

\begin{Cor} \label{C:corresp}
(a) The correspondence $(T,\varphi_{T,\I})\mapsto [\wt{w}_\I]$  \label{a:[wtwI]}
defines an injection from the set of $G^{\sc}(F)$-conjugacy classes of
pairs $(T,\varphi)$ as in Section~\ref{N:conj} to the set $\La_{\si-conj}\bs\wt{W}_{\si-\tor}$ of
$\si$-conjugacy classes of $\La$ acting on $\wt{W}_{\si-\tor}$.

\smallskip

(b) The correspondence $(T,\varphi_{T,\I})\mapsto \ov{w}_\I$
defines an injection from the set of stably conjugacy classes of
pairs $(T,\varphi)$ as in Section~\ref{N:conj} to the set $\ov{W}$.

\smallskip

%(c) Two pairs $(T,\varphi_{T,\I})$ and $(T',\varphi_{T,\I'})$ are $G(F)$-conjugate if and only if
%classes $[\wt{w}_\I]$ and $[\wt{w}_{\I'}]$ are $\Om_G^{\si}$-conjugate.

(c) If $G$ is quasi-split, then correspondences of parts (a) and (b) are bijective.
\end{Cor}

\begin{proof}
(a) Note that  if $\I',\I\supseteq\T$  are two Iwahori subgroups of $LG$
defined over $\fqbar$, then we have an equality
\begin{equation} \label{Eq:forminv}
\wt{w}_{\I'}=\wt{w}_{\I',{}^{\si}\I'}=\wt{w}_{\I',\I}\cdot\wt{w}_{\I,{}^{\si}\I}\cdot\wt{w}_{{}^{\si}\I,{}^{\si}\I'}=
\wt{w}_{\I,\I'}^{-1}\cdot\wt{w}_\I\cdot{}^{\si}\wt{w}_{\I,\I'}
\end{equation}
(compare Section~\re{affweyl}(e)). If in addition $\varphi_{T,\I}=\varphi_{T,\I'}$, then we have
$\ov{w}_{\I,\I'}=1$ (see Section~\re{admisaff}(b)), or, what is the same,
$\wt{w}_{\I,\I'}\in\La$. Thus the correspondence $(T,\varphi_{T,\I})\mapsto
[\wt{w}_\I] \in \La_{\si-\on{conj}}\bs\wt{W}_{\si-\tor}$ is well-defined.

\smallskip

Next, let $(T,\I)$ and $(T',\I')$ be two pairs as in Section~\re{basic pair} such that
$(T,\varphi_{T,\I})$ and $(T',\varphi_{T,\I'})$ are $G^{\sc}(F)$-conjugate. Then there exists
an Iwahori subgroup $\I''\supseteq \T$ such that $(T,\I'')$ and $(T',\I')$ are $G^{\sc}(F)$-conjugate and  $\varphi_{T,\I''}=\varphi_{T,\I}$.
Then, by the proven above, we have $[\wt{w}_{\I}]=[\wt{w}_{\I''}]$, and  $\wt{w}_{\I''}=\wt{w}_{\I'}$ by
\rl{cor}(a). Thus, $[\wt{w}_\I]=[\wt{w}_{\I'}]$, as claimed.

\smallskip

Conversely, assume that $(T,\I)$ and $(T',\I')$ be such that $[\wt{w}_\I]=[\wt{w}_{\I'}]$.
Then there exists $\la\in\La$ such that $\wt{w}_{\I'}=\la^{-1}\wt{w}_\I{}^{\si}\la$.
Choose an Iwahori subgroup $\I''\supseteq\T$ such that $\wt{w}_{\I,\I''}=\la$. Then, arguing as above, we see that
$\wt{w}_{\I''}=\la^{-1}\wt{w}_\I{}^{\si}\la$, thus $\wt{w}_{\I''}=\wt{w}_{\I'}$. So, by \rl{cor}(a),
the pairs   $(T,\I'')$ and $(T',\I')$ are $G^{\sc}(F)$-conjugate, thus pairs $(T,\varphi_{T,\I})=(T,\varphi_{T,\I''})$ and
$(T,\varphi_{T,\I'})$ are $G^{\sc}(F)$-conjugate.

\smallskip

(b)-(c)  Observe that $(T,\varphi_{T,\I})$ is stably conjugate   to $(T',\varphi_{T,\I'})$ if and only if there exists
an Iwahori subgroup $\I''\supseteq\T$ such that $\varphi_{T,\I''}=\varphi_{T,\I}$ and $(T,\I'')$ is stably conjugate to $(T',\I')$. Now, as in part~(a), all assertions follow from the corresponding assertions of \rl{cor}.
\end{proof}

\begin{Not} \label{N:twisting}
Choose a minimal $n\in\B{N}$ such that $\si^n$ acts trivially on $\wt{W}_G$, and let $\lan\si\ran_n$
be the cyclic group of order $n$, generated by $\si$.

\smallskip

(a) Note that an element $w\in\wt{W}$ is $\si$-torsion if and only if the element
\[
w\si\in \wt{W}\si\subseteq \wt{W}\rtimes\lan\si\ran_n
\]
is torsion. Moreover, for every $\la\in\La$ we have an equality
$\la^{-1}(w\si)\la=(\la^{-1}\cdot w\cdot {}^{\si}\la)\si$.

\smallskip

(b) For every $(T,\I)$ as in Section~\re{basic pair}, we set $\wt{u}_{\I}:=\wt{w}_{\I}\si\in\wt{W}\si$.  \label{a:wtuI}

\smallskip

(c) We denote by $(\wt{W}\si)_{\on{tor}}$ the set of elements of $\wt{W}\si$, which are torsion
as elements of $\wt{W}\rtimes\lan\si\ran_n$. Using observation of part~(a), \rco{corresp} implies that the correspondence $(T,\varphi_{T,\I})\mapsto [\wt{u}_\I]$ defines an injection from the set of $G^{\sc}(F)$-conjugacy classes of
pairs $(T,\varphi)$ as in Section~\ref{N:conj} to the set of
conjugacy classes $\La_{\on{conj}}\bs(\wt{W}\si)_{\on{tor}}$.
\end{Not}

\begin{Not} \label{N:stconj}
Let $(T,\varphi)$ be as in Section~\ref{N:conj}, and let $n$ be as in Section~\ref{N:twisting}.

\smallskip

(a) Denote by
$\ov{w}_{T,\varphi}\in\ov{W}$  the image of $(T,\varphi)$
under the correspondence of \rco{corresp}(b), and set
$\ov{u}_{T,\varphi}:=\ov{w}_{T,\varphi}\si\in \ov{W}\si\subseteq \ov{W}\rtimes\lan\si\ran_n$.  \label{a:ovuTvarphi} Notice that the notation
is compatible with that of Section~\re{basic pair}(b).

\smallskip

(b) Denote by $[\wt{u}_{T,\varphi}]\in \La_{\on{conj}}\bs(\wt{W}\si)_{\on{tor}}$  \label{a:wtuTvarphi} the image of $(T,\varphi)$
under the correspondence of \rco{corresp}(a).

\smallskip

(c) For each embedding $\fa:T\hra G$, stably conjugate to the
inclusion $T\hra G$ (see Section~\re{stableconj}), we set $T_{\fa}:= \fa(T)$  \label{a:Tfa} and
$\varphi_{\fa}:=\varphi\circ \fa^{-1}:\ov{T}_{\fa}\isom \ov{T}_G$. \label{a:varphifa}
Then $\varphi_{\fa}$ is an admissible isomorphism, pair $(T_{\fa},\varphi_{\fa})$ is stably conjugate to $(T,\varphi)$,
and we set
\[
[\wt{u}_{\fa,\varphi}]:=[\wt{u}_{T_{\fa},\varphi_{\fa}}]\in
\La_{\on{conj}}\bs(\wt{W}\si)_{\on{tor}}. \label{a:wtufavarphi}
\]
\end{Not}

\begin{Emp} \label{E:remstconj}
{\bf Remarks.} Suppose that we are in the situation of Section~\ref{N:stconj}.

\smallskip

(a)  The map $\fa\mapsto (T_{\fa},\varphi_{\fa})$ gives a bijection between the set of
embeddings $\fa:T\hra G$, stably conjugate to the inclusion $T\hra G$, and the set of pairs $(T',\varphi')$, stably conjugate to $(T,\varphi)$. Moreover, the embeddings $\fa$ and $\fa'$ are $G^{\sc}(F)$-conjugate (resp. conjugate) if and only if the pairs
$(T_{\fa},\varphi_{\fa})$ and $(T_{\fa'},\varphi_{\fa'})$ are $G^{\sc}(F)$-conjugate (resp. conjugate).

\smallskip

(b) Since ${}^{\si}\varphi=\ov{w}_{T,\varphi}\circ\varphi$ (see Section~\re{basic pair}(b)), $\varphi$ induces an isomorphism $X_*(T)\isom X_*(\ov{T})\isom X_*(\ov{T}_G)\simeq \La_G$, which interchange the action of $\si$ on $X_*(T)$ with the action of $\ov{u}_{T,\varphi}$ on $\La_G$, so induces an isomorphism $X_*(T)_{\Gm_F}\isom(\La_G)_{\ov{u}_{T,\varphi}}$. Composing it with the isomorphism
$H^1(F,T)\isom X_*(T)_{\Gm_F,\tor}$ (see Section~\re{cohom}(b)), we get an isomorphism $H^1(F,T)\isom (\La_G)_{\ov{u}_{T,\varphi},\tor}$,
hence a homomorphism
\[
\La_{\ov{u}_{T,\varphi},\tor}\to (\La_G)_{\ov{u}_{T,\varphi},\tor}\isom H^1(F,T).
\]

\smallskip

(c) Let $T^{\sc}\subseteq G^{\sc}$ be the pullback of $T\subseteq G$. Then $T^{\sc}\subseteq G^{\sc}$ is a maximal elliptic torus, and $\varphi$ induces an admissible isomorphism $\varphi:\ov{T}^{\sc}\isom \ov{T}_{G^{\sc}}$ and we have equalities  $\ov{w}_{T^{\sc},\varphi}=\ov{w}_{T,\varphi}\in\ov{W}_G=\ov{G}_{G^{\sc}}$, thus $\ov{u}_{T^{\sc},\varphi}=\ov{u}_{T,\varphi}$. Therefore applying the observation of the part~(b) to $T^{\sc}\subseteq G^{\sc}$ instead of $T\subseteq G$, we get an isomorphism $H^1(F,T^{\sc})\isom \La_{\ov{u}_{T,\varphi},\tor}$.

\smallskip

(d) Let $\pi_{G}:\wt{W}_G\rtimes\lan \si\ran_n\to \ov{W}\rtimes\lan \si\ran_n$ be the natural projection.
Since for every $\wt{u}\in \pi_{G}^{-1}(\ov{u}_{T,\varphi})$ and $\mu\in\La_G$ we have an identity
$\mu^{-1}\wt{u}\mu=(\mu^{-1}{}^{\ov{u}_{T,\varphi}}\mu)\wt{u}$, the map $([\la],[\wt{u}])\mapsto[\la\wt{u}]$ gives an action of
the group $(\La_G)_{\ov{u}_{T,\varphi}}$ on $(\La_G)_{\on{conj}}\bs\pi_{G}^{-1}(\ov{u}_{T,\varphi})$.

\smallskip

Applying this observation to $T^{\sc}\subseteq G^{\sc}$ instead of $T\subseteq G$, we get the action of the group $\La_{\ov{u}_{T,\varphi}}$ on $\La_{\on{conj}}\bs\pi_{G^{\sc}}^{-1}(\ov{u}_{T,\varphi})$. Moreover, by a straightforward calculation one checks that for every pair of elements $[\wt{u}],[\wt{u}']\in \La_{\on{conj}}\bs\pi_{G^{\sc}}^{-1}(\ov{u}_{T,\varphi})$ there exists a unique element $\inv([\wt{u}],[\wt{u}'])\in \La_{\ov{u}_{T,\varphi}}$ such that $\inv([\wt{u}],[\wt{u}'])\cdot[\wt{u}]=[\wt{u}']$.

\smallskip

Furthermore, we have
$\inv([\wt{u}],[\wt{u}'])\in (\La_{\ov{u}_{T,\varphi}})_{\tor}$, if $[\wt{u}],[\wt{u}']\in \La_{\on{conj}}\bs\pi_{G^{\sc}}^{-1}(\ov{u}_{T,\varphi})_{\tor}$.
\end{Emp}

\begin{Lem} \label{L:stconj}
In the notation of Section~\ref{N:stconj},

\smallskip

(a) The map $\fa\mapsto[\wt{u}_{\fa,\varphi}]$ gives a bijection between the set of $G^{\sc}(F)$-conjugacy
classes of embedding $\fa:T\hra G$, stably conjugate to the
inclusion $T\hra G$, and elements of $\La_{\on{conj}}\bs\pi_{G^{\sc}}^{-1}(\ov{u}_{T,\varphi})_{\on{tor}}\subseteq
\La_{\on{conj}}\bs(\wt{W}\si)_{\on{tor}}$.

\smallskip

(b) For a pair of embeddings $\fa,\fa':T\hra G$ as in part~(a), the relative position $\inv(\fa,\fa')\in H^1(F,T)$ (see Section~\re{stableconj}(c))
is the image of the relative position $\inv([\wt{u}_{\fa,\varphi}],[\wt{u}_{\fa',\varphi}])\in (\La_{\ov{u}_{T,\varphi}})_{\tor}$ (see Section~\re{remstconj}(d)) under the map of Section~\re{remstconj}(b).

\smallskip

(c) Two embeddings $\fa,\fa':T\hra G$ as in part~(a) are $G(F)$-conjugate if and only if there exists an element $\mu\in \La_G$ such that $[\wt{u}_{\fa',\varphi}]=\mu^{-1}[\wt{u}_{\fa,\varphi}]\mu$.
\end{Lem}

\begin{proof}
(a) Each $(T_{\fa},\varphi_{\fa})$ is stably conjugate to $(T,\varphi)$,
therefore we have an equality $\ov{w}_{T_{\fa},\varphi_{\fa}}=\ov{w}_{T,\varphi}$ (by \rco{corresp}(b)), thus   $\ov{u}_{T_{\fa},\varphi_{\fa}}=\ov{u}_{T,\varphi}$,
hence $[\wt{u}_{\fa,\varphi}]\in \La_{\on{conj}}\bs\pi_{G^{\sc}}^{-1}(\ov{u}_{T,\phi})_{\on{tor}}$.

Next, observe that two embeddings $\fa,\fa':T\hra G$ are conjugate if
and only if $(T_{\fa},\varphi_{\fa})$ and $(T_{\fa'},\varphi_{\fa'})$ are conjugate. Thus it follows from \rco{corresp} that the correspondence is well defined and injective.

\smallskip

Finally, to show the surjectivity, we notice that $\La_{\on{conj}}\bs\pi_{G^{\sc}}^{-1}(\ov{u}_{T,\varphi})_{\on{tor}}$
is a principal homogeneous space with respect to the group $\La_{\ov{u}_{T,\varphi},\tor}$ (by Section~\re{remstconj}(d)), so the assertion follows by a combination of observations of Sections~\re{remstconj}(c), \re{stableconj}(e) and
part~(b) of the lemma, applied to $T^{\sc}\subseteq G^{\sc}$ instead of $T\subseteq G$.

\smallskip

(b) By assumption, there exists an element $g\in G^{\sc}(F^{\nr})$ such that $\Ad_g\circ\fa=\fa'$ and $h:=g^{-1}{}^{\si}g\in LT^{\sc}(\fqbar)$.
Then the relative position $\inv(\fa,\fa')$ is the projection of $[h]\in\pi_0(LT^{\sc})\simeq X_*(T^{\sc})$ under the map
\[
X_*(T^{\sc})\to X_*(T)_{\Gm_F}\simeq H^1(F,T)
\]
(see Sections~\re{cohom}(b),(c)), while the relative position $\inv([\wt{u}_{\fa,\varphi}],[\wt{u}_{\fa',\varphi}])$
is the image of $[h]$ under the map
\[
X_*(T^{\sc})\simeq \La\to \La_{\ov{u}_{T,\varphi}}
\]
(see equality \form{basic eq} and Sections~\re{remstconj}(b),(c)).

\smallskip

(c) By definition, $\fa'$ is $G(F)$-conjugate to $\fa$ if and only if $\inv(\fa,\fa')=1$. Combining part~(b) with observations of
Section~\re{remstconj}(c), this follows if and only if elements
$[\wt{u}_{\fa,\varphi}],[\wt{u}_{\fa',\varphi}]\in \La_{\on{conj}}\bs\pi_{G^{\sc}}^{-1}(\ov{u}_{T,\varphi})$ have the same image in
$(\La_G)_{\on{conj}}\bs\pi_{G}^{-1}(\ov{u}_{T,\varphi})$, that is, there exists an element $\mu\in \La_G$ such that  $[\wt{u}_{\fa',\varphi}]=\mu^{-1}[\wt{u}_{\fa,\varphi}]\mu$.
\end{proof}

\begin{Emp} \label{E:elliptic case}
{\bf Elliptic case.} Assume that maximal torus $T\subseteq G$ is elliptic, and let $Z_0\subseteq Z(G)^0$  \label{a:Z0} be the maximal split torus.
Then $\La_{G/Z_0}\simeq\La_G/X_*(Z_0)$, and the endomorphisms $\ov{u}_{T,\varphi}|_{\La_{G/Z_0}}$ and $\ov{u}_{T,\varphi}|_{\La}$ are elliptic in the sense of Section~\re{finell}(a) (use  Section~\re{example conj} and compare Section~\re{fincor}(d)).

\smallskip

(a) Since automorphism $\ov{u}_{T,\varphi}|_{\La}$ is elliptic, all elements of $\pi_{G^{\sc}}^{-1}(\ov{u}_{T,\varphi})\subseteq \wt{W}\si$
are torsion in $\wt{W}\rtimes\lan\si\ran_n$. More precisely, using \rl{finell}(b) one sees that the order of every $\wt{u}\in \pi_{G^{\sc}}^{-1}(\ov{u}_{T,\varphi})\subseteq \wt{W}\rtimes\lan\si\ran_n$ equals the order of $\ov{u}_{T,\varphi}\in \ov{W}\rtimes\lan\si\ran_n$. Thus \rl{stconj}(a) asserts that in this case the map $\fa\mapsto
[\wt{u}_{\fa,\varphi}]$ gives a bijection between the set of $G^{\sc}(F)$-conjugacy
classes of embedding $\fa:T\hra G$, stably conjugate to the
inclusion $\fa_0:T\hra G$, and the set $\La_{\on{conj}}\bs\pi_{G^{\sc}}^{-1}(\ov{u}_{T,\varphi})$.

\smallskip

(b) By \rl{stconj}(c), for every $\mu\in\La_{G/Z_0}/\La$ such that $\mu^{-1} \wt{u}_{T,\varphi}\mu\in \La\wt{u}_{T,\varphi}$ there exists
an embedding $\fa:T\hra G$, conjugate to $\fa_0$ such that  $[\wt{u}_{\fa,\varphi}]=[\mu^{-1} \wt{u}_{T,\varphi}\mu]$. Moreover, $\fa$ is unique up to a $G^{\sc}(F)$-conjugacy.

\smallskip

(c) Since $\ov{u}_{T,\varphi}|_{\La_{G/Z_0}}$ is elliptic, for every  $1\neq \mu\in\La_{G/Z_0}$ we have
$\mu^{-1}\wt{u}_{T,\varphi}\mu\neq\wt{u}_{\fa,\varphi}$. Therefore it follows from \rl{stconj} that for every embedding $\fa:T\hra G$, conjugate to $\fa_0$, there exists a unique $\mu\in\La_{G/Z_0}/\La$ such that $[\wt{u}_{\fa,\varphi}]=[\mu^{-1} \wt{u}_{T,\varphi}\mu]$.
\end{Emp}

\subsection{Main results}

%Let $G$ be a connected reductive group over $F=\fq((t))$ splitting over $F^{\nr}:=F\otimes_{\fq}\fqbar$. Then $G$ is split over
%$K:=\fqbar((t))$, so let
%$\gm\in G(F)$ be a  regular semisimple element,  let $T\subseteq G$
%is a maximal unramified elliptic torus of $G$, and let
%$\theta:\ov{T}(\fq)\to\qlbar\m$ be a character. The goal of this
%section is to give a geometric description of the character
%$f_{T,\theta}(\gm)$, of linear combination $f^{\st}_{T,\theta}(\gm)$
%and to deduce the stability of $f^{\st}_{T,\theta}=f^{\st}_{G,T,\theta}$. Moreover, for every inner twisting
%$\phi:G\isom G'$, we are going to show the equivalence of  $f^{\st}_{G,T,\theta}$ and
%$f^{\st}_{G',T',\theta}$, where $T'\subseteq G'$ is a maximal elliptic unramified torus, stably conjugate to $T$.

The goal of this section is to formulate a geometric description of characters of cuspidal Deligne--Lusztig representations and of their stable
and $\C{E}_{T,\ka}$-stable linear combinations.

\begin{Emp} \label{E:DLpadic}
{\bf  Deligne--Lusztig representations over local non-archimedean
fields}. Let $\fa:T\hra G$ be an embedding of a maximal unramified elliptic torus of $G$ over $F$,
and let $\theta:T(F)\to\qlbar\m$   be a tamely ramified character in general position (see parts~(c),(d) below).

\smallskip

(a) Note that there exists a unique $\si$-invariant parahoric subgroup $\P_{\fa}\subseteq LG$  \label{a:pfa} containing $\T_{\fa}=\fa(\T)$  \label{a:boldTfa} (see, for example,
\cite[Section ~2.1.3]{KV}). We set $M_{\fa}:=M_{\P_{\fa}}$  \label{a:Mfa} and $G_{\fa}:=\P_{\fa}(\fq)$.  \label{a:Gfa} Then $G_{\fa}\subseteq G(F)$ is a parahoric subgroup, $G^+_{\fa}:=\P_{\fa}^+(\fq)$ is a pro-unipotent radical, and $G_{\fa}/G^+_{\fa}\simeq M_{\fa}(\fq)$.

\smallskip

(b) Since $T$ is a torus over $F$, split over $K$, its reduction $\ov{T}$ (see Section~\re{saffsetup}) is a torus over $\fq$.
Then the embedding $\fa:\T\hra\P_{\fa}$ defines an embedding $\ov{\fa}:\ov{T}\hra M_{\fa}$  \label{a:ovTfa} of a maximal elliptic torus of $M_{\fa}$.

\smallskip

(c) Assume that $\theta$ is {\em tamely ramified},  \label{a:tamely ramified} that is, the restriction
$\theta|_{T(\C{O})}:T(\C{O})\to\qlbar\m$ factors through a character
$\ov{\theta}:\ov{T}(\fq)\to\qlbar\m$.  \label{a:ovtheta} Then $(\ov{T},\ov{\theta})$ gives rise to a Deligne--Lusztig virtual representation
$R^{\ov{\theta}}_{\ov{\fa}}:=R^{\ov{\theta}}_{\ov{\fa}(\ov{T})}$  \label{a:Rovthetaovfa} of $M_{\fa}(\fq)$.

\smallskip

(d) Assume moreover that $\theta$ is {\em in general position},  \label{a:general position}
that is, not fixed by any non-trivial element of the Weyl group $W_{G,\fa(T)}(F)$. Then $\ov{\theta}$ is in general position in the sense of
\cite[Definition~5.15]{DL}, thus $R^{\ov{\theta}}_{\ov{\fa}}$ is cuspidal and irreducible (up to a sign), because $\ov{\fa}(\ov{T})\subseteq M_{\fa}$ is elliptic.

\smallskip

(e) We consider subgroup $\wt{G}_{\fa}:=G_{\fa}\cdot Z(G)(F)\subseteq G(F)$  \label{a:wtGfa} and denote by $R^{\theta}_{\fa}$  \label{a:Rthetafa} the virtual representation of $\wt{G}_{\fa}$, whose restriction to $G_{\fa}$ is the inflation of $R^{\ov{\theta}}_{\ov{\fa}}$, and the restriction to $Z(G)(F)$ is $\theta|_{Z(G)(F)}$.

\smallskip

(f) Let $\pi_{\fa,\theta}$  \label{a:pifatheta} be the induced representation $\ind_{\wt{G}_{\fa}}^{G(F)}R_{\fa}^{\theta}$ of $G(F)$. Since $R^{\ov{\theta}}_{\ov{\fa}}$ is cuspidal and irreducible, we deduce from \cite[Proposition~6.6]{MP2}, that $\pi_{\fa,\theta}$ is an irreducible virtual cuspidal representation. In particular, its character $\chi_{\pi_{\fa,\theta}}$ is an invariant generalized function on $G(F)$ supported on elements, conjugate to $\wt{G}_{\fa}$. Moreover, the restriction $f_{\fa,\theta}$  \label{a:ffatheta} of $\chi_{\pi_{\fa,\theta}}$ to  $G^{\rss}(F)$ is a locally constant function. Notice that $f_{\fa,\theta}$ only depends on the conjugacy class of $\fa$.

\smallskip

(g) If $T\subseteq G$ is a maximal unramified elliptic torus and $\fa_0:T\hra G$ is an inclusion, we write $f_{T,\theta}$  \label{a:fTtheta} instead of $f_{\fa_0,\theta}$.
\end{Emp}

Our first main result gives a geometric interpretation of $f_{T,\theta}(\gm)$ for compact regular semisimple elements $\gm\in G^{\rss}(F)_c$.

\begin{Emp} \label{E:ltheta}
{\bf Notation.} Let $T$ be an unramified torus over $F$, and let
$\theta:T(F)\to\qlbar\m$ be a tamely ramified character, that is, its restriction
$\theta|_{T(\C{O})}:T(\C{O})\to\qlbar\m$ factors through a character
$\ov{\theta}:\ov{T}(\fq)\to\qlbar\m$.

\smallskip

(a) Character $\ov{\theta}:\ov{T}(\fq)\to\fqbar$ gives rise to a
one-dimensional local system $\C{L}_{\theta}$  \label{a:Ltheta} on $\ov{T}$,
equipped with a Weil structure.

\smallskip

(b) Then every admissible isomorphism $\varphi:\ov{T}\isom \ov{T}_G$ as in Section~\re{admisaff}(a) gives
rise to  local system $\C{L}_{\theta,\varphi}:=\varphi_*(\C{L}_{\theta})$  \label{a:Lthetavarphi} on $\ov{T}_G$.

\smallskip

(c) The group $\ov{W}\rtimes\lan\si\ran$ acts naturally  on
the set admissible isomorphisms by
the rule $w(\varphi):=w\circ\varphi$ and $\si(\varphi):={}^{\si}\varphi$.
Using the identity ${}^{\si}\varphi=\ov{w}_{T,\varphi}^{-1}\circ \varphi$ (see Section~\re{basic pair}(b)), the cyclic group
$\lan \ov{u}_{T,\varphi}\ran\subseteq \ov{W}\rtimes\lan\si\ran$ is the stabilizer of $\varphi$.

\smallskip

(d) As in Section~\re{DL}, we have natural isomorphisms
$\si_*(\C{L}_{\theta,\varphi})\simeq   \C{L}_{\theta,{}^{\si}\varphi}$ and $w_*(\C{L}_{\theta,\varphi})\simeq\C{L}_{\theta,w\circ \varphi}$.
Moreover, we have natural isomorphisms $u_*(\C{L}_{\theta,\varphi})\simeq \C{L}_{\theta,u(\varphi)}$ for $u\in \ov{W}\rtimes\lan\si\ran$  compatible with compositions.

\smallskip

(e) Consider local system $\C{F}_{\theta,\varphi}:=\red_{\gm}^*(\C{L}_{\theta,\varphi})$ of $\Fl_{G,\gm}$. Then it follows from part~(e) and Section~\re{homaffspr}(c) that
for every $\wt{u}\in \wt{W}_G\rtimes\lan\si\ran$ we have a natural isomorphism
\[
a_{\wt{u},\theta,\varphi}: H_i(\Fl_{G,\gm},\C{F}_{\theta,\varphi})\isom  H_i(\Fl_{G,\gm},\C{F}_{\theta,\ov{u}(\varphi)}).
\]
Moreover, these isomorphisms are compatible with compositions.

\smallskip

(f) Combining parts~(c)-(e), each  $H_i(\Fl_{G,\gm},\C{F}_{\theta,\varphi})$ is equipped with an action of $\pi_G^{-1}(\lan\ov{u}_{T,\varphi}\ran)\subseteq  \wt{W}_G\rtimes\lan\si\ran$. In particular, each $H_i(\Fl_{G,\gm},\C{F}_{\theta,\varphi})$ is equipped with an action of $\La_G$.
Furthermore, each $H_i(\Fl_{G,\gm},\C{F}_{\theta,\varphi})$ is a finitely generated $\qlbar[\La_G]$-module (by \rp{fingen}(a)).
\end{Emp}

\begin{Emp} \label{E:interpr}
{\bf Construction.}   Let $(T,\varphi)$ be as in Section~\ref{N:conj}, and let $\theta$ be as in Section~\re{ltheta}.

\smallskip

(a) Let $[\wt{u}_{T,\varphi}]\in\La_{\on{conj}}\bs \wt{W}\si$ be the class, corresponding to a pair $(T,\varphi)$ (see \rn{stconj}(b)), and fix a representative $\wt{u}_{T,\varphi}\in\wt{W}\si$  of $[\wt{u}_{T,\varphi}]$. Then element $\wt{u}_{T,\varphi}\in \pi_{G^{\sc}}^{-1}(\lan\ov{u}_{T,\varphi}\ran)\subseteq\wt{W}\ltimes\lan\si\ran$ is a lift of $\ov{u}_{T,\varphi}$, therefore we have a natural isomorphism $\pi_G^{-1}(\lan\ov{u}_{T,\varphi}\ran)\simeq\La_G\rtimes\lan\wt{u}_{T,\varphi}\ran$.

\smallskip

(b) By part~(a) and Section~\re{ltheta}(e), we see that each $H_i(\Fl_{G,\gm},\C{F}_{\theta,\varphi})$ (resp.  $H_i(\Fl_{\gm},\C{F}_{\theta,\varphi})$) is a $\qlbar[\La_{G}\rtimes\lan\wt{u}_{T,\varphi}\ran]$-module (resp. $\qlbar[\La\rtimes\lan\wt{u}_{T,\varphi}\ran]$-module), which is finitely generated as a $\qlbar[\La_G]$-module (resp. $\qlbar[\La]$-module).

\smallskip

(c) Let $Z_0\subseteq Z(G)$ be the maximal $F$-split torus, and set $\La_0:=X_*(Z_0)\subseteq\La_G$.
Then $\La_0\subseteq \Om_G$ naturally acts on $H_i(\Fl_{G,\gm},\C{F}_{\theta,\varphi})$ (compare \rp{whaction}).
Since $\La_G/\La_0\simeq\La_{G/Z_0}$, we conclude that the space of coinvariants $H_i(\Fl_{G,\gm},\C{F}_{\theta,\varphi})_{\La_0}$ is a
$\qlbar[\La_{G/Z_0}\rtimes\lan\wt{u}_{T,\varphi}\ran]$-module, which is finitely generated as a $\qlbar[\La_{G/Z_0}]$-module.

\smallskip

(d) Since maximal torus $T\subseteq G$ is elliptic, the element $\ov{u}_{T,\varphi}$ induces an elliptic automorphisms of $\La_{G/Z_0}$ and $\La$
(compare Sections~\re{fincor}(d) and \re{finell}(a)). Moreover, modules $H_i(\Fl_{G,\gm},\C{F}_{\theta,\varphi})_{\La_0}$ and  $H_i(\Fl_{\gm},\C{F}_{\theta,\varphi})$ are $\wt{u}_{T,\varphi}$-locally finite (see \rt{dl}(a) below).

\smallskip

(e) By parts~(b)-(d) and Section~\re{trace}(a) we can form the generalized traces
\[
\Tr_{\gen}(\wt{u}_{T,\varphi}, H_i(\Fl_{G,\gm},\C{F}_{\theta,\varphi})_{\La_0})\text{ and }
\Tr_{\gen}(\wt{u}_{T,\varphi}, H_i(\Fl_{\gm},\C{F}_{\theta,\varphi})).
\]
Therefore we can form the generalized trace
\[
\Tr_{\gen}(\wt{u}_{T,\varphi}, H_*(\Fl_{G,\gm},\C{F}_{\theta,\varphi})_{\La_0}):=\sum_i(-1)^i\Tr_{\gen}(\wt{u}_{\varphi}, H_i(\Fl_{G,\gm},\C{F}_{\theta,\varphi})_{\La_0})
\]
and similarly $\Tr_{\gen}(\wt{u}_{T,\varphi}, H_*(\Fl_{\gm},\C{F}_{\theta,\varphi}))$.

\smallskip

(f) For an embedding $\fa:T\hra G$, stably conjugate to the inclusion $\fa_0:T\hra G$, we fix a representative  $\wt{u}_{\fa,\varphi}\in\pi_{G^{\sc}}^{-1}(\ov{u}_{T,\varphi})$  of $[\wt{u}_{\fa,\varphi}]=[\wt{u}_{T_{\fa},\varphi_{\fa}}]\in \La_{\on{conj}}\bs\pi_{G^{\sc}}^{-1}(\ov{u}_{T,\varphi})$ (see \rl{stconj}(a) and \rn{stconj}(c)).
\end{Emp}

%let $\wt{w}_\I\in\wt{W}$ be the corresponding
%element, and set $u=u_\I:=\wt{w}_\I\circ\si\in\wt{W}\rtimes\Gm_{\fq}$. Since
%$\La$ is a normal subgroup of $\wt{W}\rtimes\Gm_{\fq}$, we can
%form a semidirect product $\Dt:=\La\rtimes\cycl$.
%Moreover, since $T$ is elliptic, the induced automorphism
%$u|_{\La}$ is elliptic of finite order (Explain???)

%(b) Using the equality
%${}^{\si}\varphi_{T,\I}=\varphi_{T,{}^{\si}\I}=\ov{w}_{\I,{}^{\si}\I}\circ\varphi_{T,\I}=
%\ov{w}_\I^{-1}\circ \varphi_{T,\I}$ (compare Section~\re{fincor}(c)),
%we see $H_i(\Fl_{\gm},\C{F}_{\theta,\I})$ is equipped
%with an action of $\Dt$.

%(c) Notice that  each $H_i(\Fl_{\gm},\C{F}_{\theta,\I})$
%is a finitely generated $u$-locally finite $\Dt$-module. Indeed, .....

%In particular, $H_i(\Fl_{\gm},\C{F}_{\theta,\I})$
%satisfies the assumption of \rp{trform}, so we can form a trace
%$\Tr_{\gen}(u,H_i(\Fl_{\gm},\C{F}_{\theta,\I}))$.
%\end{Emp}

%(d) Since $\ov{u}_{T,\varphi}$ is elliptic, the preimage $\pi_G^{-1}(\lan\ov{u}_{T,\varphi}\ran)$ decomposes as
%$\pi_G^{-1}(\lan\ov{u}_{T,\varphi}\ran)\simeq\La_G\rtimes\lan\wt{u}_{\varphi}\lan$ for every lift
%$\wt{u}_{\varphi}\in\wt{W}_G\rtimes\lan\si\ran$ of $\ov{u}_{T,\varphi}$.

\begin{Thm} \label{T:dl}
Let $T\subseteq G$ be a maximal unramified elliptic torus of $G$, let $\fa_0:T\hra G$ be the inclusion, let $\theta:T(F)\to\qlbar\m$ be a tamely ramified character in general position, let $\gm\in G(F)$ be a compact  regular semisimple element, and let $\varphi:\ov{T}\isom\ov{T}_G$ be an admissible isomorphism.

\smallskip

Then, in the notation of Section~\re{interpr},

\smallskip

\quad\quad(a)  modules $H_i(\Fl_{G,\gm},\C{F}_{\theta,\varphi})$ and  $H_i(\Fl_{\gm},\C{F}_{\theta,\varphi})$ are
$\wt{u}_{T,\varphi}$-locally finite;

\smallskip

\quad\quad (b) we have equalities
\begin{equation} \label{Eq:2}
f_{T,\theta}(\gm)=\sum_{\fa:T\hra G} \Tr_{\gen}(\wt{u}_{\fa,\varphi},H_*(\Fl_{\gm},\C{F}_{\theta,\varphi}))=     \Tr_{\gen}(\wt{u}_{T,\varphi},H_*(\Fl_{G,\gm},\C{F}_{\theta,\varphi})_{\La_0}),
\end{equation}
where

\smallskip
$\bullet$ $\fa$ runs over a set of representatives of $G^{\sc}(F)$-conjugacy classes of embeddings $\fa:T\hra G$, conjugate to $\fa_0$;

\smallskip

$\bullet$ the generalized trace $\Tr_{\gen}(\wt{u}_{\fa,\varphi},H_*(\Fl_{\gm},\C{F}_{\theta,\varphi}))$ exists by the observation of Section~\re{interpr}(e) applied to the pair $(T_{\fa},\varphi_{\fa})$ (see \rn{stconj}(c)).
\end{Thm}

\begin{Emp} \label{E:fst}
{\bf Notation.}
Let $T, \fa_0$ and $\theta$ be as in \rt{dl}.

\smallskip

(a) We set
\[
f^{\st}_{T,\theta}:=\sum_{\fa:T\hra G}f_{\fa,\theta},
\]
where the sum runs over a set of representatives of conjugacy classes of embeddings
$\fa:T\hra G$, which are stably conjugate to $\fa_0$.

\smallskip

(b) More generally, for every $\ka\in\wh{T}^{\Gm_F}$, we set
\[
f^{\ka}_{T,\theta}:=\sum_{\fa:T\hra G}\lan \ka,\inv(\fa_0,\fa)\ran f_{\fa,\theta},
\]
where the $\fa$'s are is as in part~(a), the invariant $\inv(\fa_0,\fa)\in H^1(F,T)$ is defined in
Section~\re{stableconj}(c), and the pairing is defined in Section~\re{dualgp}(c). Note that $f^{\ka}_{T,\theta}=f^{\st}_{T,\theta}$ when $\ka=1$.
\end{Emp}

Our second main result gives a geometric interpretation of $f^{\st}_{T,\theta}(\gm)$  and $f^{\ka}_{T,\theta}(\gm)$ for compact  regular semisimple elements $\gm$.

\begin{Emp} \label{E:notstable}
{\bf Notation.}

\smallskip

(a) Consider the local system
$\C{L}^{\st}_{\theta}:=\bigoplus_{\varphi}\C{L}_{\theta,\varphi}$ on
$\ov{T}_G$, where $\varphi$ runs over the set of all admissible
isomorphisms $\ov{T}\isom\ov{T}_G$. Then it follows from Section~\re{ltheta}(d) that for every $u\in \ov{W}\rtimes\lan\si\ran$ we have a canonical isomorphism
\[
u_*(\C{L}^{\st}_{\theta})=\bigoplus_{\varphi}u_*(\C{L}_{\theta,\varphi})\isom
\bigoplus_{\varphi}\C{L}_{\theta,u(\varphi)}=\C{L}^{\st},
\]
making $\C{L}^{\st}_{\theta}$ is a $\ov{W}$-equivariant Weil sheaf.

\smallskip

(b) Set $\C{F}^{\st}_{\theta}:=\red_{\gm}^*(\C{L}^{\st}_{\theta})=\oplus_{\varphi}\C{F}_{\theta,\varphi}$. By part~(a) and Section~\re{homaffspr}(d), each homology group $H_i(\Fl_{G,\gm},\C{F}^{\st}_{\theta})$ is  equipped with an action of
$\wt{W}_G\rtimes \lan\si\ran$. In particular, we have a $\lan\si\ran$-action on each vector space
$H_j(\wt{W}_G,H_i(\Fl_{G,\gm},\C{F}^{\st}_{\theta}))$.

\smallskip

(c) By \rp{fingen}, each $H_i(\Fl_{G,\gm},\C{F}^{\st}_{\theta})$ is a
finitely-generated $\La_G$-module, therefore the space
\[
H_j(\wt{W}_G,H_i(\Fl_{G,\gm},\C{F}^{\st}_{\theta}))=H_j(\La_G,H_i(\Fl_{G,\gm},\C{F}^{\st}_{\theta}))^{\ov{W}}
\]
is a finite-dimensional $\qlbar[\lan\si\ran]$-module. Moreover, $H_j(\wt{W}_G,H_i(\Fl_{G,\gm},\C{L}^{\st}_{\theta}))$ is non-zero only for finitely
many pairs $(i,j)$.

\smallskip

(d) By part~(c), we can form the virtual finite-dimensional $\qlbar[\lan\si\ran]$-module
\[
H_*(\wt{W}_G,H_*(\Fl_{G,\gm},\C{F}^{\st}_{\theta}))):=\sum_{i,j}(-1)^{i+j}H_j(\wt{W}_G,H_i(\Fl_{G,\gm},\C{F}^{\st}_{\theta}))),
\]
hence also its trace $\Tr(\si, H_*(\wt{W}_G,H_*(\Fl_{G,\gm},\C{F}^{\st}_{\theta})))\in\qlbar$.

\smallskip

\end{Emp}

\begin{Emp} \label{E:notkappa}
{\bf Notation.}

\smallskip

(a) As it was already mentioned in Section~\re{ltheta}(f), each $H_i(\Fl_{G,\gm},\C{F}_{\theta,\varphi})$ is equipped with an action of $\pi_G^{-1}(\lan\ov{u}_{T,\varphi}\ran)\subseteq  \wt{W}_G\rtimes\lan\si\ran$ and is finitely generated as a
$\qlbar[\La_G]$-module.

\smallskip

(b) Note the algebra of regular functions $\qlbar[\wh{T}^{\Gm}]$ is naturally identified with the group algebra $\qlbar[X_*(T)_{\Gm_F}]$
(see Section~\re{dualgp}(c)), hence with $\qlbar[(\La_G)_{\ov{u}_{T,\varphi}}]$ (see Sections~\re{remstconj}(b) and \re{ltheta}(c))). In particular, each $\ka\in \wh{T}^{\Gm}$ can be viewed as a $\ov{u}_{T,\varphi}$-invariant character
$\La_G\to\qlbar\m$.

\smallskip

(c) As it was already mentioned in Section~\re{interpr}(a), we have a natural isomorphism
$\pi_G^{-1}(\lan\ov{u}_{T,\varphi}\ran)\simeq \La_G\rtimes \lan \wt{u}_{T,\varphi}\ran$. Therefore character
$\ka:\La_G\to \qlbar\m$ from part~(b) uniquely extends to a character of $\pi_G^{-1}(\lan\ov{u}_{T,\varphi}\ran)$, trivial on $\wt{u}_{T,\varphi}$. We denote by $(\qlbar)_{\ka}$ the one-dimensional representation of $\pi_G^{-1}(\lan\ov{u}_{T,\varphi}\ran)$, corresponding to $\ka$.

\smallskip

(d) Consider the representation $H_i(\Fl_{G,\gm},\C{F}_{\theta,\varphi})_{\ka}:=H_i(\Fl_{G,\gm},\C{F}_{\theta,\varphi})\otimes_{\qlbar} (\qlbar)_{\ka}$ of $\pi_G^{-1}(\lan\ov{u}_{T,\varphi}\ran)$. It is a finitely generated $\qlbar[\La_G]$-module, so the space
\[
H_j(\La_G, H_i(\Fl_{G,\gm},\C{F}_{\theta,\varphi})_{\ka})
\]
is a finite-dimensional $\qlbar[\lan\ov{u}_{T,\varphi}\ran]$-module. Again it is non-zero only for finitely
many pairs $(i,j)$.

\smallskip

(e) By part~(d), we can form the virtual finite-dimensional $\qlbar[\lan\ov{u}_{T,\varphi}\ran]$-module
\[
H_*(\La_G,H_*(\Fl_{G,\gm},\C{F}_{\theta,\varphi})_{\ka}):=\sum_{i,j}(-1)^{i+j}H_j(\La_G,H_i(\Fl_{G,\gm},\C{F}_{\theta,\varphi})_{\ka}),
\]
hence also its trace $\Tr(\ov{u}_{T,\varphi}, H_*(\La_G,H_*(\Fl_{G,\gm},\C{F}_{\theta})_{\ka}))\in\qlbar$.

\end{Emp}

\begin{Thm} \label{T:equality}
(a) For every $T,\theta$ and $\gm$ as in \rt{dl}, we have an equality
\[
f^{\st}_{T,\theta}(\gm)=\Tr(\si, H_*(\wt{W}_G,H_*(\Fl_{G,\gm},\C{F}^{\st}_{\theta}))).
\]

\smallskip

(b) Furthermore, for every $\ka\in\wh{T}^{\Gm_F}$ and $\varphi$ as in \rt{dl}, we have an equality
\[
f^{\ka}_{T,\theta}(\gm)=\Tr(\ov{u}_{T,\varphi}, H_*(\La_G,H_*(\Fl_{G,\gm},\C{F}_{\theta})_{\ka})).
\]
\end{Thm}

As an application, we will prove stability properties of $f^{\st}_{T,\theta}$ and $f^{\ka}_{T,\theta}$, thus establishing
the equal characteristic version of \cite[Theorem 2.1.6(a)]{KV}.

\begin{Thm} \label{T:stab}
Let $T$ and $\theta$ be as in \rt{dl}.

\smallskip

(a) Then the function $f^{\st}_{T,\theta}$ is stable (see Section~\re{stable}). In other words, for every pair of stably conjugate elements $\gm,\gm'\in G^{\rss}(F)$, we have an equality
\[
f^{\st}_{T,\theta}(\gm)=f^{\st}_{T,\theta}(\gm').
\]

\smallskip

(b) More generally, for every $\ka\in\wh{T}^{\Gm_F}$, the function $f^{\ka}_{T,\theta}$ is $\C{E}_{T,\ka}$-stable
(see Sections~\re{exenddatum}(a) and \re{estable}(d)).
\end{Thm}

Finally, we prove a generalization of \rt{stab}(a) which takes into an account inner twistings.

\begin{Emp} \label{E:inner}
{\bf Inner twistings and stable equivalence.}

\smallskip

(a) Let $\phi:G\isom G'$ be an inner twisting (compare \cite[Section~1.1.1]{KV}).
Using theorem of Steinberg, replacing $\phi$ by its conjugate $\phi\circ \Ad_g$ with $g\in G(F^{\sep})$, we can (and will) assume that
$\phi$ is an isomorphism over $F^{\nr}$ such that $\phi^{-1}\circ {}^{\si}\phi$ is an inner automorphism of $G$.

\smallskip

(b) For every $g\in G$, we denote by $G^{\ad}_{g}\subseteq G^{\ad}$ the stabilizer of $g$. Generalizing Section~\re{stableconj}(a),(b), we call a pair of elements  $\gm\in G^{\rss}(F)$ and $\gm'\in G'^{\rss}(F)$
{\em stably conjugate}, if there exists an element $g\in G(F^{\nr})$ such that $\phi(g\gm g^{-1})=\gm'$ and
an element $h:=\Ad_{g^{-1}}\circ (\phi^{-1}\circ{}^{\si}\phi)\circ \Ad_{{}^{\si}g}\in G^{\ad}(F^{\nr})$ belongs to $(G_{\gm}^{\ad})^0(F^{\nr})$.

\smallskip

(c) Generalizing Section~\re{stableconj}(c), two embeddings $\fa:T\hra G$ and $\fa':T\hra G'$ of maximal tori defined over $F$ are called {\em stably conjugate}, if there exists an element $g\in G(F^{\nr})$ such that $\fa'=\phi\circ\Ad_g\circ \fa$ and the induced morphism
$\phi\circ\Ad_g:\fa(T)\isom \fa'(T)$ is defined over $F$.  Note that for every embedding $\fa:T\hra G$ of a maximal elliptic torus over $F$ there exists a stably conjugate embedding $\fa':T'\hra G'$ (see \cite[Lemma~1.5.3]{KV}).

\smallskip

(d) We say that invariant functions $f$ on $G^{\rss}(F)$ and $f'$ on  $G'^{\rss}(F)$ are {\em stably equivalent},
if for every pair  $\gm\in G^{\rss}(F)$ and $\gm'\in G'^{\rss}(F)$ of stably conjugate elements, we have an equality $f(\gm)=f'(\gm')$.
In particular, this implies that $f$ and $f'$ are stable.
\end{Emp}

Now we are ready to state  a particular case of the equal characteristic version of \cite[Theorem 2.1.6(b)]{KV}, generalizing \rt{stab}(a).

\begin{Thm} \label{T:equiv}
Let $\phi:G\isom G'$ be an inner twisting, let $\fa_0$ and $\theta$ be as in \rt{dl}, let $\fa'_0:T\hra G'$ be an embedding stably conjugate to $\fa$, let $T'\subseteq G'$ be the image of $\fa'$, and let $\theta':T'(F)\to\qlbar^{\times}$ be the composition
\[
\theta\circ\fa'^{-1}_0:T'(F)\isom T(F)\to\qlbar^{\times}.
\]

Then the invariant functions $f^{\st}_{T,\theta}$ and $f^{\st}_{T',\theta'}$ are stably equivalent. %In other words, for every pair of stably conjugate  regular semisimple elements $\gm\in G(F)$ and $\gm'\in G'(F)$, we have an equality $f^{\st}_{T,\theta}(\gm)=f^{\st}_{T',\theta'}(\gm')$.
\end{Thm}

\begin{Emp}
{\bf Remark.} It should not be difficult to prove the equal characteristic version of the general case of \cite[Theorem~2.1.6(b)]{KV}, thus providing a generalization of \rt{stab}(b), which takes into an account inner twistings.
\end{Emp}

\section{Completion of proofs}

\subsection{Proof of \rt{dl}}

\begin{Emp} \label{E:step1}
First, we assume that the assertion \rt{dl}(a) holds, and are going to show the right equality of formula \form{2}.

\smallskip

(a) Using \rl{ind} and isomorphism $\La_G/\La_0\simeq \La_{G/Z_0}$
%$H_i(\Fl_{G,\gm},\C{F}_{\theta,\varphi})\simeq \ind_{\La}^{\La_{G}} H_i(\Fl_{\gm},\C{F}_{\theta,\varphi})$
we conclude that
\begin{equation*} \label{Eq:ind}
H_i(\Fl_{G,\gm},\C{F}_{\theta,\varphi})_{\La_0}\simeq\ind_{\La}^{\La_{G/Z_0}} H_i(\Fl_{\gm},\C{F}_{\theta,\varphi}).
\end{equation*}
Thus, by \rl{induced}, we have an equality
\[
\Tr_{\gen}(\wt{u}_{T,\varphi},H_*(\Fl_{G,\gm},\C{F}_{\theta,\varphi})_{\La_0})=
\sum_{\mu} \Tr_{\gen}(\mu^{-1} \wt{u}_{T,\varphi}\mu,H_*(\Fl_{\gm},\C{F}_{\theta,\varphi})),
\]
where $\mu$ runs over a set of representatives in $\La_{G/Z_0}$ of classes in $\La_{G/Z_0}/\La$ such that $\mu^{-1}\wt{u}_{T,\varphi}\mu\in \La \wt{u}_{T,\varphi}$. Moreover, and each summand on the RHS
only depends on the class $\mu\La\in \La_{G/Z_0}/\La$.

\smallskip

(b) By Sections~\re{elliptic case}(b),(c), the sum in part~(a) can be written as
\[
\sum_{\fa:T\hra G} \Tr_{\gen}(\wt{u}_{\fa,\varphi},H_*(\Fl_{\gm},\C{F}_{\theta,\varphi})),
\]
where $\fa$ runs over a set of representatives of $G^{\sc}(F)$-conjugacy classes of embeddings
$T\hra G$, conjugate to $\fa_0$.
\end{Emp}

\begin{Emp} \label{E:formula}
Next, we are going to write an explicit formula for $f_{T,\theta}$.

\smallskip

(a) Let $P\subseteq G(F)$ be an open subset and let $\Theta:G(F)/P\to \qlbar$ be a function.
We say that the sum $\sum_{g\in G(F)/P} \Theta(g)$ {\em uniformly stabilizes}, if for every open compact subgroup $Q\subseteq G(F)$
the set

\[
C_Q:=\{QxP\in Q\bs G(F)/P\,|\,\sum_{g\in QxP/P} \Theta(g)\neq 0\}
\]
is finite. Notice that the set  $QxP/P\simeq Q/(Q\cap xPx^{-1})$ is finite.

In this case, the sum $\sum_{QxP\in C_Q}\sum_{g\in QxP/P} \Theta(g)$ is independent of $Q$ and we denote the resulting sum by
$\sum_{g\in G(F)/P} \Theta(g)$.

\smallskip

(b) For every embedding $\fa:T\hra G$, stably conjugate to $\fa_0:T\hra G$, we denote by $t_{\fa,\theta}$ the smooth function on $G(F)$ such that $t_{\fa,\theta}(g)=\Tr R^{\theta}_{\fa}(g)$ if $g\in\wt{G}_{\fa}$, and
$t_{\fa,\theta}(g)=0$, if $g\notin\wt{G}_{\fa}$. Since representation $R_{\ov{\fa}}^{\ov{\theta}}$ is a cuspidal representation of $M_{\fa}(\fq)$, it follows from the formula of the induced representation that we have an equality
\[
f_{\fa,\theta}(\gm)=\sum_{g\in G(F)/ \wt{G}_{\fa}} t_{\fa,\theta}(g^{-1}\gm g),
\]
where the sum of the right hand side stabilizes uniformly.
\end{Emp}

\begin{Emp} \label{E:step2}
Our next goal is to rewrite $f_{T,\theta}(\gm)$ as a sum of terms, parameterized by a set of representatives of $G^{\sc}(F)$-conjugacy classes of embeddings $\fa:T\hra G$, which are conjugate to $\fa_0$. Namely, we claim that

\smallskip

(i) for every  $\fa:T\hra G$, conjugate to $\fa_0$, the sum
$\sum_{g\in G^{\sc}(F)/ G^{\sc}_{\fa}} t_{\fa,\theta}(g^{-1}\gm g)$, where $G^{\sc}_{\fa}:=G^{\sc}(F)\cap G_{\fa}$
stabilizes uniformly. Moreover, the resulting sum only depends of the $G^{\sc}(F)$-conjugacy class of $\fa$;

\smallskip

(ii) we have an equality
\[
f_{T,\theta}(\gm)=\sum_{\fa:T\hra G}\sum_{g\in G^{\sc}(F)/ G^{\sc}_{\fa}} t_{\fa,\theta}(g^{-1}\gm g),
\]
where $\fa$ runs over a set of representatives of $G^{\sc}(F)$-conjugacy classes of embeddings $\fa:T\hra G$, conjugate to $\fa_0$.

\begin{proof}
(a) Choose $h\in G(F)$ such that $\fa=h\fa_0 h^{-1}$, then for every $g\in G^{\sc}(F)/ G^{\sc}_{\fa}$ we have an equality
$t_{\fa,\theta}(g^{-1}\gm g)=t_{\fa_0,\theta}(h^{-1}g^{-1}\gm gh)$. Moreover, the correspondence $g\mapsto gh$
induces a bijection
\[
G^{\sc}(F)/ G^{\sc}_{\fa}=G^{\sc}(F)\wt{G}_{\fa}/ \wt{G}_{\fa}=G^{\sc}(F)h\wt{G}_{\fa_0}h^{-1}/ h\wt{G}_{\fa_0}h^{-1}\isom G^{\sc}(F)h\wt{G}_{\fa_0}/\wt{G}_{\fa_0}.
\]

\smallskip

(b) By part~(a) the sum in (i) can be rewritten as
\[
\sum_{g\in G^{\sc}(F)h\wt{G}_{\fa_0}/ \wt{G}_{\fa_0}} t_{\fa_0,\theta}(g^{-1}\gm g),
\]
therefore it conversely uniformly, because the sum $\sum_{g\in G(F)/ \wt{G}_{\fa_0}} t_{\fa_0,\theta}(g^{-1}\gm g)$ does.
Next, the second assertion in (i) follows immediately from the first one.

\smallskip

(c) Note that for every $h\in G(F)$ we have an equality $G^{\sc}(F)h\wt{G}_{\fa_0}=G^{\sc}(F)h T(F)$ (see \cite{KV}), and
the map $h\mapsto h\fa h^{-1}$ defines a bijection between the double quotient $G^{\sc}(F)\bs G(F)/T(F)$ and the set of $G^{\sc}(F)$-conjugacy classes of embeddings $\fa:T\hra G$, which are $G(F)$-conjugate to $\fa_0$.
Therefore assertion (ii) follows from part~(b).
\end{proof}
\end{Emp}

\begin{Emp} \label{E:step3}
{\bf Reduction steps.}

\smallskip

(a) By \rl{ind} and Sections~\re{step1} and \re{step2}, it remains to show that for every embedding  $\fa:T\hra G$, which is conjugate to $\fa_0$, the module $H_*(\Fl_{\gm},\C{F}_{\theta,\varphi})$ is $\wt{u}_{\fa,\varphi}$-locally finite, and we have an equality
\begin{equation*} %\label{Eq:main}
 \Tr_{\gen}(\wt{u}_{\fa,\varphi},H_*(\Fl_{\gm},\C{F}_{\theta,\varphi}))=\sum_{g\in G^{\sc}(F)/G^{\sc}_{\fa}} t_{\fa,\theta}(g^{-1}\gm g).
\end{equation*}
Moreover, we can replace $\fa_0$ by $\fa$, thus assuming that $\fa=\fa_0$.

\smallskip

(b) By definition (see \rco{corresp} and Section~\re{interpr}(a)), there exists an Iwahori subgroup $\I_0\supseteq\T$ such that $\varphi_{T,\I_0}=\varphi$ and $\wt{u}_{\I_0}=\wt{u}_{T,\varphi}$. Thus it remains to show that
for every Iwahori subgroup $\I_0\supseteq\T$,

\smallskip

\quad\quad (i)  the module $H_i(\Fl_{\gm},\C{F}_{\theta,\varphi_{T,\I_0}})$ is $\wt{u}_{\I_0}$-locally finite, and

\smallskip

\quad\quad (ii) we have an equality
\begin{equation} \label{Eq:main}
\Tr_{\gen}(\wt{u}_{\I_0},H_*(\Fl_{\gm},\C{F}_{\theta,\varphi_{T,\I_0}}))=\sum_{[g]\in G^{\sc}(F)/G^{\sc}_{\fa_0}} t_{\fa_0,\theta}(g^{-1}\gm g).
\end{equation}
\smallskip

(c) It suffices to show the assertions of part~(b) for some Iwahori subgroup $\I_0\supseteq\T$.
Indeed, let $\I\supseteq\T$ is another such Iwahori subgroup. Since $\ov{w}_{\I_0,\I}\circ \varphi_{T,\I}=\varphi_{T,\I_0}$ (see Section~\re{admisaff}(b)) and  $\wt{u}_{\I}=\wt{w}_{\I_0,\I}^{-1}\wt{u}_{\I_0}\wt{w}_{\I_0,\I}$ (see \form{forminv}), it follows from \rp{whaction} that $\wt{w}_{\I_0,\I}$ induces an isomorphism $H_i(\Fl_{\gm},\C{F}_{\theta,\varphi_{T,\I_0}})\isom H_i(\Fl_{\gm},\C{F}_{\theta,\varphi_{T,\I}})$ and interchanges the action of $\wt{u}_{\I_0}$ on $H_i(\Fl_{\gm},\C{F}_{\theta,\varphi_{T,\I_0}})$ with the action of  $\wt{u}_{\I}$ on $H_i(\Fl_{\gm},\C{F}_{\theta,\varphi_{T,\I}})$. Therefore both assertions in part~(b) for $\I$ follow from that for $\I_0$.

\smallskip

(d) By part~(c), we can assume that $\I_0\subseteq\P_{\fa_0}$. Let $J=J_{\P_{\fa_0}}\subseteq \wt{S}$ be the type of $\P_{\fa_0}$.
Then $J$ is $\si$-invariant (see Section~\re{parahoric2}), and $\wt{u}_{\I_0}$ is contained in $W_J\si$ (see Section~\re{comp}).
\end{Emp}

\begin{Emp} \label{E:step4}
In this subsection we are going to show the assertion (i) of Section~\re{step3}(b).

\smallskip

(a) Let $S_{\gm}\subseteq G^0_{\gm}$ be the maximal subtorus which splits over $F^{\nr}$, and let $\Q\subseteq LG$ be a $\si$-stable parahoric subgroup containing $L^+(S_{\gm})$ (see Section~\re{parahoric2}(c)). Choose an Iwahori subgroup $\I\subseteq \Q$, and let $J'=J_{\Q}$ be the type of $\Q$, thus $\Q=\P_{J',\I}$.

\smallskip

(b) Choose $m$ sufficiently large to satisfy \rt{inj}, and equip $\Fl_{\gm}$ with a filtration
$\{\Fl_{\gm}^{(\leq n)}\}_n$, defined by $\Fl_{\gm}^{(\leq n)}:=\Fl_{\gm}^{\leq n;J_r;J'_l;m_{reg}}$ (see Section~\re{filhom}(f)). Then each $\Fl_{\gm}^{(\leq n)}$ is $\si$-invariant (by Section~\re{affflvar2}(d)), hence
$\{H'_i(\Fl_{\gm}^{(\leq n)}, \C{F}_{\theta,\varphi_{T,\I_0}})\}_n$ is a $\si$-invariant filtration of
$H_i(\Fl_{\gm},\C{F}_{\theta,\varphi_{T,\I_0}}))$ (see Section~\re{filhom1}(b)).

\smallskip

(c) Since $\wt{u}_{\I_0}\in W_J\si$ (see Section~\re{step3}(d)), it follows from \rl{image} that the filtration $\{H'_i(\Fl_{\gm}^{(\leq n)}, \C{F}_{\theta,\varphi_{T,\I_0}})\}_n$ is $\wt{u}_{\I_0}$-invariant. Since each $H'_i(\Fl_{\gm}^{(\leq n)}, \C{F}_{\theta,\varphi_{T,\I_0}})$ is finite-dimensional, we conclude that the module $H_i(\Fl_{\gm},\C{F}_{\theta,\varphi_{T,\I_0}}))$ is $\wt{u}_{\I_0}$-locally finite. This completes the proof of assertion (i) of Section~\re{step3}(b).
\end{Emp}

In the next two subsections we will make the left hand side of \form{main} more explicit.

\begin{Emp} \label{E:step5}
(a) We equip $\La$ with a length filtration $\La^{\leq n}$ (see Section~\re{filtr0}), which is finitely generated (by \rco{fg}), $\si$-invariant (see Section~\re{affweyl2}(b)) and $\ov{W}$-invariant (see Section~\re{bruhat}(d)). Thus it is $\ov{W}\rtimes\lan\si\ran$-invariant, thus $\wt{u}_{\I_0}$-invariant.

\smallskip

(b) By Section~\re{step4}(a), $\gm$ is in good position with $\Q=\P_{J',\I}$. Therefore, by Section~\re{step4}(c), $\{H'_i(\Fl_{\gm}^{(\leq n)}, \C{F}_{\theta,\varphi_{T,\I}})\}_n$ is a $\wt{u}_{\I_0}$-invariant filtration of
$H_i(\Fl_{\gm},\C{F}_{\theta,\varphi_{T,\I_0}}))$, which is finitely generated over $\{\La^{\leq n}\}_n$ (by \rco{fingen}).

\smallskip

(c) Using the definition of the generalized trace and \rt{inj}, the left hand side of \form{main} therefore equals $\Tr(\wt{u}_{\I_0},H_*(\Fl^{(\leq n)}_{\gm},\C{F}_{\theta,\varphi_{T,\I_0}}))$ for a sufficiently large $n$. Since $H_i$ is a dual of $H^i$, the left hand side of  \form{main} is therefore equal to $\Tr(\wt{u}_{\I_0}, H^*(\Fl^{(\leq n)}_{\gm},\C{F}_{\theta,\varphi_{T,\I_0}}))$.
\end{Emp}

\begin{Emp} \label{E:step6}
(a) Since ${}^{\si}\varphi_{T,\I_0}=\ov{w}_{\I_0}^{-1}\circ \varphi_{T,\I_0}$ (compare Sections~\re{DL}(c) and \re{geom}), we have a natural isomorphism
\[
\Fr_q^*\C{F}_{\theta,\varphi_{T,\I_0}}\simeq \si_*\C{F}_{\theta,\varphi_{T,\I_0}}\simeq\C{F}_{\theta,{}^{\si}\varphi_{T,\I_0}}=\C{F}_{\theta,\ov{w}_{\I_0}^{-1}\circ \varphi_{T,\I_0}}.
\]

\smallskip

(b) Set $\Fl^{(\leq n)}_J:=\pi_J(\Fl^{(\leq n)})\subseteq\Fl_J$, where $\pi_J:\Fl\to \Fl_J$ is the projection  from Section~\re{affflvar}(c). Then $\Fl^{(\leq n)}:=\pi^{-1}_J(\Fl_J^{(\leq n)})$ (see Section~\re{filhom}(d)), the endomorphism $\wt{u}_{\I_0}$ of
\[
H^*(\Fl^{(\leq n)}_{\gm},\C{F}_{\theta,\varphi_{T,\I_0}})= H^*(\Fl^{(\leq n)}_{J,\gm},(\pi_J)_*(\C{F}_{\theta,\varphi_{T,\I_0}}))
\]
decomposes as a composition of
\[
\Fr_q^*: H^*(\Fl^{(\leq n)}_{J,\gm},(\pi_J)_*(\C{F}_{\theta,\varphi_{T,\I_0}}))\to
H^*(\Fl^{(\leq n)}_{J,\gm},(\pi_J)_*(\C{F}_{\theta,\ov{w}_{\I_0}^{-1}\circ \varphi_{T,\I_0}})),
\]
induced by the isomorphism of part~(a), and the isomorphism
 \[
H^*(\Fl^{(\leq
n)}_{J,\gm},(\pi_J)_*(\C{F}_{\theta,\ov{w}_I^{-1}\circ\varphi_{T,\I_0}}))\isom
H^*(\Fl^{(\leq n)}_{J,\gm},(\pi_J)_*(\C{F}_{\theta,\varphi_{T,\I_0}})),
 \]
induced by the isomorphism
$\wt{w}_{\I_0}:(\pi_J)_*(\C{F}_{\theta,\ov{w}_I^{-1}\circ
\varphi_{T,\I_0}})\isom(\pi_J)_*(\C{F}_{\theta,\varphi_{T,\I_0}})$ (see
Sections~\re{DL}(d) and \re{whaction}(d)).

\smallskip

(c) Now it follows from Section~\re{step5}(c) and the Grothendieck's Lefschetz trace formula,
that the left hand side of \form{main} equals the sum
\begin{equation*} \label{Eq:ltf}
\sum_{x\in \Fl^{(\leq n)}_{J,\gm}(\fq)}\Tr(\wt{w}_{\I_0}\circ
\Fr_x^*,[(\pi_J)_*(\C{F}_{\theta,\varphi_{T,\I_0}})]_x).
\end{equation*}
\end{Emp}

\begin{Emp}
{\bf Reformulation of assertion~(ii) of Section~\re{step3}(b).}

\smallskip

(a) Consider a parahoric subgroup $\Q^{\sc}:=\Q\cap L(G^{\sc})$ of $L(G^{\sc})$ and an open compact subgroup $Q^{\sc}:=\Q^{\sc}(\fq)$ of $G^{\sc}(F)$.

\smallskip

(b) Notice that we have a natural identification $\Fl_{J}(\fq)\simeq G^{\sc}(F)/G^{\sc}_{\fa_0}$.
Moreover, since $\{\Fl^{(\leq n)}\}_n$ is a $\Q^{\sc}$-invariant filtration of $\Fl$, we get that
$\{\Fl_J^{(\leq n)}(\fq)\}_n$ is a $Q^{\sc}$-invariant filtration of $G^{\sc}(F)/G^{\sc}_{\fa_0}$, which we
denote by $G^{\sc}(F)^{(\leq n)}/G^{\sc}_{\fa_0}$. Furthermore, a class $x\in G^{\sc}(F)/G^{\sc}_{\fa_0}$ corresponds
to a point of $\Fl^{(\leq n)}_{J,\gm}(\fq)\subseteq \Fl^{(\leq n)}_{J}(\fq)$ if an only if $x^{-1}\gm x\in G_{\fa_0}$.

\smallskip

(c) By definition, for each sufficiently large $n$, the right hand side of \form{main} equals the sum
$\sum_{x\in G^{\sc}(F)^{(\leq n)}/G^{\sc}_{\fa_0}} t_{\fa_0,\theta}(x^{-1}\gm x)$.
By definition, we have $t_{\fa_0,\theta}(x^{-1}\gm x)=0$, if $x^{-1}\gm x\notin G_{\fa_0}$. Likewise, we have
$t_{\fa_0,\theta}(x^{-1}\gm x)=\Tr R^{\ov{\theta}}_{\ov{\fa}_0}(\ov{x^{-1}\gm x})$, if $x^{-1}\gm x\in G_{\fa_0}$, where
$\ov{x^{-1}\gm x}\in M_{\fa_0}(\fq)$ is the class of $x^{-1}\gm x$.

\smallskip

(d) Using observations of parts~(b), (c) and Section~\re{step6}(c), it remains to show that for every
$x\in \Fl_{J,\gm}(\fq)\subseteq G^{\sc}(F)/G^{\sc}_{\fa_0}$, we have an equality
\begin{equation} \label{Eq:contr}
\Tr(\wt{w}_{\I_0}\circ
\Fr_x^*,[(\pi_J)_*(\C{F}_{\theta,\varphi_{T,\I_0}})]_x)=\Tr R^{\theta}_{\ov{\fa}_0}(\ov{x^{-1}\gm x}).
\end{equation}
\end{Emp}

\begin{Emp}
{\bf Completion of the proof.}

\smallskip

(a) Recall that an Iwahori subgroup $\I_0\subseteq\P_{\al}$ corresponds to the Borel subgroup $\ov{B}:=\I_0/\P_{\fa_0}^+$ of $M_{\fa_0}$, containing the maximal torus $\ov{\fa}_0(\ov{T})\subseteq M_{\fa_0}$.
Next, the abstract Cartan subgroup $T_{M_{\fa_0}}$ is canonically identified with $\ov{T}_{G}$, and under this identification
the admissible isomorphism $\varphi_{T,\I_0}:\ov{T}\isom \ov{T}_G$ corresponds to the admissible isomorphism
$\varphi_{\ov{B}}:\ov{\fa}_0(\ov{T})\isom T_{M_{\fa_0}}$. Moreover, we have a canonical isomorphism $W_J\simeq W_{M_{\fa_0}}$, under which  element $\wt{w}_{\I_0}\in W_J$ corresponds to $w_{\ov{B}}=w_{\varphi_{\ov{B}}}\in W_J$ (see Section~\re{parahoric}  and
Section~\re{comp}).

\smallskip

(b) Since the commutative square \form{SGaff} is
Cartesian, it follows by proper base change that the left hand side of \form{contr} equals
$\Tr(w_{\varphi_{\ov{B}}}\circ\si, H^*(\ov{\Fl}_{M_{\fa_0},\ov{x^{-1}\gm x}},\C{F}_{\ov{\theta},\varphi_{\ov{B}}})$. Therefore, by Lusztig's theorem (\rt{lus}) it is equal to  $R^{\theta}_{\ov{\fa}_0}(\ov{x^{-1}\gm x})$. This completes the proof of equality \form{contr}, from which
equality \form{main} and hence \rt{dl} follows.
\end{Emp}

\subsection{Proof of \rt{equality}}
%{Geometric interpretation of $f^{\st}_{T,\theta}(\gm)$}

\begin{Emp} \label{E:remtr}
{\bf Observations.}

\smallskip

(a) Note that for every $\qlbar[\ov{W}]$-module $V$, the idempotent
$\frac{1}{|\ov{W}|}\sum_{\ov{w}\in\ov{W}}w\in \qlbar[\ov{W}]$ defines a projection
$V\to V^{\ov{W}}$. Therefore the functor of $\ov{W}$-invariants  is exact on
$\qlbar$-vector spaces and the natural composition $V^{\ov{W}}\to V\to V_{\ov{W}}$ is an isomorphism.

\smallskip

(b) Using part~(a), for every finite-dimensional $\qlbar[\ov{W}\rtimes\lan\si\ran]$-module $V$, we have an equality
$\Tr(\si,V^{\ov{W}})=\frac{1}{|\ov{W}|} \sum_{w\in \ov{W}} \Tr(w\si,V)$.

(c) Let $V$ be a $\qlbar[\wt{W}_G\rtimes\lan\si\ran]$-module, which is finitely generated as $\qlbar[\wt{W}_G]$-module. Then each
homology group $H_j(\La_G,V)$ is a  finite-dimensional $\qlbar[\ov{W}\rtimes\lan\si\ran]$-module. Then it follows from parts~(a) and (b) that
we have an equality
\[
\Tr(\si,H_j(\wt{W}_G,V))=\Tr(\si,H_j(\La_G,V)^{\ov{W}})=\frac{1}{|\ov{W}|} \sum_{\ov{w}\in\ov{W}} \Tr(\ov{w}\si,H_j(\La_G,V)).
\]
\end{Emp}

Now we are ready to proof \rt{equality}. For convenience, we divide the proof into five steps:

\smallskip

\noindent{\bf Step 1.} By Section~\re{remtr}, for every $i,j\in\B{Z}$ we have an equality
\begin{equation} \label{Eq:trform1}
\Tr(\si,
H_j(\wt{W}_G,H_i(\Fl_{G,\gm},\C{F}^{\st}_{\theta})))=\frac{1}{|\ov{W}|} \sum_{\ov{w}\in \ov{W}} \Tr(\ov{w}\si,
H_j(\La_G,H_i(\Fl_{G,\gm},\C{F}^{\st}_{\theta}))).
\end{equation}

\smallskip  \noindent{\bf Step 2.} The decomposition $\C{L}^{\st}_{\theta}\simeq\bigoplus_{\varphi}\C{L}_{\theta,\varphi}$ (see Section~\re{notstable}) induces  decompositions $\C{F}^{\st}_{\theta}\simeq\bigoplus_{\varphi}\C{F}_{\theta,\varphi}$ (see Section~\re{ltheta}) and
$H_i(\Fl_{G,\gm},\C{F}_{\theta})\simeq\bigoplus_{\varphi}H_i(\Fl_{G,\gm},\C{F}_{\theta,\varphi})$. Moreover, $\si$ sends
summand $H_i(\Fl_{G,\gm},\C{F}_{\theta,\varphi})\subseteq H_i(\Fl_{G,\gm},\C{F}^{\st}_{\theta})$ to $H_i(\Fl_{G,\gm},\C{F}_{\theta,{}^{\si}\varphi})$, while each $w\in\wt{W}_G$ sends $H_i(\Fl_{G,\gm},\C{F}_{\theta,\varphi})$ to $H_i(\Fl_{G,\gm},\C{F}_{\theta,w\circ \varphi})$ (see \rp{whaction}).

\smallskip

In particular,  $\La_G$ stabilizes each $H_i(\Fl_{G,\gm},\C{F}_{\theta,\varphi})$, so we have a decomposition
\[H_j(\La_G,H_i(\Fl_{G,\gm},\C{F}^{\st}_{\theta}))\simeq\bigoplus_{\varphi}H_j(\La_G,H_i(\Fl_{G,\gm},\C{F}_{\theta,\varphi})).\]
Furthermore, under this decomposition $\si$ sends summand $H_j(\La_G,H_i(\Fl_{G,\gm},\C{F}_{\theta,\varphi}))$ to $H_j(\La_G,H_i(\Fl_{G,\gm},\C{F}_{\theta,{}^\si\varphi}))$, while each $\ov{w}\in\ov{W}$ sends $H_j(\La_G,H_i(\Fl_{G,\gm},\C{F}_{\theta,\varphi}))$ to $H_i(\Fl_{G,\gm},\C{F}_{\theta,\ov{w}\circ \varphi})$.

\smallskip

Since ${}^{\si}\varphi=\ov{w}_{T,\varphi}^{-1}\circ\varphi$ (see Section~\re{basic pair}(b)), the right hand side of
\form{trform1} can be therefore rewritten in the form
\begin{equation} \label{Eq:trform2}
\frac{1}{|\ov{W}|}\sum_{\varphi}\Tr(\ov{w}_{T,\varphi}\circ \si, H_j(\La_G,H_i(\Fl_{G,\gm},\C{F}_{\theta,\varphi}))).
\end{equation}

\smallskip  \noindent{\bf Step 3.}
As in \rn{stconj}, we set $\ov{u}_{T,\varphi}=\ov{w}_{T,\varphi}\circ\si\in\ov{W}\rtimes\lan\si\ran$.
We claim that the trace $\Tr(\ov{u}_{T,\varphi},   H_j(\La_G,H_i(\Fl_{G,\gm},\C{L}_{\theta,\varphi})))$ is independent
of $\varphi$, hence the expression of \form{trform2} equals
$\Tr(\ov{u}_{T,\varphi}, H_j(\La_G,H_i(\Fl_{G,\gm},\C{L}_{\theta,\varphi})))$ for each admissible
$\varphi$.

\smallskip

Indeed, for each pair of admissible isomorphisms $\varphi$ and $\varphi'$, we have an equality $\varphi'=\ov{w}\circ \varphi$ for some
$\ov{w}\in \ov{W}$ (see
Section~\re{admisaff}(b)).  Then we have the equality $u_{T,\varphi'}=\ov{w}^{-1}\ov{u}_{T,\varphi}\ov{w}\in\ov{W}\rtimes\lan\si\ran$.
Since  $\ov{w}\in\ov{W}$ induces an isomorphism
$H_j(\La_G,H_i(\Fl_{G,\gm},\C{L}_{\theta,\varphi}))\isom H_j(\La_G,H_i(\Fl_{G,\gm},\C{L}_{\theta,\varphi'}))$ (by \rp{whaction}), the independence of $\varphi$ follows.

\smallskip
\noindent{\bf Step 4.} By Step 3, the assertion of \rt{equality}(a) follows from \rt{equality}(b) with $\ka=1$. So it remains to show
\rt{equality}(b).

As it was explained in Section~\re{notkappa}(c), every $\ka\in\wh{T}^{\Gm_F}$ gives rise to character $\ka:\La_G\rtimes\lan\wt{u}_{T,\varphi}\ran\to\qlbar$,
and we denote by the same letter its restriction to $\La\rtimes\lan\wt{u}_{T,\varphi}\ran$.

\smallskip

\noindent{\bf Step 5.} Using \rl{ind}(b), we have a natural isomorphism of $\qlbar[\La_G]$-modules
\[
H_i(\Fl_{G,\gm},\C{L}_{\theta,\varphi})_{\ka}\simeq\ind_{\La}^{\La_G}(H_i(\Fl_{\gm},\C{L}_{\theta,\varphi})_{\ka}),
\]
hence a natural isomorphism of $\qlbar[\lan\wt{u}_{T,\varphi}\ran]$-modules
\[
H_j(\La_G,H_i(\Fl_{G,\gm},\C{L}_{\theta,\varphi})_{\ka})\simeq  H_j(\La,H_i(\Fl_{\gm},\C{L}_{\theta,\varphi})_{\ka}).
\]
Thus it remains to show that the equality
\[
f^{\ka}_{T,\theta}(\gm)=\Tr(\ov{u}_{T,\varphi},H_*(\La,H_*(\Fl_{\gm},\C{L}_{\theta,\varphi})_{\ka})).
\]

\noindent{\bf Step 6.}
Since $T\subseteq G$ is a maximal elliptic torus, the induced endomorphism $\ov{u}_{T,\varphi}|_{\La}\in\Aut(\La)$ is elliptic (see Sections~\re{fincor}(d) and \re{finell}(a)).
%
%Therefore for every lift $\wt{u}\in \wt{W}\Gm_{\varphi}$ of $u_{T,\varphi}$
%defines an isomorphism $\ov{\La}\rtimes\lan\wt{u}\ran_n\isom\La_A\bs\wt{W}\Gm_{\varphi}$.
%
%Since $\si^m$ acts trivially on $\La$ for some $m$, $u$ is of finite order as well.
%Next we recall that each hyperspecial parahoric subgroup of $G$
%over $F$ defines a $\Gm_{\fq}$-equivariant isomorphism
%$\wt{W}\isom \La\rtimes W_{G}$ (see Section~\re{splitting}), hence an
%isomorphism $\wt{W}_G\rtimes\Gm_{\fq}\isom \La\rtimes
%(W\rtimes\Gm_{\fq})$.
%
Hence by \rp{trform}(b), for every $i\in\B{Z}$ we have an equality
\begin{equation*}
\Tr(\ov{u}_{T,\varphi},H_*(\La,H_i(\Fl_{\gm},\C{F}_{\theta,\varphi})_{\ka}))=\sum_{\wt{u}}
\Tr_{\gen}(\wt{u}, H_i(\Fl_{\gm},\C{F}_{\theta,\varphi})_{\ka}),
\end{equation*}
%
%By observation Section~\re{???}, the right hand side of \form{3} can
%written in the form
%
%\begin{equation} \label{Eq:sum1}
%\sum_{\wt{w}}\Tr_{\gen}(\wt{w}\circ\si,H_*(\Fl_{\gm},\C{F}_{\theta,\varphi})),
%\end{equation}
%
where $\wt{u}$ runs over a set of representatives of the set of $\La$-conjugacy classes in
$\La \ov{u}_{T,\varphi}\subseteq \La\rtimes\lan \ov{u}_{T,\varphi}\ran$.
%Moreover, since $\La_A$ lies in the center of $\wt{W}\Gm$, the last set coincides
%with the set of $\bar{\La}$-conjugacy classes in
%$\ov{\La}\ov{u}_{T,\varphi}$.

Therefore we get an equality
\begin{equation} \label{Eq:3}
\Tr(\ov{u}_{T,\varphi}, H_*(\La,H_*(\Fl_{\gm},\C{F}_{\theta,\varphi})_{\ka}))=\sum_{\wt{u}}\Tr_{\gen}(\wt{u}, H_*(\Fl_{\gm},\C{F}_{\theta,\varphi})_{\ka}),
\end{equation}
where the summation as in the previous formula.

\smallskip  \noindent{\bf Step 7.}
By Section~\re{elliptic case}(a), the map
$\fa'\mapsto [\wt{u}_{\fa',\varphi}]$ is a bijection
between the set of $G^{\sc}(F)$-conjugacy classes of embeddings $\fa':T\hra G$,
which are stably conjugate to $\fa_0:T\hra G$, and $\La$-conjugacy classes in
$\La \ov{u}_{T,\varphi}\subseteq \La\rtimes\lan \ov{u}_{T,\varphi}\ran$.

%Moreover, the last set coincide with $\La_{\si}\bs\bs\La\ov{w}_{T,\varphi}$ by \rl{finell}.

\smallskip

In other words, the right hand side of \form{3} can be
written in the form
\begin{equation} \label{Eq:fa}
\sum_{\fa}(\sum_{\fa'}\Tr_{\gen}(\wt{u}_{\fa',\varphi},H_*(\Fl_{\gm},\C{F}_{\theta,\varphi})_{\ka})),
\end{equation}
where

\smallskip

\quad\quad$\bullet$ $\fa$ runs over a set of representatives of $G(F)$-conjugacy classes of embeddings $T\hra G$
which are stably conjugate to $\fa_0$;

\smallskip

\quad\quad$\bullet$ $\fa'$ runs over a set of representatives of $G^{\sc}(F)$-conjugacy classes of embeddings $T\hra G$
which are $G(F)$-conjugate to $\fa$.

\smallskip
\noindent{\bf Step 8.}
By construction (see Section~\re{interpr}(f)) for every embedding $\fa':T\hra G$, which is stably conjugate to $\fa_0$, we have
$\wt{u}_{\fa',\varphi}=\la_{\fa'}\wt{u}_{T,\varphi}$ with $\la_{\fa'}\in\La$. Then by definition (see Section~\re{notkappa}(d)),
 we have an
equality
\begin{equation} \label{Eq:fa1}
\Tr_{\gen}(\wt{u}_{\fa',\varphi},H_*(\Fl_{\gm},\C{F}_{\theta,\varphi})_{\ka})=
\lan\ka,\la_{\fa'}\ran \Tr_{\gen}(\wt{u}_{\fa',\varphi},H_*(\Fl_{\gm},\C{F}_{\theta,\varphi})).
\end{equation}

\smallskip
Moreover, we have equalities
\begin{equation} \label{Eq:fa2}
\lan\ka,\la_{\fa'}\ran =\lan\ka,\inv([\wt{u}_{T,\varphi}],[\wt{u}_{\fa',\varphi}])\ran=  \lan\ka,\inv(\fa_0,\fa')\ran,
\end{equation}
the first of which follows from the fact that $\inv([\wt{u}_{T,\varphi}],[\wt{u}_{\fa',\varphi}])\in\La_{\ov{u}_{T,\varphi}}$
(see Section~\re{remstconj}(c)) is the projection of $\la_{\fa'}$, and the second one follows from \rl{stconj}(b).

\smallskip

Combining equalities \form{fa1} and \form{fa2} with \rt{dl} we see that the expression \form{fa} is equal
to $\sum_{\fa}\lan\ka,\inv(\fa_0,\fa)\ran f_{T_{\fa},\theta_{\fa}}(\gm)=
f^{\ka}_{T,\theta}(\gm)$. This completes the proof of the theorem.
%\end{proof}

\subsection{Proof of \rt{stab}}
We start the proof with the following observations:

\begin{Emp} \label{E:dual}
{\bf The canonical maps of Section~\re{setup comp}(c) revisited.}

\smallskip

(a) Recall that for every torus $S$ over $K$, the group
algebra $\qlbar[X_*(S)]$ is naturally isomorphic to the algebra of
regular functions $\qlbar[\wh{S}]$, where $\wh{S}$ is the dual
torus of $S$ defined over $\qlbar$. Therefore we have natural isomorphisms
\[
\qlbar[X_*(S)_{\Gm_K}]\simeq \qlbar[X_*(S)]_{\Gm_K}\simeq
\qlbar[\wh{S}]_{\Gm_K}\simeq \qlbar[\wh{S}^{\Gm_K}].
\]

\smallskip

(b) By part~(a) and Section~\re{pi0}, we have a canonical isomorphism
\[
\qlbar[\pi_0(LS)]\isom\qlbar[X_*(S)_{\Gm_K}]\isom\qlbar[\wh{S}^{\Gm_K}].
\]

\smallskip

(c) The dual torus $\wh{T_G}$ is canonically isomorphic the
abstract Cartan $T_{\wh{G}}$ of the dual group $\wh{G}$.  Therefore we have canonical isomorphisms of $\qlbar$-algebras
\[
\qlbar[X_*(T_G)]\simeq \qlbar[\wh{T_G}]\simeq\qlbar[T_{\wh{G}}],
\]
which induces an isomorphism
\[
\qlbar[\La_G]^{\ov{W}}\simeq \qlbar[X_*(T_G)]^{W_G}\simeq\qlbar[T_{\wh{G}}]^{W_{\wh{G}}}\simeq\qlbar[T_{\wh{G}}/W_{\wh{G}}]\simeq\qlbar[c_{\wh{G}}],
\]
where the first isomorphism is from \form{can}, and $c_{\wh{G}}$ is the Chevalley space of $\wh{G}$.

\smallskip

(d) By the identification of the parts~(b) and (c), for every $\gm\in G^{\rss}(K)$, the canonical morphism $\pr_{\gm}:\qlbar[\La_G]^{\ov{W}}\to\qlbar[\pi_0(LG^0_{\gm})]$ defined
in Section~\re{setup comp}(c) can be viewed as a morphism $\qlbar[c_{\wh{G}}]\to \qlbar[\wh{G^0_{\gm}}^{\Gm_K}]$, hence comes from a morphism of affine algebraic varieties
\begin{equation} \label{Eq:can2}
\wh{G^0_{\gm}}^{\Gm_K}\to c_{\wh{G}}
\end{equation}
Furthermore, it is easy to see that the morphism \form{can2} is noting but the restriction $\can_{\gm}|_{\wh{G^0_{\gm}}^{\Gm_K}}$, where
$\can_{\gm}:\wh{G^0_{\gm}}\to c_{\wh{G}}$ was defined in Section~\re{connected}(a).

\smallskip

(e) By part~(d), morphism $\pr'_{\gm}$ from  Section~\re{setup comp}(c) comes from a morphism $\can'_{\gm}|_{\wh{G^0_{\gm}}^{\Gm_K}}$, where
$\can'_{\gm}=\can_{\gm}\circ\iota$ and $\iota:\wh{G^0_{\gm}}\to \wh{G^0_{\gm}}$ is the inverse map $g\mapsto g^{-1}$.

\end{Emp}

\begin{Emp} \label{E:sheafrep}
{\bf Correspondence between representations and sheaves.}

\smallskip

(a) For every finitely generated commutative $\qlbar$-algebra $A$,
there is a natural equivalence of categories $M\mapsto \C{F}_M$ between $A$-modules and (coherent) quasi-coherent
sheaves on $\Spec A$. Moreover, $M$ is finitely-generated if and only if $\C{F}_M$ is coherent.

\smallskip

(b) By part~(a), for every  finitely-generated abelian group $\Theta$,
there is a natural equivalence of categories between
representations of $\Theta$ over $\qlbar$ and quasi-coherent sheaves on $\Spec\qlbar[\Theta]$.

\smallskip

(c) Let $f:A\to B$ be a homomorphism   finitely-generated
commutative $\qlbar$-algebras, and let $\Spec f$ be the
corresponding morphism $\Spec B\to \Spec A$. Then the restriction
functor  $f^*:\Mod (B)\to \Mod(A)$ corresponds via the
correspondence of part~(a) to the push-forward of quasi-coherent
sheaves.

\smallskip

(d) Note that the multiplication map $m:A\otimes_{\qlbar} A\to A$ is an algebra homomorphism, and
the restriction functor  $m^*:\Mod (A)\to \Mod(A\otimes_{\qlbar}A)$ corresponds via the correspondence of part~(a)
to the push-forward $\Dt_*$ of quasi-coherent sheaves, where $\Dt:\Spec(A)\to \Spec(A)\times\Spec(A)$ is the diagonal map.

\smallskip

(e) Using Sections~\re{dual}(b),(c),  the correspondence of parts~(a),(b) gives an
equivalence between

\smallskip

\quad\quad$\bullet$ representations of $\pi_0(L{G}_{\gm})$ over $\qlbar$ and quasi-coherent
sheaves on $\wh{G^0_{\gm}}^{\Gm_K}$;

\smallskip

\quad\quad$\bullet$ representations of $\La_G$ over $\qlbar$ and  quasi-coherent sheaves on
$T_{\wh{G}}$;

\smallskip

\quad\quad$\bullet$ $\qlbar[\La_G]^{\ov{W}}$-modules and quasi-coherent sheaves
on $c_{\wh{G}}$.
\end{Emp}

\begin{Emp} \label{E:ggmaction}
{\bf Observations.}

\smallskip

(a) Using \rp{whaction}(d) one shows that for every local system $\C{L}$ on $\ov{T}_G$, the natural (left) action of $LG^0_{\gm}$ on $\Fl_G$ induces an action on $H_i(\Fl_{G,\gm},\C{F}_{\C{L}})$, commuting with the $\La_G$-action. Moreover, the action of $LG^0_{\gm}$ on $H_i(\Fl_{G,\gm},\C{F}_{\C{L}})$ factors though the action of
$\pi_0(LG^0_{\gm})$ (compare Section~\re{cent}(c)).

\smallskip

(b) As in Section~\re{notkappa}(d), for every character $\ka:\La_G\to\qlbar\m$ we can form a representation  $H_i(\Fl_{G,\gm},\C{F}_{\C{L}})_{\ka}$ of $\La_G$. Moreover, taking tensor product of the action of $\pi_0(LG^0_{\gm})$ on $H_i(\Fl_{G,\gm},\C{F}_{\C{L}})$ from part~(a) and the trivial action on $(\qlbar)_{\ka}$, we gen an action of $LG^0_{\gm}$  on $H_i(\Fl_{G,\gm},\C{F}_{\C{L}})_{\ka}$, commuting with the $\La_G$-action. In particular, we have an action of  $\pi_0(LG^0_{\gm})$ on each $H_j(\La_G, H_i(\Fl_{G,\gm},\C{F}_{\C{L}})_{\ka})$.

\smallskip

(c) Recall that $H_i(\Fl_{G,\gm},\C{F}_{\C{L}})$ a finitely-generated
$\qlbar[\La_G]$-module (see \rp{fingen}, \rl{fingen} and the observation of Section~\re{cent}(a)).
Therefore each $H_j(\La_G, H_i(\Fl_{G,\gm},\C{F}_{\C{L}})_{\ka})$ is finite-dimensional over $\qlbar$.

\smallskip

(d) By parts~(b),(c) and the correspondence of Section~\re{sheafrep}, $H_j(\La_G, H_i(\Fl_{G,\gm},\C{F}_{\C{L}})_{\ka})$
corresponds to a coherent sheaf $\C{A}$ on $\wh{G^0_{\gm}}^{\Gm_K}$.

\smallskip

(e) Using isomorphism of Section $\qlbar[\La_G]\simeq \qlbar[T_{\wh{G}}]$, character $\ka:\La_G\to\qlbar\m$ gives rise to a point
$\ka\in T_{\wh{G}}$, hence to a point $\nu_{\wh{G}}(\ka)\in c_{\wh{G}}$.

\end{Emp}

The following claim is crucial for what follows.

\begin{Cl} \label{C:support}
The coherent sheaf $\C{A}$ on $\wh{G^0_{\gm}}^{\Gm_K}$ defined in Section~\re{ggmaction}(d)
is supported on a finite set $\on{can}_{\gm}^{-1}(\nu_{\wh{G}}(\ka))\cap \wh{G^0_{\gm}}^{\Gm_K}$.
\end{Cl}

\begin{proof} We  present the argument in steps.

\smallskip

\noindent{\bf Step 1.} Let $\{F^rH_i(\Fl_{G,\gm},\C{F}_{\C{L}})\}_r$ be the filtration of
$H_i(\Fl_{G,\gm},\C{F}_{\C{L}})$ from \rt{action}. Then each graded
piece $V:=\gr^rH_i(\Fl_{G,\gm},\C{F}_{\C{L}})$ is a representation of the product
$\pi_0(LG^0_{\gm})\times\La_G$, hence the vector space $H_j(\La_G,V_{\ka})$
is a representation $\pi_0(LG^0_{\gm})$. Thus, by the correspondence of Section~\re{sheafrep}, $H_j(\La_G,V_{\ka})$
corresponds to a coherent sheaf $\C{W}$ on $\wh{G^0_{\gm}}^{\Gm_K}$, and it suffices to show that $\C{W}$ is
supported on $\on{can}_{\gm}^{-1}(\nu_{\wh{G}}(\ka))$.

%Notice first that we have a natural identification
%$H_*(\La_G,V_{\ka})\simeq V\otimes^L_{\qlbar[\La_G]}\ka^{-1}$, hence an identification
%$H_j(\La_G,V_{\ka})\simeq \Tor_j^{\qlbar[\La_G]}(V,\ka^{-1})$.

\smallskip

\noindent{\bf Step 2.} Recall that $V$ is a $\qlbar[\pi_0(LG^0_{\gm})\times\La_G]$-module, that is, a $\qlbar[\pi_0(LG^0_{\gm})]\otimes_{\qlbar}\qlbar[\La_G]$-module.
In particular, $V$ has a natural structure of a $\qlbar[\La_G]$-module, hence a $\qlbar[\La_G]\otimes_{\qlbar}\qlbar[\La_G]$-module
(see Section~\re{sheafrep}(d)). Moreover, it follows from \rt{action} that the action of a subalgebra
$\qlbar[\La_G]^{\ov{W}}\otimes_{\qlbar}\qlbar[\La_G]\subseteq\qlbar[\La_G]\otimes_{\qlbar}\qlbar[\La_G]$ on $V$ is induced from the action of
$\qlbar[\pi_0(LG^0_{\gm})]\otimes_{\qlbar}\qlbar[\La_G]$ via homomorphism $\pr'_{\gm}\otimes\Id$ (see Section~\re{dual}(e)).

\smallskip

\noindent{\bf Step 3.} Now we apply the correspondence Section~\re{sheafrep}. Namely,

\smallskip

\quad$\bullet$ $V$ viewed as a $\qlbar[\pi_0(LG^0_{\gm})]\otimes_{\qlbar}\qlbar[\La_G]$-module corresponds to a quasi-coherent sheaf
$\C{V}$ on $\wh{G^0_{\gm}}^{\Gm_K}\times T_{\wh{G}}$;

\smallskip

\quad $\bullet$ $V$ viewed as $\qlbar[\La_G]$-module corresponds to a quasi-coherent sheaf
$\C{U}:=(\pr_2)_*(\C{V})$ on $T_{\wh{G}}$;

\smallskip

\quad $\bullet$ $V$ viewed as $\qlbar[\La_G]\otimes_{\qlbar}\qlbar[\La_G]$-module corresponds to a quasi-coherent sheaf
$\Dt_*(\C{U})$ on $T_{\wh{G}}\times T_{\wh{G}}$;

\smallskip

\quad $\bullet$ the compatibility assertion, mentioned above, implies that we have an isomorphism
\begin{equation} \label{Eq:compat}
(\on{can}'_{\gm}\times\Id)_*(\C{V})\simeq(\nu_{\wh{G}}\times\Id)_*\Dt_*(\C{U})
\end{equation}
of quasi-coherent sheaves on $c_{\wh{G}}\times T_{\wh{G}}$.

\smallskip

\noindent{\bf Step 4.} For every $t\in T_{\wh{G}}$ we denote by $m_t:T_{\wh{G}}\to T_{\wh{G}}$ the multiplication map by $t$ and by
$i_t:\on{pt}\to T_{\wh{G}}$ the embedding of a closed point with image $t\in T$.

Then the $\qlbar[\pi_0(LG^0_{\gm})]\otimes_{\qlbar}\qlbar[\La_G]$-module $V_{\ka}$
corresponds to the quasi-coherent sheaf $(\Id\times m_{\ka})_*(\C{V})\simeq(\Id\times m_{\ka^{-1}})^*(\C{V})$,  hence the $\qlbar[\pi_0(LG^0_{\gm})]$-module
$H_j(\La_G,V_{\ka})$ corresponds to the derived pullback
\[
\C{W}:=R^j(\Id\times i_1)^*(\Id\times m_{\ka^{-1}})^*(\C{V})\simeq R^j(\Id\times i_{\ka^{-1}})^*(\C{V}),
\]
and we claim that $\C{W}$ is supported at $\on{can}_{\gm}^{-1}(\nu_{\wh{G}}(\ka))=\on{can}'^{-1}_{\gm}(\nu_{\wh{G}}(\ka^{-1}))$.
It suffices to show that the pushforward $(\on{can}'_{\gm})_*(\C{W})$ is supported at a closed point $\nu_{\wh{G}}(\ka^{-1})$.

\smallskip

\noindent{\bf Step 5.} Using base change and equation \form{compat}, we have an isomorphism
\begin{equation} \label{Eq:compat1}
(\on{can}'_{\gm})_*(\C{W})\simeq R^j(\Id\times i_{\ka^{-1}})^*(\nu_{\wh{G}}\times\Id)_*\Dt_*(\C{U}),
\end{equation}
so it remains to show that the RHS of \form{compat1} is supported at $\nu_{\wh{G}}(\ka^{-1})$.

But this is easy: the pushforward $(\nu_{\wh{G}}\times\Id)_*\Dt_*(\C{U})$ is supported at the graph $\Gm_{\nu_{\wh{G}}}\subseteq c_{\wh{G}}\times T_{\wh{G}}$, so the RHS of \form{compat1} is supported at
\[
\pr_1(\Gm_{\nu_{\wh{G}}}\cap \left(c_{\wh{G}}\times\{\ka^{-1}\}\right))=\pr_1(\{(\nu_{\wh{G}}(\ka^{-1}),\ka^{-1})\})=\{\nu_{\wh{G}}(\ka^{-1})\}.
\]
\end{proof}

%\begin{Emp} \label{E:action}
%(a) Note first that the action of $G_{\gm}(F^{\nr})$ on
%$\Fl_{\gm}$ naturally extends to the action of the groups
%scheme $\un{G}_{\gm}$ (Explain?). In particular, the action of
%$G_{\gm}(F^{\nr})$ on  $V^i_{\gm}(\C{L}^{\st}_{\theta})$ extends to the
%action of  $\un{G}_{\gm}$, thus gives size to the action of the
%group of connected components $\pi_0(\un{G}_{\gm})$. Again the
%action of  $\pi_0(\un{G}_{\gm})$ commutes with the action of
%$\wt{W}_G$.

%(b) By (a), each $V^i_{\gm}(\C{L}^{\st}_{\theta})$ is equipped with an
%action of  the group algebra $\qlbar[\pi_0(\un{G}_{\gm})]$. Using
%homomorphism $\pr'_{\gm}$ from Section~\re{dual}(d),
%$V^i_{\gm}(\C{L}^{\st}_{\theta})$ is therefore is equipped with an
%action of $\qlbar[\La_G]^{W}$.

%(c) On the other hand, each $V^i_{\gm}(\C{L}^{\st}_{\theta})$ is
%equipped with an $\wt{W}_G$-action, so it is equipped with an
%action of $\qlbar[\wt{W}_G]$-action, hence with an action of its
%subalgebra  $\qlbar[\La_G]^{W}$.
%\end{Emp}

\begin{Emp} \label{E:obsstab}
{\bf Observations.}  Consider semidirect product $\pi_0(LG^0_{\gm})\ltimes \pi_G^{-1}(\lan\ov{u}_{T,\varphi}\ran)$ given by the formula $\wt{u}\la\wt{u}^{-1}={}^{\si}\la$ for every $\wt{u}\in\pi_G^{-1}(\ov{u}_{T,\varphi})$ and $\la\in\pi_0(LG^0_{\gm})$.

\smallskip

(a) Using \rp{whaction}(d) and the commutative diagram of Section~\re{homaffspr}(a), there exists a unique action of $\pi_0(LG^0_{\gm})\ltimes \pi_G^{-1}(\lan\ov{u}_{T,\varphi}\ran)$ on $H_i(\Fl_{G,\gm},\C{F}_{\theta,\varphi})$, extending the action of $\pi_0(LG^0_{\gm})$ from Section~\re{setup comp}(a) and the action of $\pi_G^{-1}(\lan\ov{u}_{T,\varphi}\ran)$ from Section~\re{ltheta}(f).

\smallskip

(b) Character $\ka$ from Section~\re{notkappa}(c) uniquely extends to a character of the semidirect product $\pi_0(LG^0_{\gm})\ltimes\pi_G^{-1}(\lan\ov{u}_{T,\varphi}\ran)$ trivial on $\pi_0(LG^0_{\gm})$. Then, as in
Section~\re{notkappa}(d), we can form a representation $H_i(\Fl_{G,\gm},\C{F}_{\theta,\varphi})_{\ka}$ of
$\pi_0(LG^0_{\gm})\ltimes\pi_G^{-1}(\lan\ov{u}_{T,\varphi}\ran)$, making $H_j(\La_G, H_i(\Fl_{G,\gm},\C{F}_{\theta,\varphi})_{\ka})$ becomes a representation of the semi-direct product $\pi_0(LG^0_{\gm})\ltimes\lan\ov{u}_{T,\varphi}\ran$, where $\ov{u}_{T,\varphi}\cdot\la\cdot \ov{u}_{T,\varphi}^{-1}={}^{\si}\la$ for every $\la\in \pi_0(LG^0_{\gm})$.

\smallskip

(c) Let $\gm'$ be a stable conjugate of $\gm$, and $g\in G(K)$ be such that
$g\gm g^{-1}=\gm'$. Using \rp{whaction}(d), we can tensor the isomorphism of $\La_G$-representations
$l_g:H_i(\Fl_{G,\gm},\C{F}_{\C{L}})\isom H_i(\Fl_{G,\gm'},\C{F}_{\C{L}})$ with the identity map of $(\qlbar)_{\ka}$, thus getting an isomorphism $l_g:H_i(\Fl_{G,\gm},\C{F}_{\C{L}})_{\ka}\isom H_i(\Fl_{G,\gm'},\C{F}_{\C{L}})_{\ka}$ of $\La_G$-representations, hence an isomorphism
\[
l_g:H_j(\La_G, H_i(\Fl_{G,\gm},\C{F}_{\C{L}})_{\ka})\isom H_j(\La_G, H_i(\Fl_{G,\gm'},\C{F}_{\C{L}})_{\ka}).
\]
Since $({}^{\si}g)\gm ({}^{\si}g)^{-1}=\gm'$ we similarly get an isomorphism $l_{{}^{\si}g}$. Moreover, arguing as in part~(b), these isomorphisms are connected by the identity $\ov{u}_{T,\varphi}\circ l_g \circ\ov{u}_{T,\varphi}^{-1}=l_{{}^{\si}g}$.
\end{Emp}

\begin{Emp} \label{E:funcdl}
{\bf Construction.}

\smallskip

(a) Consider reductive group $G_1:=G\times^{Z(G)}T$, that is,  $G_1:=(G\times T)/Z(G)$, where $Z(G)$ acts on $G\times T$ by the formula
$(g,t)z=(gz,z^{-1}t)$, and let $\pi:G\hra G_1$ be the isogeny $g\mapsto [g,1]$.
Then $T_1:=T\times^{Z(G)}T\subseteq G\times^{Z(G)}T=G'$ is a maximal elliptic torus. Clearly, $T_1$ splits of $F^{\nr}$,
so $G_1$ splits over $F^{\nr}$.

\smallskip

(b) Note that $\pi$ gives rise to an embedding $\pi:T\hra T_1:t\mapsto[t,1]$, which  has a left inverse $m:[a,b]\mapsto ab$.
Therefore the induced map $\pi_*:H^1(F,T)\to H^1(F,T_1)$ is injective. Similarly, the induced map $\wh{\pi}:\wh{T_1}\to\wh{T}$ has a right inverse $\wh{m}:\wh{T}\to\wh{T_1}$ both of which
are $\Gm_F$-equivariant. In particular, every $\ka\in\wh{T}^{\Gm_F}$ gives rise to an element $\ka_1:=\wh{m}(\ka)\in \wh{T_1}^{\Gm_F}$, and we have  $\wh{\pi}(\ka_1)=\ka$.

\smallskip

(c) We choose a tamely ramified character $\theta_1:T_1(F)\to \qlbar\m$, extending $\theta$. Then $\theta_1$ is in general position,
since $\theta$ is such. So $(\fa_1,\theta_1)$ satisfy the assumptions of Section~\re{DLpadic}, so as in Section~\re{fst}(b) we can form a $\Ad G_1(F)$-invariant function $f^{\ka_1}_{T_1,\theta_1}$ on $G_1^{\rss}(F)$.
\end{Emp}

\begin{Prop} \label{P:restr}
In the situation of Section~\re{funcdl}, the pullback $\pi^*(f^{\ka_1}_{T_1,\theta_1})$ to $G^{\rss}(F)$ equals $f^{\ka}_{T,\theta}$.
\end{Prop}

\begin{proof} To expose the structure of the proof we divide it into steps.

\begin{Emp} \label{E:Step1}
Every embedding $\fa:T\hra G$ stably conjugate to the inclusion $T\hra G$ gives rise to the embedding $\fa_1:T_1\hra G_1$
stably conjugate to the inclusion, given by the formula $\fa_1([a,b])=(\fa(a),b)$. Furthermore,
the assignment $\fa\mapsto \fa_1$ is a bijection, where the inverse map is the restriction.

\smallskip

Moreover, for  two stably conjugate embeddings $\fa,\fa':T\hra G$ with the corresponding embeddings
$\fa_1,\fa'_1:T_1\hra G_1$ the relative positions $\inv(\fa,\fa')\in H^1(F,T)$ and  $\inv(\fa_1,\fa'_1)\in H^1(F,T_1)$
are connected by the identity
\[
\pi_*(\inv(\fa,\fa'))=\inv(\fa_1,\fa'_1).
\]

\smallskip

Since $\pi_*$ is injective (by Section~\re{funcdl}(b)),
we thus conclude that two embeddings $\fa,\fa':T\hra G$ are $G(F)$-conjugate if and only if embeddings $\fa_1,\fa'_1:T_1\hra G_1$ are $G_1(F)$-conjugate.

\end{Emp}

\begin{Emp} \label{E:Step2}
By Section~\re{Step1} it suffices to show that for every embedding  $\fa:T\hra G$ stably conjugate to the inclusion $T\hra G$ with the corresponding embedding $\fa_1:T_1\hra G_1$ the pullback $\pi^*(f_{\fa_1,\theta_1})$ equals $f_{\fa,\theta}$, that is, for every $\gm\in G^{\rss}(F)$ we have an equality
\[
f_{\fa,\theta}(\gm)=f_{\fa_1,\theta_1}(\pi(\gm)).
\]

\smallskip
We will show this assertion in Section~\re{Step3}, using an analogous assertion for Deligne--Lusztig representations over finite fields shown in
\rl{funcdl} and two identities shown in \rl{identities}.
\end{Emp}

\begin{Lem} \label{L:funcdl}
Let $\pi:H\to H_1$ be a quasi-isogeny of connected reductive groups over $\fq$, $S_1\subseteq H_1$ a maximal torus over $\fq$, and $\zeta_1: S_1(\fq)\to\qlbar\m$ a character.

Let $S:=\pi^{-1}(S_1)\subseteq H$ be the induced maximal torus of $H$,
and let $\zeta:=\zeta_1|_{S(\fq)}$ be the pullback of $\zeta_1$. Then the Deligne--Lusztig virtual representation $R^{\zeta}_{S}$ of $H(\fq)$ is isomorphic to the pullback of $R^{\zeta_1}_{S_1}$.
\end{Lem}
\begin{proof}
Note that (see \cite[Definition~11.1]{DM}) the virtual representation $R^{\zeta}_{S}$ can be described as the image of character $\theta\in \Rep S(\fq)$ under the Lusztig induction functor, given by the virtual $(H(\fq),S(\fq))$-bimodule $V:=H^*_c (\C{L}_H^{-1}(U),\qlbar)$, where $B=SU\supseteq S$ is a Borel subgroup defined over $\fqbar$, $\C{L}_H:H\to H$ is the Lang isogeny $h\mapsto h^{-1}{}^{\si}h$, and the action of
$H(\fq)\times S(\fq)$ on $V$ is induced by the action $(h,s)(x)=hxs^{-1}$ on $\C{L}_H^{-1}(U)\subseteq H$.

\smallskip

Similarly, $R^{\zeta_1}_{S_1}$ is isomorphic to the image of $\theta_1$ by the functor, given by the virtual $(H_1(\fq),S_1(\fq))$-bimodule $V_1:=H^*_c (\C{L}_{H_1}^{-1}(U_1),\qlbar)$, so  we have to show that
the $(H(\fq),S_1(\fq))$-bimodules $\Ind_{S(\fq)}^{S_1(\fq)}(V)$ and $\on{Res}_{H_1(\fq)}^{H(\fq)}(V_1)$ are isomorphic.

\smallskip

Note that quasi-isogeny $\pi:H\to H_1$ induces an isomorphism $\pi|_{U}:U\isom U_1$. Hence it induces an $(H(\fq)\times S(\fq)$-equivariant map $\phi:\C{L}_H^{-1}(U)\to\C{L}_{H_1}^{-1}(U_1)$, hence an $(H(\fq)\times S_1(\fq))$-equivariant map
$\phi':\C{L}_H^{-1}(U)\times^{S(\fq)}{S_1(\fq)}\to \C{L}_{H_1}^{-1}(U_1)$,
and it suffices to show that $\phi'$ is an isomorphism.

\smallskip

Moreover, since $\phi'$ is $S_1(\fq)$-equivariant, and $S_1(\fq)$ acts freely on both sides, we can divide by the action of $S_1(\fq)$, so it suffices to show that the induced map
\[
\ov{\phi}:\C{L}_H^{-1}(U)/S(\fq)\to\C{L}_{H_1}^{-1}(U_1)/S_1(\fq)
\]
is an isomorphism. Since Lang isogenies $\C{L}_H$ and $\C{L}_{H_1}$ are  \'etale, and $\pi|_{U}$ is an isomorphism, we conclude that the map $\phi$ is \'etale. Hence $\ov{\phi}$ is \'etale, so it suffices to check that it induces a bijection on $\fqbar$-points.

\smallskip

To show the injectivity of $\ov{\phi}$, note that the map $\ov{\pi}:H/S\to H_1/S_1$, induced by $\pi$, is an isomorphism. It remains to show that the map $\C{L}_H^{-1}(U)/S(\fq)\to H/S$, induced by the inclusion $\C{L}_H^{-1}(U)\hra H$, is injective. Explicitly, we have to show that
\begin{equation} \label{Eq:fqrational}
\text{if }h\in \C{L}_H^{-1}(U)(\fqbar)\text{ and }s\in S(\fqbar)\text{ satisfy }hs\in \C{L}_H^{-1}(U),\text{ then }s\in S(\fq).
\end{equation}

Since $h\in \C{L}_H^{-1}(U)$, we have
\begin{equation} \label{Eq:lang}
(hs)^{-1}{}^{\si}(hs)=s^{-1}(h^{-1}{}^{\si}h){}^{\si}s\in s^{-1}U{}^{\si}s=U(s^{-1}{}^{\si}s),
\end{equation}
so our assumption $hs\in \C{L}_H^{-1}(U)$ implies that $s^{-1}{}^{\si}s=1$, thus $s\in S(\fq)$.

\smallskip

To show the surjectivity of $\ov{\phi}$, we can assume that $\pi$ is the projection $H\to H^{\ad}$. Choose any $h_1\in \C{L}_{H_1}^{-1}(U_1)(\fqbar)$ and $h\in\pi^{-1}(h_1)\subseteq H(\fqbar)$. Then we have $\pi(h^{-1}{}^{\si}h)=h_1^{-1}{}^{\si}h_1\in U_1$, thus $h^{-1}{}^{\si}h\in Uz$ for some $z\in Z(H)(\fqbar)\subseteq S(\fqbar)$.

By Lang's theorem, there exists $s\in S(\fqbar)$ such that $s^{-1}{}^{\si}s=z^{-1}$. Then we have $h^{-1}{}^{\si}h\in U(s{}^{\si}(s^{-1}))$ and arguing as in \form{lang} we conclude that $(hs)^{-1}{}^{\si}(hs)\in U$, that is, $hs\in \C{L}_{H}^{-1}(U)$. Then $\pi(hs)=h_1\pi(s)\in \C{L}_{H_1}^{-1}(U_1)$ so we conclude from assertion \form{fqrational} that $\pi(s)\in S_1(\fq)$, thus $\ov{\phi}([hs])=[h_1\pi(s)]=[h_1]$, as claimed.
\end{proof}

\begin{Lem} \label{L:identities}
We claim that we have identities
$G(F)\cdot(\wt{G_1})_{\fa_1}=G_1(F)$ and $G(F)\cap (\wt{G_1})_{\fa_1}=\wt{G}_{\fa}$.

\end{Lem}
\begin{proof}
We have to show that the composition $c:(\wt{G_1})_{\fa_1}\hra G_1(F)\overset{\pr}{\lra}G_1(F)/G(F)$ is surjective and its kernel is
$\wt{G}_{\fa}$.

Since torus $\fa(T)$ is unramified, the torus $\fa(T)/Z(G)\simeq G_1/G$ is unramied. Thus, as in \cite[Lemma~1.8.17(b)]{KV}, an exact sequence $1\to G\to G_1\to G_1/G\to 1$ gives rise to an exact sequence
\begin{equation} \label{Eq:exact}
1\to G_{\fa}\to (G_1)_{\fa_1}\to (G_1/G)(\C{O})\to 1.
\end{equation}

Since $\fa(T)\subseteq G$ is elliptic, the torus $\fa(T)/Z(G)\simeq G_1/G$ is unisotropic, thus
$(G_1/G)(\C{O})=(G_1/G)(F)$ (see \cite[proof of Lemma~2.2.6]{KV}).
So the surjectivity of the projection $(G_1)_{\fa_1}\to(G_1/G)(\C{O})=(G_1/G)(F)$ implies the surjectivity of the composition
$(G_1)_{\fa_1}\hra G_1(F)\overset{\pr}{\lra} G_1(F)/G(F)$, hence the surjectivity of $c$.

\smallskip

Next let $Z_0\subseteq Z(G)^0$ be the maximal split torus, and we claim that $\wt{G}_{\fa}=G_{\fa}\cdot Z_0(F)$. Indeed,
this follows from equalities $\wt{G_{\fa}}=G_{\fa}\cdot \fa(T)(F)$ (see \cite[Corollary~2.2.7(a)]{KV}), $\fa(T)(F)=\fa(T)(\C{O})\cdot Z_0(F)$
(see \cite[Lemma~2.2.6]{KV}) and inclusion $\fa(T(\C{O}))\subseteq G_{\fa}$ (see Section~\re{DLpadic}(a)). Moreover, since $\fa(T)\subseteq G$ is an elliptic torus, we conclude that $Z_0\subseteq Z(G_1)^0$ is a maximal split torus as well, hence $(\wt{G_1})_{\fa_1}=(G_1)_{\fa_1}\cdot Z_0(F)$.

\smallskip

Write $g_1\in \Ker c\cap (\wt{G_1})_{\fa_1}$ in the form $g_1=gz$, where $g\in (G_1)_{\fa_1}$ and $z\in Z_0(F)$.
Then $z\in Z_0(F)\in G(F)\subseteq\Ker c$, so $g\in \Ker c\cap (G_1)_{\fa_1}$. Then, by the exact sequence \form{exact} we have
$g\in G_{\fa}$, thus $g_1\in \wt{G}_{\fa}$.
\end{proof}

\begin{Emp} \label{E:Step3}
Now we are ready to finish the proof of \rp{restr}.
Recall (see Section~\re{step1}(b)) that for every $\gm\in G^{\rss}(F)$ we have an equality
\[
f_{\fa,\theta}(\gm)=\sum_{g\in G(F)/\wt{G}_{\fa}}t_{\fa,\theta}(g^{-1}\gm g),
\]
where the sum of the right hand side stabilizes uniformly, and similarly
\[
f_{\fa_1,\theta_1}(\gm)=\sum_{g\in G_1(F)/(\wt{G_1})_{\fa_1}}t_{\fa_1,\theta_1}(g^{-1}\gm g).
\]

On the other hand, it follows from equalities of \rl{identities} that the natural map $G(F)/\wt{G}_{\fa}\to G_1(F)/(\wt{G_1})_{\fa_1}$ is a bijection, so (by Section~\re{Step2}) it remains to show that function $t_{\fa,\theta}:G(F)\to\qlbar$ is the restriction of $t_{\fa_1,\theta_1}:G_1(F)\to\qlbar$.

\smallskip

Using the second equality of \rl{identities} it suffices to show that the restriction $R^{\theta_1}_{\fa_1}|_{\wt{G}_{\fa}}$ is isomorphic to $R^{\theta}_{\fa}$. By definition, the restriction $R^{\theta_1}_{\fa_1}|_{Z(G)(F)}$ is $\theta_1|_{Z(G)(F)}=\theta|_{Z(G)(F)}$, while the restriction $R^{\theta_1}_{\fa_1}|_{G_{\fa}}$ is the pullback of the inflation of virtual representation $R^{\ov{\theta}_1}_{\ov{\fa}_1}$
of $M_{\fa_1}(\fq)$ with respect to the inclusion $G_{\fa}\hra (G_1)_{\fa_1}$. Alternatively, $R^{\theta_1}_{\fa_1}|_{G_{\fa}}$ is isomorphic to the inflation of the restriction $R^{\ov{\theta}_1}_{\ov{\fa}_1}|_{M_{\fa}(\fq)}$, corresponding to the injective quasi-isogeny $M_{\fa}\hra M_{\fa_1}$. Since  $R^{\ov{\theta}_1}_{\ov{\fa}_1}|_{M_{\fa}(\fq)}\simeq R^{\ov{\theta}}_{\ov{\fa}}$ by \rl{funcdl}, the restriction
$R^{\theta_1}_{\fa_1}|_{G_{\fa}}$ is thus isomorphic to the inflation of $R^{\ov{\theta}}_{\ov{\fa}}$, as claimed.
\end{Emp}
\end{proof}

Now we are ready to prove \rt{stab}.

\begin{Emp}
\begin{proof}[Proof of \rt{stab}]  To expose the structure of the proof we divide it into steps.

\smallskip

\noindent{\bf Step 1.} It suffices to show that the restriction of  $f^{\ka}_{T,\theta}$ to $G^{\rss}(F)_c$ of compact  regular semisimple elements in $G(F)$ is $\C{E}_{T,\ka}$-stable.

\begin{proof}
Note that for every $\gm\in G^{\rss}(F)$ and $z\in Z(G)(F)$, we have an equality $f_{T,\theta}(z\gm)=\theta(z)\cdot f_{T,\theta}(\gm)$, hence
\begin{equation} \label{Eq:center}
f^{\ka}_{T,\theta}(z\gm)=\theta(z)\cdot f^{\ka}_{T,\theta}(\gm).
\end{equation}
 Moreover, we have $G^0_{z\gm}=G^0_{\gm}$, and if $\gm'$ is stably conjugate to $\gm$, then $z\gm'$ is stably conjugate to $z\gm$ and we have an equality $\inv(z\gm',z\gm)=\inv(\gm',\gm)$. Now the assertion follows from the fact that each $f_{\fa,\theta}$ is supported on $\wt{G}_{\fa}=G_{\fa}\cdot Z(G)(F)$.
 \end{proof}

\noindent{\bf Step 2.} We can assume that $Z(G)$ is connected.

\begin{proof}
Consider the quasi-isogeny $\pi:G\hra G_1$ from Section~\re{funcdl}. Since $Z(G_1)\simeq T$ is connected, it suffices to show that the assertion for $G$ follows from that for $G_1$. Let $T_1,\theta_1$ and $\ka_1$ be as in Section~\re{funcdl}.
Then $\wh{\pi}(\C{E}_{T_1,\ka_1})$ is conjugate to $\C{E}_{T,\ka}$ (see Section~\re{quasiisogeny}(e)), and
$\pi^*(f^{\ka_1}_{T_1,\theta_1})=f^{\ka}_{T,\theta}$ by \rp{restr}. Since $f^{\ka_1}_{T_1,\theta_1}$ is $\C{E}_{T_1,\ka_1}$-stable by assumption, function $f^{\ka}_{T,\theta}$ is $\C{E}_{T,\ka}$-stable by \rl{quasi-isogeny}.
\end{proof}

\noindent{\bf Step 3.} By \rt{equality}, the restriction $f_{T,\theta}^{\ka}|_{G^{\rss}(F)_c}$ can be written in the form $f^{\ka}_{T,\theta}|_{G^{\rss}(F)_c}=\sum_{i,j}(-1)^{i+j}f_{i,j}^{\ka}$, where
$f_{i,j}^{\ka}$ is the function
\[
\gm\mapsto \Tr(\ov{u}_{T,\varphi},H_j(\La,H_i(\Fl_{\gm},\C{F}_{\theta,\varphi})_{\ka}).
\]
Thus, by Step 1 it suffices to show that each function $f_{i,j}^{\ka}$ is $\C{E}_{T,\ka}$-stable.
Moreover, by Step 2 we can assume that Z(G) is connected. Hence, using Section~\re{connected}(b) it suffices to show that for every $\gm\in G^{\rss}(F)_c$ the restriction $f^{\ka}_{i,j}|_{\on{Orb}^{\st}_{G(F)}(\gm)}$ lies in the span of $\{\ev^{\xi}_{\gm}\}_{\xi\in \on{can}_{\gm}^{-1}(\ka)\cap \wh{G^0_{\gm}}^{\Gm_F}}$.

\smallskip

\noindent{\bf Step 4.} By \rcl{support}, the coherent sheaf on $\wh{G^0_{\gm}}^{\Gm_K}$ corresponding to the
representation $V:=H_j(\La_G, H_i(\Fl_{G,\gm},\C{F}_{\C{L}})_{\ka})$ of $\pi_0(LG^0_{\gm})$ is supported on
a finite set $\on{can}_{\gm}^{-1}(\nu_{\wh{G}}(\ka))\cap \wh{G^0_{\gm}}^{\Gm_K}$.
Therefore vector space $V$ decomposes as a direct sum $V=\bigoplus_{\xi\in\on{can}_{\gm}^{-1}(\ka)\cap \wh{G^0_{\gm}}^{\Gm_K}}V_{\xi}$,
where the coherent sheaf on $\wh{G^0_{\gm}}^{\Gm_K}$, corresponding to the
representation $V_{\xi}$, is supported at $\xi$.

\smallskip

\noindent{\bf Step 5.} As it was mentioned in Section~\re{obsstab}(b), vector space $V$ is naturally a representation of the semi-direct product
$\pi_0(LG^0_{\gm})\ltimes\lan\ov{u}_{T,\varphi}\ran$, where $\ov{u}_{T,\varphi}\cdot\la\cdot\ov{u}_{T,\varphi}^{-1}={}^{\si}\la$ for every $\la\in \pi_0(LG^0_{\gm})$. Combining this with observations of Section~\re{genrat}(a)-(c) one gets that for every $\xi\in\on{can}_{\gm}^{-1}(\ka)\cap \wh{G^0_{\gm}}^{\Gm_K}$, we have an inclusion
$\ov{u}_{T,\varphi}(V_{\xi})\subseteq (V_{{}^{\si}\xi})$. In particular, for every $\xi\in\on{can}_{\gm}^{-1}(\ka)\cap \wh{G^0_{\gm}}^{\Gm_F}$, the subspace $V_{\xi}\subseteq V$ is $\ov{u}_{T,\varphi}$-invariant.

\smallskip

\noindent{\bf Step 6.} We claim that we have an equality
\begin{equation*} %\label{Eq:stable}
f^{\ka}_{i,j}|_{\on{Orb}^{\st}_{G(F)}(\gm)}=\sum_{\xi\in \on{can}_{\gm}^{-1}(\nu_{\wh{G}}(\ka))\cap \wh{G^0_{\gm}}^{\Gm_F}}\Tr(\ov{u}_{T,\varphi},V_{\xi})\ev^{\xi}_{\gm},
\end{equation*}
which by Step 3 implies our assertion. Explicitly, we have to show that for every stably conjugate $\gm'$ of $\gm$, we have an equality
\begin{equation} \label{Eq:stable}
\Tr(\ov{u}_{T,\varphi},H_j(\La_G, H_i(\Fl_{G,\gm'},\C{F}_{\C{L}})_{\ka})=
\sum_{\xi\in \on{can}_{\gm}^{-1}(\nu_{\wh{G}}(\ka))\cap \wh{G^0_{\gm}}^{\Gm_F}}\lan \xi,\inv(\gm,\gm')\ran\Tr(\ov{u}_{T,\varphi},V_{\xi}).
\end{equation}
\smallskip

\noindent{\bf Step 7.} Since $\gm'$ is a stable conjugate of $\gm$, there exists $g\in G(F^{\nr})$ such that $g\gm g^{-1}=\gm'$
and $h:=g^{-1}{}^{\si}g\in G^0_{\gm}(F^{\nr})$ (see Section~\re{stableconj}(a),(b)).
As it was mentioned in Section~\re{obsstab}(c), we have natural isomorphisms
\[
l_g,l_{{}^{\si}g}:H_j(\La_G, H_i(\Fl_{G,\gm},\C{F}_{\C{L}})_{\ka})\isom H_j(\La_G, H_i(\Fl_{G,\gm'},\C{F}_{\C{L}})_{\ka}),
\]
which are related by the identity $\ov{u}_{T,\varphi}\circ l_g \circ\ov{u}_{T,\varphi}^{-1}=l_{{}^{\si}g}$. Therefore we have
\[
(l_g)^{-1}\circ \ov{u}_{T,\varphi}\circ l_g=l_{g^{-1}}\circ l_{{}^{\si}g}\circ \ov{u}_{T,\varphi}=l_h\circ \ov{u}_{T,\varphi},
\]
hence the left hand side of equation \form{stable} equals
\[
\Tr((l_g)^{-1}\circ \ov{u}_{T,\varphi}\circ l_g,H_j(\La_G, H_i(\Fl_{G,\gm},\C{F}_{\C{L}})_{\ka}))=\Tr(l_{h}\circ \ov{u}_{T,\varphi},V).
\]

Furthermore, as in Section~\re{setup comp}(a), the last expression equals $\Tr([h]\circ \ov{u}_{T,\varphi},V)$, where $[h]\in\pi_0(LG^0_{\gm})$ is the class of $h\in LG^0_{\gm}$.

\smallskip

To show equality \form{stable}, it suffices to show that
\begin{equation} \label{Eq:stable1}
 \Tr([h]\circ \ov{u}_{T,\varphi},V)=\bigoplus_{\xi\in \on{can}_{\gm}^{-1}(\nu_{\wh{G}}(\ka))\cap \wh{G^0_{\gm}}^{\Gm_F}}\Tr([h]\circ \ov{u}_{T,\varphi},V_{\xi})
\end{equation}
and that for every $\xi\in \on{can}_{\gm}^{-1}(\nu_{\wh{G}}(\ka))\cap \wh{G^0_{\gm}}^{\Gm_F}$ we have
\begin{equation} \label{Eq:stable2}
 \Tr([h]\circ \ov{u}_{T,\varphi},V_{\xi})=\lan\xi,\inv(\gm,\gm')\ran\cdot\Tr(\ov{u}_{T,\varphi},V_{\xi}).
\end{equation}

\smallskip

\noindent{\bf Step 8.} To show equality \form{stable1}, notice that
\[
\Tr\left([h]\circ \ov{u}_{T,\varphi},\bigoplus_{\xi\in \on{can}_{\gm}^{-1}(\nu_{\wh{G}}(\ka))\cap (\wh{G^0_{\gm}}^{\Gm_K}\sm \wh{G^0_{\gm}}^{\Gm_F})}V_{\xi}\right)=0,
\]
which follows from the fact that endomorphism $[h]\circ \ov{u}_{T,\varphi}$ maps each $V_{\xi}$ to $V_{{}^{\si}\xi}$ (see Step 5), so it
permutes factors without fixed points.

\smallskip

Next notice that $\xi\in \on{can}_{\gm}^{-1}(\nu_{\wh{G}}(\ka))\cap \wh{G^0_{\gm}}^{\Gm_F}$ operator $[h]$ acts on the semi-simplication  $(V_{\xi})^{\on{ss}}$ of $V_{\xi}$ by the scalar $\lan\xi,[h]\ran$, where  we identify $[h]$ with its image under isomorphism $\pi_0(LG^0_{\gm})\simeq X_*(G^0_{\gm})_{\Gm_K}$ from Section~\re{pi0}, and the pairing is as in Section~\re{dualgp}(c).
So  we have an equality
\[
\Tr([h]\circ \ov{u}_{T,\varphi},V_{\xi})=\lan\xi,[h]\ran\cdot\Tr(\ov{u}_{T,\varphi},V_{\xi}).
\]
Since $\inv(\gm,\gm')\in H^1(F,G^0_{\gm})\simeq X_*(G^0_{\gm})_{\Gm_F}$
is the class of $[h]\in  X_*(G^0_{\gm})_{\Gm_F}$ (see Section~\re{cohom}(c)), we have an equality
$\lan\xi,[h]\ran=\lan\xi,\inv(\gm,\gm')\ran$, from which equality \form{stable2} follows.
\end{proof}
\end{Emp}

\subsection{Proof of \rt{equiv}}

\begin{Emp} \label{E:gmaction}
{\bf Observations.} Let $\C{L}$ be a $\ov{W}$-equivariant local system on $\ov{T}_G$.

\smallskip

(a) Using \rp{whaction}(d), the natural (left) action of $LG^0_{\gm}$ on $\Fl_G$ induces an action on $H_i(\Fl_{G,\gm},\C{F}_{\C{L}})$, commuting with the $\wt{W}_G$-action.
In particular, we have an action of  $LG^0_{\gm}$ on each $H_j(\wt{W}_G, H_i(\Fl_{G,\gm},\C{F}_{\C{L}}))$.

\smallskip

(b) Moreover, applying \rcl{support} with $\ka=1$ we conclude that the action of $LG^0_{\gm}$ on each
$H_j(\wt{W}_G,H_i(\Fl_{G,\gm},\C{F}_{\C{L}}))$ is unipotent.

\smallskip

(c) By Section~\re{inaut}, we have
the adjoint action of $LG^{\ad}_{\gm}=L(G^{\ad}_{\gm})$ on each
$H_j(\wt{W}_G, H_i(\Fl_{G,\gm},\C{F}_{\C{L}}))$.
\end{Emp}

\begin{Lem} \label{L:adunip}
The adjoint action of $L((G^{\ad}_{\gm})^0)$ on
$H_j(\wt{W}_G,H_i(\Fl_{G,\gm},\C{F}_{\C{L}}))$ is unipotent.
\end{Lem}

\begin{proof}
We will carry out the proof in three steps:

\smallskip
\noindent{\bf Step 1}. We can assume that $\C{L}$ is a constant $\ov{W}$-equivariant local system $V$.

\begin{proof}
By adjointness, a $\ov{W}$-equivariant local system $\C{L}$ has a natural embedding
\[
\C{L}\hra\ind_1^{\ov{W}}(\C{L})=\bigoplus_{\ov{w}\in \ov{W}}\ov{w}^*(\C{L}).
\]
Thus the assertion for $\C{L}$ follows from that for $\ind_1^{\ov{W}}(\C{L})$, so we can assume that $\C{L}=\ind_1^{\ov{W}}(\C{L}')$ for some local
system $\C{L}'$ on $T_G$.

\smallskip

Next, the pullback $\C{F}_{\C{L}}=\red_{\gm}^*(\C{L})$ only depends on the restriction of $\C{L}$ to the  image $\im\red_{\gm}$. Since the image
$\im\red_{\gm}\subseteq \ov{T}_G$ is a $\ov{W}$-orbit, it is finite. Then
$\C{L}'|_{\im\red_{\gm}}$ is a constant local system with fiber $V'$, hence $\C{L}|_{\im\red_{\gm}}$ is a constant local system with fiber
$V=\ind_1^{\ov{W}}(V')$. Thus $\C{F}_{\C{L}}\simeq \red_{\gm}^*(V)$, and the assertion follows.
\end{proof}

\smallskip

\noindent{\bf Step 2}. We can assume that $G=G^{\ad}$.
\begin{proof}
Using \rl{ind}, one sees that the projection $\pi:G\to G^{\ad}$
induces a natural $LG^{\ad}_{\gm}$-equivariant isomorphism
\[
H_j(\wt{W},H_i(\Fl_{\gm},\C{F}_V))\isom H_j(\wt{W}_G,H_i(\Fl_{G,\gm},\C{F}_V))\isom H_j(\wt{W}_{G^{\ad}},H_i(\Fl_{G^{\ad},\pi(\gm)},\C{F}_V)).
\]
Thus we can replace $(G,\gm)$ by $(G^{\ad},\pi(\gm))$.
\end{proof}

\smallskip

\noindent{\bf Step 3}. Assume now that $G=G^{\ad}$. Then for every $g\in LG^0_{\gm}$, the action of $\Ad_g$ on $\Fl_{G,\gm}$ decomposes as $\Ad_g=l_g\circ r_g$, where $l_g$ is the left action (induced by the endomorphism $xI\mapsto gxI$ of $\Fl_G$),
while $r_g$ is the right action
\[
\Fl_G=LG/I\to LG/gIg^{-1}\simeq\Fl_G:xI\mapsto xIg^{-1}=xg^{-1}(gIg^{-1}).
\]
In particular, $r_g$ is the automorphism of $Fl_{G,\gm}$ induced by $\om_g\in\Om_G\subseteq\wt{W}_G$ (compare Section~\re{inaut}), thus $r_g$ induces the trivial action on
$H_j(\wt{W}_G, H_i(\Fl_{G,\gm},\C{F}_{\C{L}}))$.

Therefore the action of $\Ad_g=l_g\circ r_g$ on $H_j(\wt{W}_G, H_i(\Fl_{G,\gm},\C{F}_{\C{L}}))$ coincide with the action of $l_g$, hence is it unipotent by Section~\re{gmaction}(b).
\end{proof}

Now we are ready to show \rt{equiv}.

%\begin{Emp} \label{E:iota}
%(a) Let $\gm,\gm'\in G(F)$ be two stably conjugate regular
%semisimple elements, and let $g\in G(F^{\nr})$ be such $g\gm
%g^{-1}=\gm'$. Then the action $act_g:\Fl\isom \Fl$ of $g$ of $\Fl$ it
%induces an isomorphism $\Fl_{\gm}\isom\Fl_{\gm'}$. Since
%$red_{\gm'}=red_{\gm}\circ act_g:\Fl_{\gm'}\to\ov{T}_G$, is
%therefore induces an isomorphism in homology
%$\eta_g=H_*(\Fl_{\gm},\C{L}^{\st}_{\theta})\isom
%H_*(\Fl_{\gm'},\C{L}^{\st}_{\theta})$.

%(b) The isomorphisms $\eta_g$ satisfy
%$\eta_{g^{-1}}=(\eta_g)^{-1}$ and
%$\eta_{g_1g_2}=\iota_{g_1}\iota_{g_2}$. In particular, the map
%$g\mapsto \iota_g$ defines an action of $G(F^{\nr})$ on each
%$H_i(\Fl_{\gm},\C{L}^{\st}_{\theta})$.
%\end{Emp}

\begin{Emp}
\begin{proof}[Proof of \rt{equiv}]
We have to show that for every stably conjugate elements $\gm\in G^{\rss}(F)$ and $\gm'\in G'^{\rss}(F)$ we have an equality $f^{\st}_{T,\theta}(\gm)=f^{\st}_{T',\theta'}(\gm')$. Arguing as in Step 1 of the proof of \rt{stab},
we reduce to the case when $\gm$ and $\gm'$ are compact.

\smallskip

By \rt{equality}, it suffices to show that for every $i$ and $j$, we have the equality
\begin{equation} \label{Eq:frob}
\Tr(\si,
H_j(\wt{W}_G,H_i(\Fl_{G,\gm},\C{F}^{\st}_{\theta})))=\Tr(\si,
H_j(\wt{W}_{G'},H_i(\Fl_{G',\gm'},\C{F}^{\st}_{\theta'}))).
\end{equation}

Let $g\in G(F^{\nr})$ be as in Section~\re{inner}(b). Replacing $\phi$ by $\phi\circ\Ad_g$, we may assume that
$\phi(\gm)=\gm'$ and there exists $h\in (G^{\ad}_{\gm})^0(F^{\nr})\subseteq L((G^{\ad}_{\gm})^0)(k)$ such that
$\phi^{-1}\circ {}^{\si}\phi=\Ad_h$. Using the equality ${}^{\si}\phi=\si\circ\phi\circ\si^{-1}$, we conclude that
$\Ad_h\circ \si=\phi^{-1}\circ\si\circ\phi$.

\smallskip

Let $\phi_*:\ov{T}_G\isom\ov{T}_{G'}$ be the isomorphism induced by $\phi$ (see Section~\re{functoriality}). By definition, for every admissible isomorphism $\varphi:\ov{T}\isom \ov{T}_G$, the composition $\phi_*\circ\varphi$ is an admissible isomorphism $\ov{T}\isom \ov{T}_{G'}$. Therefore we get an isomorphism  $\phi_*(\C{L}^{\st}_{\theta})\simeq\C{L}^{\st}_{\theta'}$, hence $\phi$ induces an isomorphism
\[
H_j(\wt{W}_G,H_i(\Fl_{G,\gm},\C{F}^{\st}_{\theta})))\isom H_j(\wt{W}_{G'},H_i(\Fl_{G',\gm'},\C{F}^{\st}_{\theta'})).
\]

Thus the right hand side of \form{frob} equals
\[
\Tr(\phi^{-1}\circ\si\circ\phi,
H_j(\wt{W}_G,H_i(\Fl_{G,\gm},\C{F}^{\st}_{\theta})))= \Tr(\Ad_h\circ \si, H_j(\wt{W}_G,H_i(\Fl_{G,\gm},\C{F}^{\st}_{\theta}))),
\]
therefore it remains to show the equality
\begin{equation} \label{Eq:stab}
\Tr(\Ad_h\circ \si,
H_j(\wt{W}_G,H_i(\Fl_{G,\gm},\C{F}^{\st}_{\theta})))=\Tr(\si,
H_j(\wt{W}_G,H_i(\Fl_{G,\gm},\C{F}^{\st}_{\theta}))).
\end{equation}

\smallskip

By \rl{adunip}, there exists a finite increasing filtration $F^r$ of $H_j(\wt{W}_G,H_i(\Fl_{G,\gm},\C{F}^{\st}_{\theta})))$ such that $L((G^{\ad}_{\gm})^0)$ acts trivially on each graded piece $F^{r+1}/F^r$. Moreover, since for every $g\in L((G^{\ad}_{\gm})^0)$ we have
$\si\circ\Ad_g\circ\si^{-1}=\Ad_{\si(g)}$, the filtration $F^r$ can be chosen to be $\si$-invariant. In this case, both sides of \form{stab} equal
$\sum_r\Tr(\si,F^{r+1}/F^r)$.
\end{proof}
\end{Emp}

\appendix
\section{Generalized trace}

\subsection{Construction} \label{S:gentr}
 Let $L$ be an algebraically closed field of characteristic zero, let $A$ be a finitely generated commutative algebra over $L$, and let $\psi\in \Aut(A)$ be an automorphism of $L$-algebra $A$ of order $n$.

\begin{Emp} \label{E:twisted}
{\bf Notation.}
(a) We denote by $B:=A\{u^{\pm 1};\psi\}$ the ring generated by $A$ and elements $u,u^{-1}$ subject to
relations $u\cdot a=\psi(a)\cdot u$ for all $a\in A$. Explicitly, $B$ is a set of finite sums $\sum_{i\in \B{Z}} a_i u^i$ with $a_i\in A$ and the product is given by the formula
\[
(\sum_{i} a_i u^i)\cdot(\sum_j b_j u^j):=\sum_{i,j} (a_i\cdot\psi^i(b_j)) u^{i+j}.
\]

\smallskip

(b) Notice that since $\psi^n=1$, we conclude that elements $u^n,u^{-n}\in B$ are central.

\smallskip

(c) We claim that $B$ is a Noetherian ring. Indeed, this follows from the fact that
$B$ is finite over its commutative subring $A[u^n,u^{-n}]$, and that $A[u^n,u^{-n}]$ is a finitely generated commutative algebra over $L$, thus Noetherian.

\smallskip

(d)  Let $\lan u\ran$ and $\cycl_n$  be the cyclic groups generated by
$u$ of infinite order and of order $n$, respectively, let $\lan\psi\ran$ be the (order $n$) cyclic group, generated by $\psi$,
and for every commutative ring $C$ we denote by   $C[\lan u\ran]$ and $C[\cycl_n]$ the corresponding
group algebras.
\smallskip

(e) We denote by $B_n:=A\{u;\psi\}_n$ the quotient ring $B/(u^n-1)$. Explicitly, $B_n$ is the ring generated by $A$ and element $u$ subject to
relations $u^n=1$ and $u\cdot a=\psi(a)\cdot u$ for all $a\in A$.

\smallskip

(f) Let $A_{\psi}$ be the quotient ring $A/\lan a-\psi(a),a\in A\ran$. Then we the quotient map $A\to A_{\psi}$ which induces quotient maps $B\to A_{\psi}[\cycl]$ and $B_n\to A_{\psi}[\cycl_n]$. Note that $\Spec A_{\psi}$ is the scheme of fixed points $(\Spec A)^{\psi}$. Therefore algebra $A_{\psi}$ is finite-dimensional over $L$ if and only if the scheme  $(\Spec A)^{\psi}$ is finite.
\end{Emp}

\begin{Emp} \label{E:kgroups}
{\bf $K$-groups.} (a) For every $L$-algebra $R$, we denote by $\Mod_R$ the category of $R$-modules, by
$\Mod^{\on{fd}}_R\subseteq\Mod^{\on{fg}}_R\subseteq\Mod_R$ the subcategories of finitely generated $R$-modules and of $R$-modules, which are finite dimensional over $L$, and denote by $K(R)$ and $K(R,\on{fd})$ the $K$-groups of $\Mod^{\on{fg}}_R$ and $\Mod^{\on{fd}}_R$, respectively.

\smallskip

(b) We denote by $\Mod^{\on{fg},u-\on{lf}}_B\subseteq\Mod^{\on{fg}}_B$ the subcategories of modules $M$, which are {\em $u$-locally finite}, that is, unions of $u$-invariant finite dimensional $L$-vector spaces.

\smallskip

(c) Note that  $\Mod^{\on{fd}}_{L[\lan u\ran]}$ is a tensor category with respect to the tensor product $-\otimes_L-$ and that subcategory
 $\Mod^{\on{fg}}_{L[\lan u\ran_n]}=\Mod^{\on{fd}}_{L[\lan u\ran_n]}\subseteq \Mod^{\on{fd}}_{L[\lan u\ran]}$ is a tensor subcategory.

 \smallskip
(d) Note that $\Mod^{\on{fg},u-\on{lf}}_B$ is a module category over $\Mod^{\on{fd}}_{L[\lan u\ran]}$. Namely, to a pair $M\in \Mod^{\on{fg},u-\on{lf}}_B$ and $V\in \Mod^{\on{fd}}_{L[\lan u\ran]}$ one associates a tensor product
$V\otimes_L M\in \Mod^{\on{fg},u-\on{lf}}_B$ such that $a(v\otimes m)=v\otimes a(m)$ and $u(v\otimes m)=u(v)\otimes u(m)$
for every $v\in V, m\in M$ and $a\in A$. Furthermore, $\Mod^{\on{fg}}_{A_{\psi}[\cycl_n]}\subseteq\Mod^{\on{fg}}_{B_n}\subseteq \Mod^{\on{fg},u-\on{lf}}_B$ are  $\Mod^{\on{fg}}_{L[\lan u\ran_n]}$-module subcategories.

\smallskip

(e) We denote by $K(B,u-\on{lf})$ and $K(\cycl,\on{fd})$ the $K$-groups of $\Mod^{\on{fg},u-\on{lf}}_B$ and $\Mod^{\on{fd}}_{L[\lan u\ran]}$, respectively. Moreover, it follows from part~(c) that  $K(\cycl,\on{fd})$ and $K(\cycl_n)$ are commutative rings and
that $K(B,u-\on{lf})$ and $K(B_n)$ are modules over  $K(\cycl,\on{fd})$ and $K(\cycl_n)$, respectively.

\smallskip

(f) Observe that the inclusion $\Mod^{\on{fg}}_{L[\lan u\ran_n]}\hra\Mod^{\on{fd}}_{L[\lan u\ran]}$ induces a ring homomorphism $K(\cycl_n)\to K(\cycl,u-\on{lf})$, while the inclusions
\[
\Mod^{\on{fg}}_{A_{\psi}[\lan u\ran_n]}\hra\Mod^{\on{fg}}_{B_n}\hra \Mod^{\on{fg},u-\on{lf}}_B
\]
induce morphisms
$K(A_{\psi}[\lan u\ran_n])\to K(B_n)\to K(B,u-\on{lf})$ of $K(\cycl_n)$-modules, hence
%
%while the tensor product
%\[
%\otimes: \Mod^{\on{fd}}_{L[\lan u\ran]}\times \Mod^{\on{fg}}_{B_n}\to \Mod^{\on{fg},u-\on{lf}}_B
%\] induces
a morphism
\begin{equation} \label{Eq:kgroups}
K(\cycl,\on{fd})\otimes_{K(\cycl_n)}K(B_n)\to K(B,u-\on{lf})
\end{equation}
of $K(\cycl,\on{fd})$-modules.

\smallskip

(g) Notice that $K(B_n)$ and $K(A_{\psi}[\cycl_n])$ can be interpreted as equivariant $K$-theories $K'_0(\lan\psi\ran,\Spec A)$ and
$K'_0(\lan\psi\ran,(\Spec A)^{\psi})$ of $\lan\psi\ran$-equivariant coherent sheaves on $\Spec A$ and $(\Spec A)^{\psi}=\Spec A_{\psi}$, respectively.
\end{Emp}

\begin{Lem} \label{L:ktheory}
The morphism \form{kgroups} is an isomorphism.
\end{Lem}

\begin{proof}
For every element $a\in L^{\times}$, we denote by  $\Mod^{\on{fg},u-\on{lf},a}_B\subseteq \Mod^{\on{fg},u-\on{lf}}_B$ and $\Mod^{\on{fd},a}_{L[\cycl]}\subseteq \Mod^{\on{fd}}_{L[\cycl]}$ the full subcategory consisting of modules $M$ such that $a$ is the unique generalized eigenvalue of $u^n$ on $M$,
and let $K(B;a)$ and $K(\cycl;a)$ be the $K$-groups of $\Mod^{\on{fg},u-\on{lf},a}_B$ and $\Mod^{\on{fd},a}_{L[\cycl]}$, respectively.

Note that since $u^n\in B$ is central, every $M\in\Mod^{\on{fg},u-\on{lf}}_B$ uniquely decomposes as a finite direct sum $M=\bigoplus_{a\in L\m}M_a$, where
$M_a\in\Mod^{\on{fg},u-\on{lf},a}_B$. Therefore the $K$-group $K(B,u-\on{lf})$ decomposes as $K(B,u-\on{lf})\simeq \bigoplus_{a\in L\m}K(B;a)$. Similarly,
we have a decomposition $K(\cycl,\on{fd})\simeq \bigoplus_{a\in L\m}K(\cycl;a)$, and morphism \form{kgroups} decomposes as a direct sum of morphisms
\begin{equation} \label{Eq:kgroupsa}
K(\cycl;a)\otimes_{K(\cycl_n)}K(B_n)\to K(B;a).
\end{equation}
%
%induced by the tensor product  $\otimes: \Mod^{\on{fg},u-\on{lf},a}_{L[\lan u\ran]}\times \Mod^{\on{fg}}_{B_n}\to \Mod^{\on{fg},u-\on{lf},a}_B$.

\smallskip

Choose $b\in L$ such that $b^n=a$, and let $V_b\in\Mod^{\on{fd}}_{L[\cycl]}$ be the one-dimensional module over $L$ and such that $u(v)=bv$ for every $v\in V_b$.
Then the functor $-\otimes_L V_b$ induces an equivalence of categories $\Mod^{\on{fg},u-\on{lf},1}_{B}\isom\Mod^{\on{fg},u-\on{lf},b^n}_B$, and thus an isomorphism of $K$-groups $K(B;1)\isom K(B;b^n)=K(B;a)$ and similarly an isomorphism $K(\cycl;1)\isom K(\cycl;b^n)=K(\cycl;a)$. Thus it suffices to show that morphism \form{kgroupsa} is an isomorphism for $a=1$.

The latter assertion follows from the fact that the natural inclusion $\Mod^{\on{fg}}_{B_n}\hra\Mod^{\on{fg},u-\on{lf},1}_B$ induces an isomorphism of $K$-groups
$K(B_n)\isom K(B;1)$ and hence also an isomorphism $K(\cycl_n)\isom K(\cycl;1)$.

To see that the map between $K$-groups $K(B_n)\isom K(B;1)$ is an isomorphism, we notice that every $M\in \Mod^{\on{fg},u-\on{lf},1}_B$ has a canonical  filtration
\[
0=M_0\subsetneq M_1\subsetneq M_2\subsetneq\ldots\subsetneq M_n=M,
\]
where $M_i=\Ker(u^n-1)^i$. Since $B$ is Noetherian, we have
$M_{i+1}/M_i\in \Mod^{\on{fg}}_{B_n}$ for every $i$, and the inverse map $K(B;1)\to K(B_n)$ between $K$-groups is given by the formula $[M]\mapsto\sum_i[M_{i+1}/M_i]$.
\end{proof}

\begin{Emp}
{\bf Example.} We have a ring isomorphism $K(\cycl,\on{fd})\simeq\B{Z}[L\m]$, where $\B{Z}[L\m]$ denotes the group ring of $L\m$, such that
$b\in L\m\subseteq \B{Z}[L\m]$ is mapped to the class of character $L[\cycl]\to L:u\mapsto b$, and it induces an isomorphism
$K(\cycl_n)\simeq\B{Z}[\mu_n]$. Moreover, for every $a\in L\m$ the subgroup $K(\cycl;a)\subseteq K(\cycl,\on{fd})$ corresponds
to the $\B{Z}$-span of $b\in L\m$ such that $b^n=a$.
\end{Emp}

\begin{Cor} \label{C:ktheory}
Every group homomorphism  $\Theta:K(A_{\psi}[\cycl_n])\to L$ satisfying
\begin{equation} \label{Eq:local}
\Theta(V\otimes M)=\Tr(u,V)\cdot \Theta(M)
\end{equation}
 for every $M\in \Mod^{\on{fg}}_{A_{\psi}[\cycl_n]}$ and $V\in\Mod^{\on{fg}}_{L[\cycl_n]}$ uniquely extends to a homomorphism
$\Theta:K(B,u-\on{lf})\to L$ satisfying the equality \form{local}
for every $M\in \Mod^{\on{fg},u-\on{lf}}_B$ and $V\in\Mod^{\on{fd}}_{L[\cycl]}$.
\end{Cor}
\begin{proof}
Note that a homomorphism $\Theta:K(B,u-\on{lf})\to L$ satisfies equality \form{local} if and only if it factors through a localization $K(B,u-\on{lf})_u$, where $u$ is the prime ideal of $K(\cycl,\on{fd})$, defined to be the kernel of the homomorphism $\Tr(u,-)$.
By \rl{ktheory}, the natural homomorphism $K(B_n)_u\to K(B,u-\on{lf})_u$ is an isomorphism.

Moreover, the localization theorem (see, for example, \cite[Th\'eor\'eme 2.1]{Th}) together with the observation of Section~\re{kgroups}(g) implies that the natural morphism $K(A_{\psi}[\cycl_n])_u\to K(B_n)_u$ is an isomorphism. Therefore the composition  $K(A_{\psi}[\cycl_n])_u\to  K(B,u-\on{lf})_u$ is an isomorphism, and the assertion follows from the fact that
 a homomorphism $\Theta:K(A_{\psi}[\cycl_n])\to L$ satisfies equality \form{local} if and only if it factors through a localization $K(A_{\psi}[\cycl_n])_u$.
\end{proof}

\begin{Cor} \label{C:gentr}
Assume that the scheme $(\Spec A)^{\psi}$ is finite. Then
there exists a unique homomorphism $\Tr_{\gen}(u,-):K(B,u-\on{lf})\to L$ such that

\smallskip

(i) for every $M\in \Mod^{\on{fg},u-\on{lf}}_B$ and $V\in\Mod^{\on{fg},u-\on{lf}}_{L[\cycl]}$, we have an equality
\[
\Tr_{\gen}(u,V\otimes M)=\Tr(u,V)\cdot \Tr_{\gen}(u,M);
\]

\smallskip

(ii) for every $M\in \Mod^{\on{fd}}_B$, we have an equality
$\Tr_{\gen}(u,M)=\Tr(u,M)$.

\end{Cor}

\begin{proof}
Since $(\Spec A)^{\psi}$ is finite, the algebra $A_{\psi}$ is finite-dimensional over $L$. Therefore every
$M\in \Mod^{\on{fg}}_{A_{\psi}[\cycl_n]}$ belongs to $\Mod^{\on{fd}}_{A_{\psi}[\cycl_n]}$. Moreover, the usual trace map
$\Tr(u,-):K(A_{\psi}[\cycl_n])\to L$ satisfies equality \form{local}.

Therefore it follows from \rco{ktheory} that
there exists a unique homomorphism $\Tr_{\gen}(u,-):K(B,u-\on{lf})\to L$, which satisfies condition (i) and condition (ii) for all
$M\in\Mod^{\on{fd}}_{A_{\psi}[\cycl_n]}$. It remains to show that homomorphism $\Tr_{\gen}(u,-)$  satisfies $\Tr_{\gen}(u,M)=\Tr(u,M)$ for every
$M\in\Mod^{\on{fd}}_{B}$.

Let $I:=\on{Ann}_A(M)$ be the annihilator of $M$. Since $M\in \Mod^{\on{fd}}_{B}$, we conclude that the quotient ring $A':=A/I$ is finite-dimensional over $L$. Also since $M$ is a $B$-module, we conclude that the ideal $I\subseteq A$ is $\psi$-invariant, and $M$ is induced from a $A'\{u^{\pm 1},\psi\}$-module.

It suffices to show that the composition
\[
\Tr_{\gen}(u,-)\circ \pi^*:K(A'\{u^{\pm 1},\psi\},u-\on{lf})\to K(A\{u^{\pm 1},\psi\},u-\on{lf})\to L,
\]
where $\pi:A\to A'$ is the projection, coincides with $\Tr(u,-)$.

Using \rco{ktheory} for $A'$, we have to show that these two maps have the same restrictions to
$K(A'_{\psi}[\cycl_n])$. Since $A'_{\psi}$ is a quotient of $A_{\psi}$, this follows from the definition
of $\Tr_{\gen}(u,-)$.
\end{proof}

\begin{Emp} \label{E:remgentr}
{\bf Remark.}
Note that for every element $\la\in A^{\times}$ there is an unique algebra automorphism $\phi_{\lambda}:B\isom B$ such that $\phi_{\lambda}(a)=a$ for every $a\in A$ and $\phi_{\lambda}(u)=\la u$. Therefore by \rco{gentr} one can form a generalized trace map
\[
\Tr_{\gen}(\la u,-):K(B,\la u-\on{lf})\to L,
\]
which satisfies properties (i) and (ii) of \rco{gentr} in which $u$ is replaced by $\la u$. Note also that in general the subcategory
$\Mod^{\on{fg},\la u-\on{lf}}_{B}\subseteq \Mod^{\on{fg}}_{B}$ could be different from $\Mod^{\on{fg},u-\on{lf}}_{B}$.
\end{Emp}

\subsection{Asymptotic trace}

\begin{Emp} \label{E:graded}
{\bf Set up.} (a) We view $L[t]$ as a graded algebra over $L$ with $\deg t=1$, and let
$\wt{A}=\bigoplus_{i=0}^{\infty}A_i$ be a commutative finitely generated graded $L[t]$-algebra such that
$A_0$ is finite-dimensional over $L$. In particular, each $A_i$ is finite-dimensional over $L$,
and the action of $t$ on $\wt{A}$ induces a morphism $t_i:A_i\to A_{i+1}$ of $A_0$-modules.

\smallskip

(b) For every $b\in L$ we can form a quotient algebra $\sp_b(\wt{A}):=\wt{A}/(t-b)\wt{A}$, where $\sp$ stands for {\em specialization}.
Note that $A:=\sp_1(\wt{A})$ is naturally isomorphic to the inductive limit $\colim_i A_i$ with respect to the transition maps $t_i:A_i\to A_{i+1}$ from part~(a), where the canonical maps $\eta_i:A_i\to A$ are compositions $A_i\hra\wt{A}\overset{\pr}{\lra} A$.
Also $\sp_0(\wt{A})$ is a graded ring
$\sp_0(\wt{A})=\bigoplus_i \ov{A}_i$, where $\ov{A}_0=A_0$ and $\ov{A}_{i+1}=A_{i+1}/t_i(A_i)$ for all $i>0$.

\smallskip

(c) Let $\wt{M}:=\bigoplus_{i=0}^{\infty}M_i$ be a graded $\wt{A}$-module. Then it is a graded $L[t]$-module, so
the action of $t$ on $\wt{M}$ induces a morphism $t_i:M_i\to M_{i+1}$ of $A_0$-modules. For every $b\in L$, we set
$\sp_b(\wt{M})=\wt{M}/(t-1)\wt{M}$. Then $\sp_b(\wt{M})$ is a $\sp_b(\wt{A})$-module, and $A$-module
$M:=\sp_1(\wt{M})$ can be identified with inductive limit $M=\colim_i M_i$ of $A_0$-modules. Again, each $\eta_i:M_i\to M$ decomposes as
 $M_i\hra\wt{M}\overset{\pr}{\lra} M$.

\smallskip

(d) Let $\psi\in \Aut(\wt{A})$ be a (grading preserving) automorphism of order $n$ as a  graded $L[t]$-algebra. Then $\psi$ induces automorphisms of algebra $A$ and of a graded algebra $\sp_0(\wt{A})$, hence of $\Spec A$ and $\Proj(\sp_0(\wt{A}))$. Then the construction of  Section~\re{twisted} applies, so we can form algebras $\wt{B}=\wt{A}\{u^{\pm 1},\psi\}$ and  $B=A\{u^{\pm 1},\psi\}$.
\end{Emp}

\begin{Emp} \label{E:remgraded}
{\bf Remarks.} Suppose that we are in the situation of Section~\re{graded}.

\smallskip

(a) Since $\sp_1(\wt{M})\simeq\colim_i M_i$ (see Section~\re{graded}(c)) and the functor of an inductive limit is exact, the functor $\sp_1:\Mod_{\wt{A}}^{\gr}\to\Mod_A$, where $\Mod_{\wt{A}}^{\gr}$ stands for the category of graded $\wt{A}$-modules, is exact. Moreover,
this description of $\sp_1(\wt{M})$ also implies that $\sp_1(\wt{M})\simeq 0$ if and only if $\wt{A}$-module $\wt{M}$ is $t$-torsion, that is, $t$ acts on $\wt{M}$
locally nilpotently.

\smallskip

(b) Notice that maps $\eta_i:A_i\to A$ from Section~\re{graded}(b) gives rise to the homomorphism of graded
$L[t]$-algebras $\eta:\wt{A}\to A[t]:\sum_i a_i\mapsto\sum_i\eta_i(a_i)t^i$.  Moreover, for every graded $\wt{A}$-module
$\wt{M}$ we have a morphism of graded $\wt{A}$-modules $\eta:\wt{M}\to M[t]:=A[t]\otimes_A M: \sum_i m_i\mapsto\sum_i t^i\otimes\eta_i(m_i)$.
Arguing as in part~(a) one sees that $\Ker\eta$ is $t$-torsion.

\smallskip

(c) Notice that $t$ acts injectively on $\wt{A}$ (resp. $\wt{M}$) if and only if each $\eta_i:A_i\to A$ (resp. $\eta_i:M_i\to M$) is injective
or, what is the same, if the homomorphism  $\eta:\wt{A}\to A[t]$ (resp. $\eta:\wt{M}\to M[t]$) from part~(b) is injective. In this case,
$\{A_i\}_i$ form filtration of $A$ in the sense of Section~\re{Rees}(a) and $\sp_0(\wt{A})=\bigoplus_{i=0}^{\infty} A_i/A_{i-1}$ is the graded ring of $A$.

%each morphism $t_i:M_i\to M_{i+1}$ is injective, and this happens if and only if the composition $M_i\to\wt{M}\to M$ is injective.

\smallskip

(d) Notice that graded algebra $\sp_0(\wt{A})$ is finite-dimensional over $L$ if and only if
$\Proj(\sp_0(\wt{A}))$ is empty. Moreover, by observation of Section~\re{graded}(b) this implies that $A$ is finite-fimensional over $L$
or, what is the same, that $\Spec A$ is finite. Furthermore, using part ~(c) the opposite implication holds if $t$ acts on $\wt{A}$ injectively.
\end{Emp}

\begin{Emp} \label{E:serre}
{\bf Notation.} Suppose that we are in the situation of Section~\re{graded}.

\smallskip

(a)  Denote by $\Mod^{\on{fg},\on{gr}-\on{fd}}_{\wt{B}}\subseteq \Mod^{\on{fg}}_{\wt{B}}$ be the subcategory of graded $\wt{B}$-modules $\wt{M}:=\bigoplus_{i=0}^{\infty}M_i$ such that each $M_i$ is finite-dimensional over $L$. Notice that every $\wt{M}\in \Mod^{\on{fg},\on{gr}-\on{fd}}_{\wt{B}}$ is automatically $u$-locally finite, hence its quotient $M$ is $u$-locally finite as well.
Therefore the exact functor $\sp_1$ from Section~\re{remgraded}(a) gives rise to an exact functor
$\sp_1:\Mod^{\on{fg},\on{gr}-\on{fd}}_{\wt{B}}\to \Mod^{\on{fg},u-\on{lf}}_{B}$.

\smallskip

(b) Note that the subcategory $\Mod^{\on{fg},\on{gr}-\on{fd},t-\on{tor}}_{\wt{B}}\subseteq \Mod^{\on{fg},\on{gr}-\on{fd}}_{\wt{B}}$  of $t$-torsion modules is a Serre subcategory (see \cite[Tag~02M0]{Sta}),
so we can form the quotient abelian category $\ov{\Mod}^{\on{fg},\on{gr}-\on{fd}}_{\wt{B}}:=\Mod^{\on{fg},\on{gr}-\on{fd}}_{\wt{B}}/\Mod^{\on{fg},\on{gr}-\on{fd},t-\on{tor}}_{\wt{B}}$
(see \cite[Tag~02MS]{Sta}).

\smallskip

(c) In particular, (see, for example, \cite[Tag~02MX]{Sta}), we have an exact sequence on $K$-groups
\[
K(\Mod^{\on{fg},\on{gr}-\on{fd},t-\on{tor}}_{\wt{B}})\to K(\Mod^{\on{fg},\on{gr}-\on{fd}}_{\wt{B}})\to
K(\ov{\Mod}^{\on{fg},\on{gr}-\on{fd}}_{\wt{B}})\to 0.
\]
\end{Emp}

\begin{Lem} \label{L:serre}
The exact functor $\sp_1:\Mod^{\on{fg},\on{gr}-\on{fd}}_{\wt{B}}\to \Mod^{\on{fg},u-\on{lf}}_{B}$ from Section~\re{serre}(a) induces an equivalence of categories
\[
\ov{\sp}_1:\ov{\Mod}^{\on{fg},\on{gr}-\on{fd}}_{\wt{B}}\isom\Mod^{\on{fg},u-\on{lf}}_{B}.
\]
\end{Lem}

\begin{proof}
By Section~\re{remgraded}(a), we have $\Mod^{\on{fg},\on{gr}-\on{fd},t-\on{tor}}_{\wt{B}}=\Ker(\sp_1)$.
Therefore functor $\sp_1$ induces a faithful functor $\ov{\sp}_1$ (see \cite[Tags 02MS and 06XK]{Sta})

\smallskip

Next we claim that functor $\sp_1$ is essentially surjective (hence $\ov{\sp}_1$ is such), that is, for every $M\in \Mod_B^{\on{fg},u-\on{lf}}$, there exists $\wt{M}\in\Mod^{\on{fg},\on{gr}-\on{fd}}_{\wt{B}}$ such that $\sp_1(\wt{M})\simeq M$.

\smallskip

To construct $\wt{M}$, we choose generators $m_1,\ldots,m_r$ of $M$ over $A$. For each $i$, we set $M_i:=\sum_{j=1}^r \sum_{k\in\B{Z}} A_{i}(u^k(m_j))\subseteq M$, where $A_i\subseteq\wt{A}$ acts on $M$ via projection $\wt{A}\to A$ (see Section~\re{graded}(b)).
Then $\colim_i M_i\simeq \bigcup_i M_i=M$, each $M_i$ is finite-dimensional over
$L$, because $M$ is $u$-locally finite, and $\wt{M}:=\bigoplus_{i\geq 0} M_i$
is a graded $\wt{B}$-module such that $\wt{M}\in\Mod^{\on{fg},\on{gr}-\on{fd}}_{\wt{B}}$ (by construction),
and
\[
\sp_1(\wt{M})\simeq\colim_i M_i\simeq M
\] (by Section~\re{graded}(c)).

\smallskip

It remains to show that functor $\ov{\sp}_1$ is full, that is, is surjective on morphisms. Explicitly, we have to show that
for every two modules $\wt{M}',\wt{M}''\in\Mod^{\on{fg},\on{gr}-\on{fd}}_{\wt{B}}$ and every morphism
$f:M'\to M''$ of $A$-modules such that $M'=\ov{\ev}_1(\wt{M}')$ and $M''=\ov{\ev}_1(\wt{M}'')$ comes from a morphism
$\wt{f}:\wt{M}'\to\wt{M}''$ in $\ov{\Mod}^{\on{fg},\on{gr}-\on{fd}}_{\wt{B}}$.

\smallskip

Recall that for every $\wt{M}\in\Mod^{\on{fg},\on{gr}-\on{fd}}_{\wt{B}}$ the kernel of morphism $\eta:\wt{M}\to M[t]$ from Section~\re{remgraded}(b) belongs to $\Mod^{\on{fg},\on{gr}-\on{fd},t-\on{tor}}_{\wt{B}}$, hence the induced morphism
$\eta:\wt{M}\to \eta(\wt{M})$ in $\Mod^{\on{fg},\on{gr}-\on{fd}}_{\wt{B}}$ induces an isomorphism in $\ov{\Mod}^{\on{fg},\on{gr}-\on{fd}}_{\wt{B}}$. Thus, replacing
$\wt{M}'$ and $\wt{M}''$ by $\eta(\wt{M}')$ and $\eta(\wt{M}'')$, respectively, we can assume that $\wt{M}'\subseteq M'[t]$ and $\wt{M}''\subseteq M''[t]$.

\smallskip

Homomorphism $f$ gives rise to homomorphism $f[t]=\Id\otimes f:M'[t]\to M''[t]$ of $B[t]$-modules (thus of $\wt{B}$-modules),
and we set $\wt{M}:=\wt{M}''+f[t](\wt{M}')\subseteq M''[t]$. Then $f[t]$ induces a homorphism $f[t]:\wt{M}'\to\wt{M}$ of $\wt{B}$-modules, and
the fact that functor $\sp_1$ is exact implies that the embedding $\nu:\wt{M}''\hra\wt{M}$ induces an isomorphism $\sp_1(\wt{M}'')\isom \sp_1(\wt{M})$. Thus $\nu$ induces an isomorphism in  $\ov{\Mod}^{\on{fg},\on{gr}-\on{fd}}_{\wt{B}}$, and
\[
\wt{f}:=\nu^{-1}\circ f[t]: \wt{M}'\to \wt{M}\isom\wt{M}''
\]
is a morphism we were looking for.
\end{proof}

%\begin{Emp} \label{E:notrem}
%{\bf Notation and remarks.}
%(a) Let $\wt{M}=\bigoplus_{i\in\B{Z}}M_i$ be a graded $L[t]$-module. Then $t$ acts on $\wt{M}$ injectively if and only if
%each morphism $t_i:M_i\to M_{i+1}$ is injective, and this happens if and only if the composition $M_i\to\wt{M}\to M$ is injective.

%\smallskip

%(b) For every graded $L[t]$-module $\wt{M}$, there exists the largest quotient $\wt{M}_{t-\on{inj}}$ on which $t$ acts injectively. Namely,
%$\wt{M}_{t-\on{inj}}$ is the quotient module $\wt{M}/\wt{M}_{t-\on{tor}}$, where $\wt{M}_{t-\on{tor}}=\bigcup_i\Ker t^i\subseteq\wt{M}$ is the largest $t$-torsion submodule. Notice that since $\sp_1(\wt{M}_{t-\on{tor}})=0$ and functor $\sp_1$ is exact, the projection $\wt{M}\to\wt{M}_{t-\on{inj}}$ induces an isomorphism
%$\sp_1(\wt{M})\isom \sp_1(\wt{M}_{t-\on{inj}})$.

%\smallskip

%(c) It follows from part~(b) that if $\wt{M}$ is a graded module over $\wt{A}$ or $\wt{B}$, then $\wt{M}_{t-\on{tor}}$ is a graded submodule, and
%$\wt{M}_{t-\on{inj}}$ is a graded factor module. In particular,  $\wt{A}_{t-\on{tor}}\subseteq \wt{A}$ is a graded ideal, thus
%$\wt{A}_{t-\on{inj}}$ is a graded factor algebra. Furthermore, $\wt{M}_{t-\on{inj}}$ is a graded $\wt{A}_{t-\on{inj}}$-algebra.
%\end{Emp}

\begin{Thm} \label{T:stabilization}
In the notation of Section~\re{graded}, assume that the scheme of fixed points $\Proj(\sp_0(\wt{A}))^{\psi}$ is empty.

\smallskip

(a) For every $\wt{M}\in  \Mod^{\on{fg},\on{gr}-\on{fd}}_{\wt{B}}$  the sequence of traces $\{\Tr(u, M_i)\}_{i\in\B{N}}$ stabilizes,
so we can form the ``asymptotic trace''
\[
\Tr_{\on{as}}(u, \wt{M}):=\lim_i\Tr(u,M_i).
\]

\smallskip

(b) Moreover, there exists a unique homomorphism $\Tr_{\on{as}}(u,-):K(B,u-\on{lf})\to L$ of abelian groups such that for every $\wt{M}\in  \Mod^{\on{fg},\on{gr}-\on{fd}}_{\wt{B}}$ we have
an equality
\[
\Tr_{\on{as}}(u,\sp_1(\wt{M}))=\Tr_{\on{as}}(u,\wt{M}).
\]

\smallskip

(c) Furthermore, the scheme $(\Spec A)^{\psi}$ is finite, and the homomorphism $\Tr_{\on{as}}(u,-)$ of part~(b) coincides with the
homomorphism  $\Tr_{\gen}(u,-)$ of \rco{gentr}.
\end{Thm}

%\begin{Emp}
%{\bf Remark.} By \rco{gentr} we have a generalized trace map $\Tr_{\gen}(u,-)$ under an assumption that the scheme of fixed points
%$(\Spec A)^{\psi}$ is finite.

%On the other hand, by \rt{stabilization}(b) we have a generalized trace map $\Tr'_{\gen}(u,-)$ under an assumption that the scheme of fixed points $\Proj(\gr(A))^{\psi}$ is empty. We are wondering whether one of these conditions implies the other.
%\end{Emp}

The proof of the theorem is based on the following particular case, which will be shown in the next subsection.

\begin{Cl} \label{C:vanish}
In the notation of \rt{stabilization}(a), assume that $\wt{B}$-module $\wt{M}$ is $t$-torsion. Then $\Tr(u, M_i)=0$ for each sufficiently large $i$.
\end{Cl}

%Now we are ready to prove \rl{stab}.

\begin{Emp}
\begin{proof}[Proof of \rt{stabilization}]

\smallskip

(a) The action of $t$ on $\wt{M}$ gives a graded
homomorphism $t:\wt{M}\to\wt{M}$ of degree one, hence a
homomorphism $t_i:M_i\to M_{i+1}$ for all $i$, commuting with the action of
$u$. It suffices to show that for all sufficiently large $i$ we have $\Tr(u, \Ker t_i)=\Tr(u, \Coker
t_i)=0$.

Both $\Ker t=\bigoplus_i \Ker t_i\subseteq\wt{M}$ and $\Coker t=\bigoplus_i \Coker t_i$ are graded
$\wt{B}$-modules, which are killed by $t$. Moreover they are finitely generated because $\wt{M}$ is finitely generated and $\wt{B}$ is Noetherian,
so the assertion follows from \rcl{vanish}.

\smallskip

(b) The assignment $\wt{M}\mapsto \Tr_{\on{as}}(u, \wt{M})$ from part~(a) is additive, so it gives rise to a homomorphism
$\Tr_{\on{as}}(u,-):K(\Mod^{\on{fg},\on{gr}-\on{fd}}_{\wt{B}})\to L$. By \rcl{vanish}, we have $\Tr_{\on{as}}(u, \wt{M})=0$,
if $\wt{M}$ is $t$-torsion, so by Section~\re{serre}(c), homomorphism $\Tr_{\on{as}}(u,-)$ factors unique through
$K(\ov{\Mod}^{\on{fg},\on{gr}-\on{fd}}_{\wt{B}})$. The assertion now follows from \rl{serre}.

\smallskip

(c) Notice that we have isomorphisms
\[
\Proj(\sp_0(\wt{A}))^{\psi}\simeq \Proj(\sp_0(\wt{A})_{\psi})\simeq \Proj(\sp_0(\wt{A}_{\psi}))
\]
and $(\Spec A)^{\psi}\simeq \Spec(A_{\psi})\simeq \Spec(\sp_1(\wt{A}_{\psi}))$. Therefore the first assertion follows from the observation of
Section~\re{remgraded}(d) applied to the graded algebra $\wt{A}_{\psi}$.

\smallskip

To show the second assertion, we have to show that the homomorphism $\Tr_{\on{as}}(u,-)$ satisfies properties (i) and (ii) of \rco{gentr}.
Property (i) follows from the fact that  the assignment $\wt{M}\mapsto \Tr_{\on{as}}(u,\wt{M})$ from part~(a) satisfies property (i), which is clear.
To show property (ii) notice that if $M$ is finite-dimensional over $L$, then the graded module $\wt{M}:=\bigoplus_i M$
satisfies $\sp_1(\wt{M})\simeq M$, in which case we have the equality $\Tr_{\on{as}}(u,M)=\Tr_{\on{as}}(u,\wt{M})=\Tr(u,M)$.
\end{proof}
\end{Emp}

\subsection{Proof of \rcl{vanish}}

The argument is based on the following lemma:

\begin{Lem} \label{L:fixed}
Let $Z$ be a proper scheme over a field $L$ equipped with an
action of a finite cyclic group $\Phi$ with generator $\psi$ such that
the scheme of fixed points $Z^\psi$ is empty.

\smallskip

(a) For every $\Phi$-equivariant coherent sheaf $\C{F}$ on $Z$, we
have an equality
\[
\sum_j (-1)^j\Tr(\psi, H^j(Z,\C{F}))=0.
\]

\smallskip

(b) Assume that $Z$ has an ample $\Phi$-equivariant line bundle
$\C{O}(1)$. Then for every $\Phi$-equivariant coherent sheaf
$\C{F}$ on $Z$ and every sufficiently large $i\in\B{N}$ we have an equality
$\Tr(\psi,H^0(Z,\C{F}(i)))=0$.
\end{Lem}

\begin{proof}
(a) Thought the assertion is well known to specialists, we sketch the proof for
the convenience of the reader. The map $\C{F}\mapsto\sum_j
(-1)^j\Tr(\psi, H^j(Z,\C{F}))$ is additive, so is gives rise to a homomorphism
$K'_0(\Phi,Z)\to L$ from the equivariant $K$-theory of coherent sheaves on $Z$. By definition, this homomorphism can be written as the
composition of the pushforward $R\Gm:K'_0(\Phi,Z)\to K'_0(\Phi,\on{pt})=K[\Rep \Phi]$ and the trace map $\Tr(\psi,-):K[\Rep
\Phi]\to L$. Thus it suffices to show that the composition $\Tr(\psi,-)\circ R\Gm:K'_0(\Phi,Z)\to L$
vanishes.

Next we observe that the composition $\Tr(\psi,-)\circ R\Gm$ factors through the
localization $K'_0(\Phi,Z)_{\psi}$, while by the localization theorem, we have
$K'_0(\Phi,Z)_{\psi}\simeq K'_0(\Phi,Z^{\psi})_\psi$ (see, for example,
\cite[Th\'eor\'eme 2.1]{Th}) . Since $Z^{\psi}$ is empty, the group $K'_0(\Phi,Z^{\psi})_{\psi}$
vanishes, thus the composition $\Tr(\psi,-)\circ R\Gm$ vanishes as well.

\smallskip

(b) By a theorem of Serre, we have $H^j(Z,\C{F}(i))=0$ for all $j>0$ and all sufficiently large
$i$. Therefore part~(b) follows from part~(a).
\end{proof}

%\subsection{Technical lemma} \label{S:lemma}

\begin{Emp} \label{E:setupgraded}
{\bf Set up.} (a) Let $A=\bigoplus_{i=0}^{\infty}A_i$ be a commutative finitely generated graded algebra over $L$ such that $A_0$ is finite-dimensional over $L$, let $\psi\in \Aut(A)$ be a grading-preserving automorphism of $A$ of order $n$. Then $\psi$ induces an
automorphism of $\Proj A$.

%Explicitly, $B=\bigoplus_{i=0}^{\infty}B_i$, where where $B_i=\bigoplus_{k\in\B{Z}A_i u^k$.

\smallskip

(b) Let $B_n=A\{u,\psi\}_n$ be as in Section~\re{twisted}. Then $B_n$ is naturally a graded ring such that $A\subseteq B_n$ is a graded subring and $\deg u=0$. Let $N=\bigoplus_{i=0}^{\infty} N_i$ be a finitely generated graded $B_n$-module. Then each $N_i$ is an $L[u]$-module finite dimensional over $L$, hence we can consider the trace $\Tr(u,N_i)\in L$.
\end{Emp}

\begin{Cor} \label{C:fixed}
In the situation of Section~\re{setupgraded} assume that  the scheme of fixed points
$(\Proj A)^{\psi}$ is empty. Then for each sufficiently large $i$, we have
$\Tr(\psi,N_i)=0$.
\end{Cor}

\begin{proof}
Assume first that algebra $A$ is generated by $A_1$ over $A_0$. Since $N$ is a finitely generated
$A\{u,\psi\}_n$-module, it is a $\Phi$-equivariant finitely generated $A$-module. Hence module $N$
gives rise to a $\Phi$-equivariant coherent sheaf $\Proj N$ on
$Z:=\Proj A$. Moreover, for each sufficiently large $i\in\B{Z}$ we have an isomorphism $N_i\simeq H^0(Z,\Proj(N)(i))$. Therefore the assertion
follows from \rl{fixed}(b).

\smallskip

To deduce the general case from the particular case shown above, for each $d\in\B{N}$ we consider the graded subalgebra $A^{(d)}:=\sum_{i=0}^{\infty} A_{di}\subseteq A$.
Then there exists $d$ such that the graded algebra $A^{(d)}$ is generated by $(A^{(d)})_1=A_d$ over $(A^{(d)})_0=A_0$ (by \cite[Tag 0EGH]{Sta}).
Also the inclusion $A^{(d)}\hra A$ induces an isomorphism $\Proj A\isom\Proj A^{(d)}$ (by \cite[Tag 0B5J]{Sta}).
In particular, the scheme of $(\Proj A^{(d)})^{\psi}$ is empty.

\smallskip

Notice that $A$ is integral and finitely generated (thus finite) $A^{(d)}$-algebra, thus $N$ is a finitely generated $A^{(d)}$-module.
Moreover, module $N$ decomposes as a direct sum $N=\bigoplus_{j=0}^{d-1} N^{(j)}$, where $N^{(j)}$ is a graded finitely generated $A^{(d)}\{u,\psi\}_n$-module $\bigoplus_i N_{di+j}$. Moreover, since $A^{(d)}$ is generated by $(A^{(d)})_1$,
it follows from the particular case shown above that for each $j$ and each sufficiently large $i$ we have $\Tr(\psi,N_{di+j})=0$.
Since the number of $j$'s is finite, we are done.
\end{proof}
%The following technical lemma will be crucial for what follows.

The following simple criterion will be used later.

\begin{Lem} \label{L:example}
In the situation of Section~\re{setupgraded}(a), assume that for every $i>0$
there exists a basis $\C{B}_i$ of $A_i$ over $L$ such that
$\prod_{j=0}^{n-1} \psi^j(b)=0$ for all $b\in \C{B}_i$.

Then the scheme of fixed points $(\Proj A)^{\psi}$ is empty.
\end{Lem}
\begin{proof}
Fix $\frak{p}\in (\Proj A)^{\psi}$. Then for all $i>0$ and $b\in \C{B}_i$, we have
$\prod_{j=0}^{n-1} \psi^j(b)=0\in \frak{p}$. Then we have $\psi^j(b)\in \frak{p}$ for some $j$ (because
$\frak{p}\subseteq A$ is a prime ideal), hence $b\in\frak{p}$ (because $\frak{p}$ is $\psi$-invariant).
This implies that $\C{B}_i\subseteq\frak{p}$ for every $i$, thus $\bigoplus_{i>0}A_i\subseteq\frak{p}$,
a contradiction.
\end{proof}

Now we are ready to show \rcl{vanish}.

\begin{Emp}
\begin{proof}[Proof of \rcl{vanish}]

Assume first that $t$ and $u^n-1$ act on $\wt{M}$ as zero. Since $\wt{B}/(t,u^n-1)=\sp_0(\wt{A})\{u,\psi\}_n$, $\wt{M}$ is a finitely generated graded $\sp_0(\wt{A})\{u,\psi\}_n$-module, in which case the assertion was proven in \rco{fixed}.

\smallskip

To deduce the assertion in the general case, we argue as in \rl{ktheory}:

\smallskip

Since $\wt{M}$ is a $t$-torsion module over a Noetherian ring $\wt{B}$, it has a finite filtration $0=\wt{M}^{(0)}\subsetneq \wt{M}^{(1)}\subsetneq \ldots \subsetneq \wt{M}^{(n)}=\wt{M}$ by graded submodules, where $\wt{M}^{(i)}=\Ker t^i$. Since the trace is additive, the assertion for $\wt{M}$ follows from the assertion for all quotients $\wt{M}^{(i+1)}/\wt{M}^{(i)}$. Thus we can assume that $t$ acts on $\wt{M}$ as zero.

\smallskip

Next, since $\wt{M}$ is finitely generated and $u$-locally finite, it decomposes as a finite direct sum $\wt{M}=\bigoplus_{a\in L^{\times}}\wt{M}_{(a)}$, where
$u^n$ has a unique generalized eigenvalue $a$ on $\wt{M}_{(a)}$. %Moreover, this sum is finite, because $\wt{M}$ is finitely generated.
Since it suffices to show the assertion for each $\wt{M}_{(a)}$, we can assume that $a$ is the unique generalized eigenvalue of $u^n$ on $\wt{M}_{(a)}$.

\smallskip

Moreover, as in the proof of \rl{ktheory}, choosing $b\in L$ such that $b^n=a$ and twisting, we reduce to the case $a=1$.
In this case, $\wt{M}$ is $(u^n-1)$-torsion, thus arguing as in the second paragraph of the proof, we reduce to the case when $u^n-1$ acts on $\wt{M}$ as zero,
completing the proof.
\end{proof}
\end{Emp}

\subsection{Basic example}

\begin{Emp} \label{E:Rees}
{\bf Rees rings and Rees modules.}

\smallskip

(a) Let $A$ be an $L$-algebra. By a {\em filtration} of $A$, we mean an increasing collection $A_0\subseteq A_1\subseteq A_2\subseteq\ldots\subseteq A$ of subspaces over $L$ such that
$1\in A_0, \bigcup_i A_i=A$ and $A_i\cdot A_j\subseteq A_{i+j}$ for all $i,j$. To each such filtration $\{A_i\}_i$ we associate its {\em Rees ring} $ R(A):=\bigoplus_i A_i t^i\subseteq A[t,t^{-1}]$.

%\smallskip

Then $R(A)$ is a graded $L[t]$-algebra on which $t$-acts injectively, and we have a natural isomorphism $ R(A)/(t-1) R(A)\simeq A$.
Conversely, every graded $L[t]$-algebra on which $t$-acts injectively is isomorphic to a Rees ring $R(A)$ corresponding to some
filtration of $A$ (compare Section~\re{remgraded}(c)).

\smallskip

(b) In the situation of part (a) let $M$ be an $A$-module, equipped with an increasing filtration $\{M_j\}_j$ of subspaces over $L$. We say that $\{M_j\}_j$ is {\em compatible with} $\{A_i\}_i$, if for every $i$ and $j$ we have $A_i(M_j)\subseteq M_{i+j}$.

In this case the Rees module $R(M):=\bigoplus_i M_i$ is a graded $R(A)$-module on which $t$ acts injectively, and $R(M)/(t-1)R(M)\simeq M$. Conversely, every graded $R(A)$-module on which $t$-acts injectively is isomorphic to a Rees module $R(M)$ corresponding to some
filtration of $M$ compatible with $\{A_i\}_i$.

\smallskip

(c) In the situation of part (b), we call a filtration $\{M_j\}_j$ {\em finitely generated} (over $\{A_i\}_i$) (also called {\em good} in other sources), if $R(M)$ is a finitely generated $R(A)$-module. In this case, $M$ is automatically finitely generated over $A$.

Conversely, the proof of the essential surjectivity of $\ov{\sp}_1$ (see \rl{serre}) shows that if $M$ is a finitely generated $A$-module, then $M$ has a finitely generated filtration
$\{M_i\}_i$ over $\{A_i\}_i$.

\smallskip

\end{Emp}

\begin{Emp} \label{E:fil}
{\bf Filtrations on monoids.}

\smallskip

(a) Let $\Gm$ be a monoid, which in our application will be a
group. By a {\em filtration} of $\Gm$, we mean an increasing
collection of subsets
$\Gm_0\subseteq\Gm_1\subseteq\Gm_2\subseteq\ldots\subseteq\Gm$ such that
$1\in\Gm_0, \bigcup_i\Gm_i=\Gm$ and $\Gm_i\cdot\Gm_j\subseteq\Gm_{i+j}$
for all $i,j$.

\smallskip

(b) We will say that the filtration is {\em finitely
generated}, if in addition each $\Gm_i$ is finite,
and there exists $m\in\B{N}$ such that for each $i>m$ we have
$\Gm_i=\bigcup_{j=1}^m \Gm_j\cdot \Gm_{i-j}$. Note that in this case
the monoid $\Gm$ us generated by $\bigcup_{j=0}^m\Gm_i$, thus
$\Gm$ is finitely generated.

\smallskip

(c)  Conversely, every finitely generated monoid $\Gm$ has a finitely generated filtration $\{\Gm_i\}_i$.
Indeed, let $S:=\{\gm_1,....,\gm_n\}\subseteq\Gm$ be a finite set of generators. Then the increasing collection of subsets defined by
$\Gm_0:=\{1\}, \Gm_1:=S\cup\{1\}$ and $\Gm_{i+1}:=\Gm_i\cdot \Gm_1$ form a finitely generated filtration of $\Gm$.
Moreover, if $\Gm$ is equipped with an automorphism $\psi$ of finite order, then we can assume that $S$ and hence also
$\{\Gm_i\}_i$ is $\psi$-invariant.

\smallskip

(d) Let $L$ be a field. Then a filtration $\{\Gm_i\}_i$ on a
monoid $\Gm$ gives rise to a filtration $\{L[\Gm_i]\}_i$ of its
monoid algebra $L[\Gm]$, where $L[\Gm_i]\subseteq L[\Gm]$ denotes the span of
$\Gm_i\subseteq\Gm$ over $L$.

\smallskip

(e) Note that if  the filtration $\{\Gm_i\}_i$ is finitely generated, then the induced filtration
$\{L[\Gm_i]\}_i$ of $L[\Gm]$ is finitely generated as well.
%$R(L[\Gm])$ is a finitely generated $L$-algebra.
In particular, both $L[\Gm]$ and $R(L[\Gm])$ are Noetherian, if $\Gm$ is commutative.
\end{Emp}

\begin{Emp} \label{E:finell}
{\bf Basic example.}

\smallskip

(a) Let  $\La$ be a finitely generated free
abelian group, let $\psi\in \Aut(\La)$ be an {\em elliptic} element
of finite order $n$, that is, $\La^{\psi}=\{1\}$ (compare Section~\re{fincor}(d)), and let $\La_{\psi}:=\La/(\psi-1)\La$ be the group of $\psi$-coinvariants, which is automatically finite.

\smallskip

(b) Let $\lan u\ran$ and $\cycl_n$  be the cyclic groups generated by
$u$ of infinite order and of order $n$, respectively, and let $\Dt:=\La \rtimes \lan u\ran$ and
$\Dt':=\La\rtimes\lan u\ran_n$ be the semi-direct products, where $u$ acts on $\La$ as $\psi$, that is, $u\la u^{-1}=\psi(\la)$
for every $\la\in\La$.

\smallskip

(c) We equip monoid $\La$ with a $\psi$-invariant finitely generated filtration
$\{\La_n\}_n$ with $\La_0=\{1\}$ (see Section~\re{fil}(c)). It induces filtrations on $\Dt$ and $\Dt'$
such that $\deg u=\deg u^{-1}=0$.

\smallskip

(d) Let $L$ be an algebraically closed field of characteristic zero.

\smallskip

Then $A:=L[\La]$ is a commutative finitely generated algebra over $L$ and the automorphism $\psi\in \Aut(\La)$ induces an automorphism of $A$ of finite order, thus putting us in the situation of \re{twisted}.

In particular, the algebra $B=A\{u^{\pm 1},\psi\}$ is canonically identified with $L[\Dt]$, while
$B_n=A\{u,\psi\}_n$ is identified with $L[\Dt']$.  Furthermore, by Section~\re{fil}(d)
the filtrations of part~(c) gives rise to filtrations of $A$ and $B$, thus to Rees algebras $\wt{A}=R(A)$ and $\wt{B}=R(B)$.

% and $B_n:=L[\Dt]$ and $A':=L[\Dt']$ be the
%corresponding filtered group algebras (see Section~\re{fil}(c)), and let $ R(A)$
%and $R(A')$ be the corresponding Rees rings. Then $A'=A/(u^n-1)$ and
%$  R(A')= R(A)/(u^n-1)$.

%\smallskip

%(e) Since $\La$ is commutative, its group algebra $L[\La]$ and Rees algebra $R(L[\La])$ is
%Noetherian (see Section~\re{fil}(d)). Since $R(A')$ is a finitely generated $R(L[\La])$-module (denerated by
%$\{u^i\}_{i=0}^{n-1}$), it is Noetherian as well.
\end{Emp}

\begin{Lem} \label{L:finell}
In the notation of Section~\re{finell}(a), we have

\smallskip

(a) $|\La_{\psi}|=\det(1-\psi):=\det(1-\psi,\La\otimes_{\B{Z}}\B{Q})$;

\smallskip

(b) element $u^n\in\Dt$ is central, and $(\la u)^n=u^n\in\Dt$ for each $\la\in\La$.
\end{Lem}

\begin{proof}
(a) Since $\psi$ is elliptic, endomorphism $1-\psi\in \End(\La)$ is injective.
As $\La_{\psi}$ is the cokernel of the map $1-\psi\in \End(\La)$, we get that
$\La_{\psi}$ is a finite group of cardinality $|\det(1-\psi)|$. It particular,
$\det(1-\psi)\neq 0$, and it remains to show that
$\det(1-\psi)>0$.

Let $f_{\psi}(t)=\det(t-\psi,\La\otimes_{\B{Z}}\B{Q})$ be the characteristic polynomial of
$\psi$. Since $\psi^n=1$, all roots of $f_{\psi}$ are $n$-th roots of unity.
Since $f_{\psi}(1)=\det(1-\psi)\neq 0$, and
$f_{\psi}(t)\in\B{R}[t]$, we have a decomposition
$f_{\psi}(t)=(t+1)^k\prod_i(t-\la_i)(t-\ov{\la_i})$, therefore
$\det(1-\psi)=f_{\psi}(1)=2^k\prod_i (1-\la_i)\ov{(1-\la_i)}>0$.

\smallskip

(b) Since $\psi^n\in\Aut(\La)$ is trivial, element $u^n\in\Dt$ commutes with $\La$
and $u$, hence element $u^n\in\Dt$ is central. Next for every $\la\in\La$, its $n$-power
$(\la u)^n$ equals  $\mu u^n$ with $\mu=\la \psi(\la)\ldots \psi^{n-1}(\la)\in\La$.
Moreover, $\mu\in\La^{\psi}=\{1\}$.
\end{proof}

%\begin{Emp} \label{E:finell2}
%{\bf Remarks.} (a) Note that the map $[\la]\mapsto [\la u]$
%defines a bijection $\La/(1-u)\La\isom \La_{adj}\bs\La u$, where
%as before $\La_{adj}\bs\La u$ is the quotient of the coset $\La
%u\subseteq\Dt$ by the adjoint action of $\La$. Since $\det(1-u)>0$
%(by \rl{finell}), the quotient $\La/(1-u)\La$ and hence also the
%quotient $\La_{adj}\bs \La u$ is finite of cardinality
%$\det(1-u)$.

%(b) For each finitely generated representation $V$ of $\La$ over a
%field $L$, the homology $H_*(\Lambda,V)$ is a virtual
%finite-dimensional vector space. Therefore, each finitely
%generated representation $V$ of $\Dt$ over a field $L$, the
%homology $H_*(\Lambda,V)$ is a virtual finite-dimensional
%representation of $\lan u\ran$.
%\end{Emp}

\begin{Emp} \label{E:trace}
{\bf Remarks.} Assume that we are in the situation of Section~\re{finell}.

\smallskip

(a) Since $\La_{\psi}$ is a finite group, the algebra $A_{\psi}\simeq L[\La_{\psi}]$ is finite-dimensional over $L$. Thus the scheme of fixed points $(\Spec A)^{\psi}$ is finite, hence the generalized trace homomorphism $\Tr_{\gen}(u,-)$ is defined (see \rco{gentr}).

\smallskip

(b) We claim that graded ring $\gr A$ satisfies the condition of \rl{example}, therefore the scheme of fixed points $\Proj(\gr A)^{\psi}$ is empty, hence \rt{stabilization} applies. Indeed, each $\gr_{i+1}(A)=A_{i+1}/A_{i}$ has a natural basis
$\{[\la]\}_{\la\in \La_{i+1}\sm\La_{i}}$, and for every $\la\in \La_{i+1}\sm\La_{i}$, we have $\prod_{j=1}^n\psi^j(\la)\in\La^{\psi}=\{1\}$, thus $\prod_{j=1}^n\psi^j([\la])=0$.

\smallskip

(c) Note that for every $\la\in\La\subseteq A^{\times}$, every $u$-locally finite $M\in\Mod_B^{\on{fg},u-\on{lf}}$ is $u^n=(\la u)^n$-locally finite (see
\rl{finell}(b)), hence is $\la u$-locally finite. Thus the trace
$\Tr_{\gen}(u\la,M)$ is defined (see Section~\re{remgentr}).

\smallskip

(d) Note that the map $[\la]\mapsto [\la u]$ defines a bijection
$\La_{\psi}\isom \La_{\adj}\bs\La u$, where as before $\La_{\adj}\bs\La
u$ is the quotient of the coset $\La u\subseteq\Dt$ by the adjoint
action of $\La$.

\end{Emp}
%Therefore $\Tr(u\la,M)$ only depends on the class $[\la]\in\La_u$, thus the right hand side of \form{tf} is
%independent of the choice of representatives.

\begin{Prop} \label{P:trform}
In the situation of Section~\re{finell}, let $M$ be a finitely generated $u$-locally finite $B$-module.

\smallskip

(a) For every $\la\in\La$, we have an equality $\Tr_{\gen}(\la u\la^{-1},M)=\Tr_{\gen}(u,M)$ (compare Section~\re{trace}(c)).

\smallskip

(b) Each group homology  $H_i(\La,M)$ is a finite dimensional representation of the cyclic group $\cycl$ over $L$, and
we have an equality
\begin{equation} \label{Eq:tf}
\Tr(u,H_*(\La,M)):=\sum_i (-1)^i \Tr(u,H_i(\La,M))=\sum_{\la}\Tr_{\gen}(\la u,M),
\end{equation}
where $\la u$ runs over a set of representatives of the $\La$-conjugacy classes in $\La u$ (see Section~\re{trace}(d)).
\end{Prop}

\begin{proof}
(a) Note that both homomorphisms $\Tr_{\gen}(u,-)$ and $\Tr_{\gen}(\la u\la^{-1},-)$ satisfy formula
\form{local}. Thus, by \rco{ktheory}, it suffices to show the equality when $M$ is a finitely-generated $A_{\psi}[\cycl_n]=L[\La_{\psi}\times \cycl_n]$-module. In this case, $M$ is finite-dimensional over $L$,
so we are reduced to the equality $\Tr(\la u\la^{-1},M)=\Tr(u,M)$.

\smallskip

(b) First we show that each $H_i(\La,M)$ is finite-dimensional. Arguing as in the proofs of \rl{ktheory} and \rcl{vanish}, we reduce to the
case when $u^n$ acts trivially on $M$. In this case, $M$ is a finitely generated $A=L[\La]$-module, in which case the finite-dimensionality
of $H_i(\La,M)$ is well-known.

\smallskip

Secondly we notice that it follows from part~(a) that $\Tr_{\gen}(\la u,M)$ only depends on the $\La$-conjugacy class of $\la u$, thus
the right hand side of \form{tf} is well-defined.

\smallskip

To show the identity \form{tf}, we notice that both sides are additive in $M$ and satisfy formula
\form{local}. Therefore, by \rco{ktheory}, we can assume that $M$ is an
$L[\La_{\psi}\times \cycl_n]$-module, finite-dimensional over $L$.
Moreover, we can assume that  $M$ is irreducible,
that is, $M=M_{\xi}$ corresponds to a character
$\xi:\La_{\psi}\times\cycl_n\to L\m$. Twisting $M$ by a character
of $\cycl_n$, we may assume that $\xi$ is trivial on $u$, thus is a
character of $\La_{\psi}$. In other words, we have to show that for
every character $\xi:\La_{\psi}\to L\m$, we have an equality
\[
\Tr(\psi,H_*(\La,M_{\xi}))=\sum_{\la\in\La_{\psi}}\la(\xi).
\]
If $\xi$ is non-trivial, then both sides of the last equation
vanish. So we can assume that $\xi$ is trivial, in which case we
have to show the equality $\Tr(\psi,H_*(\La,L))=|\La_{\psi}|$.

But this follows from a well-known identity
\[
\Tr(\psi,H_*(\La,L))=\sum_{i=0}^{\rk \La}(-1)^i \Tr(\psi,
\wedge^i\La)=\det(1-\psi,\La)
\]
and \rl{finell}(a).
\end{proof}

\begin{Lem} \label{L:induced}
In the situation of Section~\re{finell}, let $\La^0\subseteq \La$ be a $\psi$-invariant subgroup, let $M^0$ be a finitely generated $u$-locally finite $L[\La^0\rtimes\lan u\ran]$-module, and let $M=\ind_{\La^0}^{\La}M^0$ be the induced representation.

Then we have an equality
\begin{equation} \label{Eq:induced}
\Tr_{\gen}(u,M)=\sum_{\mu}\Tr_{\gen}(\mu^{-1} u\mu,M^0),
\end{equation}
where the sum runs over a finite set of representatives $\mu\in\La$ of classes in $\La/\La^0$ such that  $\mu^{-1} u\mu\in\La^0 u$.
\end{Lem}
\begin{proof}
Notice first that $\mu^{-1} u\mu\in\La^0 u$ if and only if $\mu^{-1}\psi(\mu)\in\La^0$. Therefore the finiteness of the sum follows from the fact that the automorphism $\psi$ is elliptic, hence the induced map $\psi-1:\La\to \La$ is injective, and hence the subgroup $(\psi-1)(\La_0)\subseteq \La^0$ is of finite index.

Next, notice that both sides of \form{induced} are additive in $M^0$. Thus, arguing as in \rp{trform} (using \rco{ktheory}) we reduce to the case when $M^0$ is finite-dimensional over $L$. Moreover, as we observed in Section~\re{trace}(b), \rt{stabilization} applies in our case.

As in Section~\re{finell}(c), choose a $\psi$-invariant finitely generated filtration $\{\La_i\}_i$ of $\La$. Then $M_i:=\La_i(M^0)\subseteq M$ gives a finitely generated filtration of $M$. Then vector space $M_i$ decomposes as $M_i=\bigoplus_{\mu}\mu M^0$, where $\mu $ runs over a set of representatives of the set of cosets $\La_i\La^0/\La^0$, and we have an equality
\[
\Tr(u,M_i)=\sum_{\mu}\Tr(\mu^{-1} u\mu,M^0),
\]
where the sum runs over the set of representatives $\mu\in\La$ of $\La_i\La^0/\La^0$ such that $\psi-(\mu)\La^0=\mu\La^0$, that is,
$\mu^{-1} u\mu\in\La^0 u$. From this the assertion follows from \rt{stabilization}(c).
\end{proof}

\section{Compatibility of actions (proof of \rt{action})}

\subsection{The topological Jordan decomposition}

\begin{Emp} \label{E:bounded}
{\bf Notation.} Let $H$ be a connected linear algebraic
group over $K:=k((t))$, and let $\gm\in H(K)$ be a semisimple element.

We
say that $\gm$ is {\em bounded} (resp. {\em strongly semisimple},
resp. {\em topologically unipotent}), if for every representation
$\rho$ of $H$ over $\ov{K}$, each eigenvalue  $\al\in\ov{K}\m$
satisfies $\val_K(\al)=0$ (resp. $\al\in k\m$, resp.
$\val_K(\al-1)>0$).
\end{Emp}

\begin{Emp} \label{E:strss}
{\bf Example.}
Let $T$ be a (split) torus over $k$. Then $\gm\in T(K)$ is strongly semisimple if and only
if $\gm\in T(k)$.
\end{Emp}

The following result, which is called the {\em topological Jordan
decomposition}, is well known and widely used in the case when $K$ is replaced by a
$p$-adic field, but does not seen to appear in the literature in our (geometric)
setting.

\begin{Lem} \label{L:tjd}
Let $\gm\in H(K)$ be a bounded semisimple element.

\smallskip

(a) There exists a unique decomposition $\gm=su=us$, called {\em the topological Jordan decomposition}, such that $s\in H(K)$ is strongly
semisimple, and $u\in H(K)$ is topologically unipotent.

\smallskip

(b) We have $H_{\gm}=H_s\cap H_u$. Moreover, if $H$ is connected, then $u,s\in H_s^0(K)$.

\smallskip

(c) An element $\gm\in H(K)={L} H(k)$ is strongly semisimple, if and only if the closure
of the cyclic group $\ov{\lan\gm\ran}\subseteq {L} H$ is a diagonalizable algebraic group.
\end{Lem}

\begin{proof}
(a) If $H=\B{G}_m$, then the lemma asserts that for every
$a\in\C{O}_K\m=k[[t]]\m$ there exists a unique decomposition
$a=su$, where $s\in k\m$ and $u\in 1+t\C{O}_K$, which is clear.
Namely, for every $a=\sum_{i\geq 0}a_i t^i\in k[[t]]\m$,
we have $a_0\in k\m$ and $a=a_0(a_0^{-1}a)$ is the required decomposition.

\smallskip

To show the uniqueness of a decomposition in general, we choose a faithful representation
$\rho:G\hra\Aut(V)$ of $G$. Extending $K$, we can assume that all
eigenvalues of $\rho(\gm)$ lie in $K$. Then $V$ decomposes as a
direct sum $V=\bigoplus_{\la} V_{\la}$, where $V_{\la}$ is the
eigenspace of $\rho(\gm)$ corresponding to $\la$. Since $\gm$ is
bounded, each $\la$ belongs $\C{O}_K\m$. Since $s$ and $\gm$
commute, $\rho(s)$ stabilizes each $V_{\la}$. Now since $s$ is
strongly semisimple, while $\gm s^{-1}$ is topologically
unipotent, $\rho(s)|_{V_{\la}}$ has to be $\la_0\Id$.

\smallskip

To show the existence of a decomposition in general, we note that $\gm$ lies in a maximal torus
$T\subseteq H$ defined over $K$. Therefore we can replace $H$ by
$T$, thus assuming that $H$ is a torus. By the uniqueness, we can
replace $K$ by a finite Galois extension, thus assuming that $H$
is split, thus $H\simeq \B{G}_m^n$. Then we reduce to the case of
$H=\B{G}_m$, if which case the assertion was shown before.

\smallskip

(b) Note first that if $\gm=su$ is the topological Jordan decomposition of $\gm$,
then $g\gm g^{-1}=(gsg^{-1})(gug^{-1})$ is the topological Jordan decomposition
of $g\gm g^{-1}$ for every $g\in H(K)$. Therefore  we have an equality
$H_{\gm}=H_s\cap H_u$. It remains to show that if $H$ is connected, then
$u\in H_s^0(K)$. Since $\gm$ is semisimple,  it lies in a maximal torus of $H$, which is connected, hence
$\gm\in H_{\gm}^0(K)\subseteq H^0_s(K)$. Therefore by the uniqueness
of the topological Jordan decomposition, we conclude that $s,u\in
H^0_s(K)$.

\smallskip

(c) As in part~(a), we reduce the assertion first to the case of torus,
then to the case of a split torus, in which case the assertion is clear.

\end{proof}

\begin{Lem} \label{L:topjd2}
Let $G$ be a split connected reductive group over $K$, $\I\subseteq LG$ an Iwahori subgroup, and
$\gm=su$ the topological Jordan decomposition of $\gm\in\I$. Then $s,u\in\I$.
\end{Lem}

\begin{proof}
Note that $\I$ has a presentation as a limit $\I\simeq \lim_n H_n$, where
$H_n:=\I/\I_n$ is an algebraic group over $k$ and $\I_n\subseteq \I^+$. Then $\gm\in \I$ corresponds to a family $\gm_n\in H_n$.
Let $\gm_n=s_n u_n=u_n s_n$ be the Jordan decomposition of $\gm_n\in H_n$. By the uniqueness of the Jordan
decomposition, both $\{s_n\}_n$ and $\{u_n\}_n$ form compatible families, thus give rise to elements $s,u\in\I$ such
that $\gm=su=us$. It suffices to show that $s$ is strongly semisimple, while $u$ is topologically unipotent.

Since $u_n$ is unipotent, we conclude that $u\in\I^+$, thus $u$ is topologically unipotent. Next,
the kernel of each projection $H_{n+1}\to H_n$ is unipotent, while each algebraic group $\ov{\lan s_n\ran}\subseteq H_n$
is multiplicative. Therefore each projection   $\ov{\lan s_{n+1}\ran}\to\ov{\lan s_n\ran}$ is an isomorphism.
Thus the projection $\ov{\lan s\ran}\to\ov{\lan s_n\ran}$ is an isomorphism, hence $s\in\I\subseteq LG$ is
strongly semisimple by \rl{tjd}(c).
\end{proof}

\begin{Lem} \label{L:topjd}
Let $G$ be a split connected reductive group over $K$,  and
let $s\in LG(k)=G(K)$ be a strongly semisimple element. Then

\smallskip

(a) the centralizer $G_s^0$ is reductive group split over $K$;

\smallskip

(b) for every Iwahori subgroup $\I\subseteq LG$ containing $s$, the intersection $\I\cap L(G_s^0)$ is an Iwahori subgroup of $L(G_s^0)$.
Conversely, every Ihawori subgroup of  $L(G_s^0)$ is obtained in which way.
\end{Lem}

\begin{proof}
Though both assertions are well-known to specialists, we provide the proof for completeness. Since $G$ is split over $K$, we can assume that $G$ is defined over $k$.

\smallskip

(a) Note first that the statement is equivalent to the assertion that $s$ is $G(K)$-conjugate to some element $s_0\in G(k)$. Indeed, assume that
$G_s^0$ is split over $K$, and let $T\subseteq G^0_s$ be a maximal split torus over $K$. Since $G$ is defined over $k$, there exists a split torus $T_0\subseteq G$ over $k$. Hence there exists $g\in G(K)$ such that $gTg^{-1}=T_0$. Then $s_0:=gsg^{-1}\in T_0(K)$ is a strongly semisimple element, thus $s_0\in T_0(k)\subseteq G(k)$. The converse assertion is clear.

Now let us prove the assertion. Replacing $G$ by $G^{\ad}$ and $s$ by its image under the projection $G\to G^{\ad}$,
we can assume that $G$ is adjoint. Next, replacing $G$ by $G^{\sc}$ and $s$ by its preimage under the projection $G^{\sc}\to G$, we can assume that $G$ is simply connected, thus $G_s$ is connected.

Moreover, since $G_s$ split over a finite Galois extension $K'$ of $K$, the first paragraph of the proof implies that $s$ is $G(K')$-conjugate to some element $s_0\in G(k)$. We claim that $s$ is $G(K)$-conjugate to $s_0$. It suffices to show that $H^1(K,G_{s_0})=1$. Since $G_{s_0}$ is a connected reductive group and field $K$ is of dimension $\leq 1$, the assertion follows from \cite[$\S$ 8.6, page 484]{BS}.

\smallskip

(b) Note that since all Ihawori subgroups of $L(G_s^0)$ are $L(G_s^0)$-conjugate, the second assertion follows from the first.

Fix a (split) maximal torus $T_0\subseteq G$ over $k$. The assertion is clear in the case when $s\in T_0\subseteq L^+(T_0)\subseteq \I$.
Since all pairs $(T,\I)$, where $T\subseteq G$ is a maximal torus split over $K$ and $\I\supseteq L^+(T)$ is an Ihawori subgroup, are
$LG$-conjugate, it suffices to show that there exists a maximal $K$-split torus $T\subseteq G$ such that
$s\in L^+(T)\subseteq \I$.

We claim that there exists a subtorus $T'\subseteq\I$ such that $s\in T'$ and the composition $T'\hra\I\to \I/\I^+$ is an isomorphism.
After this is done, the centralizer $T:=Z_G(T')\subseteq G$ is the maximal $K$-split torus we were looking for.

To construct $T'$, we choose a presentation $\I\simeq\lim_n H_n$ as in the proof of \rl{topjd2}, and let
$s_n\in H_n$ be the projection of $s$. Then $s_n\in H_n$ is semisimple, so there exists a maximal torus $T'_n\subseteq H_n$ such that
$s_n\in T'_n$. Then the composition $T'_n\hra H_n=\I/\I_n\to \I/\I^+$ is automatically an isomorphism. Moreover, we can assume that
the projection $\pi_n:H_{n+1}\to H_n$ induces an isomorphism $T'_{n+1}\isom T'_n$. Indeed, given $T'_n\subseteq H_n$, we choose
$T'_{n+1}$ be any maximal torus of $\pi_n^{-1}(T_n)\subseteq H_{n+1}$ containing $s_{n+1}$. Then $T':=\lim_{n}T'_n\subseteq\I$ be the
torus we were looking for.
\end{proof}

\begin{Emp} \label{E:notawgp}
{\bf Notation.}

\smallskip

(a) Let $s\in G(K)$ be a strongly semisimple element, thus the centralizer $G_s^0$ is reductive group, split over $K$
(by \rl{topjd}(a)). Let $T\subseteq G_s^0$ be a maximal torus split over $K$. Then $s\in T(K)=LT(k)$, thus $s\in L^+(T)$ because $s$ is bounded.

\smallskip

(b) Let $\I\subseteq LG$ be an Iwahori subgroup such that $L^+(T)\subseteq \I$. Thus $s\in \I$, hence $\I_{s}:=\I\cap {L}(G_s^0)$ is
an Iwahori subgroup of ${L}(G_s^0)$ (by \rl{topjd}(b)).

\smallskip

(c) Let $\wt{W}\subseteq\wt{W}_G$  be the affine Weyl group and the extended affine Weyl group of $G$, respectively
(see Section~\re{affweyl}), and let $\ov{\Phi}=\ov{\Phi}_{G,T},\wt{\Phi}=\wt{\Phi}_{G,T}, \wt{\Phi}^+=\wt{\Phi}^+_{G,T,\I}$ and $\wt{\Dt}$ be the set of all roots, affine roots, positive affine roots and simple affine roots of $G$, respectively (see Section~\re{affine roots}).

\smallskip

(d) Let $\ov{\Phi}_s=\ov{\Phi}_{G_s^0,T},\wt{\Phi}_s=\wt{\Phi}_{G_s^0,T}, \wt{\Phi}_s^+=\wt{\Phi}_{G_s^0,T,\I_s}$ and $\wt{\Dt}_s$ be the set of all roots, affine roots, positive affine roots and simple affine roots of $G_s^0$, respectively. By construction, we have inclusions
$\ov{\Phi}_s\subseteq\ov{\Phi},\wt{\Phi}_s\subseteq\wt{\Phi}$ and $\wt{\Phi}_s^+\subseteq\wt{\Phi}^+$.

\smallskip

(e) Let $\C{A}_{G,T}$ be the apartment, corresponding to $(G,T)$. Then every $\al\in\wt{\Phi}$ defines an affine function
on $\C{A}_{G,T}$, and the Iwahori subgroup $\I$ defines an alcove $A\subseteq\C{A}_{G,T}$. Explicitly, $A$ is the set of all
$x\in\C{A}_{G,T}$ such that $\lan\al,x\ran>0$ for all $\al\in\wt{\Dt}$.

\smallskip

(f) Let $A_s\subseteq \C{A}_{G,T}$ be the alcove, corresponding to $\I_s$. Explicitly, $A_s$ is the set of all
$x\in\C{A}_{G,T}$ such that $\lan\beta,x\ran>0$ for all $\beta\in\wt{\Dt}_s$.

\smallskip

(g) Let $\wt{W}_s\subseteq\wt{W}_G$ be the affine Weil group of $G_s^0$. Explicitly, $\wt{W}_s$ is the subgroup, generated by simple affine reflections $\{s_{\al}\}_{\al\in\wt{\Dt}_s}$. Finally, we set
\[
\wt{W}_G^s:=\{w\in\wt{W}_G\,|\,w(\al)\in\wt{\Phi}^+ \text{ for every }\al\in\wt{\Dt}_s\}.
\]

%In particular, $\Phi$ is a subset of $X^*(\ov{T}_G)=\Hom(\La_G,\B{Z})$, thus a subset of a linear functionals on
%$\La_{G,\B{R}}:=\La_G\otimes_{\B{Z}}\B{R}$, while $\wt{\Phi}$ a subset of affine functions on $\La_{G,\B{R}}$.

%(c) We set $G_s^{\sc}:=(G_s^0)^{\sc}$ and $\wt{W}_s:=\wt{W}_{G_s^{\sc}}$. Notice that we have natural embeddings
%$\wt{W}_s\hra \wt{W}_{G_s^0}\hra\wt{W}_G$.

%(d) Let  $\wt{W}_G^s:=\{w\in\wt{W}_G\,|\,w(\al)\in\wt{\Phi}^+ \text{ for every }\al\in\wt{\Dt}_s\}$.

%(e) Let $A\subseteq A_s\subseteq \La_{G,\B{R}}$ be the fundamental alcove of $\wt{\Phi}$ and $\wt{\Phi}_s$, respectively, that is, $A$ (resp. $A_s$) is the set of $x\in \La_{G,\B{R}}$ such that $0<\lan\al,x\ran<1$ for every $\al$ in $\Phi^+$ (resp. $\Phi^+_s$).
\end{Emp}

\begin{Lem} \label{L:awgp}
(a) For every $u\in\Om_G$ and $w\in \wt{W}^s_G$, we have $uw\in\wt{W}^s_G$.

\smallskip

(b) The subset $\wt{W}^s_G\subseteq \wt{W}_G$ is a set of representatives of the
quotient $\wt{W}_G/\wt{W}_{s}$.

\smallskip

(c) For every $\al\in\wt{\Dt}_s$, there exists an element
$w\in \wt{W}^s_G$ such that $w(\al)\in\wt{\Dt}$.

\smallskip

(d) For every reduced decomposition $w=s_ns_{n-1}\ldots s_2s_1$
of $w\in \wt{W}^s_G\cap \wt{W}$, we have $s_i\ldots s_2s_1\in  \wt{W}^s_G$ for
each $i=1,\ldots,n$.

\smallskip

(e) For $w\in \wt{W}^s_G$ and $\al\in\wt{\Dt}$, the following conditions are equivalent:

\quad\quad (i) $s_{\al}w\in\wt{W}^s_G$; (ii) $w^{-1}(\al)\notin \wt{\Dt}_s$; (iii) $w^{-1}(\al)\notin \wt{\Phi}^+_s$; (iv) $w^{-1}(\al)\notin \wt{\Phi}_s$.
\end{Lem}
\begin{proof}
(a) Since for every $u\in \Om_G$ we have $u(\wt{\Phi}^+)=\wt{\Phi}^+$, the assertion follows.

\smallskip

(b) We have to show that for every $w\in\wt{W}_G$ there exists a unique $w_s\in \wt{W}_s$ such that $ww_s\in\wt{W}^s_G$.
Choose a point $x\in A$. Notice that $w\in\wt{W}^s_G$ if and only if for every $\al\in\wt{\Dt}_s$ we have
$\lan\al,w^{-1}(x)\ran=\lan w(\al),x\ran >0$. In other words, $w\in\wt{W}^s_G$ if and only if $w^{-1}(x)\in A_s$.

Since the affine Weyl group acts on the set of alcoves simply transitively, there exists a unique
$w_s\in \wt{W}_s$ such that $w^{-1}(x)\in w_s(A_s)$. The last condition is equivalent to the condition that
$(ww_s)^{-1}(x)= w_s^{-1}w^{-1}(x)\in A_s$, that is, $w w_s\in\wt{W}^s_G$.

\smallskip

(c) First we claim that there exists a point $x\in A_s$ such that $\lan\al,x\ran<|\lan\beta,x\ran|$ for every
$\beta\in\wt{\Phi}\sm\pm\al$. Namely, choose a point $y$ such that $\lan\al,y\ran=0$, $\lan\beta,y\ran> 0$ for all $\beta\in\wt{\Dt}_s\sm\al$
and $\lan\beta,y\ran\neq 0$ for all $\beta\in\wt{\Phi}\sm\pm\al$. Then every point $x$, which is close enough to $y$ and such that
$\lan\al,x\ran>0$, satisfies the required property.

Choose $w\in \wt{W}_G$ such that $w(x)\in A$.
We claim that $w\in\wt{W}^s_G$ and $w(\al)\in \wt{\Dt}$. Indeed, for every $\beta\in\wt{\Phi}$ we have
$\lan w(\beta),w(x)\ran=\lan\beta,x\ran$. Since $x\in A_s$, this implies that $w(\beta)\in\wt{\Phi}^+$ for every $\beta\in\wt{\Dt}_s$, hence $w(\al)\in\wt{\Phi}^+$ and
$w\in\wt{W}^s_G$. Moreover, by our assumption on $x$, for every $\beta\in\wt{\Phi}\sm\pm w(\al)$ we have $\lan w(\al),w(x)\ran<|\lan\beta,w(x)\ran|$, which implies that $w(\al)$ has to be in $\wt{\Dt}_s$.

\smallskip

(d) Assume that $s_i\ldots s_1\notin \wt{W}^s_G$. Then  $s_i\ldots s_1(\beta)\notin\wt{\Phi}^+$ for some
$\beta\in\wt{\Dt}_s\subseteq\wt{\Phi}^+$ hence $l(s_i\ldots s_1s_{\beta})<l(s_i\ldots s_1)=i$.
Since $w=s_n\ldots s_1$ is a reduced decomposition, we conclude that
$l(ws_{\beta})<n=l(w)$, thus $w(\beta)\notin\wt{\Phi}^+$, a contradiction.

\smallskip

(e) By definition, $s_{\al}w\in\wt{W}^s_G$ if and only if $s_{\al} w(\beta)\in\wt{\Phi}^+$ for
all $\beta\in\wt{\Dt}_s$. Since $w\in\wt{W}^s_G$ we have $w(\beta)\in\wt{\Phi}^+$.
Since $\al$ is simple, it is the only positive root which $s_{\al}$ sends to
negative. Therefore we conclude that $s_{\al}w\in\wt{W}^s_G$ if and only if
$w(\beta)\neq\al$ for all $\beta\in\wt{\Dt}_s$, that is, $w^{-1}(\al)\notin\wt{\Dt}_s$.

Similarly, $s_{\al}w\in\wt{W}^s_G$ if and only if $s_{\al} w(\beta)>0$ for
all $\beta\in\wt{\Phi}^+_s$. Therefore the same argument shows that this happens
if and only if  $w^{-1}(\al)\notin\wt{\Phi}^+_s$. Finally, for every $w\in \wt{W}^s_G$ we get
$w(\wt{\Phi}_s^+)\subseteq\wt{\Phi}^+$, thus $-\al\notin w(\wt{\Phi}_s^+)$, hence
$w^{-1}(\al)\notin -\wt{\Phi}^+_s$. Thus,  $w^{-1}(\al)\notin\wt{\Phi}^+_s$ if and only if  $w^{-1}(\al)\notin\wt{\Phi}_s$,
and the proof is complete.
\end{proof}

\subsection{Application to affine Springer fibers}.
\begin{Emp} \label{E:pass cent}
{\bf Set-up.}

\smallskip

(a) Let $\gm\in G(K)$ be a bounded semisimple element, and $s\in G(K)$
a strongly semisimple element such that $\gm\in G_s^0(K)$.

\smallskip

(b) We set $G^{\sc}_s:=(G_s^0)^{\sc}$, and let $\pi$ be the
projection $G^{\sc}_s\to G^0_s\subseteq G$.

\smallskip

(c) Fix a maximal split torus $T\subseteq G$ over $K$, and Iwahori subgroups $\I\subseteq LG$ and $\I_{s}\subseteq{L}(G_s^0)$
as in Section~\re{notawgp}(a),(b).
Then the preimage $\I^{\sc}_s\subseteq{L}(G_s^{\sc})$ of $\I_s$ is an Iwahori subgroup as well.

\smallskip

(d) For every $w\in\wt{W}_G$ choose representative $n_w\in N_{LG}(LT)$.

\smallskip

(e) Let $\Fl_{G^{\sc}_s,\gm}\subseteq \Fl_{G^{\sc}_s}$ be the affine Springer fiber. Explicitly, it is the ind-scheme of all $g\in {L}(G^{\sc}_s)/\I_s^{\sc}$ such that $g^{-1}\gm g\in \I_s$.

\smallskip

(f) Note that our choice of $T$ and $\I$ induces an isomorphism $\zeta:\ov{T}_{G_s^0}\isom\ov{T}_G$. Then for every $\ov{w}\in \ov{W}$, we have an isomorphism $\zeta_{\ov{w}}:=\ov{w}\circ \zeta:\ov{T}_{G_s^0}\isom\ov{T}_G$.
\end{Emp}

\begin{Emp} \label{E:example}
{\bf Example.} Let $\gm\in G(K)$ be a bounded semisimple element with the topological Jordan decomposition
$\gm=su=us$.  Then $\gm\in G_s^0(K)$ (see \rl{tjd}(b)), thus we are in the situation of Section~\re{pass cent}.
\end{Emp}

\begin{Lem} \label{L:sprtjd}
In the situation of Section~\re{pass cent},

\smallskip

(a) for every element $w\in\wt{W}^s_G$, the map
$g\I^{\sc}_s\mapsto \pi(g) n_w^{-1}\I$ defines a morphism $\eta_w: \Fl_{G^{\sc}_s,\gm}\to\Fl_{G,\gm}$, which is  independent of the choice of $n_w$;

\smallskip

(b) if either $s=1$ or $\gm=su$ is the topological Jordan decomposition, then the induced map
$\wt{\eta}:=\bigsqcup_{w\in\wt{W}^s_G} \eta_w:\bigsqcup_{w\in\wt{W}^s_G} \Fl_{G^{\sc}_s,\gm}\to\Fl_{G,\gm}$
is a universal homeomorphism;

\smallskip

(c) for every $w\in \wt{W}^s_G$ and $u\in\Om_G$, we have $uw\in\wt{W}^s_G$, and morphism $\eta_{u w}$ decomposes as
$\Fl_{G^{\sc}_s,\gm}\overset{\eta_w}{\lra}\Fl_{G,\gm}\overset{u}{\lra}\Fl_{G,\gm}$;

\smallskip

(d) for $w\in \wt{W}^s_G$ with projection $\ov{w}\in\ov{W}$, the following diagram is commutative:
\[
\CD
\Fl_{G_s^{\sc},\gm}  @>\eta_w>> \Fl_{G,\gm}\\
@V\red_{\gm}VV       @VV\red_{\gm}V\\
\ov{T}_{G_s^0}  @>{\zeta_{\ov{w}}}>> \ov{T}_G.
\endCD
\]

\end{Lem}

\begin{proof}
(a) Since $w\in\wt{W}^s_G$, we have $w(\wt{\Phi}^+_s)\subseteq\wt{\Phi}^+$, thus
\begin{equation} \label{Eq:basic incl}
n_w \pi(\I^{\sc}_s) n_w^{-1}\subseteq n_w \I_s n_w^{-1}\subseteq\I.
\end{equation}
This inclusion implies that the map
$g\I^{\sc}_s\mapsto \pi(g) n_w^{-1}\I$ gives a well defined morphism $\eta_w:\Fl_{G^{\sc}_s}\to \Fl_{G}$. Moreover, if $g^{-1}\gm g\in \I_s$, then
$n_w(g^{-1}\gm g)n_w^{-1}\in\I$, hence $\eta_w$ restricts to a morphism
$\eta_w:\Fl_{\G^{\sc}_s,\gm}\to \Fl_{G,\gm}$. Finally, morphism $\eta_w$ is independent of the choice of $n_w$, because $\I\supseteq{L}^+(T)$.

\smallskip

(b) Since both $\Fl_{G^{\sc}_s,\gm}$ and $\Fl_{G,\gm}$ are ind-proper ind-schemes, it suffices to show that morphism
$\wt{\eta}$ induces a bijection on $K$-valued points for every algebraically closed field extension
$K\supseteq k$.

\smallskip

To show that $\wt{\eta}$ is injective, assume that  $\eta_w([g])=\eta_{w'}([g'])$ for some elements $g,g'\in LG^{\sc}_s$. By definition, this means that $[\pi(g)n_w^{-1}]=[\pi(g')n_{w'}^{-1}]$, or, equivalently, that
$n_w\pi(g^{-1}g')n^{-1}_{w'}\in\I$. Let $w''\in\wt{W}_s$ be such that $g^{-1}g'\in \I_s^{\sc}w''\I_{s}^{\sc}$. Using inclusion \form{basic incl}, we conclude that $n_w n_{w''}n^{-1}_{w'}\in\I$, thus $ww''w'^{-1}=1$, hence $w'=ww''$. Since $w,w'\in\wt{W}^s_G$ and
$w''\in  \wt{W}_s$, we conclude from  \rl{awgp}(b) that $w=w'$ and $w''=1$.
This also implies that $[g]=[g']\in \Fl_{G^{\sc}_s}$, thus $\wt{\eta}$ is injective.

\smallskip

To show that $\wt{\eta}$ is surjective, choose $[h]\in\Fl_{G,\gm}$ and a representative $h\in G(K)$ of $[h]$. Then $h^{-1}\gm h\in\I$, hence
$\gm\in h\I h^{-1}\cap L(G_s^0)$. It now follows from \rl{topjd2} that $s\in h\I h^{-1}$. Hence, by \rl{topjd}(b), the intersection $\I'_s:=h\I h^{-1}\cap L(G_s^0)$ is an Iwahori subgroup of $L(G_s^0)$, thus there exists $g\in G_s^{\sc}(K)$ such that $\I'_s=g\I_sg^{-1}$. Then $g^{-1}\gm g\in\I_s$, which means that $[g]\in\Fl_{G_s^{\sc},\gm}$.
We claim that $[h]=\eta_w([g])$ for some $w\in\wt{W}_G^w$.

\smallskip

Consider element $h_1:=h^{-1}\pi(g)\in G(K)$. Since $g\I_sg^{-1}=\I'_s\subseteq h\I h^{-1}$, we conclude that $h_1\I_s h_1^{-1}\subseteq \I$,
hence $h_1 L^+(T) h_1^{-1}\subseteq\I$. Therefore there exists $h_0\in\I$ such that $h_1 L^+(T) h_1^{-1}=h_0 L^+(T) h_0^{-1}$,
hence $h^{-1}_0 h_1\in N_{LG}(LT)$. Since $h_0^{-1}h_1=(hh_0)^{-1}\pi(g)$, we can replace $h$ by $hh_0$, thus assuming that
$h_1=h^{-1}\pi(g)$ equals $n_w$ for some $w\in\wt{W}_G$. Then  $w\in\wt{W}^s_G$ because  $n_w\I_s n_w^{-1}\subseteq \I$, and  $[h]=\eta_w([g])\in\im\wt{\eta}$.

\smallskip

(c) Since $uw\in\wt{W}^s_G$ by \rl{awgp}(a), the assertion follows from the definition.

\smallskip

(d) is straightforward.
\end{proof}

\begin{Emp} \label{E:sprtjd}{\bf Notation.}

\smallskip

(a) In the situation of Section~\re{pass cent}, for every local system $\C{L}$ on $\ov{T}_G$, we denote the local system $\zeta^*(\C{L})$ on $\ov{T}_{G_s^0}$ simply by $\C{L}$. Then it follows from \rl{sprtjd}(d) that for every $w\in\wt{W}_G^s$ we have a morphism
\[
(\eta_w)_*:H_i(\Fl_{G_s^{\sc},\gm},\C{F}_{\C{L}})\to H_i(\Fl_{G,\gm},\C{F}_{\ov{w}_*\C{L}}).
\]

\smallskip

(b) We denote the isomorphism
\[
a_{w,\C{L}}:H_i(\Fl_{G,\gm},\C{F}_{\C{L}})\isom H_i(\Fl_{G,\gm},\C{F}_{\ov{w}_*\C{L}})
\]
from \rp{whaction} simply by $w_*$.
\end{Emp}

\begin{Lem} \label{L:simref}
Let $w\in \wt{W}^s_G$ and $\al\in\wt{\Dt}$.

\smallskip

(a) Assume that $s_{\al}w\in\wt{W}^s_G$ and $\gm=su$ is the topological Jordan decomposition. Then the composition
\[
H_i(\Fl_{G_s^{\sc},\gm},\C{F}_{\C{L}})\overset{(\eta_w)_*}
{\lra}H_i(\Fl_{G,\gm},\C{F}_{\ov{w}_*\C{L}})\overset{(s_{\al})_*}{\lra}H_i(\Fl_{G,\gm},\C{F}_{(\ov{s}_{\al})_*\ov{w}_*\C{L}})
\]
equals $(\eta_{s_{\al}w})_*$.

\smallskip

(b) If $w^{-1}(\al)\in\wt{\Dt}_s$, then the
following diagram is commutative:

%\begin{equation} \label{Eq:actions}
\[
\CD
 H_i(\Fl_{G_s^{\sc},\gm},\C{F}_{\C{L}})  @>(\eta_w)_*>>  H_i(\Fl_{G,\gm},\C{F}_{\ov{w}_*\C{L}})\\
@V(s_{w^{-1}(\al)})_*VV       @VV(s_{\al})_*V\\
H_i(\Fl_{G^{\sc}_s,\gm},\C{F}_{(\ov{s}_{w^{-1}(\al)})_*\C{L}})                     @>(\eta_w)_*>>
H_i(\Fl_{G,\gm},\C{F}_{(\ov{s}_{\al})_*\ov{w}_*\C{L}}).
\endCD
\]
%\end{equation}
\end{Lem}
\begin{Emp}
{\bf Remark.}
By \rl{awgp}(e), exactly one of the assumptions $s_{\al}w\in\wt{W}^s_G$ and $w^{-1}(\al)\in\wt{\Dt}_s$ is satisfied.
%For each $w\in \wt{W}^s_G$ and each simple root $\al$ be a simple
%root of $G$ we have $s_{\al}w\notin\wt{W}^s_G$ if and only if
%$w^{-1}(\al)$ is a simple root of $G_s^0$. In other words, exactly
%one if the cases (a) or (b) always satisfied. Indeed,  since
%$w\in\wt{W}^s_G$ we have $w(I\cap G^0_s)\subseteq I$, hence $w(I\cap
%G^0_s)=I\cap G^0_{w(s)}$. Therefore $s_{\al} w\in\wt{W}^s_G$ if
%and only if $s_{\al}(I\cap G^0_{w(s)})\subseteq I$. Since $\al$ is
%simple, it is the only positive root which $s_{\al}$ sends to
%negative, therefore the last condition is equivalent to the fact
%that $\al$ is not a simple root of $G^0_{w(s)}$, or equivalently,
%that $w^{-1}(\al)$ is not a simple root of $G^0_s$.
\end{Emp}

\begin{proof}
(a) We have to show that the action of $s_{\al}$ on $\im(\eta_w)_*$ is induced by the action by $n_{s_{\al}}$.
By the definition of the $\wt{W}_G$-action it remains to show that the image of
\[
\Fl_{G_s^{\sc},\gm}\overset{\eta_w}{\lra} \Fl_{G,\gm}\overset{\wt{\red}_{\gm}}{\lra}
\left[\frac{\I}{\I}\right]\to\left[\frac{B_{\al}}{B_{\al}}\right]
%\ov{T}_G\overset{\al}{\lra}\B{G}_m %\left[\frac{\I}{\I}\right]\to\left[\frac{B_{\al}}{B_{\al}}\right]
\]
lies in the  regular semisimple locus. Explicitly, we have to show that for every representative $g\in G^{\sc}_s(K)$ of $[g]\in \Fl_{G_s^{\sc},\gm}$, we have $\al(\pr_{\I}(n_w g^{-1}\gm gn_w^{-1}))\neq 1$, where $\pr_{\I}:\I\to\ov{T}_G$ is the projection.
Notice that
\[
\pr_{\I}(n_wg^{-1}\gm gn_w^{-1})=\pr_{\I}(n_w g^{-1}s gn_w^{-1})=\pr_{\I}(n_w s n_w^{-1}),
\]
and that element $n_w s n_w^{-1}\in T(K)$ is strongly semisimple. Hence it remains to show that $\al(n_w s n_w^{-1})\neq 1$. Since $\al(n_w s n_w^{-1})=w^{-1}(\al)(s)$, we have to show that $w^{-1}(\al)\notin\wt{\Phi}_s$, but this follows from
\rl{awgp}(e).

\smallskip

(b) Set $\beta:=w^{-1}(\al)$. The the assertion follows from commutativity of the diagram
\[
\CD
\Fl_{G_s^{\sc},\gm}@>\wt{\red}_{\gm}>>[\frac{\I_s^{\sc}}{\I^{\sc}_s}]  @>>> [\frac{B_{\beta}}{B_{\beta}}] \\
 @V\eta_wVV   @. @VV\Ad n_w^{-1}V\\
\Fl_{G,\gm} @>\wt{\red}_{\gm}>>[\frac{\I}{\I}]  @>>> [\frac{B_{\al}}{B_{\al}}]
\endCD
\]
and the definion of the $\wt{W}_s$- and $\wt{W}_G$-actions.
\end{proof}

\begin{Cor} \label{C:equiv}
Assume that either $s=1$ or $\gm=su$ is the topological Jordan decomposition. Then

\smallskip

(a) for every  $w\in \wt{W}^s_G$, the composition
\[
H_i(\Fl_{G_s^{sc},\gm},\C{F}_{\C{L}})\overset{\eta_*}{\lra}
H_i(\Fl_{G,\gm},\C{F}_{\C{L}})\overset{w_*}{\lra}H_i(\Fl_{G,\gm},\C{F}_{\ov{w}_*\C{L}})
\]
equals $(\eta_w)_*$;

\smallskip

(b) For every $w\in \wt{W}_s$ the following diagram is commutative:
 \[
\CD
 H_i(\Fl_{G_s^{\sc},\gm},\C{F}_{\C{L}})  @>\eta_*>>  H_i(\Fl_{G,\gm},\C{F}_{\C{L}})\\
@Vw_*VV       @VV w_* V\\
 H_i(\Fl_{G_s^{\sc},\gm},\C{F}_{\ov{w}_*\C{L}})  @>\eta_*>>  H_i(\Fl_{G,\gm},\C{F}_{\ov{w}_*\C{L}}).
\endCD
\]

\smallskip

(c) The map $\eta_*: H_i(\Fl_{G_s^{sc},\gm},\C{F}_{\C{L}}) \to
H_i(\Fl_{G,\gm},\C{F}_{\C{L}})$ is  $(\La_{G_s^{\sc}}\times LG_{\gm})$-equivariant.
Moreover, it is the $(\wt{W}_s\times LG_{\gm})$-equivariant, if $\C{L}$ is $\ov{W}_{G_s^0}$-equivariant.
\end{Cor}

\begin{proof}
(a) Assume first that $w$ belongs to $\wt{W}$. In
this case, we will show the assertion by induction on $l(w)$. If
$l(w)=0$, thus $w=1$, there is nothing to prove. Assume now that
$l(w)>0$. In this case, $s\neq 1$, thus $\gm=su$ is the topological Jordan decomposition.

Choose a reduced decomposition $w=s_n\ldots s_1$, and
set $w':=s_{n-1}\ldots s_1$. By \rl{awgp}(d), we have $w'\in\wt{W}^s_G\cap \wt{W}$, and by
assumption $w=s_n w'\in \wt{W}^s_G$. Therefore by the induction
hypothesis, $w_*\circ\eta_*=(s_n)_*\circ (w'_*\circ\eta_*)$ equals
$(s_n)_*\circ(\eta_{w'})_*$, while by \rl{simref}(a), the
latter expression equals $(\eta_{s_nw'})_*=(\eta_w)_*$.

\smallskip

Let now $w\in \wt{W}^s_G$ be arbitrary. Write $w$ in the form
$w=uw'$, where $u\in\Om_G$ and $w'\in \wt{W}$. Then $w'=u^{-1}w$
belongs to $\wt{W}^s_G$ (by \rl{awgp}(a)), therefore the
composition $w_*\circ\eta_*=u_*\circ w'_*\circ\eta_*$ equals
$u_*\circ(\eta_{w'})_*$ by the case proven above, hence to
$(\eta_{uw'})_*=(\eta_w)_*$ by \rl{sprtjd}(c).

\smallskip

(b) We have to show that $\eta_*$ commutes
with $(s_{\beta})_*$ for every simple affine root $\beta\in\wt{\Dt}_s$.
By \rl{awgp}(c), there exists $w\in\wt{W}^s_G$ such that
$\al:=w(\beta)\in\wt{\Dt}$. Hence, by
\rl{simref}(b), we have an equality
$(s_{\al})_*\circ(\eta_w)_*=(\eta_w)_*\circ (s_{\beta})_*$, which by (a)
implies that
$(s_{\al})_*\circ (w_*\circ\eta_{*})=(w_*\circ \eta_{*})\circ (s_{\beta})_*$,
or, equivalently, that
$(w^{-1}s_{\al}w)_*\circ\eta_*=\eta_{*}\circ (s_{\beta})_*$. Since
$w^{-1}s_{\al}w=s_{w^{-1}(\al)}=s_{\beta}$, this implies that
$(s_{\beta})_*\circ\eta_*=\eta_*\circ (s_{\beta})_*$, and the proof is complete.

\smallskip

(c)  Since $\eta$ is $LG_{\gm}$-equivariant, both assertions follow from (b).
\end{proof}

\begin{Prop} \label{P:ind}
Let $\gm\in G(K)$ be a bounded semisimple element, and we have either $s=1$ or $\gm=su$ is the
topological Jordan decomposition of $\gm$.
Then for every $\ov{W}$-equivariant local system $\C{L}$ on $\ov{T}_G$, there
is a natural isomorphism
\begin{equation*} \label{Eq:induction}
H_i(\Fl_{G,\gm},\C{F}_{\C{L}})\simeq
\ind_{\wt{W}_{s}}^{\wt{W}_G}H_i(\Fl_{G_s^{\sc},\gm},\C{F}_{\C{L}})
\end{equation*}
of $(\wt{W}_G\times LG_{\gm})$-representations.
\end{Prop}

\begin{proof}
By \rco{equiv}(c), the map $\eta_*$ is $(\wt{W}_{s}\times
LG_{\gm})$-equivariant, therefore by the Frobenius reciprocity it
induces a unique $(\wt{W}_G\times LG_{\gm})$-equivariant map
\[
\wt{\eta}_*:\ind_{\wt{W}_{s}}^{\wt{W}_G}H_i(\Fl_{G_s^{sc},\gm},\C{F}_{\C{L}})\to
H_i(\Fl_{G,\gm},\C{F}_{\C{L}}).
\]
It remains to show that $\wt{\eta}_*$ is an isomorphism of vector
spaces.

\smallskip

Since $\wt{W}^s_G$ is a set of representatives of the quotient
$\wt{W}_G/\wt{W}_{s}$ (by \rl{awgp}(b)), we see that the underlining vector
space of $\ind^{\wt{W}_G}_{\wt{W}_s}H_i(\Fl_{G_s^{\sc},\gm},\C{F}_{\C{L}})$
decomposes as a direct sum $\bigoplus_{w\in\wt{W}^s_G}\{w\}\otimes
H_i(\Fl_{G_s^{sc},\gm},\C{F}_{\C{L}})$. Moreover, by the $\wt{W}_G$-equivariance  of
$\wt{\eta}_*$, we have $\wt{\eta}_*(\{w\}\otimes
x)=w_*(\eta_*(x))$ for each $w\in\wt{W}^s_G$ and $x\in
H_*(\Fl_{G_s^{sc},\gm},\C{F}_{\C{L}})$. Therefore by \rco{equiv}(a),
we get   $\wt{\eta}_*(\{w\}\otimes x)=(\eta_w)_*(x)$.

\smallskip

In other words, as the map of vector spaces $\wt{\eta}_*$ can be identified with
\[
(\bigsqcup_{w\in \wt{W}^s_G}\eta_w)_*:H_i(\bigsqcup_{w\in\wt{W}^s_G}\Fl_{G_s^{sc},\gm},\C{F}_{\C{L}})\to H_*(\Fl_{G,\gm},\C{F}_{\C{L}}),
\]
hence it is an isomorphism by \rl{sprtjd}(b).
\end{proof}

\begin{Cor} \label{C:ind}
Let $\gm\in G(K)$ be a bounded semisimple element with the topological Jordan decomposition $\gm=su$.
Then for every $\ov{W}$-equivariant local system $\C{L}$ on $\ov{T}_G$, there
is a natural isomorphism
\[
H_i(\Fl_{G,\gm},\C{F}_{\C{L}})\simeq\ind_{\wt{W}_{G^0_s}}^{\wt{W}_G}H_i(\Fl_{G_s^0,\gm},\C{F}_{\C{L}})
\]
of $(\wt{W}_G\times LG_{\gm})$-representations.
\end{Cor}

\begin{proof}
Applying \rp{ind} to the groups $G_s^0$ and $G$, we conclude that
\[
\ind_{\wt{W}_{G_s^0}}^{\wt{W}_G}H_i(\Fl_{G_s^0,\gm},\C{F}_{\C{L}})\simeq \ind_{\wt{W}_{s}}^{\wt{W}_G}H_i(\Fl_{G_s^{\sc},\gm},\C{F}_{\C{L}})\simeq
H_i(\Fl_{G,\gm},\C{F}_{\C{L}}).
\]
\end{proof}

\begin{Emp} \label{E:pflind}
\begin{proof}[Proof of \rl{ind}]
Note that part~(a) follows from \rco{equiv}(b) for $s=1$, while part~(c) follows from part~(a) and \rp{ind} for $s=1$.
Next, the first assertion of part~(b) follows from (a), while the second assertion follows from part~(c). Namely,
the local system $\C{L}^{\st}:=\bigoplus_{\ov{w}\in\ov{W}}\ov{w}_*\C{L}$ is naturally $\ov{W}$-equivariant, and the
isomorphism
\[
\ind_{\wt{W}}^{\wt{W}_G}H_i(\Fl_{G_s^{\sc},\gm},\C{F}_{\C{L}^{\st}})\simeq H_i(\Fl_{G,\gm},\C{F}_{\C{L}^{\st}})
\]
of $(\wt{W}_G\times LG_{\gm})$-representations from part~(c) restricts to an isomorphism
\[
\ind_{\La}^{\La_G}H_i(\Fl_{G_s^{sc},\gm},\C{F}_{\C{L}})\simeq
H_i(\Fl_{G,\gm},\C{F}_{\C{L}})
\]
of $(\La_G\times LG_{\gm})$-representations.
\end{proof}
\end{Emp}

\subsection{Completion of the proof}

\begin{Lem} \label{L:central}
Assume that $\gm'=z\gm$ and $z\in Z(G)(\C{O})$. Then \rt{action} for $\gm$ implies that for $\gm'$.
\end{Lem}
\begin{proof}
Denote by $m_z:\ov{T}_G\to\ov{T}_G$ the morphism induced by the multiplication by $\ov{z}\in Z(G)(k)$.  Our assumptions imply that we have equalities $\Fl_{G,\gm'}=\Fl_{G,\gm}$ and
$LG^0_{\gm}=LG^0_{\gm'}$, while the projections $\red_{\gm}:\Fl_{G,\gm}\to\ov{T}_G$ and $\red_{\gm'}:\Fl_{G,\gm'}\to\ov{T}_G$ are connected by the rule $\red_{\gm'}=m_z\circ\red_{\gm}$. Therefore we have a canonical $(\La_G\times LG^0_{\gm})$-equivariant isomorphism
$H_i(\Fl_{G,\gm'},\C{F}_{\C{L}})\simeq H_i(\Fl_{G,\gm},\C{F}_{m_z^*(\C{L})})$, which is $(\wt{W}_G\times LG^0_{\gm})$-equivariant, if
$\C{L}$ is $\ov{W}$-equivariant. This implies the assertion.
\end{proof}

\begin{Lem} \label{L:qlog}
Set $\fg:=\Lie G$. Then there exists an ${L}(G^{\ad})$-equivariant isomorphism $\Phi:LG_{\on{tu}}\isom {L}\fg_{\on{tn}}$ between locus of topologically unipotent elements of  $LG$ and topologically nilpotent elements of  ${L}\fg$, and inducing an isomorphism $\I^+\isom \Lie(\I^+)$.
\end{Lem}

\begin{proof}
Notice first that if $\pi:G'\to G$ is an isogeny, whose kernel is finite of order prime to the characteristic of $k$, then $\pi$ induces isomorphisms
$LG'_{\on{tu}}\isom LG_{\on{tu}}$ and $LG'_{\on{tn}}\isom LG_{\on{tn}}$ (compare \cite[Lemma~1.8.17]{KV}), thus the assertion for $G$ follows from from that for $G'$.

Now set $G':=G^{\sc}\times Z(G)^0$, and let $\pi:G'\to G$ be the canonical projection. Then the above observation together with
our assumption (see Section~\re{convention}(c)) on the characteristic of $k$ implies that it suffice to show the assertion for $G'$ instead for $G$. Thus it suffices to show the assertion for
$G=G^{\sc}$ and  $G=\B{G}_m$.

In both cases we know that there exists a quasi-logarithm $\Phi:G\to \fg$ in the sense of \cite[Section~1.8]{KV}, that is, an $G^{\ad}$-equivariant morphism, such that $\Phi(1)=0$,
and $d\Phi_1:\fg\to\fg$ is the identity (see \cite[Lemma~1.1.6]{KV}). Then, as in \cite[Proposition~1.8.16]{KV}, the induced morphism $\Phi:LG\to{L}\fg$ induces isomorphisms
$\Phi:LG_{\on{tu}}\isom {L}\fg_{\on{tn}}$ and  $\I^+\isom \Lie(\I^+)$, as claimed.
\end{proof}

\begin{Emp} \label{E:affsprlie}
{\bf The Lie algebra case.} Fix $x\in\fg(K)=L\fg(k)$, and let $G_x\subseteq G$ be the centralizer of $x$.

\smallskip

(a) We denote by $\Fl_{G,x}\subseteq\Fl_G$ the closed sub-indscheme (called the
{\em affine Springer fiber}), such that for every Iwahori subgroup $\I$, the image of $\Fl_{G,x}$ under the canonical isomorphism $\Fl_G\isom LG/\I$ (see Section~\re{affflvar}(b)) consists of all $[g]\in  LG/\I$ such that $\Ad g^{-1}(x)\in\Lie\I$.

\smallskip

(b) As in \cite{Lu2},\cite{Yun} each homology group $H_i(\Fl_{G,x}):=H_i(\Fl_{G,x},\qlbar)$ is equipped with an action of the group
$\wt{W}_G\times\pi_0(LG_x)$.

\smallskip

(c) Assume now that $x\in\fg(K)$ is regular semisimple. Then $G_x\subseteq G$ is a maximal torus, so the construction of
Section~\re{can} gives rise to algebra homomorphisms $\pr_x:=\pr_{G_x}$ and  $\pr'_x:=\pr'_{G_x}$ from $\qlbar[\La_G]^{\ov{W}}$ to $\qlbar[\pi_0(L{G}_{x})]$.
\end{Emp}

The following result, proved in \cite{Yun}, is a Lie algebra analog of \rt{action}.

\begin{Thm} \label{T:yun}
For every $i\in \B{Z}$, there exists a finite
filtration $\{F^j H_i(\Fl_{G,x})\}_j$ of $H_i(\Fl_{G,x})$, stable under
the action of $\wt{W}_G\times\pi_0(LG_{x})$,
such that the action of a subalgebra $\qlbar[\La_G]^{\ov{W}}\subseteq \qlbar[\wt{W}_G]$ on each graded piece
$\gr^j H_i(\Fl_{G,x})$ is induced from the action of $\qlbar[\pi_0(L{G}_{x})]$ via homomorphism $\pr'_{x}$.
\end{Thm}

\begin{Emp} \label{E:sign}
{\bf Remark.} Notice that statement of \rt{yun} is incompatible with that of \cite[Theorem~2]{Yun}, because the algebra homomorphism \cite[formula~(1)]{Yun} used in \cite{Yun} coincides with our homomorphism $\pr_x$ (see Section~\re{affsprlie}(c)), and thus differs from our homomorphism $\pr'_x$ by involution $\iota$.

On the other hand, as it was confirmed by Zhiwei Yun, his argument actually proves our \rt{yun} rather than \cite[Theorem~2]{Yun}. Namely, Yun reduces the assertion to its particular case\footnote{actually a global version of it}, when element $x$ has a regular semi-simple reduction. In this case, Lusztig's action on homology comes from the right geometric action of $\wt{W}_G$ on $\Fl_{G,x}$, which is ``compatible'' with the left action of $\pi_0(LG_{x})$. Therefore one has to precompose one of the actions with involution $\iota$ in order to have a compatibility of left actions on homology.
\end{Emp}

Now we are ready to prove \rt{action}.

\begin{Emp}
{\bf Proof of \rt{action}.}
 To expose the structure of the proof we divide it into steps.

\smallskip

\noindent{\bf Step 1.} If suffices to assume that $\C{L}$ is $\ov{W}$-equivariant.
\begin{proof}

The local system $\C{L}^{\st}:=\bigoplus_{\ov{w}\in\ov{W}}\ov{w}_*\C{L}$ is naturally $\ov{W}$-equivariant, and we claim that assertion for
$\C{L}$ follows from that for $\C{L}^{\st}$. Indeed, we have a $(\La_G\times LG^0_{\gm})$-equivariant decomposition
\[
H_i(\Fl_{G,\gm},\C{F}_{\C{L}^{\st}})=\bigoplus_{\ov{w}\in\ov{W}}H_i(\Fl_{G,\gm},\C{F}_{\ov{w}_*\C{L}}).
\]
Therefore the $(\wt{W}_G\times LG^0_{\gm})$-stable filtration $\{F^j  H_i(\Fl_{G,\gm},\C{F}_{\C{L}^{\st}})\}_j$ of $H_i(\Fl_{G,\gm},\C{F}_{\C{L}^{\st}})$ satisfying \rt{action} for $\C{L}^{\st}$ decomposes as a direct sum of
$(\La_G\times LG^0_{\gm})$-stable filtrations $\{F^j  H_i(\Fl_{G,\gm},\C{F}_{\ov{w}_*\C{L}})\}_j$ of the
$H_i(\Fl_{G,\gm},\C{F}_{\ov{w}_*\C{L}})$'s satisfying \rt{action} for $\ov{w}_*\C{L}$.
\end{proof}

From now on we will assume that $\C{L}$ is $\ov{W}$-equivariant.
\smallskip

\noindent{\bf Step 2.} The assertion holds, if $\gm$ is topologically unipotent and $\C{L}=\qlbar$.
\begin{proof}
Let $\Phi:LG_{\on{tu}}\isom {L}\fg_{\on{tn}}$ be an ${L}(G^{\ad})$-equivariant isomorphism
of \rl{qlog}, inducing an isomorphism $\I^+\isom \Lie(\I^+)$. By an ${L}(G^{\ad})$-equivariance, we have an equality $G_{\gm}=G_{\Phi(\gm)}\subseteq G$, thus $\Phi(\gm)\in\fg(K)$ is regular semisimple (and $G_{\gm}=G_{\gm}^0$). Similarly, the affine Springer fiber $\Fl_{G,\gm}\subseteq\Fl_G$  of $\gm\in LG$ coincides with the affine Springer fiber $\Fl_{G,\Phi(\gm)}\subseteq\Fl_G$  of $\Phi(\gm)\in {L}\fg$.
Moreover, unwinding the definitions, the identification  $H_i(\Fl_{G,\gm},\qlbar)\simeq   H_i(\Fl_{G,\Phi(\gm)},\qlbar)$ is $(\wt{W}_G\times LG^0_{\gm})$-equivariant, so the assertion of \rt{action} for $H_i(\Fl_{G,\gm},\qlbar)$ is equivalent to the assertion of \rt{yun} for   $H_i(\Fl_{G,\Phi(\gm)},\qlbar)$.
\end{proof}

\noindent{\bf Step 3.} The assertion holds, if $\gm$ is topologically unipotent.

\begin{proof}
Since $\gm$ is topologically unipotent, the map $\red_{\gm}:\Fl_{G,\gm}\to \ov{T}_G$ maps all of $\Fl_{G,\gm}$ to the
identity element. Therefore $\C{F}_{\C{L}}$ is the constant sheaf on $\Fl_{G,\gm}$ with fiber $\C{L}_1$, and $\C{L}_1$ is $\ov{W}$-equivariant (hence
$\wt{W}_G$-equivariant), so we have a
$(\wt{W}_G\times LG^0_{\gm})$-equivariant isomorphism
\[
H_i(\Fl_{G,\gm},\C{F}_{\C{L}})\simeq H_i(\Fl_{G,\gm},\qlbar)\otimes_{\qlbar}\C{L}_1.
\]
Therefore the assertion for general $\C{L}$ follows from the particular case $\C{L}=\qlbar$, shown in Step 2.
\end{proof}

\noindent{\bf Step 4.}  Let $\gm=su$ be the topological Jordan decomposition. Then the assertion holds if $s\in Z(G)$.

\begin{proof}
Since $u$ is topologically unipotent, the assertion follows from Step 3 and \rl{central}.
\end{proof}

\noindent{\bf Step 5.} The assertion holds in general.

\begin{proof}
Let $\gm=su$ be the topological Jordan decomposition. Since $G_s^0$ is a split reductive group over $K$ (by \rl{topjd}(a)), and $s\in Z(G_s^0)(K)$ (by \rl{tjd}(b)), the assertion holds for $\gm\in G^0_{s}$ (by Step 3). Since $(G_s^0)_{\gm}=G_{\gm}$ (by \rl{tjd}(b)), this means that there exists a finite filtration
$\{F^jH_i(\Fl_{G_s^0,\gm},\C{F}_{\C{L}})\}_j$ of $H_i(\Fl_{G_s^0,\gm},\C{F}_{\C{L}})$, stable under the action of  $\wt{W}_{G_s^0}\times \pi_0(LG^0_{\gm})$,  such that the action of $\qlbar[\La_G]^{\ov{W}_{G_s^0}}$ on each
graded piece $\gr^jH_i(\Fl_{G_s^0,\gm},\C{F}_{\C{L}})$ is induced from $\qlbar[\pi_0(LG^0_{\gm})]$ via the homomorphism
$\pr'_{G_s^0,\gm}:\qlbar[\La_G]^{\ov{W}_{G_s^0}}\to\qlbar[\pi_0(LG^0_{\gm})]$.

Using \rco{ind}, we obtain the finite $\wt{W}_G\times\pi_0(LG^0_{\gm})$-stable filtration
\[
\Big\{F^jH_i(\Fl_{G,\gm},\C{F}_{\C{L}})\Big\}_j:=\Big\{\ind_{\wt{W}_{G_s^0}}^{\wt{W}_G}F^jH_i(\Fl_{G_s^0,\gm},\C{F}_{\C{L}})\Big\}_j
\]
of $H_i(\Fl_{G,\gm},\C{F}_{\C{L}})$. Since projection $\pr'_{\gm}$ is the restriction of projection $\pr'_{G_s^0,\gm}$ to $\qlbar[\La_G]^{\ov{W}}\subseteq\qlbar[\La_G]^{\ov{W}_{G_s^0}}=\qlbar[\La_{G_s^0}]^{\ov{W}_{G_s^0}}$,
this filtration satisfies the required property.
\end{proof}
\end{Emp}

\pagebreak

\section*{List of main terms and symbols}

%Terms and symbols below are indexed by number number, they are defined.

\begin{multicols}{3}

%\e $\ov{B}_{\I,\P}$, \re{parahoric}

\e admissible isomorphism, \re{admisaff}

\e $\ov{\fa}$, \re{DLpadic}

\e bounded element, \re{bounded}

\e $\can_{\gm}$, \re{connected}

\e $\can'_{\gm}$, \re{dual}

\e Chevalley space $c_H$, \re{weylprop}

\e Demazure product, \ref{L:dem}

\e $\C{E}$-stable function, \re{estable}

\e $\C{E}_{S,\ka}$, $\C{E}_{\la,\ka}$, \re{exenddatum}

\e $\C{E}_{\gm,\ka}$, \re{estable}

\e elliptic element, \re{finell}

\e endoscopic datum, \re{enddatum}

\e $F$, $F^{\nr}$, $\fq$, \re{convention}

\e $\Fl$, $\Fl_{\gm}$, \re{variant}

\e $\Fl_G$, $\Fl_{G,\I}$, $\Fl_{G,J}$, $\Fl_{G,J,I}$, \re{affflvar}

\e $\Fl_{G,\gm}$, $\Fl_{G,J,\gm}$, \re{affsprfib}

\e $\Fl^{\leq n}$, $\Fl^{\leq n; J_r; J'_l; m_{\reg}}$, \re{filhom}

\e $\Fl_{\gm}^{(\leq n)}$,  \re{step4}

\e $\C{F}_{\C{L}}$, \re{affsprfib}

\e $\C{F}_{\theta,\varphi}$, \re{ltheta}

\e $\C{F}_{\theta}^{\st}$, \re{notstable}

\e $f_{\fa, \theta}$, $f_{T,\theta}$, \re{DLpadic}

\e $f_{T,\theta}^{\st}$, $f_{T,\theta}^{\ka}$, \re{fst}

\e filtration, \re{filtr}, \re{Rees}, \re{fil}

\e $G, G^{\rss}$ , \re{convention}

\e $G_{\fa}$, $G^+_{\fa}$, $\wt{G}_{\fa}$, \re{DLpadic}

\e $\wh{G}$, \re{dualgp}

\e $G^{\sc}$, \re{saffsetup}

\e $G_{\gm}$, \re{affsprfib}

\e general position, \re{DLpadic}

\e good position, \re{goodpos}

\e  $H_i(Y,\C{F})$, \re{convention}

\e $H_i(\Fl_{G,\gm},\C{F}_{\theta,\varphi})_{\ka}$, \re{notkappa}

\e $H'_i(Y_{\gm},\C{F}_{\C{L}})$, \re{truncation}, \re{filhom1}

\e $\I$, \re{saffsetup}

\e  $\I^{\sc}$, \re{variant}

\e  $\I^+$, \re{prop}

\e ind-scheme, \re{filtr}

\e $\inv(\gm,\gm')$, $\inv(\fa,\fa')$, \re{stableconj}

\e $J$-longest, \ref{N:k-reg}

\e $J_{\P}$, \re{parahoric}

\e $K$, $k$, \re{convention}

\e $\ov{L}$, $L^{\sep}$, $\ell$, \re{convention}

\e $\C{L}_{\theta}$, $\C{L}_{\theta,\varphi}$, \re{ltheta}

\e $\C{L}_{\theta}^{\st}$, \re{notstable}

\e $l_g$, \re{affsprfib}

\e $l(g)$, $l_{\I}(g)$, \re{bruhatlength}

\e $M_{\fa}$, \re{DLpadic}

\e $M_{\P}$, \re{parahoric}

\e $m$-regular, \re{mreg}

\e $\on{Orb}^{\st}_{G(F)}(\gm)$, \re{estable}

\e $\C{O}$, \re{saffsetup}

\e $\P$, \re{saffsetup}

\e $\P_{\fa}$, \re{DLpadic}

\e $\P^+$, $\P_{J,\I}$, \re{parahoric}

\e  $\pr_{\I}$, \re{prop}

\e $\pr_S$, $\pr'_S$, \re{can}

\e $\pr_{\gm}$, $\pr'_{\gm}$, \re{setup comp}

\e quasi-isogeny, \ref{D:quasiisogeny}

\e $R^{\theta}_{\fa}$, $R^{\ov{\theta}}_{\ov{\fa}}$, \re{DLpadic}

\e Rees module $R(M)$, \re{Rees}

\e Rees ring $R(A)$, \re{Rees}

\e $\red_{\gm}$, \re{affsprfib}

\e $S_{G,\gm}$, \re{cent}

\e $S_{\gm}$, \re{lengthfilt}

\e $\wt{S}$, $\wt{S}_{T,\I}$, \re{affweyl}

\e stable function, \re{stable}

\e stable orbit, \re{estable}

\e stably conjugate, \re{stableconj}, \ref{N:sitor}, \ref{N:conj}, \re{inner}

\e stably equivalent, \re{inner}

\e strongly semisimple, \re{bounded}

\e $T_{\C{O}}$, $\ov{T}$, $\T$, \re{saffsetup}

\e $\ov{T}_G$, $\ov{T}_{G,T}$ , \re{affweyl}

\e $\Tr_{\gen}$, \ref{C:gentr}

\e $T_{\fa}$, \ref{N:stconj}

\e $\T_{\fa}$, \re{DLpadic}

\e tamely ramified, \re{DLpadic}

\e topological Jordan decomposition, \ref{L:tjd}

\e topologically unipotent, \re{bounded}

\e $t_{\fa,\theta}$, \re{formula}

\e $\wt{u}_{\I}$, \ref{N:twisting}

\e $\wt{u}_{T,\varphi}$, $\wt{u}_{\fa,\varphi}$, \re{interpr}

\e $\ov{u}_{T,\varphi}$, $[\wt{u}_{T,\varphi}]$, $[\wt{u}_{\fa,\varphi}]$, \ref{N:stconj}

\e $\ov{W}$, $\ov{W}_G$, \re{prop}

\e $\wt{W}_G$, $\wt{W}$, $\wt{W}_{G,T}$, $W_J$, \re{affweyl}

\e $\wt{W}^{\leq n}$, \re{filtr0}

\e $(\wt{W}\si)_{\tor}$, \ref{N:twisting}

\e $\wt{w}_{\I}$, $\ov{w}_{\I}$, $\ov{w}_{T,\varphi}$, \re{basic pair}

\e  $\wt{w}_{\I,\I'}$, \re{prop}

%\e $\wt{W}^{\leq u}$, \re{notdem}

\e  $X_*(-)$, \re{prop}

\e $Z_0$, \re{interpr}

\e $\Gm_L$ , \re{convention}

\e $\Dt^{\leq n}$, \re{filtr0}

\e  $\eta$, $\eta_u$, \re{variant}

\e $\ov{\theta}$, \re{DLpadic}

\e $\La_0$, \re{interpr}

\e $\La_G$, $\La_{G,T}$, \re{affweyl}

\e $\La_{G,\gm}$, \re{cent}

\e $\La_{\gm}$, \re{lengthfilt}

\e $\nu_H$, $\nu_{T_H}$, \re{weylprop}

\e $\pi_{\fa,\theta}$, \re{DLpadic}

\e $\pi_J$, \re{affflvar}

\e $\pi_0(LG^{\sc}_{\gm})^{\leq n}$, \re{lengthfilt}

\e $\si$, \re{convention}

\e $\lan\si\ran_n$, \ref{N:twisting}

\e $\varphi_{\fa}$, \ref{N:stconj}

\e $\varphi_{T,\I}$, \re{prop}

\e $\Om_G$, $\Om_{G,\I}$, \re{affweyl}

\e

\e

\e

\e

\e

\e

\e

\e

\e

\e

\e

\e

\e

\e

\e

\e

\e

\e

\e

\e

\e

\e

\e

\e

\e

\e

\e

\e

\e

\e

\e

\e

\e

\e

\e

\e

\e

\e

\e

\e

\e

\e

\e

\e

\e

\e

\e

\e

%zzzzzzzz
\end{multicols}

\end{document}